\documentclass[10pt,
               cleardoublepage=empty,
               openany,
               numbers=dotatend,
               headings=big]{scrbook}

\usepackage{eco}
\usepackage{hyperref}
\usepackage[
	marginno,
	numberThmsWithEqns,
	english]
{stefandiss}

\usepackage{arydshln}
\usepackage{picins}
\usepackage[permil]{overpic}
\usepackage{import}
\usepackage{diagrams}
\usepackage{hyphenat}

\usepackage{natbib}
\renewcommand \citename[1] {\textsc{#1}}
\newcommand \etal {\textit{et al.}}

\usepackage{esint}

\usepackage{scrpage2}
\setcapindent{0pt}
\usepackage{remreset}
\makeatletter
\@removefromreset{section}{chapter}
\makeatother

\numberwithin{equation}{section}
 1
 2

\newcommand \newsectionpage{}

\begin{document}
\quad
\thispagestyle{empty}
\vspace*{3cm}
\begin{center}
\large \scshape \lowercase{
Stefan W. von Deylen\\[1ex]
Dissertation
}
\end{center}
\pagestyle{empty}
\vspace*{2em}
\addsubsec{Eine Chria, darin ich von meinen Academischen
           Leben\\und Wandel Nachricht gebe.}

\itshape
	'Bin auch auf Unverstädten gewesen, und hab' auch studirt.
	Ne, studirt hab' ich nicht, aber auf Unverstädten bin
	ich gewesen, und weiß von allem Bescheid. Ich ward von
	ohngefähr mit einigen Studenten bekannt, und die haben
	mir die ganze Unverstädt gewiesen, und mich allenthalben
	mit hingenommen, auch ins Collegium. Da sitzen die Herren
	Studenten alle neben 'nander auf Bänken wie in der Kirch',
	und am Fenster steht eine Hittsche, darauf sitzt 'n Profeßor
	oder so etwas, und führt über dies und das allerley Reden,
	und das heissen sie denn \emph{dociren}.
	Das auf der Hittschen saß, als ich d'rinn war, das war
	'n Magister, und hatt' eine große krause Parüque
	auf'm Kopf, und die Studenten sagten, daß seine
	Gelehrsamkeit noch viel größer und krauser, und er unter
	der Hand ein so capitaler Freygeist sey, als irgend
	einer in Frankreich und England. Mochte wohl was
	d'ran seyn, denn 's gieng ihm vom Maule weg
	als wenn's aus'm Mostschlauch gekommen wär, und
	demonstriren konnt' er, wie der Wind. Wenn er etwas vornahm,
	so fieng er nur so eben 'n bisgen an, und, eh' man sich umsah,
	da wars demonstrirt. So demonstrirt er z.\,Ex. daß 'n
	Student 'n Student und kein Rinoceros sey. Denn, sagte er,
	'n Student ist entweder 'n Student oder 'n Rinoceros;
	nun ist aber 'n Student kein Rinoceros, denn sonst
	müßte 'n Rinoceros auch 'n Student seyn; 'n Rinoceros
	ist aber kein Student, also ist 'n Student 'n Student.
	Man sollte denken, das verstünd sich von selbst, aber
	unser einer weiß das nicht besser. Er sagte, das Ding
	\glqq{}\hspace{.1em}daß 'n Student kein Rinoceros, sondern 'n Student wäre\grqq{}
	sey eine Hauptstütze der ganzen Philosophie, und die Magisters
	könnten den Rücken nicht fest genug gegenstemmen, daß sie
	nicht umkippe.

	Weil man auf Einem Fuß nicht gehn kann, so hat die
	Philosophie auch den andern, und darin war die Rede
	von mehr als Einem Etwas, und das Eine Etwas, sagte
	der Magister, sey für jedermann; zum Andern Etwas gehör'
	aber eine feinere Nas', und das sey nur für ihn und
	seine Collegen. Als wenn eine Spinn' einen Faden
	spinnt, da sey der Faden für jedermann und jedermann
	für den Faden, aber im Hintertheil der Spinne sey
	sein bescheiden Theil, nämlich das Andre Etwas das der
	zureichende Grund von dem Ersten Etwas ist, und einen
	solchen zureichenden Grund müß' ein jedes Etwas haben,
	doch brauche der nicht immer im Hintertheil zu seyn.
	Ich hätt' auch mit diesem Axioma, wie der Magister 's
	nannte, übel zu Fall kommen können. Daran hängt alles
	in der Welt, sagt er, und, wenn einer 's umstößt, so
	geht alles über und drunter.

	Denn kam er auf die Gelehrsamkeit, und die Gelehrten
	zu sprechen, und zog bey der Gelegenheit gegen die
	Ungelahrten loß. Alle Hagel, wie fegt' er sie!
	Dem ungelahrten Pöbel setzen sich die Vorurtheile
	von Alp, Leichtdörnern, Religion etc. wie Fliegen
	auf die Nase und stechen ihn; aber ihm, dem Magister,
	dürfe keine kommen, und käm' ihm eine, Schnaps schlüg'
	er sie mit der Klappe der Philosophie sich auf der
	Nasen todt. Ob, und was Gott sey, lehr' allein die
	Philosophie, und ohne sie könne man keinen Gedanken
	von Gott haben u.\,s.\,w. Dies nun sagt' der Magister
	wohl aber nur so. Mir kann kein Mensch mit Grund der
	Wahrheit nachsagen, daß ich 'n Philosoph sey, aber ich
	gehe niemahls durch'n Wald, daß mir nicht einfiele,
	wer doch die Bäume wohl wachsen mache, und denn
	ahndet mich so von ferne und leise etwas von einem
	Umbekannten, und ich wollte wetten daß ich denn an
	Gott denke, so ehrerbietig und freudig schauert mich dabey.

	Weiter sprach er von Berg und Thal, von Sonn' und Mond,
	als wenn er sie hätte machen helfen. Mir fiel dabei der
	Isop ein, der an der Wand wächßt (1.\,Reg 4,\,33);
	aber die Wahrheit
	zu sagen, 's kam mir doch nicht vor, als wenn der
	Magister so weise war, als Salomo. Mich dünkt, wer
	\emph{was rechts} weiß, muß, muß -- säh ich nur n'mahl
	einen, ich wollt 'n wohl kennen, malen wollt ich 'n
	auch wohl, mit dem hellen heitern ruhigen Auge, mit dem
	stillen großen Bewußtseyn etc. Breit muß sich ein solcher
	nicht machen können, am allerwenigsten andre verachten
	und fegen. O! Eigendünkel und Stolz ist eine feindseelige
	Leidenschaft; Gras und Blumen können in der Nachbarschaft
	nicht gedeyen.
	\begin{flushright}
	\citename{Mathias Claudius}, Asmus omnia sua secum portans, oder:\\
	Sämmtliche Werke des Wandsbecker Bothen, Erster Theil.\\
	Wandsbeck, beim Verfasser, 1774, pp. 9sqq.
	\end{flushright}
\upshape

\newpage
\quad
\thispagestyle{empty}
\begin{titlepage}
\large \scshape \lowercase{
\centering
\vspace*{\fill}
\vfill
Inauguraldissertation\\
über die
\vfill
{\LARGE \bfseries
Numerical Approximation\\
in Riemannian Manifolds\\
by Karcher Means\\
\quad \\
\quad \\ }
{\normalsize
(Numerische Approximation\\
in Riemannschen Mannigfaltigkeiten\\
mithilfe des Karcher'schen Schwerpunktes)}
\vfill
zur Erlangung des Grades eines\\
Doctoris rerum naturalium\\
\vfill
dem Fachbereich\\
Mathematik und Informatik\\
der Freien Unversität Berlin\\
im November {\upshape 2013}\\
vorgelegt\\
\vfill
\vfill
von\\[2ex]
Diplom-Mathematiker\\
Stefan Wilhelm von Deylen\\[2ex]
aus\\[2ex]
Rotenburg (Wümme)
}
\end{titlepage}
\pagestyle{empty}
\,
\newpage
\quad
\vfill
\frenchspacing
\noindent
\textbf{Gutachter:}\\
Proff. Dres.\\
Konrad Polthier, Freie Universität Berlin (Erstgutachter)\\
Max Wardetzky, Georg-August-Universität Göttingen\\
Martin Rumpf, Rheinische Friedrich-Wilhelms-Universität Bonn\\
Hermann Karcher, Rheinische Friedrich-Wilhelms-Universität Bonn

\vspace{2ex}
\noindent
\textbf{Datum der mündlichen Prüfung:}\\
10. Juni 2014

\vspace{2ex}
\noindent
Diese Dissertation wurde in den Jahren 2010--13, gefördert durch
ein Stipendium der Berlin Mathematical School (\textsc{bms}),
in der Arbeitsgruppe für Mathemati\-sche Geometrieverarbeitung
angefertigt und von Professor Dr. Konrad Polthier betreut.
\cleardoublepage
\pagestyle{scrheadings}
\setcounter{page}{1}
\renewcommand{\thepage}{\textsc{\roman{page}}}
\addchap{Introduction}

\addsubsec{Overview}


This dissertation treats questions about the definition of
``simplices'' inside Riemannian mani\-folds, the
comparison between those simplices and Euclidean ones,
as well as Galerkin methods for variational problems on manifolds.

During the last three years, the ``Riemannian centre
of mass'' technique described by \cite{Karcher77}
has been successfully employed to define the notion of a simplex
in a Riemannian manifold $M$ of non-constant curvature by
\cite{Rustamov10}, \cite{Sander12} and others.
This approach constructs, for given vertices $p_i \in M$,
a uniquely defined ``barycentric map''
$x: \stds \to M$ from the standard
simplex $\stds$ into the manifold, and calls $x(\stds)$
the ``Karcher simplex'' with vertices $p_i$.

However, the question whether $x$ is
bijective and hence actually induces
barycentric \textit{coordinates} on $x(\stds)$
remained open for most cases.
We show that under shape regularity
conditions similar to the Euclidean setting,
the distortion induced
by $x$ is of the same order as for normal
coordinates: $dx$ is almost an isometry
(of course, this
can only work if $\stds$ is endowed with an
appropriately-chosen Euclidean metric),
and $\nabla dx$ almost vanishes. The estimate
on $dx$ could have already been deduced from the
work of \cite{Jost82}, but it is the combination with the
$\nabla dx$ estimate which paves the ground for applications of
Galerkin finite element techniques.

For example, the construction can be employed to triangulate $M$
and solve problems like the Poisson problem or the
Hodge decomposition on the piecewise flat simplicial
manifold instead of $M$. This leads to analogues of the
classical estimates by \cite{Dziuk88} and subsequent
authors in the field of surface \textsc{pde}'s
(we only mention \citealt{Hildebrandt06}
and \citealt{Holst12} at this point), but as no embedding
is needed in our approach, the range of the
surface finite element method is extended to abstract
Riemannian manifolds without modification of the computational
scheme. Second, one can approximate submanifolds $S$
inside spaces other than $\R^m$ (for example,
minimal submanifolds in hyperbolic space), for which
the classical ``normal height map'' or ``orthogonal projection''
construction from the above-mentioned literature
directly carries over, and the error term generated by
the curvature of $M$ is dominated by the well-known
error from the principal curvatures of $S$.

Apart from classical conforming Galerkin methods, there are other
discretisation ideas, e.\,g. the ``discrete exterior calculus''
(\textsc{dec}, see \citealt{Hirani03})
in which variational problems such as the Poisson
problem or the Hodge decomposition can be
solved without any reference to some smooth problem.
Convergence proofs
are less developed in this area, mainly
because albeit there are interpolation
operators from discrete $k$-forms to $\Leb^2\Omega^k$,
these interpolations do not commute with the
(differing) notions of exterior derivative
on both sides. We re-interpret \textsc{dec} as non-conforming
Galerkin schemes.

\addsubsec{Results}

Let $(M,g)$ be a smooth compact Riemannian manifold.
Concerning the simplex definition and parametrisation
problem, we obtained the following (for German readers,
we also refer to the official abstract on page \pageref{offizielleZsfsg}):
\begin{subenum123}
\item	For given points $p_0,\dots,p_n \in M$
		inside a common convex ball, we
		consider the ``barycentric mapping''
		$x: \stds \to M$ from the standard
		simplex into $M$ defined by the
		Riemannian centre of mass technique.
		Its image $s := x(\stds)$ is called the
		(possibly degenerate) $n$-dimensional
		``Karcher simplex'' with vertices $p_i$. If
		$\stds$ is equipped with a flat metric $g^e$
		defined by
		edge lengths $\dist(p_i,p_j) \leq h$, where $\dist$ is
		the geodesic distance in $(M,g)$, and if
		$\vol(\stds,g^e) \geq \alpha h^n$ for some
		$\alpha > 0$ independent of $h$ (``shape regularity''),
		we give a estimate
		for the difference $g^e - x^* g$ between 
		the flat and the pulled-back metric of order $h^2$,
		as well as a first-order estimate
		for the difference $\nabla^{g^e} - \nabla^{x^*g}$
		between the Euclidean and the pulled-back connection
		(\ref{prop:comparisongandge}, \ref{prop:estimateOfChristoffelOperator}).
\item	We give estimates for the interpolation of
		functions $s \to \R$ and $s \to N$, where
		$N$ is a second Riemannian manifold
		(\ref{prop:InterpolationEstimateForRealvaluedFunctions},
		\ref{prop:interpolNtoMEstimateSingleSimplex}).
\item	Starting from the already existing
		theory of Voronoi tesselations in Riemannian
		manifolds by \cite{Leibon00} and \cite{Boissonnat11},
		we define the Karcher--Delaunay triangulation
		for a given dense and ``generic'' vertex set
		(\ref{prop:KarcherTriangIsTriang}).
\item	Concerning the Poisson problem on the space of weakly
		differentiable real-valued functions
		$\Sob^1(M,\R)$, weakly differentiable
		real-valued differential forms
		$\Sob^1\Omega^k(M)$, and weakly
		differentiable mappings
		into a second manifold
		$\Sob^1(M,N)$, we prove error estimates
		for their respective Galerkin approximations
		(\ref{prop:estimateDirichletProblem},
		\ref{prop:feDiffFormsDirichletProblem},
		\ref{prop:errorEstimateForDirPbOnMfs}).
		The same method gives estimates
		for the Hodge decomposition in
		$\Sob^1\Omega^k(M)$ if appropriate trial
		spaces as in \cite{Arnold06} are chosen
		(\ref{prop:feApproxOfHodgeDecomp}).
\item	We give proximity and metric
		comparison estimates
		for the ``normal height map'' or
		``orthogonal projection map''
		between a smooth submanifold and its
		Karcher-simplicial approximation,
		which is the classical tool for
		finite element analysis on surfaces in $\R^3$,
		but this time for submanifolds inside
		another curved manifold
		(\ref{prop:estimateExtrIntrBaryMapping},
		\ref{prop:comparisonygAndge}).
\item	We show that the differential of
		a Karcher simplex' area functional with
		respect to variations of its vertices is
		well-approximated by the area differential
		of the flat simplex $(\stds,g^e)$ with
		$g^e$ as above
		(\ref{prop:estimateAreaDifferentialOneSimplex}).
\end{subenum123}
Concerning the convergence analysis of discrete
exterior calculus schemes for a simplicial complex:
\begin{subenum123}
\addtocounter{enumi}6
\item	We define a (piecewise constant) interpolation $i_k:
		C^k \to \Pol\inv\Omega^k$
		from discrete differential forms to a subspace
		of $\Leb^2\Omega^k$, which turns the discrete exterior
		derivative into a ``differential'' $\bard:
		\Pol\inv\Omega^k \to \Pol\inv\Omega^{k+1}$ with
		Stokes' and Green's formula for simplicial domains.
		This reduces convergence issues for
		\textsc{dec} from simplicial (co-)chains to
		approximation estimates between the non-conforming
		trial space
		($\Pol\inv\Omega^k,\bard)$ and $(\Sob^1\Omega^k,d)$.
		We estimate the approximation quality of
		$\Pol\inv$ forms in $\Sob^1\Omega^k$
		(\ref{prop:Hminus1estimateForPolinvOnlyOneStage},
		\ref{prop:Hminus1estimateForPolinv})
		and compare the solutions of variational
		problems in $\Pol\inv\Omega^k$ and $\Sob^1\Omega^k$
		(\ref{prop:DirichletPbInDEC}--%
		\ref{prop:mixedFormDirichletPbInDEC:woChapter}).
\end{subenum123}

\newpage
\addsubsec{Structure and Method}

All the thesis is divided into three parts, one of which introduces
notation, the main constructions another, its applications
to standard problems in numerical analysis of geometric
problems and surfaces \textsc{pde}'s (changing the usual
setting from embedded surfaces to abstract (sub-)manifolds) the third.
Having in mind that
``the introduction of numbers as coordinates [...]
is an act of violence`` (\citealt[p. 90]{Weyl49}), we try
to stay inside the absolute Riemannian calculus
as far as possible. Our main tool are Jacobi fields, which naturally
occur when taking derivatives of the exponential map
and its inverse. Whereas the standard situation for
estimates on a Jacobi field $J(t)$ are given values $J(0)$ and
$\dot J(0)$, see e.\,g. \citet[chap. 5]{Jost11},
we will deal with Jacobi fields with
prescribed start and end value, which is convered by
(fairly rough, but satisfying) growth estimates
\ref{thm:estimatesOnDsJ} and
\ref{prop:estimatesOnDrJs}.

%

\addsubsec{Acknowledgements}

We owe thanks to people who accompanied and supported
the development of this dissertation: To our advisor
Prof. Dr Konrad Polthier, members of his work group (Dipl.-Mathh.
Konstantin Poelke, Janis Bode, Ulrich Reitebuch,
B.\,Sc. Zoi Tokoutsi and all others) as well
as other Berlin doctoral students (first and foremost Dipl.-Math. Hanne
Hardering), our \textsc{bms}
mentor Prof. Dr Günter M. Ziegler, the co-authors of the corresponding
journal publication (\citealt{Deylen13}),
Proff. Dr'es David Glickenstein (Arizona)
and Max Wardetzky (G\"ottingen), furthermore to Proff. Dr'es
Hermann Karcher (Bonn) and Ulrich Brehm (Dresden) for helpful advices.
In a more global perspective, we are indepted to our academic
teachers Proff. Dr'es Matthias Kreck and Martin Rumpf (both Bonn).

We know that this thesis, as a time-constraint human work,
will be full of smaller or larger mistakes and shortcomings.
We strongly hope that none of them will destroy main
arguments. For all others, we keep in mind the words of
an academic teacher, concerning exercise sheets: ``Do not
refuse an exercise in which you have found a mistake,
but try to find the most interesting and correct
exercise in a small neighbourhood of the original one.''

%
%
%

\newpage
\addsubsec{Symbol List}

We listed here those symbols which occur in several sections without
being introduced every time. Symbols with bracketed explanation are
also used with a different meaning, which then will be defined
in the section. Where it is useful, we added a reference to the
definition.

\begin{longtable}{ll}
\toprule
$M$					& manifold, $Mg$ is shortcut for $(M,g)$ \\
$m$					& dimension of $M$ \\
$g$					& Riemannian metric on $M$ \\
$P$					& parallel transport (beside in section \ref{sec:simplexGeometry}) \\
$R$					& curvature tensor of $M$ (\ref{parag:curvatureTensor}) \\
$\Gamma$			& Christoffel symbols (\ref{parag:connection}), Christoffel operator (\ref{obs:normalCoords}) \\
$\dist$				& geodesic distance function in $M$ \\
$X_p$				& gradient of $\frac 1 2 \dist(p,\argdot)$ (\ref{prop:definingPropertiesOfXandY}) \\
$x$         & barycentric mapping (\ref{def:mappingx}) \\
$\inj$, $\cvr$		& injectivity and convexity radius (\ref{rem:definitionOfX}) \\
$C_0$, $C_1$				& global bound for $\norm R$ and $\norm{\nabla R}$ resp.\\
$h$					& mesh size \\
$\theta$			& fullness parameter (\ref{def:thetahFullSimplex}) \\
$C_{0,1}$			& $:= C_0 + h C_1$ \\
$C_{0,1}'$			& $:= C_{0,1} \theta^{-2}$ \\
\midrule
$\complex$			& simplicial complex (\ref{def:simplexAndComplex}) \\
$n$					& dimension of $\complex$\\
$\simplexe, \simplexf, \simplexs, \simplext$ & elements, facets, simplices \\
$r$					& (realisation operator for simplicial complexes, \ref{def:abstractGeometricRealisation}) \\
$\simleq$			& $\leq$ up to a constant that only depends on $n$ \\
\midrule
$\stds$				& standard simplex, Laplace--Beltrami operator \\
$e_i$				& Euclidean basis vector \\
$1_n$				& $ = (1,\dots,1) \in \R^n$ \\
$\1$				& unit matrix \\
$\B_r(U)$			& set of points with distance $< r$ from $U$ \\
\midrule
$d$					& differential, exterior derivative \\
$\delta$			& (exterior coderivative, Kronecker symbol) \\
$\partial$			& partial / coordinate derivative, boundary of sets \\
$\nabla$			& covariant derivative \\
$D$					& covariant derivative along curves (except section \ref{sec:simplexGeometry}) \\
$L$					& (weak Laplacian, \ref{def:weakStrongLaplacian}), curve length functional \\
\midrule
\enlargethispage{2ex}%
$\absval[\ell^2] \argdot$ & canonical Euclidean norm of $\R^n$ \\
$\absval \argdot$	& pointwise norm on bundles induced by $g$, volume of sets \\
$\norm{\hspace{-1.2pt}\argdot\hspace{-1.2pt}}$		& pointwise operator norm (\ref{parag:norms}) \\
\hspace{-1pt}$\ibetrag \argdot$	& integrated (or supremum) pointwise $g$-norm (\ref{def:SobolevSpaces}) \\
\hspace{-1pt}$\inorm \argdot$	& integrated (or supremum) pointwise operator norm \\
\hspace{-1pt}$\iopnorm \argdot$	& operator norm in function spaces (\ref{def:Pol}) \\\nopagebreak
\midrule\\
\midrule
$\Cont^k$			& $k$-times continuously differentiable functions \\
$\Leb^r$			& functions whose $r$'th power is Lebesgue-integrable \\
$\SobW^{k,r}$		& functions that have $k$ covariant differentials in $\Leb^r$ (\ref{def:SobolevSpaces}) \\
$\Sob^k$			& $:= \SobW^{k,2}$ (except section \ref{sec:mfDirichletProblem}) \\
$\Sob^k_0$ etc.		& functions in $\Sob^k$ etc. with vanishing trace on the boundary \\
$\Sob^{1,0}, \Sob^{0,1}$ & forms $\alpha$ with weak $d\alpha$ or $\delta\alpha$ of class $\Leb^2$ resp.\\
$\Sob^{1,1}$		& forms $\alpha$ with weak $d\alpha$ and $\delta \alpha$ of class $\Leb^2$ \\
$\Sob^{1+1}$		& forms with weak $d\alpha$ and $\delta \alpha$ of class $\Sob^{1,1}$ \\
$\Pol$				& polynomial forms (\ref{def:PolOmegak}), functions (\ref{def:Pol}), vector fields (\ref{def:vertexTangentSpaces}) \\
$\VF$				& vector fields of class $\Cont^\infty$ \\
$\Omega^k$			& differential $k$-forms of class $\Cont^\infty$ \\
$\Omega^k_\tang$, $\Omega^k_\nor$ & diff. forms with vanishing tangential/normal trace on the boundary \\
$\Leb^2 \VF$ etc.	& vector fields of class $\Leb^2$ etc. \\
\midrule
$S$					& submanifold \\
$TM|_S$				& vector bundle over $S$ with fibres $T_pM$ \\
$TS^\perp$			& normal bundle of $S$ in $M$ \\
$\nu$				& normal on $S$ in $M$ \\
$\pi$				& projection \\
$\nor\,$			& projection onto normal part \\
$\tang\,$ 			& projection onto tangential part \\
$\Phi$				& normal height map $p \mapsto \exp_p Z$ for normal vector field $Z$ \\
$\Phi_t$			& geodesic homotopy $p \mapsto \exp_p tZ$ \\
\bottomrule
\end{longtable}

\addtocontents{toc}{\protect\pagestyle{plain}}
\tableofcontents
\cleardoublepage

\pagenumbering{arabic}
\cleardoublepage 
%
%
                          \chapter{Preliminaries}
%
%

Let us briefly recapitulate basic notions and concepts
of the concerned mathematical fields: Riemannian manifolds
and variational problems on these, simplex geometry and
simplicial complexes. For the quick reader with
experience in numerics on surfaces, a short look
on the simplex metric in barycentric
coordinates (\ref{parag:metricOnStds})
and our definition of simplicial complexes
(\ref{def:abstractGeometricRealisation}) might be of interest.

%
\section{Riemannian Geometry}
\label{set:notationManifoldsCurvature}
%

For this section, we will keep close to the notations of
\cite{Jost11} and \cite{Lee97}.---
Let $(M,g)$ or $Mg$ for short be an $m$-dimensional Riemannian manifold.
We write $\sprod XY$ or $g\sprod XY$ instead
of $g(X,Y)$ for $X,Y \in T_p M$, mainly to prevent the use
of too many round brackets.
Whereas \begriff{charts} map open sets in $M$ into $\R^m$,
we will mostly use \begriff{coordinates} $(U,x)$, i.\,e.
maps $x$ from open sets $U \subset \R^m$ into $M$
that are locally homeomorphisms.

Throughout this thesis, we apply Einstein convention
for computations in local coordinates or any other upper-lower index pair.
Only when it
explicity helps to clarify our statements, we note the
evaluation of a vector field $X$ or the metric $g$ at
a specific point $p \in M$ as $X|_p$ or $g|_p$ respectively.

\parag{Tangent Bundle and Norms.}
\label{parag:norms}
Coordinates $(U,x)$ around $p$ give rise to a basis $\frac{\partial}
{\partial x_i}$ or shortly $\partial_i$ of $T_p M$,
and a dual basis $dx^i$ on $T^*_pM$.
The tangent-cotangent isomorphism is denoted by $\flat$
and its inverse by $\sharp$.
The natural extension of $g$ to $T^*M$ has coefficients $g^{ij}$
with $g^{ij} g_{jk} = \delta^i_k$ (Kronecker symbol).
On higher tensor bundles,
$g$ also naturally induces scalar products by
$g\sprod{v \otimes \bar v}{w \otimes \bar w} :=
g\sprod vw g\sprod{\bar v}{\bar w}$ and similar for covector and mixed
tensors. The space of smooth vector fields is denotes as $\VF$,
the spaces of smooth alternating $k$-forms as $\Omega^k$.
With $\cdot$, we denote the Euclidean scalar product in $\R^n$.

We will denote the norm on all these bundles simply by $\absval \argdot$
or $\absval[g] \argdot$, because we do not see ambiguity here.
However, it differs from the operator norm of a tensor denoted as
$\norm \argdot$.
Both are equivalent,
$\norm[g] \argdot \leq \betrag[g] \argdot \leq c \norm[g] \argdot$
with a constant $c$ that only depends on the dimension $m$
and the rank of the tensor
(\citealt[eqn. 2.2-9]{Golub83}). In particular,
operator and induced norm agree on $1$-forms.

\subsection{Curvature}

\rneq
	In local coordinates $(U,x)$, the metric $g$ is a smooth field
	of positive definite $m \times m$-matrices over $U$.
	A \begriff{connection} $\nabla$ on $Mg$
	is given in local coordinates by some
	\begriff{Christoffel symbols} $\Chris ijk = \Chris jik$
	via
	\label{parag:connection}
	\begin{subeqns}
	\begin{equation}
		\label{eqn:covariantDerivative}
		\nabla_{\partial_i} \partial_j = \Chris ijk \partial_k,
		\qquad
		\nabla_X Y = (X^i \partial_i Y^k + X^i Y^j \Chris ijk) \partial_k
	\end{equation}
	for vector fields $X, Y$ around $p$ with coordinates
	$X = X^i \partial_i$ and $Y = Y^i \partial_i$ respectively.
	It naturally induces a connection on higher tensor
	bundles, e.\,g. on the bundle of linear maps
	$A: T_p M \to T_pM$, by $(\nabla_V A)(W)
	= \nabla_V(AW) - A(\nabla_V W)$.
	There is a unique connection that is symmetric
	and compatible with $g$, the \begriff{Levi--Cività
	connection} of $Mg$, whose Christoffel symbols
	can be computed by
	\begin{equation}
		\label{eqn:ChrisDefinition}
		\Chris ijk = g^{k\ell} (\partial_j g_{i\ell} + \partial_i g_{j\ell} - \partial_\ell g_{ij}).
	\end{equation}
\end{subeqns}

\rneq
	The \begriff{Riemann curvature tensor} $R$ of $Mg$
	is defined by
	\label{parag:curvatureTensor}
	\begin{subeqns}
	\begin{equation}
		\label{eqn:defCurvatureTensor}
		R(X,Y)Z = \nabla_X \nabla_Y Z - \nabla_Y \nabla_X Z - \nabla_{[X,Y]} Z.
	\end{equation}
	In local coordinates, it has coefficients
	\begin{equation}
		\label{eqn:componentsOfR}
		R_{ijk}^\ell = \partial_i \Chris jk\ell - \partial_j \Chris ik\ell
						+ \Chris jkn \Chris ni\ell - \Chris ikn \Chris nj\ell,
		\qquad
		R(\partial_i, \partial_j)\partial_k = R_{ijk}^\ell \partial_\ell
	\end{equation}
	and obeys the following (anti-)symmetries:
	\begin{equation}
		\label{eqn:symmetriesOfR}
		\begin{aligned}
			\sprod{R(X,Y)Z} W & = -\sprod{R(Y,X)Z} W, \\
			\sprod{R(X,Y)Z} W & = - \sprod{R(X,Y)W} Z, \\
			\sprod{R(X,Y)Z} W & = \sprod{R(Z,W)X} Y.
		\end{aligned}
	\end{equation}
	Let us agree that $\nabla$ and $D$ bind weaker than linear operators,
	so $D_t A W$ as above always means $D_t(AW)$, not $(D_t A) W = \dot A W$.
	\end{subeqns}

\rneq
	Along smooth curves $c: \interv ab \to M, t \mapsto c(t)$,
	any connection uniquely induces a
	\begriff{covariant differentiation} $D_t$ along $c$ by
	\begin{subeqns}
	\begin{equation}
		\label{eqn:covariantDerivativeAlongCurves}
		\dot V(t) := D_t V(t) = (\dot V^k + \dot c^i V^j \Chris ijk) \partial_k.
	\end{equation}
	A \begriff{geodesic}
	is a curve with vanishing covariant derivative, i.\,e. $D_t \dot c = 0$
	or, slightly inprecise, $\nabla_{\dot c} \dot c = 0$. In coordinates,
	\begin{equation}
		\label{eqn:geodesicEqnInCoords}
		c^k_{,tt} = -\dot c^i \dot c^j \Chris ijk
	\end{equation}
	(note that we use the symbol $\ddot c$ only for the
	covariant derivative of $\dot c$, and we denote
	the coordinate derivative by a comma-separated subscript).
	If the parametrisation does not matter, we denote a
	curve $c$ with endpoints $p,q \in M$ as
	$c: p \leadsto q$.
	The geodesic distance $\dist(p,q)$ is the length of
	the shortest geodesic $p \leadsto q$. A Riemannian manifold
	is complete if any two points can
	be joined by a geodesic.
	For some neighbourhood $B$ of $p$, we say that $B$
	is \begriff{convex} if each two points $q,r \in B$
	have a unique shortest geodesic $q \leadsto r$ in
	$M$ which lies in $B$ (\citealt{Karcher68}).
	\label{def:convexSet}
	
	Along a geodesic $c:\interv a b \to M$, there is a \begriff{parallel translation}
	$P^{t,s}: T_{c(s)} M \to T_{c(t)} M$ for every $s,t \in \interv a b$,
	defined by $P^{t,s}V = W(t)$ for the vector field $W$ along $\gamma$
	with $W(s) = V$ and $\dot W = 0$. Parallel translation is an
	isometry, as $\ddt \absval W^2 = \sprod{\dot W}W = 0$.
	The derivative of $P$ with respect to a variation of $c$
	is computed in \ref{prop:estimateOfParallelTransportDeriv}.
	
	As geodesics are unique inside a convex ball $B$, we will also
	write $P^{q,p}$ for $q,p \in B$.
	The unintuitive order of the evaluation points is inspired
	by the fact that some vector in $T_p M$ enters on the right,
	and a vector in $T_q M$ comes out on the left.---%
	We remark that in general
	$P^{r,q}P^{q,p} \neq P^{r,p}$, but instead $P^{p,r}P^{r,q}P^{q,p}$ is
	the \begriff{holonomy} of the loop $p \leadsto q \leadsto r
	\leadsto p$.
\end{subeqns}
\begin{assumption}
	\label{sit:compactManifoldWithCurvatureBounds}
	Throughout the whole thesis, we will
	assume that $Mg$ is a compact smooth $m$-dimensional manifold (without
	boundary, if not specified) with curvature
	bounds $\norm R \leq C_0$ and
	$\norm{\nabla R} \leq C_1$ everywhere. To keep definitions
	together, we give a ``forward declaration'': When a radius
	(or a mesh size) $r$ and a fullness parameter $\theta$ are defined,
	we will also use $C_0' := C_0 \theta^{-2}$ and
	$C_{0,1} := C_0 + r C_1$, analogously $C_{0,1}'$
	($C_1'$ will not be used).
\end{assumption}
\bibrembegin
\begin{remark_nn}
	Up to a factor of $\frac 4 3$, the bound $\norm R \leq C_0$ is the same as requiring
	that the sectional curvature is bounded,
	because if all sectional curvatures are
	bounded by $\pm K$, then
	$\norm R \leq \frac 4 3 K$
	(\citealt[6.1.1]{Buser81}), which is the usual assumption in the works
	of \citename{Karcher}, \citename{Jost} \etal{} Of course,
	on the other hand $K \leq C_0$.
\end{remark_nn}
\bibremend

%
\subsection{Second Derivatives}
%

Let $N\gamma_{\alpha\beta}$ and $Mg_{ij}$ be two smooth
Riemannian manifolds with coordinates $u^\alpha$ and $v^i$
respectively and $f: N \to M$ be a smooth mapping.
Its first derivative is, at each $p\in N$, a linear map $d_p f: T_p N
\to T_{f(p)}M$. Of course, the Levi Civita connections of $M$ and $N$
induce a unique way to define the Hessian $\nabla df$. For this
purpose, $df$ has to be considered as a section in $E := T^*N \otimes
f^*TM$, a bundle over $N$ with fibres $E_p = T^*_p N \times T_{f(p)} M$.
We want to give a coordinate expression for this.
\begin{definition_nn}
	Let $M$ and $N$ be two Riemannian manifolds, $f: N \to M$ smooth.
	The \begriff{Hessian} of $f$ is $\nabla^E df$, a section of
	$T^*N \otimes T^*N \otimes f^*TM$.
\end{definition_nn}
\begin{fact} \upshape
\label{prop:derivsOfCoordVFs}
\begin{subeqns}
	The connection on the cotangent bundle $T^*N$ is defined
	by
	\[
		d\big(\omega(X)\big) = \omega(\nabla^{TM} X) + (\nabla^{T^*M} \omega)(X)
		\qquad
		\text{for $\omega \in \Omega^1(N)$, $X \in \VF(N)$,}
	\]
	cf. \citet[eqn. 4.1.20]{Jost11}. This gives
	\[
		\begin{split}
		0 & = d\big(du^\alpha(\partial_\beta)\big)(\partial_\gamma)
			= du^\alpha(\nabla_{\partial_\gamma} \partial_\beta)
			+ (\nabla_{\partial_\gamma} du^\alpha) (\partial_\beta) \\
		  & = du^\alpha (\Chris \beta\gamma \delta \partial_\delta)
			+ (\nabla_{\partial_\gamma} du^\alpha)(\partial_\beta),
		\end{split}
	\]
	and with $du^\alpha(\partial_\delta) = 1$ if $\alpha = \delta$
	and $0$ else, this gives that $\nabla_{\partial_\gamma} du^\alpha$
	maps a vector $\partial_\beta$ to $-\Chris\beta\gamma\alpha$, so
	\begin{equation}
		\label{eqn:connectionInCotangentBundle}
		\nabla_{\partial_\gamma} du^\alpha = - \Chris \beta\gamma\alpha du^\beta.
	\end{equation}
	
	Vector fields $V$ on $M$ pull back to vector fields $f^*V$
	by $(f^*V)|_p = V|_{f(p)}$.
	The connection $\nabla^{TM}$ then induces
	a connection on $f^*TM$ by
	\begin{equation}
		\label{eqn:connectionInPullbackBundle}
		\nabla^{f^*TM}_X f^*V = f^*\nabla^{TM}_{df X} V.
	\end{equation}
	Let us abbreviate
	$\partial_\alpha := \frac{\partial}{\partial u^\alpha}$
	as before, and additionally
	$\partial_i := f^*{\textstyle \frac{\partial}{\partial v^i}}$, and
	$f^i_{,\alpha} := \frac{\partial f^i}{\partial u^\alpha}$.
	For example, the usual coordinate representation of $df$ is
	$df(\partial_\alpha) = f^i_{,\alpha} \frac{\partial}{\partial v^i}$.
	Therefore,
	\begin{equation}
		\label{eqn:connectionInPulledBackBundle}
		\nabla^{f^*TM}_{\partial_\alpha} \partial_j
			= f^*(\nabla^{TM}_{f^i_{,\alpha} \frac{\partial}{\partial v^i}} \partial_j)
			= f^*(f^i_{,\alpha} \Chris ijk {\textstyle \frac{\partial}{\partial v^k}})
			= f^i_{,\alpha} \Chris ijk \partial_k.
	\end{equation}

	The connections on $T^*N$ and $f^*TM$ induce
	a connection on the product bundle, cf. \citet[eqn. 4.1.23]{Jost11}:
	\begin{equation}
		\label{eqn:connectionInProductBundle}
		\nabla^E (\omega \otimes V) = (\nabla^{T^*N}\omega) \otimes V
			+ \omega \otimes (\nabla^{f^*TM} V)
		\qquad \text{for $\omega \in \Omega^1(N)$, $V \in f^*\VF(M)$}.
	\end{equation}
\end{subeqns}
\end{fact}
\begin{lemma}
	\label{prop:geomCharacterisationOfHessian}
	Let $f: N \to M$ be a smooth mapping between Riemannian
	manifolds, and let $V, W \in T_pN$. Then consider
	a variation of curves $\gamma(s,t)$ in $N$ with
	$\partial_t \gamma = W$,
	$\partial_s \gamma = V$ and
	$D_s \partial_t \gamma = 0$
	(everything is evaluated at $s = t = 0$).
	Let $c := f \circ \gamma$ be the corresponding
	variation of curves in $M$. Then
	$\partial_t c = df V$,
	$\partial_s c = df W$ and
	$(\nabla^E df)(V,W) = D_s \partial_t c$.
	If $df V \neq 0$ this is
	\[
		(\nabla^E df)(V,W) = \nabla^{TM}_{df W} df V,
	\]
	where $V$ and $W$ are extended such that $\nabla_W V = 0$.
\end{lemma}
\begin{proof}
	Inserting $df = f^i_{,\alpha} du^\alpha \otimes \partial_i$
	in \eqnref{eqn:connectionInProductBundle},
	we have
	\[
		\nabla^E_{\partial_\beta} df = \nabla^{T^*N}_{\partial_\beta} (f^i_{,\alpha} du^\alpha)
				\otimes \partial_i + f^i_{,\alpha} du^\alpha \otimes \nabla^{f^*TM}_{\partial_\beta} \partial_i.
	\]
	By \eqnref{eqn:connectionInCotangentBundle},
	\[
		\nabla^{T^*N}_{\partial_\beta} f^i_{,\alpha} du^\alpha
			= f^i_{,\alpha\beta} du^\alpha - f^i_{,\alpha} \Chris\beta\gamma\alpha du^\gamma
	\]
	and together with \eqnref{eqn:connectionInPulledBackBundle}, this gives
	\[
		\nabla^E_{\partial_\beta} df
			= (f^i_{,\alpha\beta} du^\alpha - f^i_{,\alpha} \Chris\beta\gamma\alpha du^\gamma) \otimes \partial_i
			+ f^i_{,\alpha} du^\alpha \otimes f^j_{,\beta} \Chris ijk \partial_k
	\]
	(cf. \citealt{Jost11}, eqn. 8.1.19).
	We conclude that $\nabla df$, taken
	as bilinear map $T_p N \times T_p N \to T_{f(p)} M$, acts on vectors $\partial_\beta$
	and $\partial_\delta$ as
	\[
		\begin{split}
		\nabla df (\partial_\beta, \partial_\delta)
			& = [f^i_{,\alpha\beta} du^\alpha(\partial_\delta) - f^i_{,\alpha} \Chris\beta\gamma\alpha du^\gamma(\partial_\delta)] \partial_i
			+ f^i_{,\alpha} du^\alpha(\partial_\delta) f^j_{,\beta} \Chris ijk \partial_k \\
			& = (f^i_{,\delta\beta} - f^i_{,\alpha} \Chris\beta\delta\alpha) \partial_i
			+ f^i_{,\delta} f^j_{,\beta} \Chris ijk \partial_k \\
			& = (f^i_{,\delta\beta} - f^i_{,\alpha} \Chris\beta\delta\alpha 
			+ f^j_{,\delta} f^k_{,\beta} \Chris jki) \partial_i.
		\end{split}
	\]
	This is, as it should be, symmetric in $\beta$ and $\delta$
	by the symmetry of $f^i_{,\beta\delta}$ and the Christoffel
	symbols.
	
	On the other hand, let us compute $D_s \partial_t c$.
	The derivatives of $\gamma$ are given by
	$\partial_t \gamma = \gamma^\alpha_{,t} \partial_\alpha$ and
	$\partial_s \gamma = \gamma^\beta_{,s} \partial_\beta$.
	By the chain rule,
	$\partial_t c = c^i_{,t} \partial_i
	= \gamma^\alpha_{,t} f^i_\alpha \partial_i$ and
	$\partial_s c = \gamma^\beta_{,s} f^j_\beta \partial_j$.
	By \eqnref{eqn:covariantDerivativeAlongCurves},
	\[
		D_s \partial_t c
			= \big((\gamma^\alpha_{,t} f^i_{,\alpha})_{,s}
			+ \gamma^\alpha_{,t} f^j_{,\alpha} \gamma^\beta_{,s} f^k_{,\beta} \Chris jki \big) \partial_i.
	\]
	Now $(\gamma^\alpha_{,t} f^k_{,\alpha})_{,s} = \gamma^\alpha_{,ts} f^k_{,\alpha}
	+ \gamma^\alpha_{,t} f^k_{,\alpha\beta} \gamma^\beta_{,s}$ again by the chain rule.
	As we have assumed $D_s \partial_t \gamma = 0$,
	we get $\gamma^\alpha_{,ts} = - \gamma^\beta_{,t} \gamma^\delta_{,s} \Chris \beta\delta\alpha$
	for every $\alpha$, so
	\[
		\begin{split}
		D_s \partial_t c
			& = \qquad \quad \!(-f^i_{,\alpha} \gamma^\beta_{,t} \gamma^\delta_{,s} \Chris \beta\delta\alpha
			+ f^i_{,\alpha\beta} \gamma^\alpha_{,t} \gamma^\beta_{,s}
			+ \gamma^\alpha_{,t} f^j_{,\alpha} \gamma^\beta_{,s} f^k_{,\beta} \Chris jki ) \partial_i, \\
			& = V^\delta W^\beta (-f^i_{,\alpha} \Chris \beta\delta\alpha
			+ f^i_{,\delta\beta}
			+ f^j_{,\delta} f^k_{,\beta} \Chris jki) \partial_i,
		\end{split}
	\]
\end{proof}
\begin{corollary}[\citealt{Jost11}, \textit{eqns. 4.3.48, 4.3.50}]
	\label{prop:characterisationOfRealvaluedHessian}
	\begin{subeqns}
	If $M = \R$, then the Hessian of a function $f: N \to \R$,
	applied twice to the tangent of a geodesic $\gamma$, is the
	second derivative of $f \circ \gamma$, and it holds
	\begin{equation}
		\label{eqn:HessianDef}
		\nabla d f(V,W) = \sprod{\nabla_V \grad f} W
			= \sprod{\nabla_W \grad f} V = V(Wf) - df(\nabla_V W).
	\end{equation}
	\end{subeqns}
\end{corollary}
\subsection{Scalings}
%

In most situations, we will try to prove \begriff{scale-aware}
estimates, i.\,e. estimates for coordinate expressions or
absolute terms where both sides of the inequality scale
similar when the coordinates or the diameter
of the manifold is scaled (if both sides of the inequality
even remain unchanged under rescaling, we call the estimate
\begriff{scale-invariant}).
Therefore, we will need to know the scaling behaviour
of vectors and tensors.

\parag{Coordinate change, fixed absolute manifold.}
First, consider the case where
the abstract (absolute) geometry of $Mg$ is fixed and
only coordinates are changed. A useful application is when
coordinates $(U,x)$ are given and the
eigenvalues of the matrix $g_{ij}^u$ lie
between $\theta^2 \mu^2$ and $\mu^2$, but one would like to
have eigenvalues in the order of $1$ (i.\,e. between
$\theta^2$ and $1$). This is achieved
by coordinates
\[
	y^i = \mu x^i,
	\qquad
	\smallfrac{\partial}{\partial y^\alpha}
		= \frac 1 \mu \smallfrac{\partial}{\partial x^i}.
\]
Components of vectors always scale like the coordinates:
If $W = w^{i,x} \smallfrac{\partial}{\partial x^i}
= w^{i,y} \smallfrac{\partial}{\partial y^i}$, then
$w^{i,y} = \mu w^{i,x}$. This scaling indeed
fulfills our requirements:
\[
	g_{ij}^y = g \bigsprod{\smallfrac{\partial}{\partial y^i}}{\smallfrac{\partial}{\partial y^j}}
		= \frac 1{\mu^2} g \bigsprod{\smallfrac{\partial}{\partial x^i}}{\smallfrac{\partial}{\partial x^j}}
		= \frac 1{\mu^2} g_{ij}^x
\]
The inverse matrix obviously scales with
$(g^{ij})^y = \mu^2 (g^{ij})^x$.---The Christoffel
symbols and the components of the curvature tensor scale with
\[
	(\Chris ijk)^y = \frac 1 \mu (\Chris ijk)^x,
	\qquad
	(R_{ijk}^\ell)^y = \frac 1{\mu^2} (R_{ijk}^\ell)^x.
\]

\parag{Fixed coordinates, manifold scaling.}
Consider a new Riemannian manifold $M\bar g$ with $\bar g = \mu^2 g$.
Then $\diam(M\bar g) = \mu \diam(Mg)$ and $\bar\dist(p,q) = \mu \dist(p,q)$.
The norm of a tensor that is covariant of rank $k$ and contravariant of rank $\ell$
scales with $\mu^{\ell-k}$. For example, a vector $W$, a linear
form $\omega$ and the curvature tensor $R$ scale with
\label{parag:manifoldScalingNoCoordScaling}
\begin{subeqns}
\[
	\absval[\bar g] W = \mu \absval[g] W,
	\qquad
	\absval[\bar g]\omega = \frac 1 \mu \absval[g] \omega,
	\qquad
	\norm[\bar g] R = \frac 1{\mu^2} \norm[g] R.
\]
If coordinates $(U,x)$ remain the same, then
$\bar g_{ij} = \mu^2 g_{ij}$ and
$\bar g^{ij} = \smallfrac 1{\mu^2} g^{ij}$,
and the Christoffel symbols and tensor components remain
fixed:
\begin{equation}
	\label{eqn:scalingOfCurvature}
	\bar\Gamma_{ij}^k = \Chris ijk,
	\qquad
	\bar R_{ijk}^\ell = R_{ijk}^\ell.
\end{equation}

Now suppose two manifolds $Mg$ and $N\gamma$ with
a mapping $f: N \to M$. Consider a scaling $\mu$ for $M$ and $\nu$ for $N$.
As $df$ can be regarded as a linear form on $TN$,
resulting in a vector in $TM$, it is natural that the norm of
$df$ and $\nabla df$ scale as
\begin{equation}
	\label{eqn:scalingOfDifferentials}
	\norm[\bar \gamma, \bar g]{df} = \frac \mu \nu \norm[\gamma,g]{df},
	\qquad
	\norm[\bar\gamma, \bar g]{\nabla df} = \frac \mu{\nu^2}\norm[\gamma, g]{\nabla df}.
\end{equation}
\end{subeqns}
The scaling behaviour of $\norm R$
is the reason why we never suppress curvature bounds as ``hidden constants''.
In fact, most of our results could be simply worked out in balls of radius $1$,
and their scaling behaviour could be recovered from the curvature bounds
and the scaling behaviour of left- and right-hand side operator norms.

\parag{Coordinate change with manifold scaling.}
\label{parag:manifoldScaling}
It might also be useful to use coordinates for $(M,\mu^2 g)$
where the components $g_{ij}$ remain unchanged,
for example because they had previously been normalised to have
eigenvalues in the order of $1$.
If a chart $(U,x)$ is known, such coordinates
are given by $y^i = \mu x^i$, because
vector components also scale as $w^{\alpha,v} = \mu w^{\alpha,u}$
and then
\[
	\absval[\bar g] W^2 = w^{\alpha,v} w^{\beta,v} g_{\alpha\beta}
		= \mu^2 w^{\alpha,u} w^{\beta,u} g_{\alpha\beta}
		= \mu ^2 \absval[g] W^2,
\]
as it should.
The Christoffel symbols and curvature tensor components
scale as
\[
	(\bar \Gamma_{\alpha\beta}^\gamma)^v = \frac 1 \mu (\Chris \alpha\beta\gamma)^u,
	\qquad
	(\bar R_{\alpha\beta\gamma}^\delta)^v = \frac 1 {\mu^2} (R_{\alpha\beta\gamma}^\delta)^u.
\]

If two manifold $N\gamma_{\alpha\beta}$ and $Mg_{ij}$ are
scaled with factors $\mu$ and $\nu$ in this way, resulting
in coordinate expressions $v^\alpha = \nu u^\alpha$ for $N$
and $y^i = \mu x^i$ for $M$, then the coordinate form
of $f$, which was a mapping $U_u \to U_x$, becomes
a mapping $\bar f: \nu U_u \to \mu U_x, v \mapsto \mu f(v/\nu)$,
so by chain rule
\[
	\bar f^i_{,\alpha} = \frac \mu\nu f^i_{,\alpha},
	\qquad
	\bar f^i_{,\alpha\beta} = \frac \mu {\nu^2} f^i_{,\alpha\beta}
\]
for the components in
\[
	df = f^i_{,\alpha} du^\alpha \otimes \smallfrac{\partial}{\partial x^i},
	\qquad
	d\bar f = \bar f^i_{,\alpha} dv^\alpha \otimes \smallfrac{\partial}{\partial y^i}
\]
and $\nabla df = (f^i_{,\delta\beta} - f^i_{,\alpha} \Chris\beta\delta\alpha 
+ f^j_{,\delta} f^k_{,\beta} \Chris jki) du^\beta \otimes du^\delta
\otimes \smallfrac{\partial}{\partial x^i}$
as in the proof of \ref{prop:geomCharacterisationOfHessian}.

\subsection{The Exponential Map and Special Coordinates}

On a point $p \in M$ (interior, if $M$ has boundary), there
is, at least for some small intervall $\interv{-\eps}\eps$,
a unique geodesic for each initial velocity
$X \in T_pM$. As the geodesic equation
\eqnref{eqn:geodesicEqnInCoords} is homogenous and
the unit ball in $T_p M$ is compact, this is equivalent
to the fact that for some small ball $B_\eps$ around
$0 \in T_pM$, the geodesic $c^X$ with initial velocity $X \in B_\eps$
exists on $\interv{-1}1$. As the unit sphere in $T_pM$ is compact,
there is some $\eps$ that works for any direction $X$.
The \begriff{exponential map} is
defined to map $B_\eps \to M$, $\exp_p(X) := c^X(1)$.
From this mapping, normal and Fermi coordinates can be constructed.
The former construction can be found in every Riemannian
geometry textbook (e.\,g. \citealt[p. 78]{Lee03}).

\parag{Normal coordinates}
	\label{parag:normalCoordinates}
	around $p \in M$
	are coordinates $(U,x)$ with $x(0) = p$ in which
	straight lines $t \mapsto tv$ are geodesics
	(arclength-parametrised for $\absval[\ell^2] v = 1$), which
	implies
	$g_{ij}(0) = \delta_{ij}$, $\partial_k g_{ij}(0) = 0$ and
	$\Chris ijk(0) = 0$ for all $i,j,k$.	
\begin{lemma_nn}
	Any orthonormal basis $E_i$ of $T_p M$
	induces normal coordinates $(B_\eps, x)$ around
	$p$ via $x: (u^1,\dots,u^m) \mapsto \exp_p(u^i E_i)$,
	where $\eps$ must be so small that geodesics through $p$
	are unique.
\end{lemma_nn}
\begin{proof}
	By homogenity of the geodesic equation \eqnref{eqn:geodesicEqnInCoords},
	the geodesic starting with initial velocity
	$\dot c(p) = V$ with $V = E_i v^i$ has coordinates
	\[
		c^k(t) = tv^k,
		\quad \text{so} \quad \dot c^k(t) = v^k
		\quad \text{and} \quad c^k_{,tt}(t) = 0
	\]
	for all $t$ in the
	definition interval of $c$. At the same time,
	$c^k_{,tt} = - \dot c^i \dot c^j \Chris ijk$.
	As both equalities must hold for \emph{every}
	$V \in T_pM$, this already implies $\Chris ijk = 0$. The correspondence
	between $\Chris ijk$ and $\partial_k g_{ij}$ is linear
	and of full rank, so the latter have to vanish, too,
\end{proof}
\begin{corollary}
	\label{prop:dexpIsIdentityAt0}
	$d\exp_p = \id$ at $0 \in T_pM$, that means
	$d_0(\exp_p) V = V$.
\end{corollary}
\begin{proof}
	Consider a geodesic $c$ starting from $p$ with
	velocity $V$. As the differential operator $d(\exp_p)$
	applied to $V$
	can be computed as tangent of this integral curve,
	$d(\exp_p) V = \dot c(0) = V$,
\end{proof}
\begin{observation}
	\label{obs:normalCoords}
	\begin{subeqns}
	Then the metric $g_{ij}$ in the parameter domain
	is
	\begin{equation}
		\label{eqn:metricInNormalCoords}
		g_{ij}|_u = \sprod{dx\,e_i}{dx\,e_i}
		       = \sprod{d_U(\exp_p) E_i}{d_U(\exp_p) E_j},
	\end{equation}
	where $U = u^i E_i$. Likewise, the
	\begriff{Christoffel operator}
	$\Gamma: (v,w) \mapsto \Chris ijk v^i w^j \partial_k$
	(which is bilinear, but does not behave tensorial
	under coordinate changes) is computable as pull-back
	of the connection to the parameter domain:
	The coordinate expression $\nabla_v w = \nabla^{\mathrm{eucl}}_v w
	+ \Gamma(v,w)$ given in \eqnref{eqn:covariantDerivative}
	can be understood as
	pull-back $\nabla^{x^*g}$ of the connection onto $\R^m$
	(not to be confused with the connection $x^*\nabla^g$
	on $x^*TM$ from \eqnref{eqn:connectionInPullbackBundle}),
	and such a pull-back
	is defined by $dx(\nabla^{x^*g}_v w) = \nabla_{dx\,v}dx\,w$.
	The right-hand side was identified to be
	$\nabla dx(v,w)$ in \ref{prop:geomCharacterisationOfHessian},
	and so we have
	\begin{equation}
		\label{eqn:christSymbInNormalCoords}
		d_U (\exp_p)(\Gamma(e_i,e_j)) = \nabla d_U (\exp_p)(E_i, E_j).
	\end{equation}
	\end{subeqns}
\end{observation}
\parag{Jacobi Fields.}
Let $c(s,t)$ be a smooth variation of geodesics $t \mapsto c(s,t)$.
Denote $T := \partial_t c = \dot c$ and $J := \partial_s c$.
As $T$ and $J$ are coordinate vector fields, $[J,T] = 0$,
so $\nabla_J T = \nabla_T J$ or, in other words, $D_s \partial_t c
= D_t \partial_s c$.
Differentiating the geodesic equation $D_t \dot c = \nabla_T T = 0$
gives, by \eqnref{eqn:defCurvatureTensor},
\begin{subeqns}
\[
	0 = \nabla_J \nabla_T T = \nabla_T \nabla_J T + R(J,T)T
	  = \nabla_T \nabla_T J + R(J,T)T,
\]
which is the defining equation for \begriff{Jacobi fields}:
\begin{equation}
	\label{eqn:JacobiFieldDef}
	\ddot J = R(T,J)T
\end{equation}
Conversely, every vector field $J$ along $c$ fulfilling
\eqnref{eqn:JacobiFieldDef} gives rise to a variation
of geodesics by
\begin{equation}
	c(s,t) := \exp_{\exp sJ(0)} t(P\dot c(0) + sP\dot J(0)),
\end{equation}
where $P$ is the parallel transport from $c(0)$ to
$\exp sJ(0)$ (\citealt[thm. 5.2.1]{Jost11}).
\end{subeqns}
\begin{proposition}[cf. \citealt{Karcher89}, eqn. 1.2.5]
	\label{prop:derivativeOfExpAlongCurve}
	Let $c: I \to M$ be a smooth curve and $Z$ be a vector field
	along $c$. Then the map $\phi_t:s \mapsto \exp_{c(s)} tZ(s)$
	has derivative $\dot \phi_t(s) = J(t)$ for a Jacobi field
	$J$ with initial values $J(0) = \dot c(s)$, $\dot J(0)
	= \dot Z(s)$. In particular, $d_V (\exp_p)W$ is the value
	$J(1)$ of a Jacobi field along $t \mapsto \exp_p tV$
	with initial values $J(0) = 0$ and $\dot J(0) = W$.
\end{proposition}
\begin{proof}
	$c(s,t) := \phi_t(s)$ is a variation of geodesics
	$t \mapsto c(s,t)$ for every fixed $s$, so
	$\partial_s c = \dot \phi_t$ is a Jacobi field, and
	the values for $t = 0$ are
	$J(0) = \partial_s c(s,0) = \dot c(s)$, and
	$\dot J(0) = D_t\partial_s c(s,0) = D_s \partial_t c(s,0)
	= D_s Z(s)$ again by \ref{prop:dexpIsIdentityAt0},
\end{proof}
\parag{Fermi Coordinates.}
	Let $c: {]a;b[} \to M$ be an arclength-parametrised geodesic. Then Fermi or
	\begriff{geodesic normal coordinates} along $c$ are
	an open neighbourhood $U$ of $0 \in \R^{n-1}$
	and coordinates $x: {]a;b[} \times U \to M$, in which
	$x(t,0) = c(t)$ and straight lines $s \mapsto
	c(t) + sv$ with first component $v^0 = 0$ are geodesics (arclength-parametrised
	for $\absval[\ell^2] v = 1$) perpendicular to $c$.
	This implies
	\label{parag:fermiCoordinates}
	\begin{subeqns}
	\begin{equation}
		\label{eqn:FermiCoordsDef}
		g_{ij}(t,0) = \delta_{ij}, \quad \Chris ijk(t,0) = 0
		\qquad
		\forall t \in {]a;b[}.
	\end{equation}
	If $c$ is not a geodesic, then one can still find coordinates
	with $g_{ij}(t,0) = \delta_{ij}$, but the Christoffel symbols
	cannot be controlled. In classical surface geometry, those
	are called ``parallel coordinates'' along $c$
	(we will not use them).
	\end{subeqns}
\begin{lemma}
	Let $c:{]a;b[} \to M$ be a geodesic in $M$.
	Any orthonormal basis $E_2,\dots,E_m$ of $\dot c(0)^\perp$
	induces Fermi coordinates along $c$ by
	$x: (t,u^2,\dots,u^m) \mapsto \exp_{c(t)}$ $(u^i P^{t,0}E_i)$.
\end{lemma}
\begin{proof}
	$x$ is injective because the orthogonal
	projection onto $c$ is well-defined in
	a small tube around $c$, and if a
	point $q \in M$ projects to $c(t)$, then
	the connecting geodesic $c(t) \leadsto q$
	determines the components $u^2,\dots,u^m$
	by use of normal coordinates $\dot c(t)^\perp \to M$.
	
	By definition of normal
	coordinates, the claim $g_{ij} = \delta_{ij}$ and
	$\Chris ijk = 0$ along $(t,0,\dots,0)$ is
	clear for $i,j,k \geq 2$. Because $c$ is arclength-%
	parametrised, $g_{11} = 1$, the orthogonality of
	$\dot c$ and $E_i$ at every $c(t)$ gives
	$g_{1i} = 0$ for all $i$. Because $P^{t,0}$
	is parallel, $\nabla_{\partial_1}{\partial_i}
	= \nabla_{\dot c} E_i = 0$ proves the
	vanishing of the remaining Christoffel symbols,
\end{proof}
\begin{corollary}
	\begin{subeqns}
	If $P$ is the parallel transport along
	a geodesic in $Mg$ running through $p \in M$,
	then for any vector $V \in T_p M$
	and a vector field $W$ around $p$, we have
	$P\nabla_V W = \nabla_{PV} PW$, and for a vector
	field $V$ along a geodesic
	$t \mapsto c(t)$, the
	\begriff{fundamental theorem of calculus} holds:
	\begin{equation}
		\label{eqn:fundamentalTheoremCalculus}
		V(t) = P^{t,0} V(0) + \Int_0^t P^{t,r} \dot V(r) \, \mathrm d r.
	\end{equation}
	\end{subeqns}
\end{corollary}

\subsection{The Distance and the Squared Distance Function}

The following properties already occur in \cite{Karcher89}
and \cite{Jost82}, but sometimes only hidden inside their
proofs. For the same calculations in coordinates,
see \cite{Ambrosio98}.

\rneq
The geodesic distance $\dist(\argdot, p)$
is a smooth convex function in some small neighbourhood $B$ of $p$,
excluded in $p$ itself.
\label{def:vectorFieldY}
It therefore has a gradient $Y_p$,
and its length is the Lipschitz constant of $\dist(\argdot,p)$, namely $1$
everywhere.
Additionally,
\[
	0 = V\sprod{Y_p}{Y_p} = 2 \sprod{\nabla_V Y_p}{Y_p}
	= 2 \sprod{\nabla_{Y_p} Y_p} V
	\qquad \forall V \in T_q M, q \in B,
\]
by symmetry \eqnref{eqn:HessianDef} of the Hessian $\nabla d \dist$,
so $Y_p$ is autoparallel everywhere.
The integral curves of $Y_p$ are hence geodesics emanating from $p$
with $d \dist(\dot \gamma) = 1$, so $\dist(\gamma(t),p) = t$
for each such curve.
On the other hand, $\dist(\argdot,p)$ is constant on the
distance spheres of $p$, so $Y_p$ is perpendicular to them
(\begriff{Gauss Lemma}). In normal coordinates
$(u^1,\dots,u^m)$ around $p$, we have
$\dist(\argdot,p) = \absval[\ell^2] u$ and hence
\[
	\dist(\argdot,p) Y_p = u^i \partial_i.
\]
\begin{observation}
	\label{prop:reverseBaseAndEvalPointY}
	Base and evaluation point can be reversed,
	and the vector field only changes sign:
	$Y_p|_q = - P^{q,p} Y_q|_p$, because both are
	velocities of the arclength-parametri\-sed
	geodesic $p \leadsto q$ or $q \leadsto p$ respectively.
\end{observation}
\begin{lemma}
	\label{prop:definingPropertiesOfXandY}
	In a small neighbourhood of $p$,
	\[
		X_p := \grad \smallfrac 1 2 \dist^2(p,\argdot) = \dist(p, \argdot) Y_p
	\]
	is an everywhere
	smooth vector field, its
	integral lines are (quadratically parametri\-sed)
	geodesics emanating from $p$, and
	$\exp_q(-X_p|_q) = p$, equivalently
	\[
		-X_q|_p = P^{p,q}X_p|_q = (\exp_p)\inv q
	\]
	for all $q$ in a convex
	neighbourhood of $p$.
	Loosely speaking, one also writes this
	as $PX_p = \exp\inv p$.
\end{lemma}
\begin{proof}
	Let $c$ be the arclength-parametrised geodesic with
	$c(0) = p$ and $c(\tau) = q$. By definition of $\exp$,
	we have $\exp_p \dot c(0) = q$, as well as
	$\dot c(t) = P^{t,0} \dot c(0)$ and
	$\dot c(t) = Y_p|_{c(t)}$ for all $t$ by the Gauss lemma.
	The switch of base and evaluation point is justified by
	\ref{prop:reverseBaseAndEvalPointY},
\end{proof}
\begin{lemma}
	\label{prop:derivativeOfX}
	For $V \in T_q M$, where $q$ is in a convex neighbourhood of $p$,
	let $J$ be the Jacobi field along $p \leadsto q$
	with $J(0) = 0$ and $J(\tau) = V$. Then
	\[
		\nabla_V X_p = \tau \dot J(\tau),
		\qquad
		\nabla^2_{V,V} X_p = \tau D_s \dot J(\tau).
	\]
	In particular,
	if $V$ is parallel to $X_p$, then $\nabla_V X_p = V$ and
	$\nabla^2_{V,V} X_p = 0$.
\end{lemma}
\begin{proof}
	Let
	$s \mapsto \delta(s)$ be a geodesic with $\delta(0) = q$ and
	$\dot \delta(0) = V$. Define a variation of geodesics by
	\[
		c(s,t) := \exp_p \big(t(\exp_p)\inv \delta(s) \big).
	\]
	Then $\partial_t c$ is an autoparallel vector field
	and $J := \partial_s c$ a Jacobi field along $t \mapsto c(s,t)$
	for every $s$ with boundary values $J(s, 0) = 0$ and $J(s,1) = \dot \delta(s)$.
	The $t$-derivative is
	\[
		\partial_t c(s,t) = P^{t,0}(\exp_p)\inv \delta(s)
			= P^{t,1} X_p|_{\delta(s)}
	\]
	and hence $\dot J(t) = D_t \partial_s c(0,t)
			= D_s \partial_t c(0,t)
			= D_s X_p|_{c(0,t)}
			= \nabla_{J(t)} X_p$.
	Differentiating this once more gives the claim
	for the second derivative. If $V$ is parallel to $X_p$,
	then use $\nabla_Y Y = 0$,
\end{proof}
\begin{remark_nn}
	\begin{subenum}
	\item
	Variations of $X_p$ with respect to the base point $p$
	will be considered in \ref{prop:derivativeXwrtoBasePoints}.
	\item
	Analogously to $(\exp_p)\inv = PX_p$, the derivatives of
	$X_p$ and $\exp_p$ correspond: $\nabla_V X_p$ is the derivative of some Jacobi field
	with prescribed start and end value, whereas $d_V(\exp_p) W = J(1)$
	for a Jacobi field with $J(0) = 0$ and $\dot J(0) = W$.
	\item	Although $Y_p$ is not differentiable
	at $p$, we have $\nabla X_p = \id$ at $p$, similar to $d_0 \exp_p = \id$.
	\bibrembegin
	\item
	In the notation of \cite{Grohs13}, our vector field $X_p$
	and its derivative are
	$X_p|_a = \log(a,p)$ and $\nabla X_p|_a = \nabla_2 \log(a,p)$.
	\bibremend
	\end{subenum}
\end{remark_nn}
%

\subsection{Submanifolds}

\parag{Extrinsic Curvature.}
\begin{subeqns}
For a smooth $k$-dimensional
submanifold $S \subset M$, we treat $T_pS$ as a
linear subspace of $T_pM$, denote the orthogonal
projection $T_pM \to T_pS$ as $\tang$
and the projection onto the normal
space $T_pS^\perp$ as $\nor$.
The bundle over $S$ with fibres $T_pM$
is denoted as $TM|_S = TS \oplus TS^\perp$
(meaning a fibre-wise sum of vector spaces).
The \begriff{Weingarten map} or \begriff{shape
operator} with respect to a normal field $\nu$ is
$W_\nu := \nabla \nu$, that means $U \mapsto \nabla_U \nu$.
The \begriff{second fundamental form} with respect to $\nu$ is
\begin{equation}
	\label{eqn:secondFF}
	\secondFF_\nu(U,V) := - \sprod{W_\nu U} V = \sprod{\nabla_U V}\nu
\end{equation}
because $\sprod \nu V = 0$ and hence $U \sprod \nu V = 0$.
In particular, $\secondFF_\nu(U,V)$ is in fact tensorial
in $\nu$, $U$ and $V$.
Sometimes $\secondFF(U,V) := \nor \nabla_U V$ is also called the
second fundamental form in the literature, although it is a bilinear map, not a form.
If the orthonormal parallel normal fields
$\nu_{k+1},\dots,\nu_m$ locally span $TS^\perp$, it holds
$\secondFF(U,V) = \nu_i \secondFF_{\nu_i}(U,V)$. The
covariant derivative induced by $g|_S$ is
$\nabla^S = \tang \nabla$, hence $\secondFF = \nabla - \nabla^S$.
\end{subeqns}

\parag{Generalised Fermi {\mdseries or} Graph Coordinates.}
The ``tubular neighbourhood theorem'' states that a small
neighbourhood $\B_\eps(S)$ of $S$ is diffeomorphic to
$S \times B_\eps$ with an open $\eps$-ball $B_\eps \subset \R^{m - k}$
around $0$ (\citealt[thm. II.11.4]{Bredon93}).
By explicitely constructing this diffeomorphism, upper bounds
on $\eps$ can be derived: For $t \in \interv 01$, let
\label{parag:graphCoordinates}
\begin{subeqns}
\begin{equation}
	\label{eqn:defGraphMapping}
	\Phi_t: \quad TS^\perp \to M, \quad (p,Z) \mapsto \exp_p tZ.
\end{equation}
If $p$ moves with velocity $\dot p$, we know by
\ref{prop:derivativeOfExpAlongCurve} that $d\Phi_t(\dot p) = J(t)$
for a Jacobi field with $J(0) = \dot p$, $\dot J(0) = \nabla_{\dot p} Z$.

The case where $Z = \nu$ is parallel along $p$
is particularly interesting. Then
$\Phi_t$ parametrises the level sets of the distance
function from $S$; we have $\dot J(0)
= W_\nu \dot p$, and the parallel transport of $\nu$ along
$t$ is normal to the image of $\Phi_t$,
so the whole curve fulfills $\dot J = W_\nu J$. Therefore,
the pull-back metric $\Phi_t^* g(\dot p,\dot p)$ changes
with respect to $t$ as
$\ddt \Phi_t^* g\sprod{\dot p}{\dot p} = 2 g\sprod J{\dot J}
= 2 g \sprod J{W_\nu J}$, see
\citet[eqn. 1.2.7]{Karcher89}. Hence the maximal eigenvalue
of $W_\nu$ over $ \nu \in \mathbb S^{m-k} \subset TS^\perp$ and over $t$
bounds $\eps$. We will pursue this more explicitely in
\ref{prop:boundOnReach}. By $\ddot J = \ddt (WJ) = \dot W J + W \dot J
= \dot W J + W^2 J$, one then obtains a Riccati-type equation
for the Weingarten map (\citealt[eqn. 1.3.1]{Karcher89})
\begin{equation}
	\label{eqn:RiccatiEqnForWeingartenMap}
	\dot W = R_\nu - W^2
	\qquad \text{for} \quad
	R_\nu = R(\nu, \argdot)\nu.
\end{equation}

Generally, a tangent vector $U \in T_{(p,Z)}TS^\perp$
is induced by a curve $s \mapsto \exp_{p(s)} tZ(s)$,
where $\dot p$ is tangential to $S$ and
$\dot Z = \tang \dot Z + \nor \dot Z$. The above-mentioned
Jacobi field $J$ can be split into
two Jacobi fields $J_p(s) + J_\nu(t)$
with initial values
\begin{equation}
	\label{eqn:JacobiFieldsForOrthProjection}
	\begin{aligned}
		J_p(0) & = \dot p				& \qquad J_\nu(0) & = 0 \\
		\dot J_p(0) & = \tang \dot Z	& \dot J_\nu(0) & = \nor \dot Z.
	\end{aligned}
\end{equation}
The part $\tang \dot Z$ is in fact
$\tang \nabla_{\dot p} Z$ (if we assume $Z$ to be extended parallel along
$t$), so it is uniquely determined by $\dot p$,
and thus $U$ has the representation $(\dot p, \nor \dot Z)$ in
the chart $\Phi_t$.
Let $\psi$ be the orthogonal projection $\B_\eps(S) \to S$.
As Jacobi fields with orthogonal initial values and
velocities 	stay
orthogonal, we have an orthogonal splitting
$V = V_p + V_\nu$ for $V \in T_p M$,
$p \in \B_\eps(S)$, with $V_p = J_p(1)$, $V_\nu = J_\nu(1)$.
This gives a simple
representation of $d\psi$, namely $d\psi(V_\nu) = 0$ and
$d\psi(V_p) = \dot p$.
The geometric interpretation of the splitting is
\begin{equation}
	\label{eqn:JacobiFieldsOrthSplitting}
	V_p = P^{p,\psi(p)} \tang P^{\psi(p),p} V,
	\qquad
	V_\nu = P^{p,\psi(p)} \nor P^{\psi(p),p} V,
\end{equation}
that means $V_p$ and $V_\nu$ are the orthogonal projections
onto $P TS$ and $P TS^\perp$ respectively.
This is proven by $\frac{\d^2}{\d t^2} \sprod{J_p}Z = 0$
(if $Z$ is extended parallel along $t$) and the initial
conditions $\sprod{J_p(0)} Z = 0$ and $\sprod{\dot J_p(0)} Z = 0$.
\end{subeqns}

%
\newsectionpage
\section{Functional Analysis and Exterior Calculus}
\label{sec:functionalAnalysis}
%

We will quickly review the Dirichlet problem and the Hodge
decomposition in this section. All proofs are reformulations
from \cite{Schwarz95}, but we tried to take special care that
not the vector bundle structure of $\Omega^k$, but only its
functional analytical nature has been use (the only exception
is \ref{prop:HodgeFriedrichsDecomposition}).

\begin{notation}
	\begin{subenum}
	\item	We have defined $\VF$ and $\Omega^k$ as the
	spaces of smooth vector fields and $k$-forms on $M$. The
	pointwise scalar product $g$ on all tensor products of
	$TM$ and $T^*M$ naturally induces an $\Leb^2$ product
	on them:
	\[
		\dprod vw := \Int_M g\sprod vw
	\]
	The completion $\VF$ with respect to the $\Leb^2$ norm
	will be called $\Leb^2\VF$, analogously $\Leb^2\Omega^k$
	for the differential forms. The notation
	$\dprod\argdot\argdot$ will only be used for the
	$\Leb^2$ scalar product, so all indices like
	${\dprod \argdot\argdot}_{\Leb^2}$,
	${\dprod \argdot\argdot}_{\Leb^2(Mg)}$ and
	${\dprod \argdot\argdot}_{\Leb^2\VF}$ or
	${\dprod \argdot\argdot}_{\Leb^2\Omega^k}$
	are only added for ease of reading.
	
	\item	Let $M$ have a boundary $\Rand M$.
	The projections $\tang$ and $\nor$ from
	$TM|_{\Rand M}$ onto $T\Rand M$ and $T\Rand M^\perp$
	pull back $k$-forms as
	$\tang^*v(V_1,\dots,V_k)
	= v(\tang V_1,\dots, \tang V_k)$ and similarly $\nor^*\omega$.
	The spaces of
	$k$-forms with vanishing tangential part on $\Rand M$ are called
	$\Omega^k_\tang$.
	
	Together with
	the usual \begriff{exterior derivative} $d$, the $\Omega^k$
	form the smooth \begriff{de Rham cochain complex}
	\[
		\Omega^0 \to \dots \to \Omega^n \to 0
	\]
	The \begriff{exterior coderivative} $\delta$ is, for forms with
	appropriate boundary conditions, adjoint to $d$
	with respect to the $\Leb^2$ scalar product:
	\begin{equation}
		\label{eqn:Greensformula}
		\begin{split}
		{\dprod v {dw}}_{\Leb^2\Omega^{k+1}} = {\dprod {\delta v} w}_{\Leb^2\Omega^k}
		\qquad & \forall v \in \Omega^{k+1},
						 w \in \Omega^k_\tang \\
		& \andall v \in \Omega^{k+1}_\nor, w \in \Omega^k
		\end{split}
	\end{equation}
	The image and the kernel of $d$ in $\Omega^k$ are
	called the spaces of \begriff{boundaries} and \begriff{cycles},
	$\Bdry^k := \im d|_{\Omega^{k-1}}$ and
	$\Cyc^k := \ker d|_{\Omega^k}$. The space of
	\begriff{harmonic} forms is $\Harm^k := \Cyc^k \cap d(\Omega^{k-1}_\tang)^\perp$.
	For $\delta$, we have $\Bdry_k^*$ and $\Cyc_k^*$ defined analogously,
	so $\Harm^k = \Cyc^k \cap \Cyc^*_k$ by \eqnref{eqn:Greensformula}.
	Denote
	\begin{equation}
		\label{eqn:defLapDir}
		\Lap(v,w) := \dprod{dv}{dw} + \dprod{\delta v}{\delta w},
		\qquad
		\Dir(v) := \Lap(v,v).
	\end{equation}
	\end{subenum}
\end{notation}
\begin{definition}
	\label{def:SobolevSpaces}
	Define the following six norms on each $\Omega^k$:
	\[
		\renewcommand\arraycolsep{0.3ex}
		\begin{array}{llclcl}
		\ibetrag[\Sob^{1,0}] v^2 & := \halfquad \ibetrag[\Leb^2] v^2 &+& \ibetrag[\Leb^2]{dv}^2 \\[0.7ex]
		\ibetrag[\Sob^{0,1}] v^2 & := \halfquad \ibetrag[\Leb^2] v^2 &+& \ibetrag[\Leb^2]{\delta v}^2 \\[0.7ex]
		\ibetrag[\Sob^{1,1}] v^2 & := \halfquad \ibetrag[\Leb^2] v^2 &+& \ibetrag[\Leb^2]{dv}^2 \!\!&+& \ibetrag[\Leb^2]{\delta v}^2
									= \ibetrag[\Leb^2] v^2 + \Dir(v) \\[0.7ex]
		\ibetrag[\Sob^{1+1}] v^2 & := \halfquad \ibetrag[\Sob^{1,1}] v^2 \!&+& \ibetrag[\Leb^2]{d\delta v}^2
										\!\!&+& \ibetrag[\Leb^2]{\delta dv} \\[0.7ex]
		\ibetrag[\Sob^1]     v^2 & := \halfquad \ibetrag[\Leb^2] v^2 &+& \ibetrag[\Leb^2]{\nabla v}^2 \\[0.7ex]
		\ibetrag[\Sob^2]     v^2 & := \halfquad \ibetrag[\Sob^1] v^2 &+& \ibetrag[\Leb^2]{\nabla^2 v}^2\!\!\!\!.
		\end{array}
	\]
	Let $\Sob^{1,0}\Omega^k$ etc. be the completion of $\Omega^k$
	with respect to these norms.
	The $\Leb^r$, $\SobW^{1,r}$ and $\SobW^{2,r}$ norms
	are the usual modification of the $\Leb^2$, $\Sob^1$ and $\Sob^2$
	norms for exponents $r \neq 2$.
\end{definition}
\begin{observation}
	\begin{subenum}
	\item	$(\Sob^{1,0}\Omega, d)$ is a cochain and $(\Sob^{0,1}\Omega,\delta)$
	is a chain Hilbert complex, that means that $d$ or $\delta$
	are bounded linear operators with $d^2 = 0$ or $\delta^2 = 0$ respectively
	(to be notationally precise, a Hilbert complex requires $d$ or $\delta$ only to
	be closed operators).
	\item	The $\Sob^1$ norm dominates the $\Sob^{1,0}$ and the
	$\Sob^{0,1}$ norm; the $\Sob^2$ norm dominates all the other norms.
	For functions, $\Sob^{0,1} = \Leb^2$ and $\Sob^{1,0} = \Sob^{1,1} = \Sob^1$.
	\end{subenum}
\end{observation}
\bibrembegin
\begin{remark_nn}
	\citet[sec. 1.3]{Schwarz95} uses a different definition of $\Sob^2$ which depends on
	local choices of orthonormal bases. He then uses the $\Sob^1$
	and $\Sob^2$ norms to control the exterior (co-)derivatives. We consider
	the use of $\Sob^{1,0}$ and similar norms a sharper tool for this, as only
	the actually needed derivatives have to exist.
	\citet[eqn. 3.4.4]{Jost11} writes $\Sob^2$ for what we call $\Sob^{1+1}$.
\end{remark_nn}
\bibremend
\begin{fact}[\citealt{Bruening92}, corr. 2.6]
	The spaces $\Sob^{1,1}\Bdry^k$ are closed in $\Sob^{1,1}\Omega^k$
	if and only if $\Sob^{1,1}\Bdry_k^*$ are closed in $\Sob^{1,1}\Omega^k$.
	If this is the case, and if $\Harm^k$ is finite-dimensional,
	then $(\Sob^{1,1}\Omega^k,d)$ is called
	a \begriff{Fredholm complex.}
\end{fact}

\subsection{Laplace Operator and Dirichlet Problem}

\begin{observation}
	\label{lem:strongLaplacianRepresentation}
	Directly from Green's formula \eqnref{eqn:Greensformula},
	one gets for $v \in \Sob^{1+1}$, $w \in \Sob^{1,1}$
	\[
		\Lap(v,w) = \dprod{(d\delta + \delta d)v}w
	\]
	in either of
	these four cases:
	\begin{align*}
		\tang^* w = 0, &\,\,	\nor^* w = 0 &
		\tang^* w = 0, &\,\, 	\tang^* \delta v = 0 \\
		\nor^* w = 0,  &\,\,	\nor^* dv = 0 &
		\nor^* dv = 0, &\,\,	\tang^* \delta u = 0.
	\end{align*}
\end{observation}
\begin{definition}
	\label{def:weakStrongLaplacian}
	\begin{subeqns}
	The \begriff{strong Laplacian} is $\laplace := d \delta + \delta d: \Sob^{1+1}\Omega^k
	\to \Leb^2\Omega$. The \begriff{weak Laplacian} is $L: \Sob^{1,1}\Omega^k \to (\Sob^{1,1}\Omega^k)^*,
	v \mapsto \Lap(\argdot,v)$. The \begriff{strong Dirichlet problem} is to find
	$u \in \Sob^{1+1}\Omega^k$ with
	\begin{equation}
		\laplace u = f,
		\qquad \tang^* u = 0,\, \tang^* \delta u = 0.
	\end{equation}
	The \begriff{weak Dirichlet problem} is to find $u \in \Sob^{1,1}\Omega^k_\tang$
	with $Lu = f$ in $(\Sob^{1,1}\Omega^k_\tang)^*$, that means
	\begin{equation}
		\dprod{du}{dv} + \dprod{\delta u}{\delta v} = \dprod fv
		\qquad \forall v \in \Sob^{1,1}\Omega^k_\tang.
	\end{equation}
	Such a $u$ is called a \begriff{Dirichlet potential} for $f$.
	\end{subeqns}
\end{definition}
\begin{fact_nn}[\begriff{Dirichlet principle}]
	A form $u \in \Sob^{1,1}\Omega^k_\tang$ is a solution of
	the weak Dirichlet problem if and only if it minimises
	$\Dir(v) - \dprod fv$ over all $v \in \Sob^{1,1}\Omega^k_\tang$.
\end{fact_nn}
\begin{remark_nn}
	For sections of smooth vector bundles over $M$, the
	trace of the second covariant derivative
	gives a ``metric'' Laplace operator $\spur \nabla^2$,
	connected to our Laplacian or ``Laplace--Beltrami'' operator by
	the Weizenböck formula (\citealt[thm 4.3.3.]{Jost11}),
	which we do not use, and only mention to avoid confusion.
	They agree if and only if $Mg$ is flat.
\end{remark_nn}
\begin{proposition}
	\label{prop:weakAndStrongDirichletProblemOnkForms}
	If $u \in \Sob^{1,1}\Omega^k_t$ is a solution of the weak
	Dirichlet problem and is in addition contained in $\Sob^{1+1}\Omega^k$,
	then it solves the strong Dirichlet problem.
\end{proposition}
\begin{proof}
	Suppose $\Lap(u,v) = \dprod fv$ for all $v \in \Sob^{1,1}\Omega^k_\tang$.
	Then a fortiori this holds for $v \in \Sob^{1,1}\Omega^k_{\tang,\nor}$
	and so, by \ref{lem:strongLaplacianRepresentation},
	$\dprod{\laplace u}v = \dprod fv$ for all $v \in \Sob^{1,1}\Omega^k_{\tang,\nor}$,
	which shows $\laplace u = f$ by the fundamental lemma. As the
	vanishing boundary values for $v$ were only needed in the
	use of Green's formula, not in the $\Leb^2$ testing,
	we can infer $\dprod{\laplace u} v
	= \dprod fv$ for all $v \in \Sob^{1,1}\Omega^k$ by continuity.
	But as $\Sob^{1,1}\Omega^k$ contains functions whose
	normal trace does not vanish, $\dprod{\laplace u} v = \Lap(u,v)$
	can only hold if $\tang \delta u = 0$,
\end{proof}
\begin{remark_nn}
	If we had used, for the weak Dirichlet problem,
	$\Sob^{1+1}\Omega^k$ for the space of
	test functions instead of $\Sob^{1,1}\Omega^k_\tang$, then
	the space of Dirichlet potentials for $f = 0$ would
	agree with $\Sob^{1,1}\Harm^k_\tang$. But by our definition, this
	requires the additional assumption $u \in \Sob^{1+1}\Omega^k$.
\end{remark_nn}
\begin{observation}
	\begin{subeqns}
	By definition of the spaces involved and Green's formula
	\eqnref{eqn:Greensformula}, one directly obtains:
	\begin{equation}
		\label{eqn:inclusionsOfExtDerivSpaces}
		\begin{aligned}
		d(\Sob^{1,0}\Omega^{k-1}_\tang) & \subset \delta(\Sob^{1+1}\Omega^{k+1})^\perp &
		\qquad\quad
		\delta(\Sob^{0,1}\Omega^{k+1}_\nor) & \subset d(\Sob^{1+1}\Omega^{k-1})^\perp \\
		\Sob^{1,0}\Cyc^k & \supset \delta(\Sob^{1+1}\Omega^{k+1}_\nor)^\perp &
		\Sob^{0,1}\Cyc_k^* & \supset d(\Sob^{1+1}\Omega^{k-1}_\tang)^\perp \\
		\Sob^{1,1}\Harm^k_\nor & \perp \Sob^{0,1}\Bdry^k &
		\Sob^{1,1}\Harm^k_\tang & \perp \Sob^{1,0}\Bdry_k^*
		\end{aligned}
	\end{equation}
	As $\Sob^{1,0}\Omega^k$ and $\Sob^{0,1}\Omega^k$ are
	completions of spaces with $d^2 = 0$ and $\delta ^2 = 0$,
	this property carries over:
	\begin{equation}
		\label{eqn:weakSecondExtDeriv}
		d(\Sob^{1,0}\Omega^k) \subset \Sob^{1,0}\Cyc^k,
		\qquad
		\delta(\Sob^{0,1}\Omega^k) \subset \Sob^{0,1}\Cyc_k^*.
	\end{equation}
	\end{subeqns}
\end{observation}
\begin{proposition}[\begriff{Poincaré inequality}]
	Let $(\Sob^1\Omega^k,d)$ be a Fredholm complex where the inclusion
	map $\Sob^1\Omega^k \to \Leb^2\Omega^k$ is compact. Then:
	\begin{subenum}
	\item	\label{prop:coercivityOnHarmPerp}
	$\Dir$ is $\Sob^1$-coercive on $\Sob^1(\Harm^k)^\perp$, that
	means there is $C_\boxdot > 0$ with
	\[
		\ibetrag[\Sob^1] v^2 \leq C_\boxdot \Dir(v)
		\qquad \forall v \in \Sob^1\Omega^k, v \perp \Harm^k.
	\]
	\item	\label{prop:coercivityOnHarmPerpTang}
	If the trace operator $v \mapsto \tang^* v$ is a continuous
	mapping $\Sob^1\Omega^k(M) \to \Leb^2\Omega^k(\Rand M)$,
	then $\Dir$ is $\Sob^1$-coercive on $(\Harm^k_\tang)^\perp_\tang$,
	that means there is $C_\boxdot > 0$ with
	\[
		\ibetrag[\Sob^1] v^2 \leq C_\boxdot \Dir(v)
		\qquad \forall v \perp \Harm^k_\tang,\, \tang^* v = 0.
	\]
	\item	\label{prop:poincareIneqVanishingBdryValues}
	If $\homfont A$ is a closed affine subspace in $\Sob^1\Omega^k$
	that does not include constant forms $\neq 0$,
	then there is $C_\boxdot > 0$ with
	\[
		\ibetrag[\Leb^2] v^2 \leq C_\boxdot \ibetrag[\Leb^2]{\nabla v}^2
		\qquad \forall v \in \homfont A.
	\]
	Examples are $\homfont A = \{ u \in \Sob^1\Omega^k \mit u|_{\Rand M} = 0\}$ if $M$
	has a boundary, or $\homfont A = \{u \in \Sob^1\Omega^k \mit \fint_M u = 0\}$ in cases
	where the integral $\int_U u$ makes sense.
	By linear translation, the latter one gives the consequence
	$\ibetrag[\Leb^2]{v - \fint_M v} \leq C_\boxdot \ibetrag[\Leb^2]{\nabla v}$
	for functions $v \in \Sob^1\Omega^0$.
	\end{subenum}
\end{proposition}
\begin{proof}
	It obviously suffices to show the last claim with $\ibetrag[\Sob^1] v^2$
	instead of $\ibetrag[\Leb^2] v^2$ on the left-hand side.
	Then the proof
	always follows the same lines: If the claim is wrong,
	there has to be a sequence $\ser[i]v \subset \Sob^1\Omega^k$ with
	$\ibetrag[\Sob^1] {v_i} = 1$, and the right-hand side
	tends to $0$. Because this sequence is bounded in $\Sob^1\Omega^k$,
	there has to be a weakly convergent subsequence, which we again
	denote by $\ser v$. This will suffice to extract a
	contradiction in all three cases.
	
	\textit{ad primum:}
	Because $\Harm^k$ is finite-dimensional, it is closed,
	and so is $(\Harm^k)^\perp$, that means $v \perp \Harm^k$.
	At the same time, $\Dir(v) = \lim \Dir(v_i) = 0$,
	so $v \in \Sob^{1,1}\Harm^k$. Therefore, $v = 0$, but
	at the same time $\ibetrag[\Leb^2] v = \lim \ibetrag[\Leb^2]{v_i} > 0$
	because the imbedding
	$\Sob^1\Omega^k \to \Leb^2\Omega^k$ is compact.
	
	\textit{ad sec.:}
	By assumption, $(\Harm^k_{\tang})_{\tang}$ is closed,
	so the argument works again, as $\tang^* v = \lim \tang^* v_i = 0$.
	
	\textit{ad tertium:}
	Here the convergence of the right-hand side
	means $\nabla v_i \to 0$ strongly in $\Leb^2$, hence
	$\nabla v = 0$, so $v$ has to be constant almost everywhere,
	so $v = 0$ by assumption on $\homfont A$,
\end{proof}
\begin{remark}
	\begin{subenum}
	\item
	It is common to prove the last part constructively,
	see \ref{prop:poincareIneqConstructive}. We are not
	aware of a constructive proof for the first and second case.
	\item	\label{rem:PoincareConstantWithoutScalingFactor}
	By a scaling argument, one can see that $C_\boxdot
	= \tilde C_\boxdot \diam M$ with a constant $\tilde C_\boxdot$
	that does not depend on the size of $M$.
	\end{subenum}
\end{remark}
\begin{proposition}
	\label{prop:DirichletProblemForkForms}
	Situation as in \ref{prop:coercivityOnHarmPerpTang}.
	Let $f \in \Leb^2\Omega^k$ with $f \perp \Sob^{1,1}\Harm^k_\tang$.
	Then there is exactly one $u \in (\Sob^{1,1}\Harm^k_\tang)^\perp_\tang$
	with $\Lap(u,v) = \dprod fv$ for all $v \in \Sob^{1,1}\Omega^k_\tang$.
\end{proposition}
\begin{proof}
	By the Lax--Milgram theorem, there is exactly one
	$u \in (\Sob^{1,1}\Harm^k_\tang)^\perp_\tang$ with
	$\Lap(u,v)$ $= \dprod fv$ for all $v \in (\Sob^{1,1}\Harm^k_\tang)^\perp_\tang$.
	Now observe that not only $\Sob^{1,1}\Omega^k_\tang =
	\Sob^{1,1}\Harm^k_\tang \oplus (\Sob^{1,1}\Harm^k_\tang)^\perp$,
	but that the second summand must also have zero boundary values, so
	$\Sob^{1,1}\Omega^k_\tang =
	\Sob^{1,1}\Harm^k_\tang \oplus (\Sob^{1,1}\Harm^k_\tang)^\perp_\tang$.
	So everything that is missing is to prove this equality also
	for $v \in \Sob^{1,1}\Harm^k_\tang$. But that is not difficult:
	On the one hand, $\Lap(\argdot,v) = 0$ for such $v$, on the
	other $\dprod fv = 0$ by assumption on $f$,
\end{proof}
\begin{remark_nn}
	\begin{subenum}
	\item
	For each $u^* \in \Sob^{1,1}\Harm^k_\tang$, one also has
	$L(u + u^*) = f$. So the solution is unique up to
	harmonic components. This non-uniqueness for manifolds of
	higher genus can indeed be observed in numerics,
	cf. \citet[section 2.3.3]{Arnold10}.
	\item
	If $f$ is not orthogonal to $\Sob^{1,1}\Harm^k_\tang$,
	then there is an orthogonal projection $p$ of $f$
	to this space and a Dirichlet potential $u$ for $f - p$.
	\end{subenum}
\end{remark_nn}
\noindent

\subsection{Hodge Decompositions}

\begin{proposition}[\begriff{weak Hodge decomposition},
	\citealt{Bruening92}, lemma 2.1]
	\label{prop:HodgeDecompositionInH10}
	There is an orthogonal
	decomposition
	\[
		\Sob^{1,0}\Omega^k = d(\Sob^{1,0}\Omega^{k-1}_\tang) \oplus \Sob^{1,0}(\Cyc^k)^\perp \oplus \Sob^{1,0}\Harm^k,
	\]
	where $\Sob^{1,0}\Harm^k := \Sob^{1,0}\Cyc^k \cap d(\Sob^{1,0}\Omega^{k-1}_\tang)^\perp$.
\end{proposition}
\begin{proof}
	By definition of $\Sob^{1,0}\Harm^k$, this
	sum exhausts $\Sob^{1,0}\Omega^k$, and the last
	summand is orthogonal to the two other ones. It only
	remains to show the orthogonality of $d(\Sob^{1,0}\Omega^{k-1}_\tang)$
	and $\Sob^{1,0}(\Cyc^k)^\perp$. So let $u = da$ with $a \in \Sob^{1,0}\Omega^{k-1}_\tang$.
	Then by \eqnref{eqn:weakSecondExtDeriv} $u \in \Sob^{1,0}\Cyc^k$, hence
	it is perpendicular to each element in $\Sob^{1,0}(\Cyc^k)^\perp$,
\end{proof}
\begin{proposition}
	[\begriff{strong Hodge decomposition},
	\citealt{Bruening92}, cor. 2.5; \citealt{Schwarz95}, thm. 2.4.2]
	\label{prop:HodgeDecomposition}
	For a Fredholm complex,
	there is an orthogonal decomposition
	\[
		\Sob^{1,1}\Omega^k = d(\Sob^{1+1}\Omega^{k-1}_\tang) \oplus
		\delta(\Sob^{1+1}\Omega^{k+1}_\nor) \oplus \Sob^{1,1}\Harm^k.
	\]
	In other words: each $u \in \Sob^{1,1}\Omega^k$ can
	be decomposed as $u = da + \delta b + c$ with
	$\tang a = 0$, $\nor b = 0$, $dc = 0$, and $\delta c = 0$.
	The parts $a$ and $b$ can be computed as minimisers of
	$F[u](a) = \dprod{da}{da} - 2 \dprod{da}u$ over
	$a \in \Sob^{1+1}(\Cyc^{k-1})^\perp_\tang$ and
	$G[u](b) = \dprod{\delta b}{\delta b} - 2 \dprod{\delta b}u$
	over $b \in \Sob^{1+1}(\Cyc^*_{k+1})^\perp_\nor$ respectively.
\end{proposition}
\begin{proof}
	By the Fredholm property,
	$\Sob^{1,1}\Harm^k$ is closed and hence convex. By the
	projection theorem (cf. e.\,g. \citealt[thm. 2.2]{Alt06}),
	$\Sob^{1,1}\Omega^k$ has an orthogonal decomposition
	$\Sob^{1,1}\Omega^k = \Sob^{1,1}\Harm^k \oplus
	(\Sob^{1,1}\Harm^k)^\perp$. So everything we have
	to show is
	\begin{align*}
		\Sob^{1,1}(\Harm^k)^\perp & = d(\Sob^{1+1}\Omega^{k-1}_\tang)
			\oplus \delta(\Sob^{1+1}\Omega^{k+1}_\nor) \\
		\Leftrightarrow \quad
		\Sob^{1,1}\Harm^k & = \big(d(\Sob^{1+1}\Omega^{k-1}_\tang)
			\oplus \delta(\Sob^{1+1}\Omega^{k+1}_\nor) \big)^\perp.
	\end{align*}
	The inclusion $\subset$ is clear by the last
	part of \eqnref{eqn:inclusionsOfExtDerivSpaces}.
	For the other direction consider $u \in \Sob^{1,1}\Omega^k$
	such that
	\begin{align*}
		\dprod u {dv} & = 0 & \forall v \in \Sob^{1,0}\Omega^{k-1}_\tang, \\
		\dprod u {\delta w} & = 0 & \forall w \in \Sob^{0,1}\Omega^{k+1}_\nor.
	\end{align*}
	Because of Green's formula, this means
	\begin{align*}
		\dprod {\delta u} v & = 0 & \forall v \in \Sob^{1,0}\Omega^{k-1}_\tang,\\
		\dprod {d u} w & = 0 & \forall w \in \Sob^{0,1}\Omega^{k+1}_\nor.
	\end{align*}
	These $v$ and $w$ suffice to test for $du = 0$ and $\delta u = 0$
	in the interior of $M$, so $u \in \Sob^{1,1}\Harm^k$.
	
	For the variational property, observe that the orthogonal
	projection $da$ of $u \in \Sob^{1,1}\Omega^k$ onto $d(\Sob^{1+1}\Omega^{k-1}_\tang)$
	fulfills $\dprod{u - da}{dv} = 0$ for every $v \in \Sob^{1+1}\Omega^{k-1}_\tang$,
	which is exactly the optimality condition for $F$. The
	minimiser is unique up to elements of $\Cyc^k$, for
	which reason we only seek $a$ in $(\Cyc^k)^\perp$. The
	analogous argument applies for $G$,
\end{proof}
\begin{remark}
	\label{rem:HodgeDecompositionL2}
	By assumption, the ranges of $d$ and $\delta$ are closed,
	so a continuity argument also shows that
	$\Leb^2[d(\Sob^{1,0}\Omega^k_\tang \oplus
	\delta(\Sob^{0,1}\Omega^k_\nor)]^\perp$ is the
	$\Leb^2$ completion $\bar \Harm^k$ of $\Harm^k$,
	which gives the $\Leb^2$ Hodge decomposition
	\[
		\Leb^2\Omega^k = d(\Sob^{1,0}\Omega^{k-1}_\tang) \oplus
		\delta(\Sob^{0,1}\Omega^{k+1}_\nor) \oplus \bar \Harm^k.
	\]
\end{remark}
\begin{proposition}
	[\begriff{Hodge--Friedrichs decomposition},
	\citealt{Schwarz95}, thm 2.4.8]
	\label{prop:HodgeFriedrichsDecomposition}
	For a Fredholm complex that is $\Sob^{1+1}$-regular, that
	means the Dirichlet potential of an $\Leb^2$ right-hand side
	is in $\Sob^{1+1}\Omega^k$, there is an orthogonal
	decomposition
	\[
		\begin{split}
		\Sob^{1,1}\Harm^k & = \Sob^{1,1}\Harm^k_{\tang} \oplus \Sob^{1,1}\Harm^k \cap \Sob^{1,0}\Bdry_k^*, \\
			& = \Sob^{1,1}\Harm^k_{\nor} \oplus \Sob^{1,1}\Harm^k \cap \Sob^{0,1}\Bdry^k.
		\end{split}
	\]
	Together with \ref{prop:HodgeDecomposition}, this gives
	the decomposition
	\[
		\begin{split}
		\Sob^{1,1}\Omega^k
			& = d(\Sob^{1+1}\Omega^{k-1}_\tang)
				\oplus \delta(\Sob^{1+1}\Omega^{k+1}_\nor)
				\oplus \Sob^{1,1}\Harm^k_{\tang} \oplus \Sob^{1,1}\Harm^k \cap \Sob^{1,0}\Bdry_k^* \\
			& = d(\Sob^{1+1}\Omega^{k-1}_\tang)
				\oplus \delta(\Sob^{1+1}\Omega^{k+1}_\nor)
				\oplus \Sob^{1,1}\Harm^k_{\nor} \oplus \Sob^{1,1}\Harm^k \cap \Sob^{0,1}\Bdry^k.
		\end{split}
	\]
\end{proposition}
\begin{proof}
	We only prove the first decomposition, the second one is literally the same.
	By the last statement of \eqnref{eqn:inclusionsOfExtDerivSpaces},
	$\Sob^{1,1}\Harm^k_{\tang} \perp \Sob^{1,0}\Bdry_k^*$. So if the
	decomposition exists, it will be orthogonal.
	We are done if we can show that every $u \in \Sob^{1,1}\Harm^k$ with
	$u \perp \Sob^{1,1}\Harm^k_{\tang}$ is a coboundary, that means there is
	some $v \in \Sob^{1+1}\Omega^{k-1}$ with $u = \delta v$.
	
	Therefore, suppose $u \in \Sob^{1,1}\Harm^k$ and
	$u \perp \Sob^{1,1}\Harm^k_{\tang}$. Then by
	\ref{prop:DirichletProblemForkForms} there is a Dirichlet potential
	$w \in \Sob^{1,1}\Omega^k_\tang$ for $u$. Let $v := dw$.
	
	(I.) It holds $u - \delta v \perp \Sob^{1,1}\Harm^k_{\tang}$ due to
	the assumption on $u$ and $\delta v \in \Sob^{1,0}\Bdry_k^*$.
	
	(II.) As $\delta u = 0$ and $\delta^2 v = 0$, we have
	$\delta(u - \delta v) = 0$. Because the Dirichlet problem
	is $\Sob^{1+1}$-regular, $u = (d\delta + \delta d)w$ and
	hence $d(u - \delta v) = d(d\delta w + \delta d w - \delta d w)
	= d(d \delta w) = 0$. This shows that $u - \delta v \in \Sob^{1,1}\Harm^k$.
	And because $u - \delta v = d \delta w$ is a boundary, it has
	vanishing tangent component: Take any $(k-1)$-dimensional domain
	$U \subset \Rand M$. There is a domain $U' \subset M$ with
	$U = U' \cap \Rand M$, and $\int_U \tang^*(u - \delta v)
	= \int_{U'} d(u - \delta v) = 0$. For this reason,
	$u - \delta v \in \Sob^{1,1}\Harm^k_\tang$.
	
	Now $u - \delta v$ is at the same time in some space
	and its orthogonal complement, which can only
	hold if $u - \delta v = 0$,
\end{proof}

\subsection{Mixed Form of the Dirichlet Problem in $\Omega^k$}

\begin{observation}
	\label{obs:mixedFormDirichletPb}
	The Dirichlet problem also has a \begriff{mixed form} without
	coderivatives
	(\citealt[sec. 7.1]{Arnold06}): If one introduces the
	auxiliary variable $\sigma$, which replaces
	$\delta u$ in a weak sense, i.\,e. which fulfills
	$\dprod \sigma\tau = \dprod{d\tau}u$ for all
	$\tau \in \Sob^{1,0}\Omega^{k-1}_\tang$, then
	$Lu = f-p$ can be written as
	\begin{align*}
		\dprod{d\sigma} v + \dprod{du}{dv} & = \dprod{f-p}v && \forall v \in \Sob^{1,0}\Omega^k_\tang, \\
		\dprod \sigma\tau - \dprod u{d\tau} & = 0 && \forall \tau \in \Sob^{1,0}\Omega^{k-1}_\tang, \\
		\dprod u q & = 0 && \forall q \in \Sob^{1,1}\Harm^k_\tang.
	\end{align*}
	Let $\homfont S := \Sob^{1,0}\Omega^{k-1}_\tang
	\times \Sob^{1,0} \Omega^k_\tang \times \Sob^{1,1}\Harm^k_\tang$.
	By computing the Euler--Lagrange equation, one sees
	that $(\sigma,u,p) \in \homfont S$ solves the
	equations above if and only if it is a minimiser of
	\[
		I(\sigma,u,p) := \smallfrac 12 \dprod \sigma \sigma - \dprod{d\sigma} u
			- \smallfrac 12 \dprod{du}{du} + \dprod {f-p}u.
	\]
	In such a critical point, every $(\sigma,0,0) \in \homfont S$
	is a descent direction, and every $(0, v, q) \in \homfont S$
	is an ascent direction, so $I$ has a saddle point at $(\sigma, u, p)$.
\end{observation}
\begin{proposition}
	[\citealt{Arnold06}, thm. 7.2]
	\label{prop:wellposednessOfWeakDirichletProblemOnkForms}
	In a Fredholm complex,
	the weak Dirichlet problem
	in mixed formulation is well-posed, that means:
	There is a solution $(\sigma,u,p) \in \homfont S$
	for every $f \in \Leb^2\Omega^k$, and there is a constant $c$
	only depending on $M$ such that
	$\ibetrag[\homfont S]{\sigma,u,p} \leq c\ibetrag[\Leb^2] f$.
\end{proposition}
\begin{proof}
	\begin{subeqns}
	For $s = (\sigma,u,p)$ and $t = (\tau,v,q)$, let
	\begin{equation}
		\label{eqn:defBilfb}
		b(\sigma,u,p; \tau,v,q) := \dprod \sigma \tau - \dprod u {d\tau}
		+ \dprod{d\sigma} v + \dprod{du}{dv} + \dprod p v - \dprod u q,
	\end{equation}
	let $B: \homfont S \to \homfont S^*$ be the linear
	operator with $b(s,t) = {\dprod{Bs}t}_{\homfont S^*,\homfont S}$
	and $F: t \mapsto \dprod fv_{\Leb^2}$. We have
	to show that there is a solution to the operator equation
	$B(s) = F$ in $\homfont S^*$.
	
	There is an inf-sup condition for $b$:
	According to \ref{prop:HodgeDecompositionInH10}, decompose $u = da + x + y$
	with $a \in \Sob^{1,0}\Omega^k_\tang \cap \Cyc^{k-1}$ and
	$x \perp \Cyc^k$. Then by \ref{prop:coercivityOnHarmPerpTang},
	we have $\ibetrag[\Sob^{1,0}] a \leq C_\boxdot \ibetrag[\Leb^2] a
	\leq C_\boxdot \ibetrag[\Leb^2] u$ and by
	\ref{prop:coercivityOnHarmPerp},	$\ibetrag[\Sob^{1,0}] x
	\leq C_\boxdot \ibetrag[\Leb^2]{dx} = C_\boxdot \ibetrag[\Leb^2]{du}$.
	With $\tau = \sigma - C_\boxdot^{-2} a$, $v = u + d\sigma + p$
	and $q = p - y$, one obtains (all norms are $\Leb^2$ norms here)
	$b(\sigma,u,p; \tau,v,q) = \ibetrag \sigma^2 - C_\boxdot^{-2} \dprod \sigma a
	+ C_\boxdot^{-2} \ibetrag{da}^2 + \ibetrag{d\sigma}^2 + \ibetrag{du}^2
	+ \ibetrag p^2 + \ibetrag y^2$, which is (for $C_\boxdot \geq 1$)
	greater than $C_\boxdot^{-2}(\ibetrag[\Sob^{1,0}]\sigma^2
	+ \ibetrag[\Sob^{1,0}] u^2 + \ibetrag[\Leb^2] p^2)$.
	As $\ibetrag[\homfont S]{\tau, v,q}
	\leq \alpha \ibetrag[\homfont S]{\sigma,u,p}$
	for some $\alpha \in \R$,
	we have shown that there is a constant
	$\gamma$ with
	\begin{equation}
		\label{eqn:infsupCondForb}
		\sup_{t \in \homfont S} \frac{b(s,t)}{\ibetrag[\homfont S] t} \geq \gamma \ibetrag[\homfont S]{s}
		\qquad \forall s \in \homfont S.
	\end{equation}
	Once this inf-sup-condition is established,
	the way is well-known
	(\citealt[thm. 5.2.1]{Babuska72}):
	Consider the linear operator $B: \homfont S \to \homfont S^*$
	belonging to $b$. Because of $\ibetrag[\homfont S^*]{Bs} \geq \gamma \ibetrag[\homfont S] s$,
	it must be injective. And as we have that for any $t \in \homfont S$ there
	is some $s \in \homfont S$ with $b(s,t) \neq 0$, the image of $B$
	must be the whole space $\homfont S^*$ by the closed range theorem.
	Using the inf-sup condition once again, we get that
	$\ibetrag[\homfont S^*]\alpha \geq \gamma \ibetrag[\homfont S]{B\inv \alpha}$,
	which shows that $B\inv$ is continuous,
	\end{subeqns}
\end{proof}
\begin{remark}
	\label{rem:GaffneysInequality}
	The only thing that \emph{cannot} be inferred from our
	high-level point of view is
	that $(\Sob^{1,0}\Omega,d)$ indeed forms a Fredholm
	complex, and that the
	imbedding $\Sob^1\Omega^k \to \Leb^2\Omega^k$ is compact.
	The finite-dimensionality of $\Harm^k$
	and hence the Fredholm property is proven by the inequality
	$\ibetrag[\Sob^1] v \leq C (\Dir(v) + \ibetrag[\Leb^2] v)$
	of \cite{Gaffney51}, see e.\,g. \citet[cor. 2.1.6]{Schwarz95}.
	The compact embedding is Rellich's inequality
	(see e.\,g. \citealt[A\,6.1]{Alt06}), which carries
	over from the Euclidean case without modification. Both need the vector
	bundle structure of $(\Sob^{1,0}\Omega,d)$, but they
	are obviously also true if $(\Sob^{1,0}\Omega,d)$ is
	a complex of finite-dimensional vector spaces. Therefore
	all proofs above literally carry over to the
	``discrete exterior calculus'' from section \ref{sec:dec}.
	A very preliminary version of this attempt of formulation
	exterior calculus without recurrence to the vector-bundle
	structure has been given in \cite{Deylen12}.
\end{remark}
\begin{definition}
	\label{def:variationalProblem}
	When we speak of \begriff{variational problems} in the
	forthcoming sections, we will
	always refer to the Hodge decomposition, the
	Dirichlet problem and other strongly elliptic problems.
\end{definition}

%
\newsectionpage
\section{Geometry of a Single Simplex}
\label{sec:simplexGeometry}
%
%
	As typical domain for the parametrisation of
	simplices, the numerical community mostly uses the
	$n$-dimensional \begriff{unit simplex}
	\[
		D := \conv(0,e_1,\dots,e_n) = \{p \in \R^n_{\geq 0} \mit p \cdot 1_n \leq 1\},
	\]
	where $1_n$ is the vector in $\R^n$ with all entries $1$,
	and $e_i$ is the $i$'th Euclidean unit vector.
	We will
	investigate parametrisation of simplices over $D$
	for given edge lengths and see how the Riemannian
	metric over $D$ changes when those edge lengths
	are slightly distorted. In contrast,	
	geometers tend to employ
	the \begriff{standard simplex}
	\[
		\stds := \conv(e_0,\dots,e_n) = \{\lambda \in \R^{n+1}_{\geq 0} \mit \lambda \cdot 1_{n+1} = 1 \}
	\]
	for the same purpose (here and in the following, we will
	use the enumeration $e_0,\dots,e_n$ for the canonical Euclidean
	basis of $\R^{n+1}$). Although the parametrisation over $\stds$ is not
	a parametrisation in the strict sense, as not some $\R^n$
	itself is used, but some linear subspace of it, we
	will see that there will be no problems
	with this additional direction.

\subsection{The Unit Simplex}
%
\parag{Metric on $TD$.}
	Consider points $p_0,\dots,p_n \in \R^n$ that are
	supposed to be vertices of a simplex $s$. As we are
	only interested in its isometry-invariant properties,
	we can assume that $p_0$ is the origin of $\R^n$.
	Then the matrix $P := [p_1|\cdots|p_n]$ represents a linear
	map $D \to s$. The first fundamental form has entries
	\label{sec:metricOnUnitSimplex}
	\begin{subeqns}
	\[
		C_{ij} := (P^t P)_{ij} = p_i \cdot p_j,
		\qquad i,j = 1,\dots,n,
	\]
	and the volume of $s$ is computable as
	$\vol s = \smallfrac 1 {n!} (\Det C)^\half$.
	If $\ell_{ij} = \absval{p_i - p_j}$ are the
	edge lengths of $s$, we have
	by the cosine law
	\begin{equation}
		\label{eqn:lawOfCosines}
		C_{ij} = \smallfrac 1 2 (\ell_{0i}^2 + \ell_{0j}^2 - \ell_{ij}^2).
	\end{equation}
	If only a system of prescribed ``edge lengths'' $\bar \ell_{ij}$
	is given, then there is a simplex with such edge lengths if
	and only if the matrix with entries $\frac 1 2 (\bar \ell_{0i}^2 + \bar \ell_{0j}^2 -
	\bar \ell_{ij}^2)$ is positive definite.
\end{subeqns}

\parag{Metric on $T^* D$.}
	Let $D_i$, $i = 1,\dots,n$, be the \begriff{facet} (subsimplex of
	codimension $1$) of $D$ opposite to the vertex $e_i$.
	The vector $e_i$ is normal to $D_i$,
	and as normal directions transform with
	$P^{-t}$, the vectors $v^i := P^{-t} e_i$, $i = 1,\dots,n$,
	are normal to the facets $s_i$ of $s$. In other words,
	$P\inv$ has the normals $v^i$ as rows.
	At the same time, these $v^i$ are the gradients
	of the barycentric coordinate functions $\lambda^i$,
	defined by the representation $p = \lambda^i p_i$
	for any point in $s$. The length
	of these gradients decreases as the simplex'
	height $h_i$ above $p_i$ decreases, more precisely
	\label{sec:cometricOnUnitSimplex}
	\begin{subeqns}
	\begin{equation}
		\label{eqn:defOuterNormalvi}
		v^i = \grad \lambda^i \perp s_i,
		\qquad
		\absval{v^i} = \smallfrac 1 {h_i},
	\end{equation}
	this in particular implies $\absval{s_i} / \absval{v^i} =
	\smallfrac 1 n \vol s$ for each $i$.
	This formula does not only hold for $i = 1,\dots,n$,
	but also the appropriately-scaled normal $v^0$ opposite to the
	origin $p_0$ is the gradient of $\lambda^0$.
	Define $V := [v^1|\cdots|v^n]$.
	As the barycentric coefficients sum up to one,
	$v^0 = - V 1_n$.
	By definition of $V$, we have
	$V^t P = \1$, which means that the $v^i$ and $p_j$ form
	a ``biorthogonal system''
	(\citealt[thm. 1.1.2]{Fiedler11}).
	The matrix
	\[
		\tilde Q^{ij} := (V^t V)^{ij} = v^i \cdot v^j,
		\qquad i,j = 1,\dots,n,
	\]
	is inverse to $C_{ij}$ and hence represents the scalar
	product of the cotangent space. Clearly $V^t = P \tilde Q$, that means
	$v^k = p_j \sprod{v^j}{v^k}$,
	so $v^0 = - V 1_n = p_j \sprod{v^j}{v^0}$.
	As $p_j$ also equals the edge vector
	$e_{j0}$ (due to $p_0 = 0$),
	this shows the translation-independent form
	\begin{equation}
		\label{eqn:normalRepresentationByEdges}
		v^i = e_{kj} \, v^k \cdot v^j
		\qquad \forall i = 0,\dots,n
	\end{equation}
	(where $i = j$ might be included in the summation
	or not, which does not matter due to $e_{ii} = 0$).
	In the following, we will only consider the
	$(n+1)\times(n+1)$ matrix $Q$, which extends
	$\tilde Q$ by a $0$'th row and column:
	\[
		Q^{ij} := v^i \cdot v^j,
		\qquad i,j = 0,\dots,n.
	\]
	It is made to have vanishing row-sum and column-sum,
	i.\,e. $Q 1_{n+1} = 0$ and $1_{n+1}^t Q = 0$.
	In the special case $n = 2$, we know from
	\citeauthor{Pinkall93} (\citeyear{Pinkall93},
	\citename{Pinkall} nowadays dates the 
	formula back to \citealt{Duffin59} or even
	\citealt{McNeal49}, but as the formula
	itself is easy to discover by classical trigonometry,
	we value the application to computational
	mathematics higher than the first occurence of
	two opposite angles' cotangents)
	\begin{equation}
		\label{eqn:dualEdgeFormulaDEC}
		\betrag{ijk} v^j \cdot v^i = \frac{\betrag{*ij}}{\betrag{ij}}
		= \cot \alpha_{ij}^k,
	\end{equation}
	(as is frequently used in discrete exterior calculus, see
	\citealt{Hirani03}), where $*ij$ is the straight
	line from $ijk$'s circumcentre to the midpoint
	of $ij$, and $\alpha_{ij}^k$ is the angle
	in $ijk$ opposite to vertex $k$.
\end{subeqns}
\begin{definition}
	\label{def:thetahFullSimplex}
	We say that $s$ is \textbf{$(\theta,h)$-small} if
	all edge lengths $\ell_{ij}$ are smaller
	than $h$ and $s$
	has volume greater than $\theta h^n \sigma_n$,
	where $\sigma_n := \sqrt{n+1}/(2^{\nicefrac n2} n!)$ is
	the volume of the regular $n$-simplex with unit edge length,
	i.\,e. the scaled standard simplex \smash{$\frac 1 {\sqrt 2} \stds$}.
	In terms of the first fundamental form,
	$\det C \geq (\theta h^n n! \, \sigma_n)^2$.
\end{definition}
\begin{remark}
	\begin{subenum}
	\item	The standard simplex has maximal
	volume among all $n$-simplices with the same edge length
	bounds, so $\theta \leq 1$.
	\bibrembegin
	\item	The parameter $\theta$
	is $1/\sigma_n$ times the \begriff{fullness}
	$\Theta(s)$ in \citet[sec. IV.14]{Whitney57}. The
	fullness parameter $\theta$ from \cite{Deylen13}
	is $n! \, \sigma_n \theta$ in the notation employed here.
	It would be equivalent to
	require a lower bound on the angles between subsimplices.
	\bibremend
	\item	Weaker requirements on the simplex quality that
	still ensure well-posedness of the interpolation problem, like
	the famous maximum angle condition of
	\cite{Babuska76}, are circumstantially treated by
	\cite{Shewchuk02}.
	\end{subenum}
\end{remark}
\begin{lemma}
	\label{prop:eigenvaluesOfFirstFF}
	Let $\underline \alpha_n := n!\,\sigma_n n^{1-n}$
	for all $n \in \N$. Then
	the eigenvalues $\lambda_i$ of $C$ fulfill
	\[
		\theta h \, \underline\alpha_n
		\halfquad \leq \halfquad
		\sqrt{\lambda_i}
		\halfquad \leq \halfquad
		h n.
	\]
\end{lemma}
\begin{proof}
	We have to estimate $\norm P$ and $\norm{P\inv}$
	from \ref{sec:metricOnUnitSimplex}.---Recall
	that the $n$-dimensional unit simplex
	has interior and boundary measure
	\[
		\vol_n(D) = \frac 1 {n!},
		\qquad
		\vol_{n-1}(\bdry D) = \frac n {(n-1)!} + \frac{\sqrt n}{(n-1)!}
	\]
	(the latter one because the $(n-1)$-dimensional
	standard simplex $\conv(e_1,\dots,e_n)$ has volume
	$\frac{\sqrt n}{(n-1)!}$).
	For any $n$-simplex $s$, the radius $r$ of the insphere is
	connected to volume and surface via $\vol_n(s) = \frac r n \vol_{n-1}(\bdry s)$.
	This can be easily seen by considering the
	simplices $s^*_i := \conv(s_i,c)$, where $s_i$ is a facet of $s$
	and $c$ is the circumcentre of the insphere. These $s^*_i$ all
	have volume $\vol_n(s^*_i) = \frac r n \vol_{n-1}(s_i)$, and
	$\vol_n(s) = \vol_n(s^*_1) + \dots + \vol_n(s^*_{n+1})$.
	Now, solving $\vol_n(D) = \frac r n \vol_{n-1}(\bdry D)$ for $r$
	gives
	\[
		r = \frac 1 {n + \sqrt n} \geq \frac 1 {2n}.
	\]
	This means that any vector $v \in TD$ with
	length $\frac 1 n \leq 2r$ can be represented as $p-q$
	with points $p,q \in \stds$. Its image in $s$ is
	$Pp - Pq$, which must be shorter
	than the diameter of $s$. So $\norm P \leq nh$.
	On the other hand,
	\[
		\lambda_{\min} (nh)^{2n-2}
		\geq \lambda_{\min} \lambda_{\max}^{n-1} \geq \det C
		> (\theta h^n n! \, \sigma_n)^2
	\]
\end{proof}
\begin{corollary}
	\label{prop:estimateEntriesOfOrthoBasis}
	For the norm $\absval[g] w^2 := w^i w^j C_{ij}$ on $T\stds$ holds
	$\theta h \, \underline\alpha_n \absval[\ell^2] w \leq
	\absval[g] w \leq h n \absval[\ell^2] w$.
	In particular, all edges are longer than $\theta h \, \underline\alpha_n$.
	The columns $v^i$ of $V$ form a $g$-ortho\-normal basis,
	and $\absval[\ell^2]{v^i} \leq (\theta h \underline\alpha_n)\inv$ for all $i$.
\end{corollary}
\begin{proof}
	We have $\absval[g] w = \absval[\ell^2]{C^\half w}$, and
	the extremal eigenvalues of $C^\half$ are $\lambda_{\min}^\half$
	and $\lambda_{\max}^\half$, which shows the first claim.
	The second is clear from $V^tP = \1$ and the fact that
	$C\inv$ has eigenvalues between $(hn)^{-2}$
	and $(\theta h \underline \alpha_n)^{-2}$,
\end{proof}
\begin{lemma}
	\label{prop:fromNormToScalarProduct}
	Assume two symmetric matrices $C,\bar C \in \R^{n \times n}$
	with $C$ being positive definite and
	$\absval{(C - \bar C)v \cdot v} \leq \eps C v\cdot v$.
	Then also $\absval{(C - \bar C)v \cdot w} \leq \eps \absval{C v \cdot v}^\half
	\absval{C w \cdot w}^\half$.
\end{lemma}
\begin{proof}
	The claim is independent of scaling $v$ and $w$, so let
	$Cv\cdot v = Cw \cdot w = 1$. We will first show
	the claim for $C$-orthogonal vectors and
	then for linear combinations. So assume $Cv\cdot w = 0$
	for the moment. Then
	$C(v+w)\cdot(v+w) = C(v-w)\cdot(v-w) = 2$, and the
	parallelogram identity (polarisation formula)
	\[
		4A(v,w) = A(v+w,v+w) - A(v-w,v-w)
	\]
	for any symmetric $2$-tensor $A$
	gives
	$4 \absval{(C - \bar C)v \cdot w} \leq
	\eps C(v+w)\cdot (v+w) + \eps
	C(v-w)\cdot (v-w) = 4\eps = 4 \eps \absval{Cv\cdot v}^\half
	\absval{Cw \cdot w}^\half$. Now for a linear combination,
	we obtain $\absval{(C - \bar C)(v+w) \cdot v}^2
	\leq \absval{(C - \bar C)v \cdot v}^2
	+ \absval{(C - \bar C)w \cdot v}^2
	\leq 2 \eps = \eps \absval{Cv\cdot v} \, 
	|C(v+w)\,	\cdot $ $(v+w)|$,
\end{proof}
\begin{remark}
	The polarisation argument is also feasible for higher-order symmetric tensors,
	as e.\,g.
	\[
		6C(v,v,w) = C(v+w,v+w,v+w) - C(v-w,v-w,v-w) - 2C(w,w,w),
	\]
	which (together with the usual paralellogram
	identity) shows that all evaluations of a symmetric $3$-tensor can be
	reduced to linear combinations of equal argument evaluations.
\end{remark}
\begin{lemma}
	\label{prop:fromEntriesToBilfsUnitSimplex}
	Let $C,\bar C \in \R^{n \times n}$ be symmetric matrices,
	where all eigenvalues of $C$ are larger than $\lambda_{\min} > 0$
	(in particular, $C$ is positive definite),
	and $\absval{C_{ij} - \bar C_{ij}} \leq \eps \lambda_{\min} / n$.
	Then
	$\absval{(C - \bar C) v \cdot w} \leq \eps
	\absval{Cv\cdot v}^\half \, \absval{Cw\cdot w}^\half$.
\end{lemma}
\begin{proof}
	Due to \ref{prop:fromNormToScalarProduct}, the case $v = w$
	is sufficient.
	Then
	\[
		\absval{(C - \bar C) v\cdot v} =
		\Bigabsval{\sum_{i,j}(C_{ij} - \bar C_{ij})v_i v_j}
			\leq \frac{\eps \lambda_{\min}} n \sum_{i,j} \absval{v_i} \absval{v_j}
			= \frac{\eps \lambda_{\min}} n \Big(\sum_i \absval{v_i}\Big)^2,
	\]
	and the square is, by Jensen's inequality, smaller than
	$n \sum \absval{v_i}^2 = n \absval v^2$.
	As $C$ is positive definite, $Cv \cdot v \geq \lambda_{\min} \absval v^2$,
\end{proof}
\bibrembegin
\begin{remark}
	Compare the classical eigenvalue distortion theorem
	of \cite{Bauer60} in the formulation of \citet[theorem 2.1]{Ipsen98}:
	If $C$ can be orthogonally diagonalised, $\lambda$ is an eigenvalue
	of $C$, then there is an eigenvalue $\bar \lambda$ of $\bar C$
	with $\smash{\frac{\absval{\lambda - \bar \lambda}}{\absval{\bar \lambda}}}
	\leq \norm{C\inv(\bar C - C)}$.
\end{remark}
\bibremend

\subsection{The Standard Simplex}

Things are less obvious in barycentric coordinates,
for which reason we refer to \citet[chapters 1 and 2]{Fiedler11}
for proofs of the following statements.

\parag{Metric on $T\stds$.}
	\label{parag:metricOnStds}
	We drop our assumption $p_0 = 0$ and let $p_0,\dots,p_n$ be
	arbitrary points in $\R^m$. Their convex hull $s$
	has a parametrisation over the standard simplex
	$\stds$, represented by the matrix
	$P_+ = [p_0| \cdots| p_n]$. The Riemannian metric on
	$T\stds = \{v \in \R^{n+1} \mit v \cdot 1_{n+1} = 0\}$
	is given by
	\begin{subeqns}
	\begin{equation}
		\label{eqn:gijInTermsOfEllij}
		E_{ij} := - \smallfrac 1 2 \betrag{p_i - p_j}^2,
		\qquad i,j = 0,\dots,n.
	\end{equation}
	By \eqnref{eqn:lawOfCosines},
	we know $C_{ij} = E_{ij} - E_{0i} - E_{0j}$
	for $i,j = 1,\dots,n$. The volume of $s$ can be
	computed as
	\begin{equation}
		\label{eqn:CayleyMengerMatrix}
		\vol s = \smallfrac 2 {n!} (-\Det M_+)^\half,
		\quad \text{where} \quad
		M_+ = \begin{pmatrix} 0 & -\smallfrac 1 2 1_{n+1}^t \\
		                      -\smallfrac 1 2 1_{n+1} & E
		      \end{pmatrix} \in \R^{(n+2)\times(n+2)}
	\end{equation}
	is $-\frac 1 2$ times the usual \begriff{Cayley--Menger} or \begriff{extended
	Menger matrix}. The \begriff{volume element} hence is $G := 2 (-\Det M_+)^\half$.
	\end{subeqns}
\parag{Metric on $T^*\stds$.}
	Assume $\Det M_+ \neq 0$.
	The cotangent space on $\stds$ consists of 
	all linear combinations $v_i d\lambda^i$
	with $v \cdot 1_{n+1} = 0$. As $Q^{ij}$,
	$i,j = 0,\dots,n$, already contains the correct
	scalar products between \emph{all} possible
	linear combinations of the $d\lambda^i$,
	it is in particular the correct representation
	for the scalar product on $T^*\stds$. The ambiguity in the choice
	of $g_{ij}$ and $g^{ij}$ is especially visible in
	the fact that normally both are inverse matrices,
	whereas for the choices made here, we only have
	\begin{subeqns}
	\begin{equation}
		\label{eqn:inverseOfCayleyMengerMatrix}
		M_+\inv = \begin{pmatrix} 4r^2 & -2q^t \\
		                           -2q & Q
		           \end{pmatrix},
	\end{equation}
	where $r$ is the circumradius and
	$q$ are the barycentric coordinates of the circumcentre of $s$.
	\end{subeqns}
\bibrembegin
\begin{remark_nn}
	\cite{Fiedler11} uses the symbols $M$ and $M_0$ for the objects
	$-2E$ and $-2M_+$.
\end{remark_nn}
\bibremend
\begin{definition}
	\label{def:gijOnlyInBaycentricCoords}
	From now on, we will only consider barycentric coordinates
	on the standard simplex and will always use
	\[
		g_{ij} = E_{ij}, \qquad g^{ij} = Q^{ij},
		\qquad i,j = 0,\dots,n
	\]
	and say that $g$ is a \begriff{$(\theta,h)$-small metric}
	if $s$ is $(\theta,h)$-small. The latter one is
	of course only possible if $s$ is non-degenerate,
	in particular $n \leq m$.
	Note that for any $c \in \R$,
	the matrices $g_{ij} + c$ and $g^{jk} + c$
	induce the same scalar product on $T\stds$ and $T^*\stds$
	as $g_{ij}$ and $g^{jk}$ respectively, so we can
	assume that both are positive definite
	on $\R^{n+1}$ with the same eigenvector bounds
	as in \ref{prop:eigenvaluesOfFirstFF}.
\end{definition}
\begin{proposition}
	\label{prop:tangentCotangentIsomorphism}
	Although $g_{ij}$ and $g^{jk}$ are not inverse matrices,
	the tangent-cotangent isomorphism is given
	by the usual identities:
	\[
		(v^i\partial_i)^\flat = g_{ij} v^i d\lambda^j,
		\qquad
		(\alpha_j d\lambda^j)^\sharp = g^{ij} \alpha_i \partial_j.
	\]
\end{proposition}
\begin{proof}
	By \eqnref{eqn:CayleyMengerMatrix} and
	\eqnref{eqn:inverseOfCayleyMengerMatrix},
	$Q^{ij} E_{jk} = \delta^i_k + q^i$ for all
	$i,k = 0,\dots,n$. In other words,
	\[
		QE = \1 + [q|\cdots|q].
	\]
	But if $\alpha \cdot 1_{n+1} = 0$, then
	$[q|\cdots|q] \alpha = 0$, irrespective of
	the entries of $q$. So on $T\stds$
	(and equally on $T^*\stds$), both
	matrices are indeed inverse to each other,
\end{proof}
\begin{lemma}
	\label{prop:estimatesForHigherPowers}
	Assume real numbers $x,y > 0$ with
	$\absval{x-y} \leq \eps x$ for some $\eps < 1$. Then
	\[
		\absval{x^p - y^p} \leq c_p \eps x^p,
		\qquad \text{where} \quad c_p =
			\begin{cases}
				p \, (1 + \eps)^{p-1} & \text{for $p \geq 1$}, \\
				p \, (1 - \eps)^{p-1} & \text{for $p \in \interv 0 1$}.
			\end{cases}
	\]
	In particular,
	\[
		\absval{\sqrt x - \sqrt y} \leq 2 \eps \sqrt x
		\quad \text{and} \quad
		\absval{x^2 - y^2} \leq 3 \eps x^2
		\qquad \text{for $\eps < \smallfrac 1 2$.}
	\]
\end{lemma}
\begin{proof}
	Let $f: x \mapsto x^p$. Then
	$\absval{f(x) - f(y)} \leq \max_{\xi \in \interv x y}f'(\xi) \absval{x-y}$.
	So we only have to find the maximum of $f'$.
	Suppose $p < 1$. Here $f'$ is monotonously
	decreasing. If $x \leq y$, then $\max f' = f'(x) = p x^{p-1}$.
	If $y \leq x$, then $\max f' = f'(y) = p y^{p-1} \leq p(1-\eps)^{p-1}x^{p-1}$,
	and this case dominates the first one.---The
	argument for $p > 1$ is litterally the same,
	but with inversed rôles of $x$ and $y$,
\end{proof}
\begin{proposition}
	\label{prop:almostEqualMetricsFromEdgeLengths}
	Let $\ell_{ij}$ be the edge lengths of a
	simplex, defining a $(\theta,h)$-small metric $g$
	on $\stds$, and let $\bar \ell_{ij}$ be
	a second system of desired edge lengths with
	$\absval{\ell_{ij} - \bar \ell_{ij}}
	\leq \frac 2 3 \eps n^{-1} \underline \alpha_n^2 \theta^2 \ell_{ij}$
	where $\eps \leq \frac 1 2$.
	Then there is a simplex $\bar s \subset \R^n$
	with edge lengths $\bar \ell_{ij}$,
	and its Riemannian metric $\bar g$ over $\stds$
	fulfills $\absval{(g - \bar g)\sprod v w}
	\leq \eps \absval[g] v\, \absval[g] w$.
\end{proposition}
\begin{proof}
	For the existence claim,
	it suffices to show that $\bar g_{ij} = - \smallfrac 1 2 \bar \ell_{ij}^2$
	is positive definite on $T\stds$.
	By \ref{prop:estimatesForHigherPowers} and
	the assumption, we have
	$\absval{\ell_{ij}^2 - \bar \ell_{ij}^2} \leq
	2 \eps n^{-1} \underline \alpha_n \theta^2 \ell_{ij}^2$, hence
	$\absval{E_{ij} - \bar E_{ij}}
	\leq \eps n^{-1} \underline \alpha_n^2 \theta^2 h^2$.
	Now apply \ref{prop:fromEntriesToBilfsUnitSimplex}
	to get $\absval{(g - \bar g)\sprod v w} \leq
	\eps \absval[g] v \absval[g] w$. In particular,
	$\bar g$ is positive definite for $\eps < 1$,
\end{proof}
\begin{remark}
	Although the proof of \ref{prop:fromEntriesToBilfsUnitSimplex}
	seems to be fairly rough, the fullness
	parameter $\theta$ is essential to
	\ref{prop:almostEqualMetricsFromEdgeLengths} and cannot be weakened
	to a bound e.\,g. on the minimal edge length
	(which would be equivalent to a minimal angle bound, which
	is a sharp tool for some interpolation estimates, see
	\citealt{Babuska72}). In fact,
	edge lengths $2$, $1+\eps$ and $1+\eps$ of a thin triangle
	may only be relatively distorted by a factor $\delta
	\in {]\frac1{1+\eps}; 1 + \eps[}$
	to guarantee the existence of a corresponding simplex.
\end{remark}
\begin{corollary}
	\label{prop:normEquivalencegBarg}
	Situation as in \ref{prop:almostEqualMetricsFromEdgeLengths}.
	Then $g$ and $\bar g$ are
	equivalent norms, $(1-\eps)\bar g \sprod v v \leq g \sprod v v
	\leq (1+\eps)\bar g \sprod v v$.
\end{corollary}
\begin{proposition}
	\label{prop:comparisonOfTensorMetrics}
	Situation from \ref{prop:almostEqualMetricsFromEdgeLengths}.
	Then the estimate carries over to all higher
	tensor spaces $T^r_s\stds$ with constants
	depending on $r$ and $s$.
\end{proposition}
\begin{proof}
	Let us first consider the extension of $g$
	to bivectors, that means elements of $T^0_2\stds$.
	Without regarding the tensor-product structure,
	we could just say this
	it is a scalar product with components
	$g_{ij} g_{k\ell}$, where $(i,k)$ and $(j,\ell)$ are the
	indices of the first and second argument respectively. And
	these components are almost equal to those of $\bar g$
	(we abbreviate $\absval{g_{ij} - \bar g_{ij}} \leq \delta g_{ij}$):
	\[
		\begin{split}
			\absval{g_{ij} g_{k\ell} - \bar g_{ij} \bar g_{k\ell}}
				& \leq \absval{g_{ij} g_{k\ell} - \bar g_{ij} g_{k\ell}}
				     + \absval{g_{ij} \bar g_{k\ell} - \bar g_{ij} \bar g_{k\ell}} \\
				& \leq \delta \absval{g_{ij} g_{k\ell}} + \delta \absval{g_{ij} \bar g_{k\ell}} \\
				& \leq \delta \absval{g_{ij} g_{k\ell}} + \delta(1 + \delta) \absval{g_{ij} g_{k\ell}}
				\leq 3\delta \absval{g_{ij} g_{k\ell}}.
		\end{split}
	\]
	For the $n^2 \times n^2$ matrices $g_{ij}g_{k\ell}$
	and $\bar g_{ij} \bar g_{k\ell}$, apply
	\ref{prop:fromEntriesToBilfsUnitSimplex} again,
\end{proof}
\begin{corollary}
	\label{prop:comparisonEuclidSimplexVolumeForms}
	Situation as in \ref{prop:almostEqualMetricsFromEdgeLengths}.
	The volume elements of $g$ and $\bar g$ fulfill
	$\absval{G - \bar G} \leq c \,\eps G$
	with a constant $c$ depending only on $n$.
\end{corollary}
\begin{proof}
	The usual proof would be to use $G = 2 (-\Det M_+)^\half$,
	together with the first order approximation of
	the determinant:
	$\Det(F + tA) = \big(1 + t \spur F\inv A + O(t^2) \big) \Det F$
	(a classical reference is \citealt[pp. 96sqq.]{Bellmann60}, although
	the nicest proof that we know is in \citealt[lemma 8.2.1]{Eschenburg07}).
	However, we want to control more than the first order.
	So observe that if $\partial_i$ is a coordinate-induced
	basis of the tangent space, then
	$G = \absval[g]{\partial_1 \wedge \dots \wedge \partial_n}$,
	and $\bar g - g$ on the space of $n$-forms is controlled by the
	previous proposition,
\end{proof}
\begin{proposition}
	\label{prop:comparisonEuclidSimplexMetricsOnForms}
	Let $g,\bar g$ be two metrics with 
	$\absval{(g - \bar g)\sprod v w} \leq \eps
	\absval[g] v \absval[g] w$, where $\eps < \frac 1 2$.
	Then the
	metric $g^{ij}$ on the cotangent space fulfills
	$\betrag{g\sprod\alpha\beta - \bar g\sprod\alpha\beta}
	\leq 2 \eps \betrag[g]\alpha \betrag[g]\beta$
	for $\alpha,\beta \in T^*\stds$.	
\end{proposition}
\begin{proof}
	Again, we do not use the conventional approach
	to bound $\bar M_+\inv$ by the differential
	of the matrix inversion
	$\smallfrac{\d}{\d F} F\inv (A) = - F\inv A F\inv$
	\cite[lemma 2.8]{Deuflhard03}. Instead, the
	definition of $\absval[g]\argdot$ on $T^*\stds$ as
	operator norms will help:
	For every $v \in T\stds$, we have
	$\absval[\bar g] \alpha^2 \geq (\frac{\alpha(v)}{\absval[\bar g]v})^2
	\geq \frac 1 {1+\eps} (\frac{\alpha(v)}{\absval[g]v})^2$,
	in particular for the $v$ realising $\absval[g] \alpha$.
	On the other hand, if $w$ realises $\absval[\bar g]\alpha$,
	then $\absval[\bar g] \alpha^2 = (\frac{\alpha(w)}{\absval[\bar g]w})^2
	\leq \frac 1 {1-\eps} (\frac{\alpha(w)}{\absval[g]w})^2$. So
	\[
		                 \frac 1 {1 + \eps} \, \absval[g] \alpha^2
		\halfquad \leq \halfquad \absval[\bar g] \alpha^2
		\halfquad \leq \halfquad \frac 1 {1 - \eps} \, \absval[g] \alpha^2.
	\]
	Now as $\eps < \smallfrac 1 2$ by assumption,
	we have $\frac 1 {1 - \eps} \leq 1 + 2\eps$
	and $\frac 1 {1 + \eps} \geq 1 - \eps > 1 - 2\eps$, therefore
	$\bigabsval{\absval[g]\alpha^2 - \absval[\bar g]\alpha^2}
	\leq 2 \eps \absval[g] \alpha^2$, which suffices
	due to \ref{prop:fromNormToScalarProduct},
\end{proof}

%
\newsectionpage
\section{Simplicial Complexes and Discrete Riemannian Metrics}
\label{sec:simplicialComplexes}
%
In computational geometry, it is common to describe simplicial
complexes as the union of simplices in Euclidean spaces with
appropriate conditions on their intersections. We consider these
intersection conditions as tedious and use the more abstract
definition via barycentric coordinates, as is usually done
in topology. We follow the lines of \cite{Munkres84}, but we
repeat the definitions in order to directly deal with abstract
simplicial complexes as (almost everywhere smooth) Riemannian
manifolds.
\subsection{Non-Oriented Complexes}

\begin{definition}
	\label{def:simplexAndComplex}
	An $n$-dimensional \begriff{combinatorial simplex} ($n$-simplex) is a
	set of $n+1$ elements, its $\ell$-dimensional \textbf{subsimplices}
	are subsets of cardinality $\ell+1$. An $n$-dimensio\-nal \begriff{combinatorial
	simplicial complex} is a collection $\complex = (\complex^0,\dots,\complex^n)$,
	where each $\complex^\ell$ is a collection of $\ell$-dimensional simplices
	such that if $\simplext$ is a $k$-dimensional subsimplex
	of $\simplexs \in \complex^\ell$, then $\simplext \in \complex^k$. The complex
	is called \begriff{regular} if each simplex is contained in
	an $n$-simplex and each $(n-1)$-simplex is the subsimplex of at most
	two $n$-simplices. When we speak of simplicial complexes, we always
	mean regular ones.
	
	An $(n-1)$-simplex $\simplexf$ is called a \begriff{boundary simplex} if
	there is only one $\simplexe \in \complex^n$ with $\simplexf \subset \simplexe$.
	The $(n-1)$-dimensional complex formed out of these boundary simplices
	and their subsimplices is called the \begriff{boundary complex} $\Rand \complex$
	of $\complex$.
\end{definition}
\pagebreak[3]
\begin{notation}
	We use special notations for the most interesting dimensions
	(here $k$ is any dimension between $0$ and $1$, kept fixed inside
	an argumentation):
	\begin{center}
	\begin{tabular}{llc|clc|cll}
	vertices & $p_i$ or $i \in \complex^0$ & \hspace{2ex} & \hspace{2ex}
		& $\simplext \in \complex^{k-1}$ & \hspace{2ex} & \hspace{2ex}
		& facets & $\simplexf \in \complex^{n-1}$ \\
	edges    & $ij \in \complex^1$ &&
		& $\simplexs \in \complex^k$ &&
		& elements & $\simplexe \in \complex^n$
	\end{tabular}
	\end{center}
	Sometimes we will also use the convention $\simplext \in \complex^k$ and
	$\simplexs \in \complex^{k+1}$. In every case $\simplext$ will
	be one dimension smaller than $\simplexs$.
\end{notation}
\begin{definition}[\cite{tomDieck00}, p. 63]
	\label{def:abstractGeometricRealisation}
	Let $\simplexs := \{p_0,\dots,p_k\}$ be a combinatorial $k$-simplex.
	For a function $\lambda: \simplexs \to \R$, abbreviate $\lambda(p_i)$
	as $\lambda^i$.
	The \begriff{geometric realisation} of $\simplexs$ is $r\simplexs :=
	\{\lambda: \simplexs \to \interv 0 1 \mit \lambda \cdot 1_{n+1} = 1\}$.
	For a complex $\complex = (\complex^0,\dots,\complex^n)$, the realisation is
	defined as $r\complex := \bigcup_{\simplexe \in \complex^n} r\simplexe$.
\end{definition}
\begin{remark_nn}
	\begin{subenum}
	\item
	This definition is equivalent to, but much
	more elegant than the usual way of ``annotating'' the
	vertices of the euclidean standard simplex $\stds$ with the elements of $\complex^0$,
	considering the disjoint union of $\absval{\complex^n}$ many such annotated simplices
	and glueing them whenever two sides have equal annotations.
	\item
	By setting $\lambda^i = 0$ for all unused vertices $p_i \in \complex^0$,
	the elements of $r\complex$ can naturally be considered as
	functions $\lambda: \complex^0 \to \interv 0 1$ with
	$\supp \lambda = \simplexs$ for some $\simplexs \in \complex^k$.
	\item
	We say that some property is fulfilled \begriff{piecewise} on $r\complex$
	if it is fulfilled on each $r\simplexe$, $\simplexe \in \complex^n$.
	\end{subenum}
\end{remark_nn}
\begin{proposition}
	\label{prop:rKIsManifold}
	Let $\complex$ be an $n$-dimensional simplicial complex (regular, as always). Then
	$r\complex$ is an $n$-dimensional manifold, which is smooth everywhere
	except at $(n-2)$-simplices.
\end{proposition}
\begin{proof}
	\begin{subeqns}
	For each $\simplexf \in \complex^{n-1}$, belonging to $\simplexe, \simplexe' \in \complex^n$,
	define a chart $x_\simplexf: {r\simplexe} \cup {r \simplexe'} \to \R^n$
	in the following way: Without loss of generality, assume
	$\simplexe = \{p_0,\dots,p_n\}$ and $\simplexe' = \{p_1,\dots,p_{n+1}\}$.
	Let $e_1,\dots, e_n$ be the usual euclidean
	basis vectors in $\R^n$, let $e_0$ be the origin and
	$e_{n+1} := \smallfrac 2n (e_1 + \dots + e_n)$. Then the convex hulls
	$D := \conv(e_0,\dots,e_n)$ and $D' := \conv(e_1,\dots,e_{n+1})$ are isometric
	(up to a change of orientation). Now define
	\begin{equation}
		\label{eqn:xsForAbstractRealisation}
		x_\simplexf(\lambda) := \lambda^i e_i =
		\begin{cases}
			\quad \lambda^0 e_0 + \dots + \lambda^n e_n
			& \text{on $r\simplexe$,} \\[1ex]
			\quad \lambda^1 e_1 + \dots + \lambda^{n+1} e_{n+1}
			& \text{on $r\simplexe'$.}
		\end{cases}
	\end{equation}
	This $x_\simplexf$ is a bijection $r\simplexe \to D$
	and $r\simplexe' \to D'$. For any other
	chart $x_{\simplexf'}$ that also covers $r\simplexe$, the chart transition
	is an affine map that maps $E$ to either $\simplexe$ or $\simplexe'$,
	hence $r\complex$ is smooth in the interior of each $n$-simplex.
	
	Around an $(n-2)$-simplex, we do not
	give a chart, but we just remark that a topological manifold
	is sufficiently defined by a finite cover of closed chart domains.
	Open charts are only needed for the definition of smooth functions.
	
	Note that another choice than $e_0,\dots,e_{n+1}$ would have led
	to the same differentiable structure on $r\complex$
	(as long as the convex hulls are full-dimensional simplices in $\R^n$),
	\end{subeqns}
\end{proof}
\begin{observation}
	\begin{subenum}
	\item
	Consider $\simplexs \in \complex^k$.
	By \ref{def:abstractGeometricRealisation}, $r\simplexs$ is a full-dimensional
	subset of the $k$-dimensional affine space $\{\lambda: \simplexs \to \R
	\mit \lambda \cdot 1_{n+1} = 1 \}$, so its tangent space is
	$T_\lambda r\simplexs = \{ v: \simplexs \to \R \mit
	v \cdot 1_{n+1} = 0 \}$ at every internal $\lambda \in r\simplexs$.
	As $r\simplexs$ is an affine space, we deliberately drop the foot point
	$\lambda$ in most cases, just as we do with $T\stds$.
	\item
	The obvious linear isomorphism $\stds \to r\simplexs$ is
	$e_i \mapsto rp_i$.
	This means that $\lambda \in \stds$ and $v \in T\stds$ are mapped
	to $\lambda^i rp_i$ and $v^i rp_i$.
	\item The realisation of the boundary complex is the boundary
	of the realisation: $\partial r\complex = r \partial \complex$. In particular,
	$r\complex$ is a manifold without boundary iff each $(n-1)$-simplex in $\complex$
	belongs to two $n$-simplices.
	\end{subenum}
\end{observation}
\begin{definition}
	[\citealt{Wardetzky06} or, similar but shorter, \citealt{Hildebrandt06}]
	\label{def:differentiableStructureOnrK}
	Define a \begriff{differentiable structure} on $r\complex$ by the
	requirement that some function is smooth (or of class $\Cont^{k,\alpha})$
	if it has this smoothness property piecewise and is
	continuous up to the boundary. Consequently,
	define $\Sob^k$ as the completion of $\Cont^k$ with
	respect to the $\Sob^k$ norm.
\end{definition}
\begin{definition}[\citealt{Bobenko10}]
	\label{def:discreteRmMetric}
	Let $\complex$ be a simplicial complex.
	A function $\ell: \complex^1 \to \R_{\geq 0}$
	with the property that $C_{ij}$ from \eqnref{eqn:lawOfCosines}
	is positive semidefinite for each $\simplexe \in \complex^n$
	is called a \begriff{discrete Riemannian metric}.
	In particular, $\ell$ fulfills the triangle
	inequality on each $\simplext \in \complex^2$.
	
	On each $T_\lambda r\simplexs$, $\simplexs \in \complex^k$, the
	discrete Riemannian metric $\ell$ induces a usual
	Riemannian metric $g_\ell\sprod v w = v^i w^j g_{ij}$
	by $g_{ij} := - \frac 1 2 \ell_{ji}^2$,
	cf. \eqnref{eqn:gijInTermsOfEllij}.
	As this metric does not change with $\lambda$, $r\simplexs$ will
	be flat. When we deal with a piecewise flat metric, we always
	assume that it is defined via a discrete Riemannian metric.
\end{definition}
\begin{observation}
	\begin{subenum}
	\item	Let $\simplext$ be a facet of $\simplexs$. The restriction of $\ell$
	to edges in $\simplext$ is a discrete Riemannian metric
	for itself, and its induced Riemannian metric $g_{\ell,\simplext}$
	on $r\simplext$ is the restriction of $g_\ell$.
	So the glueing of two supersimplices $\simplexs, \simplexs'$
	of $\simplext$ along $\simplext$ is done isometrically with
	respect to $g_\ell$.
	\item	Consequently, every set $U \subset r\complex$ that does not
	contain any $(n-2)$-simplex is flat. In fact, also
	$(n-2)$-simplices might be included if they have
	some flat neighbourhood, which is equivalent to
	requiring that their internal dihedral angles
	as defined by \citet[p. 412]{Cheeger84} sum up to $1$. In this
	sense, curvature of piecewise flat spaces is concentrated
	in the $(n-2)$-simplices.
	\item	\label{rem:isometryChartToStds}
	If $p_i$ are points in Euclidean space with $\ell_{ij}
	= \absval[\ell^2]{p_i - p_j}$, then $g_\ell$ coincides with
	the pull-back metric $g_s$ of $s := \conv(p_0,\dots,p_k)$ to
	the standard simplex $\stds$.
	Hence, $x^\stds: (r\simplexs, g_\ell) \to (\stds,g_s)$,
	$\lambda \mapsto \lambda^i e_i$ and $x^s:
	(r\simplexs, g_\ell) \to (s,\ell^2)$, $\lambda \mapsto \lambda^ip_i$
	are both isometries.
	\item	\label{rem:isometricEmbeddingAdjacentSimplices}
	In the construction \eqnref{eqn:xsForAbstractRealisation},
	one may use points $q_i$ with distances $\absval[\ell^2]{q_i - q_j}
	= \ell_{ij}$ instead of the points $e_i$ (of course,
	$\absval[\ell^2]{q_0 - q_{n+1}}$ does not undergo
	any restriction). Up to Euclidean isometries, these $q_i$
	are unique. This defines an atlas $\{x_\simplexf \mit \simplexf \in \complex^{n-1}\}$
	of isometries.
	\item	\label{eqn:defBarycentricCoordinates}
	Consider a triangle $\{p_i,p_j,p_k\} \in \complex^2$, shortly
	written as $ijk$.
	By the usual trigonometric formulas incorporating only edge lengths,
	one can define angles $\alpha_{ij}^k$ opposite to the edge $ij$
	and area $\absval{ijk}$
	on the basis of $\ell$ only, without using $g_\ell$. The metric
	$\dist_\ell$ on $r\complex$ obtained by the requirement
	\[
		\absval{jk\lambda} = \lambda^i \absval{ijk}, \qquad
		\absval{ki\lambda} = \lambda^j \absval{ijk}, \qquad
		\absval{ij\lambda} = \lambda^k \absval{ijk}.
	\]
	is the same as the metric induced by $g_\ell$.
	The generalisation of this approach to higher dimensions is of course
	feasible and natural, but notationally tedious.
	\end{subenum}
\end{observation}
\begin{proposition}
	\begin{subeqns}
	Let $r\complex$ be a realised simplicial complex with a piecewise
	flat metric $g$.
	Consider two adjacent elements $\simplexe = \{p_0,\dots,p_n\}$ and
	$\simplexe' = \{p_1,\dots,p_{n+1}\} \in \complex^n$
	with common subsimplex $\simplexf$.
	Then for any $\lambda$ in the interior of $r\simplexf$,
	the differential of the transition map $Tr\simplexe \to Tr\simplexe'$
	has dual
	\begin{equation}
		\label{eqn:transitionDifferentialBetweenSimplices}
		\begin{aligned}
			(d\tau_{\simplexe',\simplexe})^\flat: \, T^*r\simplexe & \to T^*r\simplexe', &
			d\lambda^i & \mapsto \qquad\qquad\! d\lambda^i \qquad \text{for $i = 1,\dots,n$}, \\
			&&d\lambda^0 & \mapsto -\smallfrac{\absval{d\lambda^0}}{\absval{d\lambda^{n+1}}} d\lambda^{n+1}.
		\end{aligned}
	\end{equation}
	\end{subeqns}
\end{proposition}
\begin{proof}
	It is clear that the common differentials
	$d\lambda^1,\dots,d\lambda^n$ remain unchanged.
	Under an isometric embedding as in
	\ref{rem:isometricEmbeddingAdjacentSimplices},
	$(d\lambda^0)^\sharp$ and
	$(d\lambda^{n+1})^\sharp$ are normal to the common
	facet (cf. \ref{sec:cometricOnUnitSimplex}),
	pointing in opposite directions, which
	gives $\smallfrac 1 {\absval{d\lambda^0}} d\lambda^0
	= -\smallfrac 1 {\absval{d\lambda^{n+1}}} d\lambda^{n+1}$,
\end{proof}
\begin{remark_nn}
	\begin{subenum}
	\item	We have chosen to give $(d\tau_{\simplexs',\simplexs})^\flat$
	and not $d\tau_{\simplexs,\simplexs'}$ just to obtain a nicer formula.
	One could as well say
	\[
		\grad \lambda^0 \mapsto
		\smallfrac{\absval{d\lambda^0}}{\absval{d\lambda^{n+1}}} \grad \lambda^{n+1}.
	\]
	\item	Formally, $\tau_{\simplexe',\simplexe}$
	is only defined on $r\simplexe
	\cup r\simplexe'$, where it is the identity. But the
	charts $x^\stds_\simplexe$ and $x^\stds_{\simplexe'}$
	from \ref{rem:isometryChartToStds}
	can be extended to some neighbourhood of the standard
	simplex, as $(r\simplexe,g_\ell)$ and
	$(r\simplexe',g_\ell)$ are glued isometrically.
	\end{subenum}
\end{remark_nn}

\subsection{Oriented Complexes}

\begin{definition}%
	[\glqq{}Nanu, Sie kennen $\complex_\ori$ noch nicht?\,\grqq{}]
	Let $V$ be a set. Define an equivalence relation $\sim$ on the
	set $V^{n+1}$ of $(n+1)$-tuples over $V$ by $a \sim b$ iff there is a permutation
	with positive sign that maps $a$ into $b$. Let $[a_0,\dots,a_n]$ be
	the equivalence class of $a \in V^n$. The quotient of $V^{n+1}$
	under $\sim$ is called the set of \begriff{oriented $k$-simplices}
	with vertices in $V$ and is denoted by $[V^n]$.
	
	If $[b] \in [V^k]$ is an oriented simplex, its \begriff{facets}
	are the oriented $(k-1)$-simplices obtained by dropping one of
	its vertices: $[b_0,\dots,\widehat{b_i},\dots,b_k] < [b_0,\dots,b_k]$.
	The \begriff{subsimplices} of $[b]$ are obtained by dropping one
	or more vertices. If dimensions do not
	matter, we also abbreviate $[a] < \dots < [b]$ as $[a] < [b]$
	if $[a]$ is a subsimplex of $[b]$.
	
	An oriented \begriff{simplicial complex} $\complex_\ori$ of dimension $n$ with vertex set
	$V$ is a collection $\complex^0_\ori,\dots,\complex^n_\ori$, where $\complex^k_\ori
	\subset [V^k]$, such that $[a] < [b]$ for some $[b] \in \complex^k_\ori$
	implies $[a] \in \complex^{k-1}_\ori$. The complex is \begriff{regular}
	if no vertex occurs twice in any simplex,
	each simplex is contained in an $n$-dimensional simplex,
	each $(n-1)$-simplex is the boundary of exactly one
	$n$-simplices, and each two $n$-simplices in $\complex^n_\ori$
	have different vertex sets.
	
	If $\complex_\ori$ is a regular orientable simplicial complex,
	we denote the corresponding complex made out of
	non-oriented simplices by $\complex$.
	The \begriff{realisation} of an oriented complex $\complex_\ori$
	is defined as $r\complex_\ori := r\complex$.
\end{definition}
\begin{remark}
	\begin{subenum}
	\item	There are exactly two distinct oriented simplices with
	the same set of vertices $a_0,\dots,a_n$, which we write
	$[a_0,\dots,a_n]$ and $[a_0,\dots,a_n]^-$. As non-oriented simplices
	were defined as sets, each non-oriented simplex corresponds to
	two oriented simplices. So the last condition on a regular complex
	says that not $[a]$ and $[a]^- \in [V^n]$ can belong to an
	$n$-dimensional complex at the same time.
	\item	The vertices of a non-degenerate
	euclidean simplex $s = \conv(p_0,\dots,p_n)
	\subset \R^m$ can be ordered such that $P = [p_1-p_0|\cdots|
	p_n-p_0]$ as in \ref{sec:metricOnUnitSimplex}
	has positive determinant. This
	is what we call the \begriff{canonical} orientation of
	$\{p_0,\dots,p_n\}$. (On the other hand, if $p_0,\dots,p_n$
	are not taken out of some oriented space, there is no
	canonical choice.)
	\end{subenum}
\end{remark}
\begin{proposition}
	$r\complex_\ori$ is an orientable piecewise smooth manifold
	for any regular oriented simplicial complex $\complex_\ori$.
\end{proposition}
\begin{proof}
	We will show that if we use only those charts from
	the proof of \ref{prop:rKIsManifold} that respect the orientation
	of $n$-simplices, we obtain an oriented atlas of $r\complex_\ori$:
	
	Suppose there are two simplices $\simplexs, \simplexs' \in \complex_\ori^n$ that
	share $n$ vertices, say $p_1,\dots,p_n$. Then because
	$\simplext := [p_1,\dots,p_n]$ can only be contained in one of them,
	we can assume that
	\[
		\simplexs = [p_0,p_1,\dots,p_n],
		\qquad
		\simplexs' = [p_{n+1}, p_1,\dots,p_n]^-
	\]
	for two vertices $p_0,p_{n+1} \in \complex_\ori^0$. Now let $x_\simplext$
	be the chart as in \eqnref{eqn:xsForAbstractRealisation}.
	Obviously, $[e_0,\dots,e_n]$ and $[e_{n+1},e_1,\dots,e_n]^-$ are
	both canonically oriented. As there was no choice in this
	construction, every other chart that covers $\simplexs$
	must also map $r\simplexs$ to a euclidean simplex with
	this orientation, therefore every transition map is
	orientation-pre\-serving,
\end{proof}

\subsection{Barycentric Subdivision}

\begin{definition}
	Let $\complex$ be a simplicial complex, regular as usual,
	and $\complex^*:= \complex^1 \cup \cdots \cup \complex^n$ be the
	set of all its simplices. An (ascending) \begriff{$k$-flag} in $\complex$ is
	a set $\simplexa := \{\simplexa_0,\dots,\simplexa_k\} \subset (\complex^*)^{k+1}$
	such that, if its elements are ordered by magnitude,
	$\simplexa_i \subset \simplexa_{i+1}$. In other words,
	a $k$-flag is a sequence of $k+1$ nested simplices.
	If $\simplexa_i \in \complex^{n_i}$, we also write
	$\simplexa = (\flag{n_0},\dots,\flag{n_k})$,
	meaning that $\flag j$ is a ``generic $j$-simplex''.
\end{definition}
\begin{remark}
	\begin{subenum}
	\item
	The notation $\flag i$
	is uncommon, but not more ambigous than other authors'
	notations such as $\sigma^i$. Our notation
	is made to save double subscripts.
	\item
	Of course, flags are simplices, only in some special
	complex. But having a different name will (hopefully)
	prevent confusion.
	The term ``flag'' is borrowed from algebra, where it
	signifies sequences of nested linear spaces,
	whereas set theory mostly speaks of ``ascending chains''
	for nested sets. But the term ``chain'' already
	has a canonical meaning in simplicial homology theory,
	and in section \ref{sec:dec} we need to use both at a time.
	\item	\label{rem:kChainInnSimplex}
	All elements of a $k$-flag lie in a common $n$-simplex.
	An $n$-flag contains exactly one $k$-simplex for each $k$.
	\end{subenum}
\end{remark}
\begin{example}
	Suppose $\complex$ consists of one triangle $ijk$, its edges
	and its vertices. Then the $0$-flags are the elements
	of $\complex^*$ (to be totally precise, the
	$0$-flags are singletons containing elements of
	$\complex^*$). The $1$-flags are of the form $(\flag 0,
	\flag 1)$, that means combinations of
	a vertex and an edge containing it, or of the form
	$(\flag 0,\flag 2)$, i.\,e. a vertex
	and the triangle, or $(\flag 1,\flag 2)$, an edge and the triangle:
	\[
		\{i,ij\}, \{i, ik\}, \{i,ijk\}, \{ij,ijk\}, \{ik,ijk\}
		\quad \text{and similar for the vertices $j$ and $k$.}
	\]
	The $1$-flags $(\flag 0,\flag 1)$
	are interpreted as straight line segments from
	the point $r\flag 0$ to a point $\lambda_{\flag 1}$
	somewhere on the edge $\flag 1$, and the flags
	$(\flag 1,\flag 2)$ connect the points $\lambda_{\flag 1}$
	to the ``barycentre'' $\lambda_{\flag 2}$.
	The $2$-flags consist of a vertex, an edge containing
	this vertex, and the triangle, they are all of the form
	$(\flag 0,\flag 1, \flag 2)$:
	\[
		\{i,ij,ijk\}, \{i,jk,ijk\}
		\quad \text{and similar for other vertices.}
	\]
\end{example}
\begin{definition}
	The \begriff{(barycentric) subdivision} $\sd \complex$
	of the complex $\complex$ is a complex of the same
	dimension whose $k$-simplices are the $k$-flags in
	$\complex$.
	
	Suppose there is some $\lambda_\simplexs \in r\simplexs$ given for
	each $\simplexs \in \complex^*$. Because of
	\ref{rem:kChainInnSimplex}, the mapping
	\[
		r(\sd \complex)^0 \to r\complex,
		\quad
		r\{\simplexs\} \mapsto \lambda_\simplexs
	\]
	can be uniquely extended to a continuous, piecewise
	affine mapping $i: r(\sd \complex) \to r\complex$,
	mapping the realisation of a flag $r(\simplexa_0,
	\dots,\simplexa_k)$ to the convex hull of
	$\lambda_{\simplexa_0},\dots, \lambda_{\simplexa_l}$.
	If $\ell$ is a discrete Riemannian metric on $\complex$,
	then $r(\sd \complex)$ can be endowed with the induced
	metric $\ell_{\{\simplexs\},\{\simplexs'\}}
	= \absval[g_\ell]{\lambda_{\simplexs} - \lambda_{\simplexs'}}$,
	and $i$ becomes an isometry. Let $r':= i \circ r$
	be the ``realisation of $\sd \complex$ in $r\complex$''.
	
	If $\complex_\ori$ is an oriented complex, one can obviously
	define an oriented subdivision by considering the $n$-flags
	as tuples instead of sets and using the orientation induced
	by~$r'$.
\end{definition}
\begin{observation}
	There are several obvious conclusions from the
	fact that $r'\{\simplexa_0,\dots,\simplexa_k\}$
	$= \conv(\lambda_{\simplexa_0},\dots,\lambda_{\simplexa_k})$.
	Most prominently, one can decompose the
	realisation of a $k$-sim\-plex
	$r\simplext$ into the realisations of
	$k$-flags ending at $\simplext$.
	The boundary $\Rand(r\simplexs)$ of a realised
	\mbox{$(k+1)$}-simplex $\simplexs$ is covered by (the realisation of)
	$k$-flags ending at facets of $\simplexs$:
	\[
		r\simplext = \bigcup_{\flag k = \simplext} r'(\flag 0,\dots,\flag k),
		\qquad
		\Rand(r\simplexs) = \bigcup_{\flag k \subset \simplexs} r'(\flag 0,\dots,\flag k).
	\]
\end{observation}
\begin{definition_nn}
	\begin{subeqns}
	For $\simplexs \in \complex^k$, aggregate the
	$n$-flags containing $\simplexs$ in the \begriff{neighbourhood} $U(\simplexs)$ of
	$\simplexs$ and the $(n-k)$-flags starting with $\simplexs$
	in the \begriff{dual cell} $*\simplexs$:
	\begin{equation}
		\label{eqn:defDualCell}
		U(\simplexs) := \bigcup_{\flag k = \simplexs} r'(\flag 0,\dots,\flag n),
		\qquad
		r(*\simplexs) := \bigcup_{\flag k = \simplexs} r'(\flag k,\dots,\flag n).
	\end{equation}
	\end{subeqns}
\end{definition_nn}
\begin{observation}
	\begin{subenum}
	\item
	The flags occuring in $U(\simplexs)$ must obviously be different
	from the flags occuring in $U(\simplexs')$ for $\simplexs
	\neq \simplexs'$, so these neighbourhoods form a covering of
	$r\complex$ with disjoint interior for each $k$.
	\item	\label{obs:neighbourhoodIsSimplexAndDual}
	The set of all $n$-flags ``running through $\simplexs$''
	can be represented as a product of two flag sets: $k$-flags
	ending at $\simplexs$, whose union is $r\simplexs$,
	and the $(n-k)$-flags beginning at $\simplexs$,
	whose union is $r(*\simplexs)$. For us, the latter is
	just a way to write this union, we will not define
	the combinatorial dual of $\complex$. The interested reader is
	referred to \citet[§ 64]{Munkres84}.
	\item	\label{rem:decompositionOfRandUt}
	The boundary of a neighbourhood consists of those flags
	where $\flag k = \simplext$ is left out:
	\[
		\Rand U(\simplexs) = \bigcup_{\flag k = \simplexs}
			r'(\flag 0,\dots,\widehat{\flag k},\dots,\flag n).
	\]
	This can be seen as follows: The boundary of any $n$-flag $\simplexa$
	consists of $(n-1)$-flags $\simplexa'$ where any one of the elements in
	$\simplexa$ is left out. The boundary of the union $U(\simplexs)$
	now consists of those facets $r'\simplexa'$ where some element
	is left out \emph{and} there is no other $n$-flag
	belonging to $U(\simplexs)$ on the other side of
	$r'\simplexa'$. This second condition is satisfied
	only if $\simplexs$ is left out, because if
	$\flag i \neq \simplexs$ is left out, there is
	another flag $(\flag 0',\dots,\flag n')$
	running through $\simplexs$ with $\flag i' \neq \flag i$.
	\end{subenum}
\end{observation}
\begin{lemma}
	\label{prop:centreDiffsArePerpendicular}
	Let $r\complex$ be a realised simplicial complex with piecewise flat
	metric, and let $\lambda_\simplexs$ be the circumcentre
	of $r\simplexs$ for each $\simplexs \in \complex^*$.
	Then for each $n$-flag $(\flag 0,\dots,\flag n)$,
	the vectors $v_{\flag i,\flag{i+1}} := \lambda_{\flag{i+1}}
	- \lambda_{\flag i}$ are perpendicular to $r\flag i$
	and thus pairwise orthogonal.
\end{lemma}
\begin{proof}
	Consider the two-dimensional case: If $\lambda_{ijk}$ is
	the circumcentre of $r(ijk)$, then $\absval{v_{i,ijk}}
	= \absval{v_{j,ijk}}$. The ``circumcentre'' of the edge
	$ij$ is $\lambda_{ij} = \frac 1 2(ri + rj)$. So we have
	two equilateral triangles $(\lambda_i,\lambda_{ij},\lambda_{ijk})$
	and $(\lambda_j,\lambda_{ij},\lambda_{ijk})$, which must hence
	have the same angle $\pi/2$ at $\lambda_{ij}$. The same argument
	applies in higher dimensions: If $\lambda_\simplext$ is
	the barycentre of $r\simplext$, then all $v_{i,\simplext}$
	have the same length. If $\simplext$
	is a facet of $\simplexs$, then all triangles $(i,\lambda_\simplext,
	\lambda_\simplexs)$ are equilateral and hence have the same angle
	at $\lambda_\simplext$. This can only be (because the vectors
	$v_{i,\simplext}$ span the supporting plane of $r\simplext$)
	if $v_{\simplext,\simplexs}$ is perpendicular to the
	supporting plane of $\simplext$,
\end{proof}
\begin{corollary}
	\label{prop:volumeOfPrimalAndDualCell}
	If the complex is \begriff{well-centred}, i.\,e.
	if the circumcentre $\lambda_\simplexs$ always lies inside
	$r\simplexs$, then
	the volume of a $k$-flag $\simplexa \in (\sd \complex)^k$
	can be computed as
	\[
		\absval{r'\simplexa} = \frac 1 {k!}
		\absval{v_{\simplexa_0,\simplexa_1}} \cdots
		\absval{v_{\simplexa_{k-1},\simplexa_k}}
		\halfquad = \halfquad
		 \frac 1 {k!}
		 \absval{{v_{\simplexa_0,\simplexa_1}} \wedge \cdots \wedge
		v_{\simplexa_{k-1},\simplexa_k}}.
	\]
	Together with \ref{obs:neighbourhoodIsSimplexAndDual},
	we get for an $n$-dimensional complex
	\[
		\absval{\simplexs} \,\absval{*\simplexs}
		= \binom n k \, \absval{U(\simplexs)}
		\qquad \text{for } \simplext \in \complex^k,
	\]
	where we have written $\absval \simplexs$
	instead of $\absval{r\simplexs}$ for short, as we
	will always do in the following (no ambiguity will
	occur, as the magnitude $k+1$ of $\simplexs$ is
	always indicated by saying $\simplexs \in \complex^k$).
\end{corollary}
\begin{remark_nn}
	This last volume equation
	is, to the best of our knowledge, not yet used in
	discrete exterior calculus, but frequently in Regge calculus, see e.\,g.
	\cite{Miller13}, and its use for discrete calculus
	seems to date back to \cite{Miller97}.
\end{remark_nn}
\cleardoublepage 
%
%
               \chapter{Main Constructions}
%
%
\label{sec:mainConstruction}

%
\section{The Karcher Simplex: Definition}
%

\begin{notation}
	Let $\dist$ be the geodesic distance on $M$.
	For points $p_0,\dots,p_n \in M$ and the $n$-dimensional
	standard simplex $\stds$, consider the function
	\[
		E: M \times \stds \to \R, \quad
		(a,\lambda) \mapsto \lambda^0 \dist^2(a,p_0) + \dots + \lambda^n \dist^2(a,p_n).
	\]
\end{notation}
\parag{Convexity.}
	\label{sec:convexityRadius}
	Let $\cvr M$ be the largest radius such that all
	geodesic balls $\B_{\cvr M}(p)$, $p \in M$, are convex
	in the sense of \ref{def:convexSet}.
	It can be estimated by
	\[
		\cvr M \leq \smallfrac 1 2 \min \big\{\smallfrac \pi {\sqrt{C_0}}, \inj M\big\},
	\]
	where $\inj M$
	is the injectivity radius of the manifold
	(\citealt[thm. 5.14]{Cheeger75}), and $\dist(p,\argdot)$ is
	convex in $\B_{\cvr M}(p)$. Consequently, for a smaller ball
	$B := \B_{\frac 1 2 \cvr M}(p)$, all
	functions $\dist(q,\argdot)$, $q \in M$,
	are convex in $B$.
\bibrembegin
\begin{remark_nn}
	\label{rem:boundOnInj}
	One knows that $\inj M \geq \min\{$minimal distance between
	conjugate points$\}, \{\frac 1 2$min. length of a closed
	geodesic$\}$, and because conjugate points need to have
	distance greater than $\pi / \sqrt{C_0}$ by Rauch's
	comparison theorem (which is true also if the sectional
	curvature is somewhere negative, so the restriction of nowhere
	positive sectional curvature
	in \citealt[corr. 5.7]{Cheeger75}, is only historically determined and
	factually unneccessary),
	\[
		\cvr M \leq \smallfrac 1 2 \Big\{\smallfrac \pi {\sqrt{C_0}}, \,
				\smallfrac 1 2 \text{min. length of a closed geodesic}\Big\}.
	\]
	The probably best known estimate for
	the latter term in arbitrary dimension is
	\[
		\Len(\gamma) \geq 2 \pi \frac{\vol M}{\vol \mathbb S^m}
		         \Big(\frac{\sqrt{C_0}}{\sinh(\sqrt{C_0} \diam M)}\Big)^{m-1}
	\]
	for a closed geodesic $\gamma$ (\citealt{Heintze78}),
	where $\mathbb S^m$ is the $m$-dimensional unit sphere in $\R^{m+1}$.
	For even dimensions, \cite{Klingenberg59} showed
	$L(\gamma) \geq 2\pi\sqrt{C_0}$ for orientable and $L(\gamma) \geq \pi\sqrt{C_0}$
	for non-orientable $M$.
\end{remark_nn}
\bibremend
\begin{observation}
	\label{rem:definitionOfX}
	Local minimisers of $E(\argdot, \lambda)$ for fixed $\lambda$
	are zeroes of the section $F: M \times \stds \to TM$,
	\[
		F(a,\lambda) := \lambda^i X_i|_a
		\qquad \text{with }X_i = \smallfrac 1 2 \grad \dist^2(\argdot, p_i)
		\text{ from \ref{prop:definingPropertiesOfXandY}.}
	\]
	If the points $p_i$ lie in a common ball $B := \B_{\frac 1 2 \cvr M}(p)$
	for some $p \in M$, then $E(\argdot, \lambda)$ is convex,
	and hence there is a unique
	minimiser in $B$. But this can be sharpened:
\end{observation}
\begin{proposition}
	\label{prop:xIsGlobalMinimiser}
	Let $p_0,\dots,p_n$ be contained in a ball
	$B = \B_r(q) \subset M$ with $r \leq \frac 12 \cvr M$.
	Then for each $\lambda$ in
	the standard simplex,
	$E(\argdot, \lambda)$ has a global minimiser
	in $B$.
\end{proposition}
\begin{proof}
	Let $x$ be the minimiser of $E(\argdot, \lambda)$ over $a \in B$.
	The key observation is that this function cannot
	have more than one minimiser in $\B_{2r}(M)$,
	as \citet[thm. 7.3]{Kendall90} has shown.
	His approach uses convex functions on $M$; but
	\cite{Afsari11} has given a direct proof:
	$F(\argdot,\lambda)$ points inwards on
	the boundary $\Rand \B_{2r}(p)$, and the Hessian
	of $E$ is positive definite at all critical points.
	So by the Poincaré--Hopf index theorem, $F$
	needs to have exactly one zero.
	
	Outside $\B_{2r}(p)$, there cannot be any
	minimisers:
	If $a \in M \setminus \B_{2r}(q)$, then
	$\dist(a,p_i) > r \geq \dist(x,p_i)$ and hence
	$E(a,\lambda) > E(x,\lambda)$,
\end{proof}
\begin{assumption_nn}
	From now on, we only consider $p_0,\dots,p_n$ that
	lie in a common ball $B$ of radius smaller than $\frac 1 2 \cvr M$,
	in particular $C_0 r^2 \leq \frac{\pi^2} 4$.
\end{assumption_nn}
\begin{definition}
	\label{def:mappingx}
	For a given $\lambda \in \stds$, let $x(\lambda)$ be the
	minimizer of $E(\argdot,\lambda)$ in $B$. We call
	this map $x: \stds \to M$ the \begriff{barycentric
	mapping} with respect to vertices $p_i$, and its
	image $s := x(\stds)$ the corresponding
	\begriff{Karcher simplex}.
\end{definition}
\begin{remark}
	\begin{subenum}
	\item In case $M$ is the Euclidean space, $x$ is just the canonical
	parametrisation $\lambda \mapsto \lambda^i p_i$, because
	$\dist^2(p,a) = \absval[\ell^2]{a - p}^2$ gives $X_i|_a = a - p_i$.

	\item	
	\label{rem:properSubsimplices}
	For $\lambda^i = 0$, the value $x(\lambda)$ is independent of $p_i$.
	So the subsimplices of the standard simplex are mapped to ``Karcher
	subsimplices'' which only depend on the vertices of the subsimplex.
	
	\item If $e_i$ is the $i$-th Euclidean basis vector of $\R^{n+1}$,
	then $x(t e_j + (1-t) e_i) = \gamma(t)$, where $\gamma$ is
	the unique shortest geodesic with $\gamma(0) = p_i$ and
	$\gamma(1) = p_j$.

	\item Because $x$ is continuous, the Karcher subsimplices
	form the boundary of a Karcher simplex: $\bdry (x(\stds)) = x(\bdry \stds)$.
	
	\item \label{rem:negativeLambdai}
	Concerning the definition of $x$,
	we will not make use of the fact that
	all $\lambda^i$ are positive beside in
	\ref{prop:EstimateFordx}.
	It was only
	needed to have an easy access to the well-definedness
	of the minimiser. \cite{Sander13} showed
	that negative weights also lead to a well-defined
	minimum if the $p_i$ are contained in a ball
	whose radius is bounded by a constant depending
	on $\inj M$, the curvature of $M$ and
	$\max \absval{\lambda^i - \lambda^j}$.
	\end{subenum}
\end{remark}
\begin{proposition}
	If all $p_i$ lie in a totally geodesic submanifold $S$,
	then $x(\stds) \subset S$. Therefore the usual notion of
	simplices in spaces of constant curvature as convex hull
	of the vertices (cf. e.\,g.
	\citealt[ex. 3.3.6]{Thurston97}) is recovered.
\end{proposition}
\begin{proof}
	Let $k_i := \frac 1 2 \dist^2(\argdot, p_i)$.
	If $B$ is convex, then so is $B \cap S$. So
	$E(\lambda,\argdot)$ is convex on $B \cap S$,
	and hence there is a unique minimizer $a$
	of $E(\lambda, \argdot)|_S$, so there are coefficients
	$\lambda^i$ with
	$\lambda^i \grad (k_i|_S) = 0$
	at $a$. As $S$ is totally geodesic in $M$, it holds $\grad (k_i|_S) = 
	(\grad k_i)|_S = X_i|_S$. Hence $\lambda^i X_i = 0$
	at the point $a \in S$,
\end{proof}
\begin{proposition}
	\label{prop:derivativesOfF}
	\begin{subeqns}
	Define bundle maps $\sigma, A_v$ (for $v \in \R^{n+1}$\!) by
	\begin{alignat*}{2}
		\sigma|_\lambda: \quad
			& T_\lambda \stds \to T_{x(\lambda)} M,
			& v \mapsto & -v^i X_i|_{x(\lambda)}, \\
		A_v|_\lambda : \quad
			& T_{x(\lambda)} M \to T_{x(\lambda)} M, \qquad
			& V \mapsto & v^i \nabla_V X_i|_{x(\lambda)}.
	\end{alignat*}
	If $x$ is smooth at $\lambda \in \stds$,
	then its first and second derivative there fulfill
	\begin{align}
		\label{eqn:xFirstDeriv}
		A_\lambda & dx\, v - \sigma(v) = 0, \\
		\label{eqn:xSecondDeriv}
		A_\lambda & \nabla dx(v,w) +
				A_w dx\,v + A_v dx\,w + \lambda^i \nabla^2_{dx\,w,dx\,v} X_i = 0.
	\end{align}
	\end{subeqns}
\end{proposition}
\begin{proof}
	Similar to the proof of the implicit function theorem: The derivatives of $F$ in a direction
	$(V,v) \in T_{x(\lambda)} M \times T_\lambda \stds$ are
	\[
		\nabla_{(V,v)} F = \lambda^i \nabla_V X_i + v^i X_i
			= A_\lambda \, V - \sigma_\lambda(v).
	\]
	Now consider a curve $\gamma: t \mapsto \lambda + tv$ with
	derivative $\dot \gamma(t) = v$. We have $F(x(\gamma(t)), \gamma(t)) = 0$,
	and differentiating this gives, like in the usual proof of the
	implicit function theorem, $0 = D_t F = \nabla_{(dx\,v,v)} F$.
	This shows the first claim.
	
	The second claim is a totally analogous computation,
	but involves second covariant derivatives $\nabla^2_{V,W}
	:= \nabla_V \nabla_W - \nabla_{\nabla_V W}$.
	We differentiate $F$ once more and obtain
	\[
		\begin{split}
		\nabla_{(W,w)} \nabla_{(V,v)} F
			& = w^i \nabla_V X_i + v^i \nabla_W X_i + \lambda^i \nabla_W \nabla_V X_i \\
			& = w^i \nabla_V X_i + v^i \nabla_W X_i + \lambda^i \nabla^2_{W,V} X_i + \lambda^i \nabla_{\nabla_W V} X_i
		\end{split}
	\]
	Again, $F$ neither changes in direction $(dx\,v,v)$ nor
	in direction $(dx\,w,w)$, so we get
	\[
		\begin{split}
		0 & = w^i \nabla_{dx\,v} X_i + v^i \nabla_{dx\,w} X_i + \lambda^i \nabla^2_{dx\,w,dx\,v} X_i
		  + \lambda^i \nabla_{\nabla_{dx\,w} dx\,v} X_i \\
		  & = A_w dx\,v + A_v dx\,w + \lambda^i \nabla^2_{dx\,w,dx\,v} X_i
		  + A_\lambda \nabla_{dx\,w} dx\,v.
		\end{split}
	\]
	And because $\stds$ is flat, we have
	$\nabla_{dx\,w} dx\,v = \nabla dx(v,w)$
	by \ref{prop:geomCharacterisationOfHessian},
\end{proof}
\bibrembegin
\begin{remark}
	One can consider $\lambda$ as a point measure on
	$M$ that assigns the mass $\lambda^i$ to the
	point $p_i$. For a general probability measure $\mu$
	on $M$, \cite{Karcher77} speaks of the minimiser
	of
	\[
		E_\mu(a) := \int_M \dist^2(a,p) \d\mu(p)
	\]
	as \begriff{Riemannian centre of mass}, but the
	subsequent literature has mostly called it the
	\begriff{Karcher mean} with respect to the measure $\mu$
	(cf. e.\,g. \citealt{Jermyn05}, probably initiated by
	\citealt{Kendall90}). The concept seems to go back
	to \citename{Cartan} (see the historic overview
	in \citealt{Afsari11}), but has not been used by others
	until the work of \cite{Grove73}.
	
	Karcher himself used the centre of mass to
	retrace the standard mollification procedure
	of Gauss kernel convolution in the case
	of functions that map into a manifold.
	Considering the centre of mass as a function
	from an interesing finite-dimensional space of measures into $M$,
	as we use it, has been done
	by \cite{Rustamov10}, and the barycentric
	coordinates we deal with have been used by
	\cite{Sander12,Sander13} and \cite{Grohs13}. All these emphasise the possibility
	to glue the Karcher simplices along corresponding
	facets, but do not investigate the distortion
	properties of this mapping.
	Recently, we were
	informed that \citename{Dyer} and \citename{Wintraecken}
	(Rijksuniversiteit Groningen) have also proven a result similar
	to \ref{prop:comparisongandge} by Topogonov's angle comparison
	theorems. However, this approach seems not to deliver
	an analogue of \ref{prop:EstimateForNabladx}.

	In contrast, there is a large literature
	for ``barycentric coordinates''
	on general convex polygons in the plane,
	cf. \cite{Warren07, Meyer07} and references therein.
\end{remark}
\bibremend

%
\newsectionpage
\section{Approximation of the Geometry}
%
\label{sec:DistanceFunction}

\subsection{Estimates for Jacobi Fields}

\begin{sloppypar}
It is well-known that locally, Jacobi fields grow approximately linearly:
\begin{equation}
	\label{eqn:JacobiFieldsUsualEstimate}
	\absval{J(t) - P^{t,0}(J(0) + t \dot J(0))}
		\leq C_0 t^2 \absval{\dot \gamma}^2
			\big(\absval{J(0)} + \smallfrac 1 4 t \absval{\dot J(0)}\big)
		\qquad \text{for } C_0t^2 \leq \smallfrac{\pi^2}4.
\end{equation}
In fact, \citet[thm. 5.5.3]{Jost11} proves that the left-hand side
is smaller than $\absval{J(0)}$ $(\cosh ct - 1) +
\ddt \absval J(0)(\smallfrac 1 c \sinh ct -t)$
for $c = \sqrt{C_0}$. By Taylor expansion and
$\ddt \absval J \leq \absval{\dot J}$,
this estimate is weakened to our form.
Clearly, if the values $J(0)$ and $J(\tau)$ are given,
one can expect $\dot J$ to behave like
$\smallfrac 1 \tau (J(\tau) - J(0))$,
but as \textsc{Richard}
\citet[p.~11]{Dedekind93} said, ``nothing that is provable ought to be
believed without proof in science.''
\end{sloppypar}

\begin{situation}
	\label{sit:JacobiField}
	Suppose $\gamma: [0;\tau] \to M$ is an arclength-parametrised
	geodesic with $\gamma(0) = p$ and
	$\gamma(\tau) = q$, and $V \in T_q M$. Let
	$s \mapsto \delta(s)$ be a geodesic with $\delta(0) = \gamma(\tau)$ and
	$\dot \delta(0) = V$. Define a variation of geodesics by
	\[
		c(s,t) := \exp_p \big(\smallfrac t \tau (\exp_p)\inv \delta(s) \big).
	\]
	Then $T := \partial_t c$ is an autoparallel vector field
	and $J := \partial_s c$ a Jacobi field along $t \mapsto c(s,t)$
	for every $s$ with boundary values $J(s, 0) = 0$ and $J(s,\tau) = \dot \delta(s)$.
\end{situation}
\begin{proposition}
	\label{thm:estimatesOnDtJ} Situation as in \ref{sit:JacobiField}.
	Define $V(s,t) := P^{t,\tau} \dot\delta(s)$ and
	$\ell(s) := \tau \absval T(s)$, the distance from $p$ to
	$\delta(s)$.
	(By construction, $\absval{V(s,t)}$ is constant in $s$ and $t$,
	and $\absval{T(s,t)}$ is constant in $t$, so we
	drop the unneeded arguments.)
	If $C_0\ell^2(s) < \smallfrac{\pi^2} 4$
	for all $s$, then
	\begin{align*}
		\, & \bigabsval{J (s,t) - \smallfrac t \tau V(s,t)} \leq 2 C_0 \ell^2(s) \absval V,\\
		& \absval{\dot J(s,t) - \smallfrac 1 \tau V(s,t)} \leq \smallfrac 3 2 C_0 \ell(s) \, \absval T(s)\absval V ,\\
		& \absval{\ddot J(s,t)} \qquad \qquad \quad \!\!\leq C_0 \absval T^2(s) \absval V.
	\end{align*}
	If the derivatives of $R$ up to order $k$
	are bounded by constants, then so are the
	$t$-derivati\-ves of $J$ up to order $k + 2$.
\end{proposition}
\begin{proof}
	\begin{subeqns}
	From the usual Jacobi field estimates,
	e.\,g. \citet[thm. 5.5.1]{Jost11}, we
	get that $\absval J$ is increasing
	for all $t < \tau$ in case $C_0\ell^2 < \smallfrac{\pi^2} 4$.
	By the Jacobi equation
	\eqnref{eqn:JacobiFieldDef},
	this already shows the last claim. Now
	observe $J(s,0) - 0 \dot J(s,0) = 0$ and
	\[
		\left| D_t\left( J(s,t) - t \dot J(s,t) \right) \right|
			= t \, \absval{\ddot J(s,t)}
			\leq C_0 t \, \absval T^2(s) \absval V.
	\]
	So the vector field $U: t \mapsto J(s,t) - t \dot J(s,t)$
	vanishes at $t = 0$, and we have bounded its derivative.
	The fundamental theorem of calculus \eqnref{eqn:fundamentalTheoremCalculus}
	gives
	\begin{equation}
		\label{eqn:JminustdotJ}
		\absval{\dot J(s,t) - t J(t,s)} \leq \smallfrac 1 2 C_0 t^2 \absval T^2(s) \absval V.
	\end{equation}
	By $J(s,\tau) = V(s,\tau)$, we have
	\[
		\absval{V(s,\tau) - \tau \dot J(s, \tau)}
		\leq \smallfrac 1 2 C_0 \, \ell^2(s)\, \absval V.
	\]
	Now $\absval{P^{t,\tau}\dot J(s,\tau) - \dot J(s,t)}
	\leq (\tau - t) \max \absval{\ddot J}$ by the mean
	value theorem, and thus
	\[
		\begin{split}
		\absval{V(s,t) - \tau \dot J(s,t)} & \leq
			\absval{P^{t,\tau}V(s,\tau) - \tau P^{t,\tau} \dot J(s,\tau)}
				+ \tau \absval{P^{t,\tau} \dot J(s,\tau) - \dot J(s,t)} \\
			& \leq \quad \smallfrac 1 2 C_0 \absval V \, \ell^2(s) \hspace{13ex}
				+ C_0 (\tau - t) \tau \absval V \, \absval T^2(s) \\
			& \leq \quad \smallfrac 3 2 C_0 \ell^2(s) \absval V.
		\end{split}
	\]
	This proves the comparison between $\dot J$ and $\smallfrac 1 \tau V$.
	For the comparison to $J$, consider
	\[
		\begin{split}
		\big|J (s,t) - \smallfrac t \tau V(s,t) \big| & \leq
			\absval{J(s,t) - t \dot J(s,t)}
			+ t \absval{\dot J(s,t) - \smallfrac 1 \tau V} \\
			& \leq \smallfrac 1 2  C_0 t^2 \absval V \, \absval T^2(s)
				\hspace{1ex} + \smallfrac 3 2 C_0 t \tau \absval V \, \absval T^2(s) \\
			& \leq 2 C_0 \ell^2(s) \absval V.
		\end{split}
	\]
	
	The statement about higher derivatives of $J$ is
	justified by the fact that one can easily give linear
	\textsc{ode}'s for them by differentiating the Jacobi
	equation, e.\,g. $\dddot J + R(\dot J,T)T + \dot R(J, T)T = 0$
	as $\dot T = 0$,
	\end{subeqns}
\end{proof}
\begin{remark}
	\label{rem:scaleAwarenessOfJacobiEstimates}
	These estimates are scale-aware with respect to
	reparametrisations of $\gamma$: If $t$ is replaced by $\lambda t$,
	then also $\tau$ becomes $\lambda \tau$, whereas $\absval T$
	becomes $\frac 1 \lambda \absval T$. So $\frac t \tau$
	and hence the whole first inequality in \ref{thm:estimatesOnDtJ}
	is scale-independend.
	As $\dot J = \nabla_T J$ (loosely speaking), the second inequality scales
	with $1 / \lambda$ and the third one with $1 / \lambda^2$.
\end{remark}
\begin{lemma}
	\label{prop:estimateSecondOrderODE}
	Consider some $\Cont^2$ function $U: [0;\tau] \to \R^m$ 
	satisfying the linear second-order differential equation
	$\ddot U = AU + B$ with smooth time-dependent data $A(t) \in \R^{m \times m}$
	and $B(t) \in \R^m$ as well as boundary conditions $U(0) = U(\tau) = 0$.
	Then, provided that $\norm{A(t)} \tau^2 \leq 1$ everywhere, it holds
	\[
		\absval{\dot U(t)} \leq 3 \absval B \tau,
		\qquad
		\absval{U(t)} \leq 6 \absval B t(\tau - t).
	\]
\end{lemma}
\begin{proof}[Proof (by David Glickenstein).]
	\begin{subeqns}
	Denote the maxima of $\norm A$ and $\absval B$ over
	$\interv 0 \tau$ as $a$ and $b$ respectively.
	As $U$ is $\Cont^2$, there is an upper bound $K$ for $\absval U$ on $\interv 0 \tau$,
	attained at $t = \theta$.
	As this point is critical for $\absval U^2$, we have
	$\sprod U {\dot U} = 0$ there. So
	\[
		K^2 + \theta^2 \absval{\dot U(\theta)}^2
		= \absval{U - t\dot U}^2 (\theta)
		= \Bigabsval{\Int_0^\theta t \ddot U \d t}^2
		\leq \Bigabsval{\int t(aK + b)}^2
		= \big(\smallfrac 1 2 \theta^2 (a K + b) \big)^2.
	\]
	This shows $K \leq \frac 1 2 \theta^2(a K + b)$, so
	$K \leq \tau^2 b$ by assumtion and hence
	$\absval{\ddot U} \leq 2b$. (Note that this argument, which
	first roughly bounds $\absval U$ and then re-inserts this
	bound into the differential inequality to get a sharper
	estimate, is the same as in
	\ref{prop:weakEstimateForPullBackConnectionDifference}sq.)
	Furthermore,
	the inequality chain also shows $\theta \absval{\dot U(\theta)}
	\leq b \theta^2$, which means
	$\absval{\dot U(\theta)} \leq b \tau$. For
	other values of $t$, we have
	$\absval{\dot U(t)} \leq b \tau + \int \absval{\ddot U(t)}
	\leq 3 b \tau$
	and, by integrating once more,
	$\absval{U(t)} \leq 3b\tau t$ as well as $\absval{U(t)} \leq 3b\tau (\tau-t)$,
	whose minimum is dominated by $6bt(\tau-t)$,
	\end{subeqns}
\end{proof}
\begin{proposition}
	\label{thm:estimatesOnDsJ}
	Sitation as in \ref{sit:JacobiField}, $C_0\ell^2(s) \leq \frac{\pi^2}4$.
	Then
	\[
		\begin{split}
			\absval{D_s J(s,t)} & \leq 90 C_{0,1}(s) \smallfrac{t(\tau - t)}\tau \, \absval V^2 \, \absval T(s)\\
			& \leq 90 C_{0,1}(s) \, t \, \absval V^2 \, \absval T(s),
		\end{split}
		\qquad
		\absval{D_s \dot J(s,t)} \leq 50 C_{0,1}(s) \, \absval V^2 \, \absval T(s).
	\]
	with $C_{0,1}(s) := C_0 + \ell(s) C_1$.
	If derivatives of $R$ up to order $k$ are
	bounded by constants $C_1,\dots,C_k$, then
	$\tau \absval{D^k_{s\dots s}\dot J} \leq c(C_0,\dots,C_k) \,\absval V^2$.
	Under reparametrisations of $\gamma$ as in
	\ref{rem:scaleAwarenessOfJacobiEstimates}, the first
	estimate remains unchanged, the second one scales with
	$\frac 1 \lambda$. 
\end{proposition}
\begin{proof}
	Our approach is to derive some differential equation for
	$D_s J = \nabla_J J$, which has boundary values $D_s J (s,0) = 0$ and
	$D_s J(s,\tau) = 0$ for all $s$ because $J(s,0) = 0$ is constant
	in $s$ and $J(s,\tau) = \dot \delta(s)$ is
	the tangent of a geodesic.
	
	\textit{ad primum:}
	Because $J$ and $T$ are coordinate vector fields,
	\eqnref{eqn:defCurvatureTensor} gives
	$D_s D_t U = D_t D_s U + R(J,T)U$ for every vector field $U$,
	so we have
	\[
		\begin{split}
		D_s \ddot J = D_s D_t D_t \partial_s c
			& = D_t D_s D_t \partial_s c + R(J,T)\dot J \\
			& = D_t D_t D_s \partial_s c + D_t R(J,T)J + R(J,T) \dot J \\
			& = D^2_{tt} D_s J + \dot R(J,T)J + R(\dot J,T)J + 2R(J, T)\dot J
		\end{split}
	\]
	whereas the (negative) left-hand side is, due to the Jacobi equation,
	\[
		-D_s \ddot J = D_s\big(R(J,T)T\big) = (D_s R)(J,T)T + R(D_s J, T)T
			+ R(J, \dot J)T + R(J,T)\dot J
	\]
	(note $D_s T = D_t J = \dot J$). From now on, we consider $J$ and $\dot J$
	as being part of the given data (which is allowed,
	as we have already sufficiently described their
	behaviour in \ref{thm:estimatesOnDtJ}). So we have a linear second-order
	\textsc{ode} for $U := D_s J$:
	\[
		\ddot U = A U + B,
	\]
	where both sides scale with $1 / \lambda^2$ under reparametrisation,
	and the norm of $A$ is bounded through $\norm A
	\leq C_0 \absval T^2(s)$.
	For ease of notation, we will thus assume that we consider
	a $t$-line with $\absval T(s) = 1$ and rescale our
	results afterwards. By assumption on the smallness of $\tau$,
	\[
		\begin{split}
		\absval B & \leq 2 C_1 \absval J^2 + 5 C_0 \absval J \, \absval{\dot J} \\
			& \leq 2 C_1 \absval V^2 + 5 C_0 \absval V (\smallfrac 1 \tau + \smallfrac 3 2 C_0 \tau) \absval V \\
			& \leq 15 C_{0,1} \, \smallfrac 1 \tau \absval V^2.
		\end{split}
	\]
	Now consider Fermi coordinates along $c(s,\argdot)$
	as in \ref{parag:fermiCoordinates}
	to obtain an \textsc{ode} in Euclidean space. For
	any smooth vector field
	$V = V^i \partial_i$, the covariant
	derivative in direction $T = \partial_t c$
	is just $\nabla_T V = V^i_{,1}
	\partial_i$. Hence, our \textsc{ode} has the
	coordinate expression
	\[
		U^i_{,11} \partial_i
		= (A^i_j U^j + B^i) \partial_i.
	\]
	As we only need to know the values of $U$ on $x = (t,0,\dots,0)$,
	this gives a euclidean differential equation for the components $U^i$
	of the same form as above.
	The claim on $U$ is then contained in
	\ref{prop:estimateSecondOrderODE}.
	
	\textit{ad sec.:} With $U = D_s J$ as above, we have
	$D_s \dot J = \dot U + R(J,T)J$ and thus
	\[
		\absval{D_s \dot J} \leq \absval{\dot U} + C_0 \absval J^2
			\leq 45 C_{0,1} \absval V^2 + 4 C_0 \absval V^2.
	\]

	\textit{ad tertium:}
	For higher $s$-derivatives, one can proceed by induction:
	The statement is true for $k = 0,1$, as we have shown
	above. By analogous computations, one can control
	$D^k_{s\dots s} \dot J$ by a linear second-order \textsc{ode}, in which
	all lower derivatives might enter as ``given data''. This data
	is bounded by a constant, and hence the solution will be bounded as well,
\end{proof}

\subsection{Estimates for Normal Coordinates}

\begin{situation}
	\label{sit:normalCoords}
	Fix some $p \in M$ and consider normal coordinates
	around $p$ as in \ref{parag:normalCoordinates}:
	\[
		x: (u^1,\dots,u^m) \mapsto \exp_p u^i E_i
	\]
	for some orthonormal basis $E_i$ of $T_p M$.
	Recall from \ref{obs:normalCoords}
	that the deviation of metric and connection
	from its Euclidean counterparts can be described
	by differential and Hessian of the
	exponential map.
	In the following, we let $r := \dist(p,\argdot)$ be the
	geodesic distance to $p$.
\end{situation}
\begin{lemma}
	\label{prop:estimateMetricEntriesNormalCoords}
	Situation as above,
	$C_0 r^2 \leq \smallfrac{\pi^2} 4$.
	Then $\absval{g_{ij} - \delta_{ij}} \leq C_0 r^2$.
\end{lemma}
\begin{proof}
	By \ref{prop:fromNormToScalarProduct},
	it suffices to consider $i = j$. So we only have
	to show $\bigabsval{\absval{dx\,e_i}^2-1}
	\leq C_0 r^2$.
	By means of \eqnref{eqn:metricInNormalCoords},
	this amounts to control
	$\bigabsval{\absval{d_U(\exp_p) E_i} -
	\absval{E_i}}$.
	From \ref{prop:derivativeOfExpAlongCurve}, we
	know that $d_U(\exp_p) E_i$ is the
	terminal value $J(1)$ of a Jacobi field with
	$J(0) = 0$ and $\dot J(0) = E_i$. 
	Now
	$\bigabsval{\absval{d_U(\exp_p) E_i} -
	\absval{E_i}} \leq \absval{d_U(\exp_p) E_i -
	PE_i} \leq \frac 1 4 C_0 r^2$ by
	\eqnref{eqn:JacobiFieldsUsualEstimate}, and the squared
	norms thus cannot differ by more than
	$2(1+\smallfrac{\pi^2}{16}) \cdot \smallfrac 1 4 C_0 r^2
	\leq 0.809 C_0 r^2$
	due to \ref{prop:estimatesForHigherPowers},
\end{proof}
\begin{lemma}
	\label{prop:estimateChristoffelOperatorNormalCoords}
	Situation as above, $C_0 r^2 \leq \smallfrac{\pi^2} 4$.
	Then $\norm \Gamma \leq 10 C_0 r + 5 C_1 r^2$.
\end{lemma}
\begin{proof}
	\begin{subeqns}
	Again, the case $i = j$ is sufficient. Additionally,
	we will only prove the claim for $r = 1$. The
	correct scaling is then automatically
	enforced by \ref{parag:manifoldScalingNoCoordScaling}.
	So let $T,V \in T_p M$
	be unit vectors, assume
	$C_0 \leq \smallfrac{\pi^2} 4$,
	and consider a variation of geodesics
	$c(s,t) = \exp_p t(T+sV)$. As the exponential mapping
	has no radial distortion, we may assume $V \perp T$
	wthout loss of generality. This delivers us
	a Jacobi field $J(s,\argdot) = \partial_s c(s,\argdot)$
	for each $s$, and \eqnref{eqn:christSymbInNormalCoords}
	tells us that $d_T(\exp_p)(\Gamma(v,v)) =
	\nabla d_T (\exp_p)(V,V) = \nabla_J J(0,1)
	= D_s \partial_s c(0,1)$ for $V = v^i E_i$. As observed in
	\ref{thm:estimatesOnDsJ}, the vector field
	$U := D_s \partial_s(0,\argdot)$
	along $c(0,\argdot)$ obeys the linear second-order
	\textsc{ode}
	\[
		\ddot U = R(T,U)T + R(\dot J, J)T + 3 R(T,J)\dot J + R(T,\dot J)J
			+ \dot R(T,J)J + (D_s R)(T,J)T
	\]
	where the obvious notation $T$ for $\partial_t c$ has been used.
	So again we have $\ddot U = AU + B$ with $\norm A \leq C_0$
	and $\absval B \leq 2 C_1 \absval J^2 + 5 C_0 \absval J \absval{\dot J}$,
	but this time as an initial-value problem with
	$U(0) = \dot U(0) = 0$. Denoting the supremum over $\absval B$
	by $b$ again, the norm $\absval U$ will be dominated by
	the solution of $\ddot u = c^2 u + b$, $c = \sqrt{C_0}$,
	which is $\frac b{2c^2}\e^{-ct}(\e^{ct}-1)^2$, which
	itself is smaller than $\frac 5 8 b$ for
	$ct \leq \frac \pi 2$. This means
	$\absval{\nabla d_T (\exp_p)(V,V)} \leq \frac 5 8 b$,
	and our task is to estimate $B$ against $V = \dot J(0)$.
	
	From \eqnref{eqn:JacobiFieldsUsualEstimate}, we get
	$\absval{J(t)} \leq (1 + \frac 1 4 C_0 t^2) t \absval{\dot J(0)}
	\leq (1 + \frac{\pi^2}{16}) t \absval{\dot J(0)}$
	for all $t \leq 1$. On the other hand,
	\eqnref{eqn:JminustdotJ} gives
	$t \absval{\dot J(t)} \leq (1 + \frac 1 2 C_0 t^2) \absval{J(t)}$,
	and combining both leads us to the rought,
	but sufficient estimate
	$\absval{\dot J(t)} \leq (1 + \frac{\pi^2} 8)(1 + \frac{\pi^2}{16})
	\absval{\dot J(0)}$. So we have, as $V = \dot J(0)$ is of
	unit length,
	\begin{equation}
		\label{eqn:proofChristSymb}
		\absval B \leq (1 + \smallfrac{\pi^2}{16})^2 C_1
			+ (1 + \smallfrac{\pi^2}8)(1 + \smallfrac{\pi^2}{16})^2 C_0.
	\end{equation}
	So far, we have only estimated the norm of
	$\nabla d_T (\exp_p)(V,V) = d_T (\exp_p)(\Gamma(v,v))$
	by $\frac 5 8 b$, and this needs to be compared to
	$\Gamma(v,v)$. By \ref{prop:derivativeOfExpAlongCurve},
	the former is the value $Z(1)$ of a Jacobi field $Z$
	along $c(0, \argdot)$, and the latter is $\dot Z(0)$.
	Using \eqnref{eqn:JacobiFieldsUsualEstimate} for $Z$,
	we obtain $\absval{\dot Z(1)} \leq (1 - \frac{\pi^2}{16})^{-1}
	\absval{Z(1)}$, and by inserting this
	into \eqnref{eqn:proofChristSymb}, we finally get
	\[
		\absval{\Gamma(e_i,e_i)}
		\leq \frac{\frac 5 8 \absval B}{1 - \frac{\pi^2}{16}}
		\leq \frac{\frac 5 8(1 + \frac{\pi^2}{16})^2}{1 - \frac{\pi^2}{16}}\, C_1
			+ \frac{\frac 5 8(1 + \frac{\pi^2}8)(1 + \frac{\pi^2}{16})^2}{1 - \frac{\pi^2}{16}} \, C_0
		\leq 5 C_1 + 10 C_0,
	\]
	\end{subeqns}
\end{proof}
{
\begin{remark}
	\begin{subenum}
	\bibrembegin
	\item
	As one can easily see in the proof, our numerical
	constants are by no means optimal.
	A sharper result, but with much more technical effort,
	has been given by \cite{Kaul76}. This author also
	deals with the case that the sectional curvature might
	be asymmetrically bounded between $c_0$ and $C_0$,
	whereas we are only interested in the simpler case
	$c_0 = - C_0$.
	\bibremend
	\item	Considering the Christoffel symbols as objects
	that store ``derivative information'' for the metric,
	the classical procedure of numerical analysis would have
	been to first estimate the Christoffel symbols and then
	integrate this to obtain a bound for the metric tensor.
	It is a specific property of the $g_{ij}$ that they
	can be bounded by a right-hand side which includes
	fewer terms than the bound for their derivatives.
	\item	Under scaling of $Mg$ as in \ref{parag:manifoldScaling},
	the estimate \ref{prop:estimateChristoffelOperatorNormalCoords}
	scales like $\smallfrac 1 \mu$, and
	\ref{prop:estimateMetricEntriesNormalCoords} is scale-independent.
	The assumptions in both propositions are scale-independent.
	\item	\label{rem:harmonicCoords}
	Regarding \eqnref{eqn:componentsOfR} rises
	the question if
	derivatives of $R$ are actually needed to bound
	$\norm \Gamma$. In fact they \emph{are}
	needed in normal coordinates (\citealt[ex. 2.3]{deTurck81}),
	but not in harmonic coordinates, which would lead
	to estimates that only depend on $C_0 r^2$
	(\textit{loc. cit.}, thm. 2.1).
	As \cite{Bemelmans84} remarked, the metric $g$ can be
	infinitesimally abridged by a short time of Ricci flow,
	and the new metric $\bar g$ has $\norm{\nabla^i \bar R}
	\leq \bar C_i(C_0)$ for all $i$.
	Furthermore, a bound on
	$\nabla R$ will be needed in \ref{thm:estimatesOnDsJ}
	anyway, so we decided to take normal coordinates, which
	make it easier to give explicit numerical constants
	in the estimates.
\end{subenum}
\end{remark}
}
\begin{conclusion}
	\begin{subeqns}
	\label{prop:ChrisBoundNormalCoords}
	In a normal coordinate ball $(B,u)$ of radius $r$ with
	$C_0 r^2 < 1$ and $2r < \inj M$, $g$ and the Euclidean
	standard metric are equivalent, and
	\begin{equation}
		\bigabsval{\absval[g(u)] V - \absval[\ell^2] V} \leq C_0 \absval u^2 \absval[\ell^2] V,
		\qquad
		\norm{\Gamma(u)} \leq 10C_0 \absval u + 5C_1 \absval u^2.
	\end{equation}
	\end{subeqns}
\end{conclusion}
\begin{corollary}
	\label{prop:ChrisBoundFermiCoords}
	\begin{subeqns}
	In a Fermi coordinate tube of radius $r$ with
	$C_0 r^2 < 1$ and $2r < \inj M$, $g$ and the Euclidean
	standard metric are equivalent, and
	\begin{equation}
		\label{eqn:ChrisBoundFermiCoords}
		\bigabsval{\absval[g(t,u)] V - \absval[\ell^2] V} \leq C_0 \absval u^2 \absval[\ell^2] V,
		\qquad
		\norm{\Gamma(t,u)} \leq 10C_0 \absval u + 5C_1 \absval u^2.
	\end{equation}
	\end{subeqns}
\end{corollary}
\begin{lemma}
	\label{lem:lengthOfCurvesWithTwoCoincPoints}
	Let $g$ and $g^e$ be two Riemannian metrics 
	with $\bigabsval{\absval[g] v
	- \absval[g^e] v} \leq \eps \absval[g^e] v$,
	$\eps < 1$. Then the curve lengths and
	geodesic distances with 
	respect to $g$ and $g^e$ fulfill
	\[
		\absval{\Len_g(c) - \Len_{g^e}(c)} \leq \eps \Len_{g^e}(c),
		\qquad
		\absval{\dist_g(p,q) - \dist_{g^e}(p,q)} \leq \eps \dist_{g^e}(p,q).
	\]
\end{lemma}
\begin{proof}
	The first claim is proven in the
	obvious way by integrating
	$\bigabsval{\absval[g]{\dot c}
	- \absval[g^e]{\dot c}} \leq \eps \absval[g^e]{\dot c}$
	along $c$. The second claim is a
	combination with
	$\Len_g(c) \leq \Len_g(c^e)$ and
	$\Len_{g^e}(c^e) \leq \Len_{g^e}(c)$
	if $c$ and $c^e$ are the distance-realising
	geodesics for $g$ and $g^e$ respectively,
\end{proof}

\subsection{Approximation of the Metric}

\begin{corollary}
	\label{corol:estimateForNablaX}
	Let $q$
	be in a convex neighbourhood of $p$, $\ell := \dist(p,q)$
	with $C_0 \ell^2 \leq \frac{\pi^2}4$,
	and let $U \in T_q M$ be an arbitrary direction. Then
	\begin{align*}
		\absval{\nabla_V X_p - V} \leq \smallfrac 3 2 C_0 \ell^2 \absval{\pi_Y^\perp V}
			\leq \smallfrac 3 2 C_0 \ell^2 \absval V,\\[0.5ex]
		\absval{\nabla^2_{V,V} X_p} \leq 50(C_0 + \ell C_1)\ell \, \absval{\pi_Y^\perp V}^2.
	\end{align*}
	Here $\pi_Y^\perp$ is the orthogonal projection onto
	the orthogonal complement of $Y_p|_q$ in $T_q M$.
\end{corollary}
\begin{proof}
	Direct consequence of 
	\ref{thm:estimatesOnDtJ} and \ref{thm:estimatesOnDsJ}
	together with \ref{prop:derivativeOfX},
\end{proof}
\bibrembegin
\begin{remark_nn}
	With $\absval U$ instead of $\absval{\pi_Y^\perp U}$,
	but with a smaller constant, the first claim
	is directly proven in \citet[also cf.
	\citealt{Karcher77}, \textsc a.5.4]{Jost82}. For the
	improvement, see \cite{Kaul76}. An exact computation of
	$\nabla d\exp$ for symmetric spaces is given by
	\cite{Fletcher12}.
\end{remark_nn}
\bibremend
\begin{lemma}
	\label{prop:estimateIdMinusInverseOperator}
	Let $A: V \to V$ be an endomorphism of a normed
	vector space $V$ with $\norm{\id - A} \leq \eps < 1$.
	Then $\norm{\id - A\inv} \leq \eps / (1 - \eps)$.
\end{lemma}
\begin{proof}
	By the Neumann series (\citealt[ex. 3.7]{Alt06}):
	\[
		A\inv = \sum_{i = 0}^\infty (\id - A)^i,
		\quad \text{so} \quad
		\norm{\id - A\inv} \leq \sum_{i = 1}^\infty \eps^i = \frac \eps {1 - \eps},
		\vspace{-3.8ex}
	\]
\end{proof}
\begin{lemma}
	\label{prop:EstimateFordx}
	Let $p_0,\dots,p_n$ be distinct points inside a convex
	ball of radius $h$ and $x$ be their barycentric mapping.
	If $6C_0 h^2 \leq 1$, then for
	a tangent vector $v \in T_\lambda \stds$ at any $\lambda \in \stds$
	and $\sigma$ as in Proposition\,\ref{prop:derivativesOfF},
	\[
		\bigabsval{dx\, v - \sigma(v)} \leq
				2 C_0 h^2 \, \absval{\sigma(v)}.
	\]
\end{lemma}
\begin{proof}
	By \ref{prop:derivativesOfF}, $d_\lambda x \,v = (A_\lambda)\inv \sigma|_\lambda(v)$
	and hence $\bigabsval{dx\,v - \sigma(v)} \leq
	\norm{A_\lambda\inv - \id} \, \absval{\sigma(v)}$.
	By \ref{corol:estimateForNablaX}, one has $\absval{\nabla_V X_i - V}
	\leq \frac 3 2 C_0 \dist^2(\argdot, p_i) \absval V$ for all tangent vectors $V$, or, in terms
	of operator norms, $\norm{\nabla X_i - \id} \leq \frac 3 2 C_0 \dist^2(\argdot, p_i)
	\leq \frac 3 2 C_0 h^2$.
	Thus, as $\lambda \cdot 1_{n+1} = 1$ and $\lambda^i \geq 0$,
	\[
		\norm{A_\lambda - \id} = \norm{\lambda^i(\nabla X_i - \id)}
		\leq \absval{\lambda^i} \, \norm{\nabla X_i - \id}
		\leq \smallfrac 3 2 C_0 h^2.
	\]
	Now if $6 C_0 h^2 \leq 1$, then
	$1 - \frac 3 2 C_0 h^2 \geq \frac 3 4$, and
	the claim follows from \ref{prop:estimateIdMinusInverseOperator},
\end{proof}
\begin{notation}
	We write $a \simleq b$ if there
	is some constant $c$ which only depends on $n$
	such that $a \leq c\, b$ (saying ``$a \leq b$ up to a constant.'').
	Equivalently, we will also write $a = O(b)$. We in particular remark
	that our suppressed constants do not depend on the geometry
	parameters.
\end{notation}
\pagebreak[3]
\begin{theorem}
	\label{prop:comparisongandge}
	\begin{subeqns}
	Let $p_0,\dots,p_n$ be distinct points inside a convex
	ball and $x$ be their barycentric mapping.
	Let $g^e$ be the flat Riemannian
	metric on $\stds$ induced by geodesic distances
	$\dist(p_i,p_j)$. Suppose $g^e$ is $(\theta,h)$-small,
	$3 n C_0'h^2 < \underline \alpha_n^2$
	with $\underline \alpha_n$ from
	\ref{prop:eigenvaluesOfFirstFF}.
	Then it holds for tangent vectors
	$v,w \in T_\lambda \stds$
	\begin{equation}
		\label{eqn:comparisongge}
		\absval{(x^* g-g^e)\sprod v w} \simleq C_0' h^2
								\absval v\, \absval w.
	\end{equation}
	The norms on the right-hand side can be either $x^*g$
	or $g^e$ norms, as both are equivalent.
	\end{subeqns}
\end{theorem}
\begin{proof}
	Note that the assumption on $h$ includes the
	requirements of \ref{prop:EstimateFordx} and
	\ref{lem:lengthOfCurvesWithTwoCoincPoints}.
	Due to \ref{prop:fromNormToScalarProduct},
	it suffices to show the claim for $v = w$.
	Consider a
	point $\lambda \in \stds$ with image $a = x(\lambda)$.
	We first compare $x^* g$ to the Euclidean metric of the simplex
	$\bar s_a =\conv(X_i|_a) \subset T_a M$, and compare this metric
	to $g^e$ afterwards.
	
	Parametrise $\bar s_a$ in the canonical way over the unit
	simplex via $\bar x: \lambda^i e_i \mapsto \lambda^i X_i|_a$.
	Now clearly $d\bar x = \sigma$ from
	\ref{prop:derivativesOfF}.
	The metric of $\bar s_a$ is the induced metric of the surronding
	vector space, namely $g|_a$. Now use \ref{prop:EstimateFordx}
	to get
	\[
		\begin{split}
		\absval{(x^* g|_a)\sprod v v^{\nicefrac 1 2} - (\bar x^* g|_a)\sprod v v^{\nicefrac 1 2}}
			& = \big|\absval[g|_a]{dx \, v} - \absval[g|_a]{d\bar x \, v} \big| \\
			& \leq \absval[g|_a]{dx \, v - d\bar x \, v} \leq 2 C_0 h^2 \absval[g|_a]{d\bar x \, v}
				= 2 C_0 h^2 \absval[\bar x^*g] v.
		\end{split}
	\]
	And of course, the same is true for the squared norms
	by \ref{prop:estimatesForHigherPowers}: $\absval{(x^* g - \bar x^* g)\sprod vv}
	\leq 6 C_0 h^2 \absval[\bar g^e] v^2$.
	Hence we have successfully
	compared $x^* g$ to the euclidean metric of $\bar s_a$.
	If we can show that $s$ and $\bar s_a$ have almost equal metrics,
	we are done with \ref{prop:almostEqualMetricsFromEdgeLengths}.
	
	The edge lengths of $\bar s_a$ are $\absval[g|_a]{X_i - X_j}$,
	and the edge lengths of $s$ are the geodesic distance between
	$p_i = \exp_a(X_i)$ and $p_j = \exp_a(X_j)$.
	By \ref{lem:lengthOfCurvesWithTwoCoincPoints},
	we have for their edge lengths $\ell_{ij}$ and $\bar \ell_{ij}$
	\[
		\absval{\ell_{ij} - \bar \ell_{ij}}
			   = \bigabsval{\dist(p_i, p_j) - \absval{X_i - X_j}}
			\leq C_0 h^2 \dist(p_i, p_j)
			   = C_0 h^2 \ell_{ij},
	\]
	so $g^e$ and $\bar g^e$ match
	\ref{prop:almostEqualMetricsFromEdgeLengths}
	with $\frac 2 3 \eps n^{-1} \underline \alpha_n^2 \theta^2 = C_0 h^2$,
\end{proof}
\begin{corollary}
	\label{prop:boundsOndx}
	$h \theta \absval[\ell^2] \argdot \simleq \absval[g] \argdot \simleq h \absval[\ell^2] \argdot$.
\end{corollary}
\begin{definition_nn}
	We say that points $p_1,\dots,p_n \in M$
	lie in \begriff{$(\theta,h)$-close position},
	if there is some $p \in x(\stds)$ such that
	$\bar g^e_{ij} = -\frac 1 2\absval[g|_p]{X_i - X_j}^2$
	defines a $(\theta,h)$-small metric
	in the notation of
	\ref{def:thetahFullSimplex}.
	(Note that this can only be if $n \leq m$.)
\end{definition_nn}
\begin{corollary}
	\label{prop:bijectivex}
	Each collection of points $p_0,\dots,p_n$ in
	$(\theta,h)$-close position that fulfill
	$3 n C_0' h^2 \leq \underline \alpha_n$ defines
	an injective barycentric map.
\end{corollary}
\bibrembegin
\begin{remark}
	\label{rem:metricDistortionTensorJ}
	As $x^*g$ and $g$ are equivalent metrics,
	there is a self-adjoint automorphism $J$ of $T_\lambda \stds$
	such that
	$x^*g \sprod vw = g^e \sprod{J v} w$,
	as has been empoyed by \citet[thm. 3.8]{Holst12}.
	For a comparison to the metric distortion
	tensor $A$ of \cite{Wardetzky06}
	and the $A_h$ of
	\cite{Dziuk88} et al.,
	see \ref{rem:metricDistortionTensorA}.	
\end{remark}
\bibremend

\subsection{Approximation of Covariant Derivatives}

\begin{remark}
	The second-order approximation qualities of a parametrisation
	would usually be measured by bounds on the Christoffel
	symbols. However, our definition of $g^{ij}$ is not exactly
	the inverse matrix of $g_{ij}$, so the usual definition
	\eqnref{eqn:ChrisDefinition} would not work.
	
	Instead, we employ the idea from \cite{Kaul76} to consider
	the operator $\Gamma = \nabla^{x^*g} - \nabla^{g^e}$,
	which would have the coordinate expression
	$\Gamma(V,W) = \Gamma_{ij}^k V^i W^j \partial_k$ in a usual $n$-dimensio\-nal
	chart. Recall from \ref{sit:normalCoords}
	that $\nabla^{x^*g}$ is defined by
	$dx \nabla^{x^*g}_v w = \nabla_{dx\,v} dx\,w$
	We suppress the $g$ subscript for norms.
\end{remark}
\begin{theorem}
	\label{prop:EstimateForNabladx}
	Situation as in \ref{prop:comparisongandge}.
	Then $\norm[\stds g^e,Mg]{\nabla dx}
	\simleq C_{0,1}' h$.
\end{theorem}
\begin{proof}
	Due to \ref{prop:fromNormToScalarProduct},
	it suffices to show the theorem for $v = w$.
	Similar to $\norm{A_\lambda\inv-\id} \simleq C_0 h^2$ we have,
	as the $v^i$ sum up to zero,
	\[
		\absval{A_v(V)} = 
		\absval{v^i \nabla_V X_i - {\textstyle \sum v^i V}}
			\leq \absval{v^i} \absval{\nabla_V X_i - V}
			\leq \smallfrac 3 2 \absval[\ell^1] v \, C_0 h^2 \absval[g] V.
	\]
	
	Now we use \eqnref{eqn:xSecondDeriv}
	and again that $\norm{A_\lambda\inv} \simleq
	1 + C_0 h^2$:
	\[
		\begin{split}
		\frac 1 {1 + C_0 h^2} \absval{\nabla dx(v,v)}
			& \leq 2 \absval{A_v(dx\,v)} + \absval{\lambda^i \nabla^2_{dx\,v,dx\,v} X_i} \\
			& \simleq C_0 h^2 \absval[\ell_1] v \, \absval{dx\,v}
				+ C_{0,1} h \absval{dx\,v}^2 \simleq C_{0,1}h^2 \absval[\ell^2]v\,\absval v,
		\end{split}
	\]
	where $\nabla^2 X_i$ has been estimated by
	\ref{corol:estimateForNablaX},
\end{proof}
\begin{corollary}
	\label{prop:estimateOfChristoffelOperator}
	$\absval{(\nabla^{x^*g} - \nabla^{g^e})_v w}
	\simleq C_{0,1}'h \absval v \, \absval w$.
\end{corollary}
\begin{proof}
	It suffices to consider $v$ and
	$w$ with constant coefficients,
	so $\nabla^{g^e}_v w = 0$.
	By definition of $\nabla^{x^*g}$ and
	\ref{prop:geomCharacterisationOfHessian},
	$\absval[x^*g]{\nabla^{x^*g}_v w}
		= \absval[g]{\nabla^g_{dx\,v} dx\,w}
		= \absval{\nabla dx(v,w)}$,
	and so the preceding theorem applies,
\end{proof}
\begin{corollary}
	\label{prop:comparisondxvAtDifferentPoints}
	For $\lambda,\mu \in \stds$, it holds
	$\bigabsval{d_\lambda x \, v - P d_\mu x \, v}
	\simleq C_{0,1}'h \, \absval{\lambda - \mu} \, \absval v$.
\end{corollary}
\begin{proof}
	By the fundamental theorem
	\eqnref{eqn:fundamentalTheoremCalculus},
	\[
		d_\lambda x \, v = Pd_\mu x \, v + \Int_\gamma
		P\nabla_{dx\,\dot \gamma} dx\,v
		= P d_\mu x \, v + \int P \nabla dx(\dot \gamma,v)
	\]
	for
	a curve $\gamma: \lambda \leadsto \mu$,
\end{proof}
\begin{remark}
	\begin{subenum}
	\item	\label{rem:anyOtherApproxWouldDo}
	The piecewise flat metric $g^e_{ij} = - \frac 1 2 \ell_{ij}^2$
	is nothing more than the ``most natural candidate''
	for a constant Riemannian metric. Any other $g^e$ that
	is a second-order approximation of $g|_a$ for some $a \in s$ would
	give the same result. The particulary interesting observation is that $g$
	can be approximated up to second order by a piecewise
	\emph{constant} metric, whereas an arbitrary function
	would require a piecewise linear function for a similar
	approximation order.
	\item
	The convergence result for the connection does
	\emph{not} mean that if $Mg$ is triangulated
	over a sequence of finer and finer simplicial
	complices $r\complex_h$, the connections of
	$x_h^*g$ and $g_h^e$ would converge. In fact,
	$g_h^e$ would be always piecewise flat, so the
	connection would vanish and hence can never
	approximate the connection of a curved $Mg$.
	This global impossibility is consistent with our
	simplex-wise convergence result because the
	connection for $g^e$ on two adjacent simplices
	cannot be compared to each other, as the metric
	is not continuous across the simplex boundary.
	The connection $\nabla^{g^e}$ can henceforth not be
	connected to the derivative of $g^e$ globally,
	but only in those matters which make sense in
	this situation, e.\,g. higher derivatives of
	real-valued functions as in \eqnref{eqn:estimateHessiansFirstOrder}.
	\item
	The convergence of Lipschitz--Killing
	curvatures from \cite{Cheeger84} applies for our
	situation, although their triangulation is
	defined in a slighly different way, see
	\ref{rem:triangulationDefinitionCheegerMuellerSchrader}.
	It is a convergence in measure of rate
	$h^\half$. For submanifolds of Euclidean space,
	\cite{CohenSteiner06} give a convergence order
	$h$ in measure, but we did not check if their arguments
	can be carried over to our setting.
	\bibrembegin
	\item	\label{comp:fulldimApproximation}
	The metric approximation result is similar to,
	and of the same order as the usual one using
	orthogonal projection of a triangular surface onto
	some nearby smooth surface (cf. \citealt{Dziuk88}). We will reproduce
	this conventional approach for
	the approximation of submanifolds in section
	\ref{sec:graphSubmanifolds}.
	\bibremend
	\end{subenum}
\end{remark}

%
\newsectionpage
\section{Approximation of Functions}
%
\label{sec:approximationOfFunctions}

\begin{goal}
	In this section, we want to apply the Karcher mean construction and the
	results from the previous section to functions between manifolds:
	First, we consider the case of functions $(\stds, x^*g) \to \R$, where the preimage
	is to be approximated by $(\stds,g^e)$.
	After that, we will consider the case of approximation in the image,
	which means we will interpolate some function $y: \stds \to M$
	by the Karcher simplex parametrisation $x$ with prescribed values
	$x(e_i) = y(e_i)$.
\end{goal}

\subsection{Approximation in the Preimage}

\begin{situation}
	\label{sit:approxPreimage}
	Suppose $x^*g$ and $g^e$ are two Riemannian metrics on $\stds$,
	that $g^e$ is flat and that
	\ref{prop:comparisongandge} for the metric as well as
	\ref{prop:estimateOfChristoffelOperator} for the Christoffel
	operator holds. Vector and operator norms,
	if not explicitely qualified, are taken with respect to one of
	these equivalent norms.
\end{situation}
\begin{proposition}
	\label{prop:estimateGradAndHessianForRealvaluedInterpolation}
	\begin{subeqns}
	Situation as in \ref{sit:approxPreimage}.
	For given smooth $u: \stds \to \R$,
	\begin{gather}
		\absval{\grad^{x^*g} u - \grad^{g^e} u} \simleq
			C_0' h^2 \absval{\grad^{x^*g} u}, \\
		\label{eqn:estimateHessiansFirstOrder}
		\norm{\nabla^{x^*g} du - \nabla^{g^e} du}
			\simleq C_{0,1}'h \absval{du}
	\end{gather}
	\end{subeqns}
\end{proposition}
\begin{remark_nn}
	It is easier to estimate the \emph{operator} norm
	$\norm{\nabla^{x^*g} du - \nabla^{g^e} du}$, although
	we will actually need the induced
	norm $\absval{\nabla^{x^*g} du - \nabla^{g^e} du}$ for bilinear
	forms or bi-covectors. Recall that the equivalence constant
	for these two norms only depends on the dimension,
	which will be neglected as usual, so
	$\norm[g]\argdot \leq \betrag[g] \argdot \simleq \norm[g] \argdot$
	on any tensor bundle over $TM$.
\end{remark_nn}
\begin{proof}
	\textit{ad primum:}
	Represent $du = u_i d\lambda^i$.
	In the notation of \ref{def:gijOnlyInBaycentricCoords},
	we have $\grad^{x^*g} u = Q^{ij} u_i \partial_j$ and
	$\grad^{g^e} u = (Q^e)^{ij} u_i \partial_j$.
	So with $\bar u = (u_1,\dots,u_n)$,
	\[
		\begin{split}
		\absval[g]{\grad^{x^*g} u - \grad^{g^e} u}^2
			& = E(Q - Q^e) \bar u \cdot (Q - Q^e) \bar u \\
			& \simleq (C_0'h^2)^2 EQ \bar u \cdot Q \bar u \\
			& = (C_0' h^2)^2 Q\bar u \cdot \bar u
			= (C_0' h^2 \betrag{du})^2.
		\end{split}
	\]
	\textit{ad sec.:} By 
	\eqnref{eqn:covariantDerivative}, the
	difference between two connections only depends on
	their Christoffel symbols.
	Extend the vectors $v, w \in T_\lambda\stds$ to vector fields with
	constant coefficients. As $g^e$ is flat, this gives $\nabla^{g^e} v = 0$ and
	$\nabla^{g^e} w = 0$. Now by
	\eqnref{eqn:HessianDef},
	\[
		(\nabla^{g^e}du - \nabla^{x^*g}du)(v,w)
			= du (\nabla^{g^e}_v w) - du(\nabla^{x^*g}_v w)
			= du((\nabla^{g^e} - \nabla^{x^*g})_v w)
	\]
	and together with
	\ref{prop:estimateOfChristoffelOperator},
	\[
		\begin{split}
		\absval{(\nabla^{x^*g} du - \nabla^{g^e} du)(v,w)}
			& = \absval{du(\nabla^{x^*g}_v w - \nabla^{g^e}_v w)}
			  \leq \absval{du} \, \absval{\Gamma(v,w)} \\
			& \leq \absval{du} C_{0,1}'h \absval v\,\absval w,
		\end{split}
	\]
\end{proof}
\begin{proposition}
	\label{prop:normEquivalences}
	\begin{subeqns}
	Situation as in \ref{sit:approxPreimage}.
	The $\SobW^{k,p}$-norms, $k = 0,1,2$, with respect to $x^*g$ and $g^e$
	are equivalent for every $p \in [1;\infty[$:
	\begin{align}
		\ibetrag[\Leb^p(\stds x^*g)] u^p
			& = \ibetrag[\Leb^p(\stds g^e)] u^p (1 + O(C_0' h^2)),
			\label{eqn:estimateL1normSecondOrder} \\
		\ibetrag[\Leb^p(\stds x^*g)]{du}^p
			& = \ibetrag[\Leb^p(\stds g^e)]{du}^p (1 + O(C_0' c_p h^2)),
			\label{eqn:estimateH1normFirstOrder} \\
		\ibetrag[\SobW^{1,p}(\stds x^*g)]{du}^p
			& = \ibetrag[\SobW^{1,p}(\stds g^e)]{du}^p (1 + O(C_{0,1}' c_p h)),
	\end{align}
	with $c_p$ from \ref{prop:estimatesForHigherPowers}.
	The same holds, without power $p$ and factor $c_p$,
	for the $\SobW^{k,\infty}$ norms.
	\end{subeqns}
\end{proposition}
\begin{remark_nn}
	Note that the estimates speak about $\ibetrag \argdot ^p$ instead of
	$\ibetrag \argdot$. This means that the estimates
	become worse for $p \to \infty$. Therefore, an additional argument
	for the case $p = \infty$ is needed.
\end{remark_nn}
\begin{proof}
	\textit{Case $k = 0$:} The Lebesgue norms on $\stds x^*g$ and $\stds g^e$
	only differ by their volume elements $G$ and $G^e$, which fulfills the claimed
	equivalences thanks to \ref{prop:comparisonEuclidSimplexVolumeForms}.
	So
	\[
		\Bigabsval{\Int_\stds \absval u^p G - \Int_\stds \absval u^p G^e}
			\simleq C_0' h^2 \Int_\stds \absval u^p G.
	\]
	In the $\Leb^{\infty}$ norm, there is nothing to show, as both norms agree.
	
	\textit{Case $k = 1$:} Here an approximation of the volume element and the
	gradient norm enter:
	\begin{multline*}
		\Bigabsval{\int g\sprod{du}{du}^{\nicefrac p 2}G - \int g^e\sprod{du}{du}^{\nicefrac p 2}G^e}\\
			\leq \Bigabsval{\int g\sprod{du}{du}^{\nicefrac p 2}(G - G^e)} + c_p \Bigabsval{\int(g-g^e)\sprod{du}{du}^{\nicefrac p 2}G^e}\\
			\simleq C_0' c_p h^2 \int g\sprod{du}{du}^{\nicefrac p 2}G,
	\end{multline*}
	because $c_p \geq 1$.
	For the $\Leb^\infty$ norm of $du$, it suffices to observe that
	if $\absval[g^e]{d_{\lambda^*}u}$ is maximal among
	all $\lambda \in \stds$, then
	$\absval[g^e]{d_{\lambda^*} u} \simleq (1 + O(C_0' h^2))
	\absval[g]{d_{\lambda^*} u} \leq (1 + O(C_0' h^2))
	\max_\lambda \absval[g]{d_\lambda u}$.
	
	\textit{Case $k = 2$:} We do not have an estimate of our usual form
	$\absval{x - y} \leq \eps \absval x$ for the Hessian, but the proof
	of \ref{prop:estimatesForHigherPowers} also admits this situation:
	\[
		\begin{split}
		\bigabsval{\absval[g]{\nabla^g du}^p - \absval[g]{\nabla^{g^e} du}^p}
			& \leq c_p \absval[g]{\nabla^g du}^{p-1} \bigabsval{\absval[g]{\nabla^g du} - \absval[g]{\nabla^{g^e} du}} \\
			& \leq c_p \absval[g]{\nabla^g du}^{p-1} \absval[g]{du} \norm \Gamma \\
			& \leq c_p \big( \smallfrac{p-1}p \absval[g]{\nabla^g du}^p + \smallfrac 1 p \absval[g]{du}^p \big) \norm \Gamma \\
			& \leq c_p (\absval[g]{\nabla^g du}^p + \absval[g]{du}^p) \norm \Gamma,
		\end{split}
	\]
	thanks to Young's inequality (\citealt[eqn. 1--11]{Alt06}).
	Now one needs approximations of the volume form,
	the norm on covectors and bi-covectors from \ref{prop:comparisonOfTensorMetrics},
	as well as of the Hessian:
	\[
		\begin{split}
		&
		\Bigabsval{\int \absval[g]{\nabla^g du}^p G - \int \absval[g^e]{\nabla^{g^e}du}^p G^e}\\
			& \quad \leq \int \bigabsval{\absval[g]{\nabla^g du}^p - \absval[g]{\nabla^{g^e} du}^p} G
				+ \int \bigabsval{\absval[g]{\nabla^{g^e} du}^p - \absval[g^e]{\nabla^{g^e} du}^p} G
				+ \int \absval[g^e]{\nabla^{g^e} du}^p (G - G^e) \\
			& \quad \simleq \int \bigabsval{\absval[g]{\nabla^g du}^p - \absval[g]{\nabla^{g^e} du}^p} G
				\qquad \,\,\, + \qquad C_0' c_p h^2 \int \absval[g]{\nabla^{g^e} du}^p G \\
			& \quad \simleq (C_0' c_p h^2 + C_{0,1}' c_p h)
					\int \absval{\nabla^g du}^p G
					+ C_{0,1}' c_p h \int \absval[g]{du}^p G,
		\end{split}
	\]
\end{proof}
\begin{theorem}
	\label{prop:InterpolationEstimateForRealvaluedFunctions}
	\begin{subeqns}
	Situation as in \ref{sit:approxPreimage}.
	For a $\Cont^2$ function $u: \stds \to \R$, let
	$u_h: \stds \to \R$ be its Lagrange interpolation,
	that means $u_h$ is linear and $u_h(e_i) = u(e_i)$. Then
	\[
		\ibetrag[\Leb^\infty(\stds)]{u - u_h} + h \ibetrag[\Leb^\infty(\stds)]{d (u - u_h)}
			\simleq h^2\theta^{-1} \inorm[\Leb^\infty(\stds g^e)]{\nabla^{g^e} du}.
	\]
	The right-hand side can be replaced by
	$h^2\theta^{-1} (1 + C_{0,1}' h)\inorm[\Leb^\infty(\stds x^*g)]{\nabla^{x^*g} du}$.
	\end{subeqns}
\end{theorem}
\begin{proof}
	\begin{subeqns}
	If we were only interested in this interpolation of real-valued
	functions, the easiest method of proof would be to use the 
	interpolation estimates in Euclidean space. But  when
	we come to mappings into a second manifold
	in \ref{prop:dxMinusdyAtVertices},
	these methods would not be applicable without
	further work. Therefore we decided to use a more
	``geometric'' approach.
	
	\textit{ad primum:}
	Let $\mu \in \stds$ be an arbitrary point,
	consider the tangent $e_{ij} = e_j - e_i$
	to the geodesic
	$\gamma_{ij}: e_j \leadsto e_i$ and
	\[
		r_1: \lambda \mapsto (d_\lambda u - d_\lambda u_h) e_{ij}.
	\]
	This scalar-valued function has a zero along the
	geodesic $\gamma_{ij}$, because $r_1 \circ \gamma_{ij}$
	is the map $t \mapsto (du-du_h)(\dot \gamma_{ij})
	= \ddt \absval{u - u_h}(\gamma(t))$, and $\absval{u - u_h}$
	is zero at both endpoints of $\gamma_{ij}$. Let $\nu \in \stds$
	be the position of this extremum.
	
	Now let $\gamma$ be the geodesic $\nu \leadsto \mu$ and
	$\psi(t) := r_1 (\gamma(t)) = (d_{\gamma(t)} u - d_{\gamma(t)} u_h) e_{ij}$. Then
	\[
		\dot\psi(t) = g^e \sprod{\grad^{g^e}(u - u_h)}{\nabla^{g^e}_{\dot \gamma} e_{ij}}
			+ g^e \sprod{\nabla^{g^e}_{\dot \gamma} \grad^{g^e}(u - u_h)}{e_{ij}}.
	\]
	The first summand vanishes because $e_{ij}$ is
	parallel with respect to $g^e$, and the second one
	is $\nabla^{g^e}d(u-u_h)(e_{ij},\dot \gamma)$ due to
	\ref{prop:geomCharacterisationOfHessian}.
	So
	\begin{equation}
		\label{eqn:proofInterpolationEstimateRealValued}
		\absval{\dot\psi(t)} \leq h \norm[g^e]{\nabla^{g^e}d(u-u_h)} \, \absval[g^e]{\dot \gamma}
		\qquad \forall t,
	\end{equation}
	and because $u_h$ is linear, $\nabla^{g^e} du_h = 0$. Hence
	$\absval{\psi(t)} \leq \int \absval{\dot\psi(s)} \d s
	\leq h^2 \inorm[\Leb^\infty(\gamma_{ij} g^e)]{\nabla^{g^e} du}$.
	
	If $E_k$ form an orthonormal basis, then
	$\absval{du}^2 = \sum du(E_k)^2$ Because of
	\ref{prop:estimateEntriesOfOrthoBasis},
	the $E_k$ have an expression in the $e_{ij}$ with
	coefficients smaller than $1/\theta h$, which gives
	\[
		\absval[g^e|_\mu]{du-du_h} \simleq h\theta^{-1} \inorm[\Leb^\infty(\stds,g^e)]{\nabla du}.
	\]
	As $\mu$ was chosen arbitrarily, this holds for every point in $\stds$.
	
	\begin{sloppypar}
	\textit{ad sec.:} Now consider a new arbitrary point $\mu \in \stds$,
	the function
	\[
		r_0: \lambda \mapsto \absval{u(\lambda) - u_h(\lambda)}^2
	\]
	and a geodesic $\gamma: e_i \leadsto \lambda$ for some
	vertex $e_i$ of $\stds$. Then let $\phi(t) :=
	r_0(\gamma(t))$. As $r_0$ vanishes at the interpolation
	points, we have $\phi(0) = 0$, and everywhere
	$\absval{\dot \phi(t)} = \absval{d(u - u_h) \dot \gamma}
	\leq \absval[g^e]{d(u - u_h)} \, \absval[g^e]{\dot \gamma}
	\simleq h\theta^{-1} \absval[g^e]{\dot \gamma} \, \inorm[\Leb^\infty(\stds,g^e)]{\nabla^{g^e} du}$
	and thus $\absval{\phi(t)} \leq \int \absval{\dot \phi(s)} \d s
	\simleq h^2\theta^{-1} \inorm[\Leb^\infty(\stds,g^e)]{\nabla^{g^e}du}$.
	\end{sloppypar}
	
	\textit{ad tertium:} The last statement is a direct application of
	\ref{prop:normEquivalences},
	\end{subeqns}
\end{proof}
\begin{corollary}
	\label{prop:InterpolationEstimateLpForRealvaluedFunctions}
	\begin{subeqns}
	The same result also applies for the $\Leb^p$ norms:
	\[
		\ibetrag[\Leb^p(\stds)]{u - u_h} + h \ibetrag[\Leb^p(\stds)]{d (u - u_h)}
			\simleq h^2\theta^{-1} \inorm[\Leb^p(\stds g^e)]{\nabla^{g^e} du}.
	\]
	The right-hand side can be replaced by
	$h^2\theta^{-1} (1 + C_{0,1}' h)\inorm[\Leb^p(\stds x^*g)]{\nabla^{g^e} du}$.
	\end{subeqns}
\end{corollary}
\begin{proof}
	\begin{subeqns}
	Only the estimate \eqnref{eqn:proofInterpolationEstimateRealValued}
	has to be refined by the ``H\"older 1-trick'',
	a common application of H\"older's inequality (\citealt[lemma 1.10]{Alt06}):
	Suppose some function $a \in \Leb^\infty(\stds)$ is estimated
	pointwise by $\absval{a(\lambda)} \leq
	\int_{\gamma[\lambda]} b$, where the integration path
	$\gamma[\lambda]: e_0 \leadsto \lambda$ is of size $h$. Then
	as in the most ``basic'' proof (there are others, cf.
	\ref{prop:coercivityOnHarmPerpTang}) of the Poincaré
	inequality (\citealt[sec. 6.26]{Adams75}),
	\begin{equation}
		\label{eqn:interpolEstimatRealValued1}
		\begin{split}
		\ibetrag[\Leb^p(\stds)] a^p = \Int_\stds \absval a^p
			& \leq \Int_\stds \Big(\Int_{\gamma[\lambda]} b \, 1\Big)^p \\
			& \leq \Int_\stds \big(\Int_{\gamma[\lambda]}b^p\big)
				\big(\int 1\big)^{\nicefrac pq}
			\quarquad \leq \quarquad
			\Int_\stds \big(\Int_{\gamma[\lambda]}b^p\big) h^{\nicefrac pq}.
		\end{split}
	\end{equation}
	Then compute the $\stds$ integral by first integrating
	over the subsimplex $\stds_0$ opposite to the vertex $e_0$
	and then over the ray $e_0 \leadsto \mu \in \stds_0$. Then
	$\lambda = t\mu + (1-t)e_0$ for some $t$ between $0$ and $1$, and
	each function $c \in \Leb^\infty(\R)$ with $c \geq 0$ fulfills
	$\int_0^r \int_0^t c(s) \d s \d t \leq r \int_0^r c(s) \d s$,
	we have
	\begin{equation}
		\label{eqn:interpolEstimatRealValued2}
		\Int_\stds \Int_{\gamma[\lambda]} b^p
		\leq h \Int_\stds b^p.
	\end{equation}
	Then observe $\frac pq = p-1$,
	so we have $\ibetrag[\Leb^p(\stds)] a^p \leq h^p \ibetrag[\Leb^p(\stds)] b^p$
	for such a function $a$. As there does not occur any $\Leb^\infty$
	term in the final estimate, it remains valid for
	$a,b \in \Leb^p(\stds)$,
	\end{subeqns}
\end{proof}

\subsection{Approximation in the Image}

\bibrembegin
\begin{remark_nn}
	For curves in $M$, there are already
	interpolation estimates for high-order (quasi\hbox{-}) interpolation
	methods by \cite{Wallner05} and \cite{Grohs11}.
	
	During the finishing of this thesis,
	\cite{Grohs13} have given a very elaborate
	estimate for higher-order ``polynomial'' interpolation using
	the Karcher mean construction. We decided to nevertheless
	publish our proof here, as we hope that our approach
	gives more geometric intuition, involves simpler constants,
	and is used in sections \ref{sec:graphSubmanifolds}--%
	\ref{sec:mfDirichletProblem}.
\end{remark_nn}
\bibremend
	
\begin{situation}
	\label{sit:interpolationInManifolds}
	In the following, we assume that
	$\stds$ carries a $(\theta,h)$-small Euclidean metric $g^e$
	(which is not assumed to come from geodesic distances in $M$).
	We consider a smooth function $y: \stds g^e \to Mg$
	(and assume that $y(\stds)$ lies in a convex ball of radius $r$
	as in \ref{prop:ChrisBoundNormalCoords} with $C_0 r^2 < 1$ )
	and define $x$ to be the barycentric mapping with respect
	to the vertices $y(e_i)$. We will
	usually write $x$ and $y$ instead of $x(\lambda)$ and $y(\lambda)$.
\end{situation}
\begin{lemma}
	\label{prop:derivativesOfDist}
	Situation as in \ref{sit:interpolationInManifolds}.
	Let $P$ be the parallel transport $T_y M \to T_x M$.
	Consider $\dist(x,y)$ and $\dist^2(x,y)$
	as functions $\stds \to \R$.
	Then
	\begin{align*}
		d(\dist^2(x,y))v & = 2 g|_x\sprod{X_y}{(dx-Pdy)v},\\
		d(\dist(x,y))v   & = \hspace{1.16ex} g|_x\sprod{Y_y}{(dx-Pdy)v},
	\end{align*}
	with $X_p$, $Y_p$ as in \ref{prop:definingPropertiesOfXandY}.
\end{lemma}
\begin{proof}
	From \ref{prop:definingPropertiesOfXandY}, we know the
	gradients of $\dist$ and $\dist^2$ if only one of the
	two arguments is varying. Then for $\phi: M \times M
	\to \R$, $(p,q) \mapsto \dist^2(p,q)$, we have
	for tangent vectors $V \in T_p M$ and $W \in T_q M$
	that
	\[
		d\phi(V,W) = g|_p\sprod V{X_q} + g|_q\sprod W{X_p}.
	\]
	Now $\dist^2(x,y)$ is the concatenation
	of the map $\lambda \mapsto (x,y)$, which
	has derivative $v \mapsto (dx\,v, dy\,v)$ with $\phi$,
	so
	\[
		d(\dist^2(x,y)) v = g|_x\sprod{X_y}{dx\,v} + g|_y\sprod{X_x}{dy\,v}
	\]
	As $X_y$ is the starting tangent of the
	geodesic $x \leadsto y$ parametrised over $\interv 0 1$,
	we have $PX_x = - X_y$, and $P$ is an isometry,
	so $g|_y\sprod{X_x}{dy\,v} = - g|_x\sprod{X_y}{Pdy\,v}$,
\end{proof}
\begin{lemma}
	\label{prop:estimateOfParallelTransportDeriv}
	Let $c(s,t)$ be a smooth variation of curves
	$c(s,\argdot)$, and let $P_s^{b,a}: T_{c(s,a)}M
	\to T_{c(s,b)}M$ be the parallel transport along these
	curves. Then
	\[
		D_s P_s^{b,a} = \Int_a^b P_s^{b,t} R(\partial_t, \partial_s) P_s^{t,a} \d t
	\]
	(note that the integrand is always a linear map $T_{c(s,a)}M
	\to T_{c(s,b)}M$, the integration is therefore
	defined without problems)
	and hence $\norm{D_s P_s^{b,a}} \leq
	C_0 \int_{c(s,\argdot)}\absval{\partial_s}$.
\end{lemma}
\bibrembegin
\bibremend
\begin{proof}
	Consider a vector field $V(s) \in T_{c(s,a)}M$, and
	let $V(s,t) := P_s^{t,a}V(s)$. As first step, observe that
	$D_s(P_s^{b,a}V(s)) = P_s^{b,a}(D_sV(s)) + (D_sP_s^{b,a})V(s)$.
	This formula seems obvious, but actually requires a little
	argumentation: It symbolically resembles $\nabla(AV)
	= (\nabla A)V + A(\nabla V)$ for linear bundle maps $A$
	from \ref{parag:connection},
	but as $P$ mediates between different tangent spaces for
	preimage and image, the $\nabla$ operator is not the
	same on both sides. Instead, consider the function
	$f: c(s,a) \mapsto c(s,b)$ between the $a$ and the $b$ isoline
	$A$ and $B$ respectively.
	The parallel transport from $a$ to $b$ is a mapping $T_x M \to
	T_{f(x)} M$ and hence an element of $TM|_A \otimes f^*(TM|_B)$.
	As in \eqnref{eqn:connectionInPullbackBundle},
	the induced connection on this bundle is given by
	\[
		\nabla_{\partial_s} (\omega \otimes f^*V)
			= (\nabla_{\partial_s} \omega) \otimes f^* V
			+ \omega \otimes f^* \nabla_{df(\partial_s)} V,
	\]
	and indeed $df(\partial_s c(s,a)) = \partial_s c(s,b)$,
	giving the Leibniz rule for $PV$.
	On the other hand, the fundamental theorem of calculus
	gives
	\[
		D_s(V(s,b)) = P_s^{b,a}(D_s V(s,a)) + \Int_a^b P_s^{b,t} (D_t D_s V(s,t)) \d t.
	\]
	Because $V(s, \argdot)$ is parallel, the vector field in the integrand is
	\[
		D_t D_s V(s,t) = D_s D_t V(s,t) + R(\partial_t, \partial_s) V(s,t)
			= 0 + R(\partial_t, \partial_s) P_s^{t,a} V(s).
	\]
	Because $V(s)$ is independent of $t$, it can be pulled
	out of the integral.---The second claims results from
	\[
		\norm{D_s P_s^{b,a}} \leq \Int_a^b \norm R \, \absval{\partial_t} \, \absval{\partial_s}
			\, \norm{P_s^{b,t}} \, \norm{P_s^{t,a}} \d t
	\]
	and $\norm{P_s^{t,t'}} = 1$ everywhere because parallel
	transport is isometric,
\end{proof}
\begin{remark_nn}
	Generally, it is well-known that the curvature tensor can
	be characterised as infinitesimal version of
	holonomy, i.\,e. the parallel transport along
	a closed curve (see e.\,g. \citealt[sec 8.6]{Petersen06}).
	We found this specific version in \citet[lemma 3.2.2]{Rani09}.
	The estimate can obviously be sharpened by replacing
	$\absval{\partial_t} \absval{\partial_s}$ by
	$\absval{\partial_t \wedge \partial_s}$, see
	\citet[6.2.1]{Buser81}.
\end{remark_nn}
\begin{lemma}
	\label{prop:dxMinusdyAtVertices}
	Situation as in \ref{sit:interpolationInManifolds}.
	Then if $\dist(x,y) \leq \rho$ everywhere in $\stds$, we have at every vertex $e_i$
	\[
		\norm[\stds g^e,Mg]{d_{e_i} x - d_{e_i} y}
		\simleq h \theta^{-1} \left(\inorm[\Leb^\infty(\stds g^e, Mg)]{\nabla dy}
			+ C_{0,1} \rho \inorm[\Leb^\infty(\stds g^e, Mg)]{dy}^2 \right).
	\]
\end{lemma}
\begin{proof}
	First, consider $v$ to be an edge vector $e_j - e_i$, so $c:
	t \mapsto e_i + tv$ parametrises the $ij$ edge over $\interv 0 1$.	
	Then choose Fermi coordinates $(t,u^2,\dots,u^m)$ along an arclength\-%
	parametrised version of $\gamma := x \circ c$. As $\gamma$ itself is not
	parametrised by arclength, it has coordinates
	$\gamma(t) = (\frac t \alpha, 0,\dots,0)$ with $\alpha = \dist(p_i,p_j)$.
	The image of $c$ under
	$y$ is another curve $\delta$ which intersects $\gamma$
	at $p_i$ and $p_j$, so
	\[
		\delta(0) = \gamma(0),
		\qquad
		\delta(1) = \gamma(1).
	\]
	By the intermediate value theorem, each component
	$\dot \gamma^i - \dot \delta^i$ must have a zero at some
	$\tau^i \in \interv 0 1$. As $D_t \dot \gamma = 0$
	and $\Chris ijk = 0$ along $\gamma$, the
	second derivatives $\gamma^i_{,tt}$ of the components
	vanish, too. So
	\[
		\absval{(\dot \gamma^i - \dot \delta^i)(0)}
		\leq \Int_0^{\tau^i} \absval{\delta_{,tt}^i} \d t
		\leq \tau^i \ibetrag[\Leb^\infty(\interv 0 1)]{\delta^i_{,tt}}.
	\]
	By \ref{prop:geomCharacterisationOfHessian}, $D_t \delta = \nabla dy(v,v)$,
	and together with $D_t \delta = (\delta^i_{,tt} + \delta^j_{,t} \delta^k_{,t} \Chris jki) \partial_i$
	from \eqnref{eqn:covariantDerivativeAlongCurves}, we have
	\[
		\absval[g]{\delta_{,tt}} \leq \absval[g]{\nabla dy(v,v)} + \absval[g]{dy\,v}^2 \, \max \norm \Gamma
	\]
	By \eqnref{eqn:ChrisBoundFermiCoords}, we have
	$\absval{dy\,v} = \absval[g]{\dot \delta}
	\simleq (1 + C_{0,1} \rho^2) \absval[\ell^2]{\dot \delta}$, which
	means that both norms are equivalent for small $\rho$.
	Similarly, $\max \norm \Gamma \simleq C_{0,1} \rho$. Together with
	$\absval[g^e] v = \alpha \leq h$,
	we have
	\[
		\absval[g|_{p_i}]{(dx - dy)v} \simleq h 
			\left(\inorm[\Leb^\infty(\interv 0 a g^e,Mg)]{\nabla dy}
			+ C_{0,1} \rho \inorm[\Leb^\infty(c g^e,Mg)]{dy}^2 \right) \absval[g^e]v.
	\]
	This shows
	the claimed estimate for edge vectors. And	
	some general $v$ that is not tangent to an edge
	can be represented as linear combination of edge tangents $e_i$,
	and all coefficients $v^i$ are estimated from above by $\absval[g] v / \theta$
	up to a constant,
\end{proof}
\begin{remark}
	\begin{subenum}
	\item
	For triangles, the fullness parameter $\theta$ controls
	the minimum angle at each vertex. This is exactly the
	parameter that enters in the last argument, so there
	is a direct geometry meaning of the factor $\theta\inv$.
	\item
	There are also coordinate-free methods to prove
	\ref{prop:dxMinusdyAtVertices}, but we did not find any
	method that is ``so intrinsic'' that no curvature
	term like $C_{0,1} h \rho \inorm{dy}$ comes in. For
	example, one could transport $\dot \delta$ and $\dot \gamma$
	both to the vertex $p=y(e_i)$ and do all comparisons there.
	Then the estimate $\delta_{,tt} - \ddot \delta$ is
	not needed anymore, but some $\nabla P$ and the holonomy
	$P^{\gamma(t),\delta(t)} - P^{\gamma(t),p}P^{p,\delta(t)}$
	have to be estimated by \ref{prop:estimateOfParallelTransportDeriv}
	and \ref{prop:holonomyEstimate}.
	\bibrembegin
	\item The term in parentheses on the right-hand side of
	\ref{prop:dxMinusdyAtVertices} is what \cite{Grohs13}
	estimate by their ``smoothness descriptor''. Our computation
	shows that the nonlinear lower-order term $\ibetrag{dy}^2$
	only enters with an additional distance factor $\rho$.
	\bibremend
	\end{subenum}
\end{remark}
\begin{proposition}
	\label{prop:estimatedxMinusPdy}
	Situation as in \ref{sit:interpolationInManifolds}.
	Then if $\dist(x,y) \leq \rho$ everywhere in $\stds$,
	\[
		\inorm[\Leb^\infty(\stds g^e,Mg)]{dx-Pdy}
		\simleq h \theta^{-1}\left(\inorm[\Leb^\infty(c g^e,Mg)]{\nabla dy - P\nabla dx}
						+ C_{0,1} \rho \inorm[\Leb^\infty(c g^e,Mg)]{dy}^2 \right).
	\]
\end{proposition}
\begin{proof}
	Let us prove the claim at some $p = x(\mu)$.
	Consider some vector $v \in T\stds$, and let
	$V := (dx-Pdy)v$. Along a geodesic
	$\gamma = x \circ c: p_i \leadsto p$, which comes from
	a curve $c: e_i \leadsto \mu$ in $\stds$, we have
	by the fundamental theorem \eqnref{eqn:fundamentalTheoremCalculus}
	\[
		V|_p = \tilde P^{1,0}V|_{p_i} + \Int_0^1 \tilde P^{1,t} D_t V|_{\gamma(t)} \d t,
	\]
	where $\tilde P$ is the parallel transport along $\gamma$
	(not to be confused with the parallel transport $P$
	along geodesics $y \leadsto x$).
	Inside the integral, we have
	$D_t V = \nabla_{dx\,\dot c} V = \nabla_{dx\,\dot c} dx\,v
	- \nabla_{dx\,\dot c}(Pdy)v$. As in the proof of
	\ref{prop:estimateOfParallelTransportDeriv}, define a
	mapping $f: x(\lambda) \mapsto y(\lambda)$.
	Then $df(dx\,w) = dy\,w$ and hence
	$\nabla_{dx\,\dot c} (P dy) v = P \nabla_{dy\,\dot c} dy\,v
	+ (\nabla_{dx\,\dot c} P)(dy\,v)$. Together, this gives
	\[
		\begin{split}
		D_t V
		& = \nabla_{dx\,\dot c} dx\,v - P \nabla_{dy\,\dot c}dy\,v - (\nabla_{dx\,\dot c}P)dy\,v \\
		& = \nabla dx(\dot c,v)- P \nabla dy(\dot c,v) - (\nabla_{dx\,\dot c}P)dy\,v.
		\end{split}
	\]
	By \ref{prop:estimateOfParallelTransportDeriv},
	we have $\absval{\nabla_{dx\,\dot c} P} \leq C_0 \rho
	\max \absval{\partial_s}$, where $\rho$ is again the maximum
	distance between $x$ and $y$, and $\partial_s$ is the
	vector field defined in the proof above and has values
	$dx\,\dot c$ and $dy\,\dot c$ at the endpoints $x$ and $y$. Thus
	\[
		\absval[g]{D_t V}
			\leq \norm{\nabla dy - P \nabla dx} \absval[g^e]{\dot c} \, \absval[g^e] v
			+ C_0 \rho \max \absval{\partial_s} \, \norm{dy} \, \absval[g^e] v
	\]
	Now observe $\absval{\partial_s} \simleq \absval{dy\,\dot c}$,
	which gives
	\[
		\absval[g]{D_t V}
		\simleq h \big(\norm{\nabla dy - P\nabla dx}
			  + C_0 \rho \norm{dy}^2 \big) \absval[g^e] v \absval[g^e]{\dot c}.
	\]
	By $\absval[g|_p] V \leq
	\absval[g|_{p_i}] V + \max \absval{D_t V}$,
	the claim is proven with help of \ref{prop:dxMinusdyAtVertices}
	for the initial value at $p_i$,
\end{proof}
\begin{proposition}
	\label{prop:estimateDistancexAndy}
	Situation as in \ref{sit:interpolationInManifolds}. If
	$C_{0,1} \theta^{-1} h \inorm{dy}^2$ is small, then
	\[
		\ibetrag[\Leb^\infty(\stds)]{\dist(x,y)} \simleq h^2 \theta^{-1}
			\inorm[\Leb^\infty(c g^e,Mg)]{\nabla dy - P\nabla dx}.
	\]
\end{proposition}
\begin{proof}
	Consider any point $\lambda \in \stds$ and a geodesic $c: e_i \leadsto \lambda$.
	Then, with \ref{prop:derivativesOfDist},
	\[
		\dist(x(\lambda),y(\lambda)) = \int d(\dist(x,y))\dot c
		\leq \int \absval[g]{(dx-Pdy)\dot c}
		\leq h \inorm{dx-Pdy},
	\]
	everywhere, and this norm is estimated by \ref{prop:estimatedxMinusPdy}: There
	is a constant $\alpha$ such that
	\[
		\frac 1 {1 - \alpha C_{0,1} h \theta^{-1} \inorm{dy}^2}
		\dist(x(\lambda),y(\lambda))
		\leq  h \theta^{-1} \inorm{\nabla dy - P\nabla dx},
	\]
	and the assumption means that the fraction is greater
	than, say, $\frac 1 2$,
\end{proof}
\begin{remark}
	\begin{subenum}
	\item	The smallness assumption on
	$C_{0,1} \theta^{-1} h \inorm{dy}^2$ is reasonable
	because the interesting situation is when the domain
	is decomposed into finer and finer simplicial
	complexes. In this case $h \to 0$, whereas
	(given that the subdivision is performed intelligently)
	$\theta$ can be bounded from below independent of $h$,
	and $C_{0,1}$ as well as $\inorm{dy}$ are
	independent of this refinement (here it is important
	that $\norm{dy}$ is taken with respect to $g^e$ on $\stds$,
	not $\ell^2$).
	\item	\label{rem:scaleAwarenessOfInterpolEstimate}
	The estimates are scale-aware:
	When $\stds$ is scaled by $\bar g^e = \nu^2 g^e$
	and $M$ is scaled by $\bar g = \mu^2 g$
	like in \eqnref{eqn:scalingOfDifferentials}, then
	both sides of the estimates \ref{prop:estimatedxMinusPdy}
	and \ref{prop:estimateDistancexAndy} scale similarly, namely
	like $\frac \mu\nu$ and like $\mu$ respectively: In fact,
	$\bar r = \mu r$, $\bar h = \nu h$, $\norm[\bar g] R
	= \frac 1{\mu^2} \norm[g] R$, and derivatives of $x$ and $y$
	scale like in \eqnref{eqn:scalingOfDifferentials}:
	$\norm[\bar g^e, \bar g]{dx} = \smallfrac \mu\nu \norm[g^e,g]{dx}$,
	$\norm[\bar g^e, \bar g]{\nabla dx} = \smallfrac\mu{\nu^2}
	\norm[g^e,g]{\nabla dx}$ and similar for $y$.
	\end{subenum}
\end{remark}
\begin{conclusion}
	\label{prop:estimateDistancexAndyH1}
	Taking \ref{prop:estimateDistancexAndy}
	and \ref{prop:estimatedxMinusPdy} together,
	we get
	\[
		\ibetrag[\Leb^\infty(\stds)]{\dist(x,y)}
		+ h \inorm[\Leb^\infty(\stds g^e,Mg)]{dx - Pdy}
		\simleq h^2 \theta^{-1}
			\inorm[\Leb^\infty(\stds g^e,Mg)]{\nabla dy - P\nabla dx}.
	\]
\end{conclusion}
\begin{theorem}
	\label{prop:interpolNtoMEstimateSingleSimplex}
	Let $N$ and $M$ be Riemannian manifolds with
	curvature bounds $C_0$ and $C_1$ as usual, and $y: N \to M$
	be a given smooth function.
	Suppose $p_0,\dots,p_n \in N$ are given points in $(\theta,h)$-close
	position with $h$ so small that their
	barycentric mapping $\stds \to s$ is injective,
	where  $s \subset N$ is the Karcher simplex with respect to
	vertices $p_i$, and furthermore suppose
	that the barycentric mapping $\stds \to M$
	with respect to vertices $y(p_i)$ is well-defined.
	Then if $C_{0,1}' h \inorm[\Leb^\infty]{dy}^2$
	is small in comparison to the dimensions,
	there is a function $y_h: s \to M$
	interpolating $y$ at the $p_i$,
	with
	\[
		\ibetrag[\Leb^\infty(s)]{\dist(y_h,y)}
		+ h \inorm[\Leb^\infty(s,M)]{dy_h - Pdy} \\
		\simleq h^2 \theta^{-1}
			\inorm[\Leb^\infty(s,M)]{\nabla dy_h - P\nabla dy}.
	\]
\end{theorem}
\begin{proof}
	By \ref{prop:bijectivex}, there is a bijective barycentric
	mapping $x_N: \stds \to s$ with $e_i \mapsto p_i$. If
	$x_M: \stds \to M$ is the barycentric mapping with respect
	to vertices $y(p_i)$, which have distance less than
	$h \inorm{dy}$, set $y_h:= x_M \circ x_N\inv$.
	The the estimate is a combination of
	\ref{prop:estimateDistancexAndyH1} and
	\ref{prop:normEquivalences},
\end{proof}
\begin{remark}
	\begin{subenum}
	\item	One could have proven the intermediate estimates
	\ref{prop:dxMinusdyAtVertices} and
	\ref{prop:estimatedxMinusPdy} for scaled versions of $\stds$
	and $M$, for example with $\ell^2$ instead of $g^e$ or
	$\diam M \leq 1$. But we did not consider the situation
	above complicated enough to justify a separate scaling
	argument. But if one likes, the argument obviously could
	have been executed for $\stds$ and $M$ having both
	unit size. Then \ref{rem:scaleAwarenessOfInterpolEstimate}
	is the equivalent of the usual ``transformation from the reference
	element''.
	\item	The step from the $\Leb^\infty$ estimate
	\ref{prop:estimateDistancexAndyH1} to an $\Leb^p$ estimate
	works exactly as in
	\ref{prop:InterpolationEstimateLpForRealvaluedFunctions}, so we
	save paper by not repeating all the integrals.
	\item	The estimate could be considered as ``incomplete
	work'', as the right-hand side still contains a
	$\nabla dx$ term. We decided not to estimate it
	by \ref{prop:EstimateForNabladx} to make clear that
	the right-hand side tends to zero if $y$ is an
	``almost barycentric'' map.
	\item	\label{rem:higherOrderInterpolSander}
	For a ``higher-order'' interpolation,
	\cite{Grohs13} use basis functions $\phi^i: \stds \to M$
	of higher order, not just $\phi^i(\lambda) = \lambda^i$ as we did,
	that fulfill $\phi^1 + \dots + \phi^k = 1$ and
	$\phi^i(\mu_j) = \delta^i_j$ for control points
	$\mu_j \in \stds$. Then the interpolation with
	respect to points $p_i = y(\mu_i)$ is the minimiser
	of $\phi^i \dist^2(p_i,\argdot)$.
	
	If one chooses the $\phi^i$ such that there are $k+1$
	control points on each edge, as is usually done
	(e.\,g. for the quadratic basis functions
	$\lambda^i(2\lambda^i - 1)$ and $4\lambda^i\lambda^j$,
	$i \neq j$), then \ref{prop:dxMinusdyAtVertices} and
	\ref{prop:estimatedxMinusPdy} can obviously be interated
	to give
	\[
		\inorm{\nabla^\ell dx - P \nabla^\ell dy}
		\simleq h^{k-\ell} \theta^{\ell-k} \inorm{\nabla^k dx - \nabla^k Pdy}
			+ \text{ curvature terms}.
	\]
	The only point of difficulty is to show that $\inorm{\nabla^k dx}$
	is actually bounded by $\inorm{dy}$ and the geometry, which
	needs quite some
	computation, as it involves many derivative norms
	$\inorm{\nabla^j X_i}$, but is provable along the same lines.
	\end{subenum}
\end{remark}

%
\newsectionpage
\section{The Karcher--Delaunay Triangulation}
\label{sec:KarcherDelaunayTriangulation}
%

\begin{notation}
	The term ``triangulation'' is used in various senses in (discrete)
	geometry, topology, and computational mathematics. We will
	use it only in the topological meaning as a
	map $r\complex \to M$ for some simplicial complex $\complex$. (To obtain
	the correct homology, one usually requires this map to be
	bijective. We will construct this map and give
	conditions for its injectivity, but we will also call
	it a triangulation without these conditions.)
	The partition of a space $M$ into topological disks will
	in contrast be referred to as a ``tesselation''.
\end{notation}
\begin{goal}
	In this section, we want to explore how a simplicial
	structure can be imposed on a ``$\delta$-dense'' point set
	in a manifold. Conversely, if the simplicial
	structure and the vertex set are given, the resulting triangulation
	weill be considered in section \ref{sec:graphSubmanifolds}.
\end{goal}
\begin{situation}
	\label{sit:deltaDensePointSet}
	In the following, the usual assumption from
	\ref{sit:compactManifoldWithCurvatureBounds} that $M$
	is closed (i.\,e. complete and without boundary)
	is essential. Let $V \subset M$ be a
	set of finitely many, but at least $m$ points in $M$ and $\delta > 0$
	be such that each $\delta$-ball in $M$ contains
	at least one point from $V$. We say that $V$ is
	\begriff{$\delta$-dense} in $M$.
\end{situation}
\begin{definition}[\citealt{Leibon00}]
	Situation as in \ref{sit:deltaDensePointSet}. Let
	$p \in V$. The \begriff{Voronoi cell} of $p$ is
	the set of points in $M$ that are nearer to $p$
	than to any other point in $V$:
	\[
		V_p := \{a \in M \mit \dist(a,p) \leq \dist(a,q)
		\forall q \in V \}
	\]
	These sets cover $M$, overlapping only on
	their boundaries. The cover $\{V_p \mit p \in V\}$
	is called the \begriff{Voronoi tesselation} of $M$
	with vertex set $V$.
	
	The \begriff{bisector} $B_{pq}$ of $p$ and $q \in V$ is
	the set of points which have equal distance to $p$
	and $q$, but larger distance to all other points in $V$:
	\[
		B_{pq} := \{ a \in M \mit \dist(a,p) = \dist(a,q)
		\leq \dist(a,r) \forall r \in V\}
	\]
	Obviously, $B_{pq} = V_p \cap V_q$.
	Similarly, the bisector $B_\simplexs$ of a set $\simplexs \subset V$
	is defined as $\bigcap_{p \in \simplexs} V_p$. The set
	$V$ is said to be \begriff{generic} if each
	non-empty bisector $B_\simplexs$ is a disk-type submanifold
	(with boundary) of codimension $\absval \simplexs - 1$ and
	$B_\simplexs$ is empty for $\absval \simplexs > m+1$.
\end{definition}
\begin{proposition}
	Situation as in \ref{sit:deltaDensePointSet}, $p \in V$.
	Then $V_p$ has diameter less than $2\delta$. If
	$2 \delta \leq \cvr M$,
	then $V_p$ is a topological ball.
\end{proposition}
\begin{proof}
	\textit{ad primum:}
	No point in $V_p$ has distance greater than $\delta$
	from $p$, so the claim is simply the triangle inequality.
	
	\textit{ad sec.:} If geodesics starting from $p$ are unique,
	then $V_p$ is star-shaped, cf. \ref{sec:convexityRadius},
\end{proof}
\begin{remark}
	\begin{subenum}
	\item	Let $\simplexs \subset V$ with nonempty $B_\simplexs$.
	Locally, $B_\simplexs$ is a smooth submanifold, and its tangent space
	in a non-boundary point $a \in B_\simplexs$ is
	\[
		\{W \in T_a M \mit g\sprod W{X_p - X_q} = 0
		\forall p,q \in \simplexs\},
	\]
	where $X_p$ and $X_q$ are the gradients of squared
	distances as in \ref{rem:definitionOfX}.
	In fact, a curve $\gamma$ with image in $B_\simplexs$
	fulfills $\dist^2(\gamma(t),p) - \dist^2(\gamma(t),q) = 0$
	everywhere, which has derivative $g\sprod{\dot\gamma}
	{X_p - X_q}$.
	\item	Note that $V_p$ will in general not be convex
	because $V_p$ and $V_q$ could only be both
	convex if $B_{pq}$ were totally geodesic.
	\cite{Beem75} showed that all bisectors are totally
	geodesic if and only if $M$ has constant sectional
	curvature.%
	%
	\item	Generally, the properties of topological
	spheres in Riemannian manifolds are treated by \cite{Karcher68}:
	A topological sphere that does not meet its cut locus
	cuts $M$ in two open sets, some ``interior'' ball and
	some ``outside''. A set $B \subset M$ is convex if and
	only if each $p \in \Rand B$ has a ``geodesic support plane'',
	that is a subspace $H_p \subset T_p M$ of codimension 1, such that
	all starting directions $(\exp_p)\inv a$ for points $a \in B$
	lie on the same side of $H_p$.
	\item
	\cite{Boissonnat11} remark that genericity of a point
	set, which can be achieved in Euclidean space by
	an arbitrarily small perturbation of a degenerate
	point set, is not always removable by infinitesimally
	small changes of $V$. For this reason, we assume a
	generic $V$ and disregard the question of
	sharp conditions that ensure this.
	The problem is currently treated in detail by
	\citename{Dyer} and \citename{Wintraecken}
	(Rijksuniversiteit Groningen).
	\end{subenum}
\end{remark}
\begin{proposition}
	\label{prop:delaunayComplexIsComplex}
	Situation as in \ref{sit:deltaDensePointSet}
	with generic $V$ and $2 \delta < \cvr M$.
	Then $\complex^\ell:= \{ \simplexs \subset V \mit B_\simplexs$ is non-empty,
	$\simplexs \in \complex^{\ell-1} \}$,
	$\ell = 0,\dots,m$, define a regular simplicial complex without boundary,
	called the \begriff{Delaunay complex} for $M$, with vertex set
	$\complex^0 = V$.
\end{proposition}
\begin{proof}
	The only property for a simplicial complex, that some
	$\simplext \subset \simplexs$ with cardinality $k$ is contained
	in the set of $k$-simplices $\complex^k$, is clear, because
	$B_\simplext \subset B_\simplexs$ for $\simplext \subset \simplexs$.
	
	It remains to show that $\complex$ is regular and has no boundary. An $(m-1)$-simplex
	$\simplext$ cannot be part of more than two $m$-simplices,
	because a non-constraint bisector
	$\{ a \in M \mit \dist(a,p) = \dist(a,q)\}$ divides $M$ into two
	distinct sets. On the other hand, there cannot be only one
	$m$-simplex containing $\simplext$, because boundaries of the Voronoi
	cells can only occur where two cells meet if $M$ has no boundary
	for itself.
	And each $\ell$-simplex belongs to an $(\ell+1)$-simplex:
	Let $\simplext \in \complex^\ell$. Because $B_\simplext$ is a topological
	disk of dimension $n - \ell$, it must have a boundary, which
	in turn can only consist of bisectors $B_\simplexs$ with
	$\simplext \subset \simplexs$,
\end{proof}
\begin{remark}
	Note that the following situations are ruled out by our assumptions:
	\begin{subenum}
	\item	an $m$-dimensional sphere with $V = \{p_0,\dots,p_m\}$, because
	$B_V$ would consist of two points of equidistance, which is not a $0$-ball,
	but a $0$-disk. However, the definition of Voronoi regions
	would be feasible, but its dual would consist of two $m$-simplices
	with the same vertices, and our notation does not allow
	to distinguish between them.
	\item	$m + 2$ equidistant points $V = \{p_0,\dots,p_{m+1}\}$
	in an $m$-dimensional manifold, because this $V$ is not
	generic: In fact, $B_V$ would be the point of equidistance,
	but this set should be empty, as $\absval V > m + 1$.
	\item	The counterexample of \citet[pp. 38sqq.]{Boissonnat11},
	because the bisector $B_{\{p,u,v,w\}}$ is not empty.
	\end{subenum}
\end{remark}
\begin{definition}
	\label{def:globalMappingx}
	Situation as in \ref{sit:deltaDensePointSet} with generic $V$ and $2 \delta < \cvr M$.
	Let $\complex$ be the complex from \ref{prop:delaunayComplexIsComplex}.
	For $\simplexe \in \complex^n$, let $x_\simplexe$ be the
	mapping from \ref{def:mappingx}. As $x_\simplexe|_{r\simplexf}$
	only depends on $\simplexf$ for $\simplexf \subset \simplexe$,
	this piecewise definition gives a well-defined
	mapping $x: r\complex \to M$, called the \begriff{Karcher--Delaunay triangulation}
	of $M$ with vertex set $V$.
\end{definition}
\begin{proposition}
	\label{prop:KarcherTriangIsTriang}
	The Karcher--Delaunay triangulation is indeed a triangulation
	in the usual sense:
	Situation as in \ref{sit:deltaDensePointSet} with generic $V$ and $2 \delta < \cvr M$.
	If $\delta$ is so small that the requirements of
	\ref{prop:bijectivex} are met on each Karcher simplex, then $x$ is bijective.
\end{proposition}
\begin{proof}
	The map $x$ is surjective because its image is non-empty and
	has no boundary in $M$. By \ref{prop:bijectivex}
	each $x_\simplexe$ is injective,
	and as the Karcher simplices do not overlap except on
	their boundaries, so is $x$,
\end{proof}
\begin{remark}
	\begin{subenum}
	\item
	If $M$ is not closed, but compact and with boundary,
	the construction is of course feasible, but will only be bijective
	if there are also points on the boundary and the boundary is aligned
	with their Karcher--Delaunay triangulation.
	\bibrembegin
	\item	\label{rem:triangulationDefinitionCheegerMuellerSchrader}
	The construction of \cite{Cheeger84} seems similar, but (of course)
	does not use our barycentric mapping. It starts with a triangulation
	$x: r\complex \to M$, considers finer and finer
	subdivisions $s: r\complex' \to r\complex$
	of the complex, and then compares the metric $(s\circ x)^*g$ on
	$r\complex'$ to the piecewise flat metric induced by edge lengths
	$\ell_{ij} = \dist(x\circ s(ri), x \circ s(rj))$, for
	edges $ij \in (\complex')^1$ in the subdivided complex.
	\item	We know, however, that \cite{Burago13} state that ``it is
	now clear that in dimensions beyond three polyhedral structures
	are too rigid to serve as discrete models of Riemannian spaces with
	curvature bounds'', but nevertheless there will certainly be
	rigidity results for spaces of piecewise constant curvature
	without counterparts in the smooth category, we are not convinced
	that the references they give support this statement in its
	full generality.
	\bibremend
	\end{subenum}
\end{remark}

%
\newsectionpage
\section{A Piecewise Constant Interpolation of \textsc{dec}}
\label{sec:dec}
%

\begin{goal}
	As second main construction of this thesis, we will now give an interpretation
	of the discrete exterior calculus as piecewise constant differential forms,
	which turns variational problems in the simplicial cohomology $(C^k,\Rand^*)$
	into problems in a complex $(\Pol\inv\Omega^k,\bard)$. The main question will
	be the connection between $\bard$ and the usual exterior derivative $d$
	on $\Sob^{1,0}\Omega^k$. The introductional definitions are the basics of
	simplicial homology as they can be found in any topology textbook,
	e.\,g. \cite{Munkres84} or \cite{Hatcher01}.
\end{goal}


\subsection{Discrete Exterior Calculus (\textsc{dec})}


\begin{definition}
	Let $R$ be a ring with neutrals $0$ and $1$, and
	let $\complex_\ori$ be a regular $n$-dimensio\-nal
	oriented simplicial complex.
	For any simplex $\simplexs \in \complex_\ori^k$, let $\chi_\simplexs:
	\complex_\ori^k \to R$ be defined by $\chi_\simplexs(\simplexs) = 1$ and
	$\chi_\simplexs(\simplexs') = 0$ for any $\simplexs' \neq \simplexs$.
	
	Consider the $R$-module $\tilde C_k(\complex)$ that is spanned
	by all $\chi_\simplexs$,
	$\simplexs \in \complex_\ori^k$. Let $C_k(\complex)$, the space
	of \begriff{$k$-chains over $\complex$} (with coefficients in $R$),
	be the quotient of $\tilde C_k(\complex)$ under the identification of
	$\chi_{\simplexs^-}$ and $-\chi_\simplexs$.
	Its dual space $C^k(\complex)$, the $R$-module of all homomorphisms
	$C_k(\complex) \to R$, is called the space of \begriff{$k$-cochains
	over $\complex$} (with coefficients in $R$).
	Let $f^\simplexs$ be the generators of $C^k(\complex)$ dual to $\delta_\simplexs$,
	that means $f^\simplexs(\chi_\simplexs) = 1$
	and $f^\simplexs(\chi_{\simplexs'}) = 0$ for $\simplexs \neq \simplexs'$.
	In the following, we will only use $R = \R$.
	
	The \begriff{boundary operator} is the linear
	map $C_k(\complex) \to C_{k-1}(\complex)$, defined on the generators by
	\[
		\Rand \, \chi_{[p_0,\dots,p_k]}
			:= (-1)^i \chi_{[p_0,\dots,\hat p_i,\dots,p_k]}
	\]
	(as usual, summation over $i$ is intended),
	where $\hat p_i$ means that this vertex is omitted.
	With respect to the basis $\chi_\simplexs$, we write
	$\Rand$ in coefficients:
	\[
		\Rand \, \chi_\simplexs = \Rand^\simplext_\simplexs \chi_\simplext
		\qquad
		\text{for $\simplexs \in \complex^k$},
	\]
	where summation over $\simplext \in \complex^{k-1}$ is
	intended. For the whole section, we will sum about indices
	occuring twice in a product, irrespective if
	they are superscripts oder subscripts. Volume terms
	like $\absval \simplexs$ or $\absval{U(\simplexs)}$
	do not count for this, as $\simplexs$ is no sub- or
	superscript in them.
	The \begriff{co-boundary operator} $\Rand^*$ is
	the dual of $\Rand$, i.\,e. a map $C^{k-1}(\complex)
	\to C^k(\complex)$ uniquely characterised by
	$\Rand^* \alpha(c) = \alpha(\Rand c)$ for all
	$\alpha \in C^{k-1}(\complex)$ and all $c \in C_k(\complex)$.
\end{definition}
\begin{remark}
	\begin{subenum}
	\item	By a direct computation, or by common linear
	algebra knowledge, one obtains that
	the matrix representation of $\Rand^*$ is the
	transposed of the matrix representation of $\Rand$.
	In other words,
	$(\Rand^*)^\simplext_\simplexs = \Rand_\simplexs^\simplext$
	for $\Rand^* f^\simplext = (\Rand^*)^\simplext_\simplexs f^\simplexs$.
	\item	It would be very natural to write $\delta_\simplexs$ instead
	of $\chi_\simplexs$, because $\chi_\simplexs$ actually \emph{is}
	the Kronecker delta on $\complex^k$. But we will already have some operator
	$\delta$ acting on differential forms, we will define some $\delta$ for cochains
	and some $\bardelta$ on piecewise constant forms. In the whole
	following section, we will not use the Kronecker symbol.
	\bibrembegin
	\item	The use of functions $\chi_\simplexs$ as generators
	of $C_k(\complex)$ is only one possible definition. The other frequently
	encountered possibility is to speak of ``formal linear combinations''
	of the $\simplexs$ themselves (e.\,g. \citealt{Hatcher01} and \citealt{Hirani03} use this
	definition). Logically, there is no difference between both definitions,
	as the only strict way to define ``formal linear combinations'' is
	to use the characteristic functions $\chi_\simplexs$. However, the existence
	of both approaches introduces an unpleasant notational ambiguity
	that may disturb a quick reader:
	Linear maps from simplices to $R$ are chains in our notation,
	whereas they represent \emph{co}chains in the other.
	Our notation has the advantage to employ $\simplexs$ only
	as sub- and superscript, but not as term, which allows
	for usual summation convention.
	\bibremend
	\end{subenum}
\end{remark}
\begin{lemma}
	\label{prop:rankOfBoundaryMatrix}
	Let $\complex$ consist of one single $n$-simplex.
	Then the boundary map $\Rand_k: C_k \to C_{k-1}$,
	which can be written as $\binom{n+1}{k+1} \times
	\binom{n+1} k$-matrix,
	has rank $\binom n k$.
\end{lemma}
\begin{proof}
	The matrix size just stems from counting the
	elements in $\complex^k$, which arise from choosing
	$k+1$ vertices out of $n+1$.
	
	The rank of $\Rand_k$ is proven by induction
	over $k$, starting with $k = n$. Here the statement
	is that $\Rand_n: C_n \to C_{n-1}$ has rank one,
	which is true because $\Rand_n \neq 0$.
	For any $k < n$, the rank-nullity theorem gives
	that the rank of $\Rand_k$ can be computed as
	dimension $\binom{n+1}{k+1}$ of its image
	space minus the dimension of its kernel, which is
	the rank of $\Rand_{k+1}$ because the $k$'th
	homology group of the simplex vanishes. And by
	assumption, $\Rand_{k+1}$ has rank $\binom n{k+1}$,
	which gives rank $\Rand_k = \binom{n+1}{k+1} - \binom n {k+1}
	= \binom n k$,
\end{proof}
\parag{Short introduction to discrete exterior calculus (\textsc{dec}).}
\label{sec:introductionToDec}
\begin{subeqns}
The discrete exterior calculus (\citealt{Desbrun05}, \citealt{Hirani03})
attempts to build a simple and useable
finite-dimensional version of the de Rham cohomology based on an
intelligent interpretation of simplical cohomology. It calls $\Rand^*$
the \begriff{discrete exterior derivative} $d$, which gives that $d:
C^k(\complex) \to C^{k+1}(\complex)$ acts as
\begin{equation}
	\label{eqn:defExtDerivDEC}
	d f^\simplext = d_\simplexs^\simplext f^\simplexs
	\qquad \text{with $d_\simplexs^\simplext = \Rand^\simplext_\simplexs$}.
\end{equation}
If points $\lambda_\simplexs \in r\simplexs$ for all
simplices $\simplexs$ and numbers $a_k \in \R$ are given,
leading to dual cells $r(*\simplexs)$ as in \eqnref{eqn:defDualCell},
it defines the scalar product of two discrete $k$-forms as
\begin{equation}
	\label{eqn:defScalarProductDEC}
	\dprod{\alpha_\simplexs f^\simplexs}{\beta_{\simplexs'} f^{\simplexs'}}_{C^k}
		:= a_k \alpha_\simplexs \beta_\simplexs
		\smallfrac{\betrag{*\simplexs}}{\betrag{\simplexs}}.
\end{equation}
\begin{remark_nn}
	The numbers $a_k$ do usually not occur in
	the definition of the scalar product, but we will see
	that they must be chosen as $a_k = \binom nk$ to obtain
	a correspondence to piecewise constant forms.
	The points $\lambda_\simplexs$ are classically chosen to be the
	circumcentres of the $r\simplexs$.
\end{remark_nn}

The coderivative $\delta$ is supposed to be dual to $d$
with respect to this scalar product, that means
${\dprod \alpha {d\beta}}_{C^k} = {\dprod{\delta \alpha}\beta}_{C^{k-1}}$
for all $\alpha \in C^k(\complex)$ and all $\beta \in C^{k-1}(\complex)$.
Spelling out both sides for $\alpha = f^\simplexs$
and $\beta = f^\simplext$ gives
\begin{equation}
	d^\simplext_\simplexs
	\smallfrac{\betrag{*\simplexs}}{\betrag{\simplexs}} a_k
	= \delta^\simplexs_\simplext
	\smallfrac{\betrag{*\simplext}}{\betrag{\simplext}} a_{k-1},
	\qquad \Leftrightarrow \qquad
	\delta^\simplexs_\simplext =
		\frac{a_k}{a_{k-1}\,}\!\!\frac{\betrag{*\simplexs} \betrag{\simplext}}
		{\betrag{*\simplext} \betrag{\simplexs}}
		\Rand_\simplexs^\simplext.
\end{equation}
Other definitions are obvious: A form $\alpha$ is called
harmonic if $(\delta d + d \delta)\alpha = 0$ etc.
\end{subeqns}

\subsection{Piecewise Constant Differential Forms}

\begin{situation}
	\label{sit:dec}
	Let $\complex_\ori$ be an oriented regular $n$-dimensional
	simplicial complex with a discrete Riemannian metric $g$,
	let $\complex$ be the corresponding non-oriented complex, and let
	$\lambda_\simplexs$ for any simplex $\simplexs$ define
	subdivision neighbourhoods $U(\simplexs)$.
\end{situation}
\begin{definition}
	\label{def:PolOmegak}
	Situation as in \ref{sit:dec}. Let $\Pol^0\Omega^k$ be the
	space of $\Leb^\infty\Omega^k$ forms that are constant
	in $U(\simplexs)$ for each $\simplexs \in \complex^k$.
	
	Any simplex $\simplexs \in \complex^k$
	has a volume $k$-form $\dvol_{r\simplexs}$
	which can be extended to a constant $k$-form
	in whole $U(\simplexs)$. Denoting the extension also
	as $\dvol_{r\simplexs}$, let
	\[
		\omega^\simplexs := \begin{cases}
							\dvol_{r\simplexs} & \text{in } U(\simplexs) \\
							0 & \text{elsewhere},
							\end{cases}
		\qquad \text{and} \quad
		\Pol\inv\Omega^k := \Span\{\omega^\simplexs
		\with \simplexs \in \complex^k\} \subset \Pol^0\Omega^k.
	\]
\end{definition}
\begin{example}
	Situation as in \ref{sit:dec}, dimension $n = 2$.
	Consider two triangles
	$rijk$ and $rjil$, which together contain the
	subdivision neighbourhood $U(ij)$. Then
	$\omega^{ij}$ is the flattened unit vector vector field
	in direction $ri-rj$ in $U(ij)$ and zero elsewhere,
	$\omega^i$ is the characteristic function of $U(i)$,
	and similarly $\omega^{ijk}$ is the volume form of $r\complex$
	in $rijk$ and the zero $2$-form elsewhere.
\end{example}
\begin{observation}
	\begin{subeqns}
	All basis elements have pointwise unit length with respect
	to the metric induced on the tensor bundles by $g$,
	and have distinct support up to null sets,
	so the $\Leb^2$ scalar product has diagonal form
	in the basis $\omega^\simplexs$:
	\begin{equation}
		\label{eqn:scalarProductOnPol}
		{\dprod{\alpha_\simplexs \omega^\simplexs}
			{\beta_{\simplexs'} \omega^{\simplexs'}}}_{\Leb^2\Omega^k} =
			\betrag{U(\simplexs)} \alpha_\simplexs \beta_\simplexs
	\end{equation}
	\end{subeqns}
\end{observation}
\begin{definition_nn}
	\label{def:bard}
	\begin{subeqns}
	\addtocounter{equation}1
	Situation as in \ref{sit:dec}. Let $\bard: \Pol\inv\Omega^{k-1}
	\to \Pol\inv\Omega^k$ be defined by
	\begin{equation}
		\label{eqn:defExtDerivPol}
		\bard \omega^\simplext = \bard_\simplexs^\simplext \omega^\simplexs,
		\qquad
		\bard_\simplexs^\simplext :=
		\frac{\betrag\simplext}{\betrag\simplexs} \Rand_\simplexs^\simplext.
	\end{equation}
	Let $\bardelta: \Pol\inv\Omega^k \to \Pol\inv\Omega^{k-1}$ be defined by
	\begin{equation}
		\label{eqn:defExtCoderivPol}
		\bardelta \omega^\simplexs = \bardelta_\simplext^\simplexs\omega^\simplext,
		\qquad
		\bardelta_\simplext^\simplexs =
		\frac{\betrag{U(\simplexs)}}
			{\betrag{U(\simplext)}} \bard_\simplexs^\simplext.
	\end{equation}
	\end{subeqns}
\end{definition_nn}
\begin{proposition}
	\label{prop:bardIsDualToRand}
	\label{prop:RandDualIsIsomophicToBard}
	\begin{subeqns}
	Situation as in \ref{sit:dec}. The maps $\bard$ and $\bardelta$ fulfill the
	\begin{alignat}{2}
		\label{eqn:PolStokesFormula}
		&\text{``discrete Stokes' formula''}&
			\Int_{r\simplexs} \bard \alpha & = \Int_{\Rand r\simplexs} \alpha, \\[2ex]
		\label{eqn:PolGreensFormula}
		&\text{``discrete Green's formula''} &
		\quad {\dprod{\bard\alpha} \beta}_{\Leb^2\Omega^k} & = {\dprod\alpha{\bardelta\beta}}_{\Leb^2\Omega^{k-1}}
	\end{alignat}
	for all $\simplexs \in \complex^k_\ori$, $\alpha \in \Pol\inv\Omega^{k-1}$, and
	$\beta \in \Pol\inv\Omega^k$.
	In particular, $\bard^2 = 0$.
	\end{subeqns}
\end{proposition}
\begin{proof}
	\textit{ad primum:} For any $\alpha \in \Omega^{k-1}$,
	we have $\smash{\Int_{\Rand r\simplexs} \alpha = \Rand_\simplexs^{\simplext'}
	\Int_{r\simplext'} \alpha}$. Now let $\alpha = \omega^\simplext$
	for some $\simplext \in \complex^{k-1}$. Then we have
	\[
		\Int_{\Rand r\simplexs} \omega^\simplext
		= \Rand_\simplexs^{\simplext'} \Int_{r\simplext'} \omega^\simplext
		= \Rand_\simplexs^\simplext \betrag \simplext
	\]
	(without summation over $\simplext$). On the other hand,
	\[
		\Int_{r\simplexs} \bard \omega^\simplext
		= \Int_{r\simplexs} \bard_{\simplexs'}^\simplext \omega^{\simplexs'}
		= \bard_\simplexs^\simplext \betrag \simplexs.
	\]
	
	\textit{ad sec.:} If one spells out both scalar products with help of
	\eqnref{eqn:scalarProductOnPol} for $\alpha = \omega^\simplext \in \Pol\inv\Omega^{k-1}$
	and $\beta = \omega^\simplexs \in \Pol\inv\Omega^k$, one gets
		$\bardelta^\simplexs_\simplext \betrag{U(\simplext)}
			\overset != \bard^\simplext_\simplexs \betrag{U(\simplexs)}$
	for \eqnref{eqn:PolGreensFormula} to hold,
\end{proof}
\begin{remark}
	\label{rem:bdryValuesForDEC}
	The discrete Green's formula \eqnref{eqn:PolGreensFormula}
	holds without assumption on the boundary values because
	the weights $\absval{U(\simplexs)}$ and $\absval{U(\simplext)}$
	already incorporate the smaller extent of $\bardelta$.
	For a correct treatment of boundary conditions in variational
	problems, one would have to modify \eqnref{eqn:defExtCoderivPol}.
	We decided to investigate the original \textsc{dec} setup here.
\end{remark}
\begin{proposition}
	\label{prop:interpolationOfDec}
	Situation as in \ref{sit:dec}. The map $i_k: C^k \to \Pol\inv\Omega^k$,
	$f^\simplexs \mapsto \smallfrac 1{\betrag\simplexs} \omega^\simplexs$
	is a cochain map, i.\,e. each
	square in the following diagram commutes:
	\begin{diagram}[height=5ex]
	C^0 & \rTo_{d} & C^1 & \rTo_{d} & \dots & \rTo_d & C^n\\
	\dTo^{i_0} &     & \dTo^{i_1} &     &&& \dTo^{i_n} \\
	\Pol\Omega^0 & \rTo^{\bard} & \Pol\Omega^1 & \rTo^{\bard} & \dots & \rTo^\bard & \Pol\Omega^n\\
	\end{diagram}
	If $a_k = \binom nk$, it is an isometry for each $k$ and a chain map, i.\,e.
	each square in the following diagram commutes:
	\begin{diagram}[height=5ex]
	C^0 & \lTo_{\delta} & C^1 & \lTo_{\delta} & \dots & \lTo_\delta & C^n\\
	\dTo^{i_0} &     & \dTo^{i_1} &     &&& \dTo^{i_n} \\
	\Pol\Omega^0 & \lTo^{\bardelta} & \Pol\Omega^1 & \lTo^{\bardelta} & \dots & \lTo^\bardelta & \Pol\Omega^n\\
	\end{diagram}
\end{proposition}
\begin{proof}
	The isometry property is clear by the expressions
	\eqnref{eqn:defScalarProductDEC} and \eqnref{eqn:scalarProductOnPol}
	for the scalar product of $C^k$ and $\Pol\inv\Omega^k$ respectively.
	The properties $\bard i_{k-1} = i_k d$
	and $\bardelta i_k = i_{k-1} \delta$
	only need to be checked for basis elements, so it
	suffices to show
	\[
		\bard_\simplexs^\simplext \smallfrac{1}{\betrag\simplext} =
		d_\simplexs^\simplext \smallfrac{1}{\betrag\simplexs},
		\qquad
		\bardelta_\simplext^\simplexs \smallfrac{1}{\betrag\simplexs} =
		\delta_\simplext^\simplexs\smallfrac{1}{\betrag\simplext}
		\qquad
		\forall \simplexs \in \complex^k, \simplext \in \complex^{k-1}.
	\]
	The first one is obvious from definitions
	\eqnref{eqn:defExtDerivDEC}
	and \eqnref{eqn:defExtDerivPol}. The second
	one comes from \eqnref{eqn:defScalarProductDEC},
	as
	\[
		\bardelta_\simplext^\simplexs
			= \frac{\betrag\simplext \betrag{U(\simplexs)}}
				{\betrag\simplexs \betrag{U(\simplext)}} \Rand^\simplexs_\simplext
			= \frac{\textstyle \binom nk}{\textstyle \binom n{k-1}\,}
				\!\!\frac{\betrag{*\simplexs}}{\betrag{*\simplext}} \Rand^\simplexs_\simplext
			= \frac{a_{k-1}\,}{a_k} \!\!\frac{\textstyle \binom nk}{\textstyle \binom n{k-1}\,}\!\!
				\frac{\betrag{\simplexs}}{\betrag{\simplext}} \delta^\simplexs_\simplexs,
	\]
\end{proof}
\bibrembegin
\begin{remark_nn}
	It might seem a little bit queer to use piecewise constant forms
	for this construction and not the elementary forms introduced
	by \citet[sec. IV.27]{Whitney57}
	\[
		\tilde \omega^{[p_0\dots p_k]} = k! \,
			\lambda^i d\lambda_0\wedge\cdots\widehat{d\lambda^i}\cdots\wedge d\lambda^k,
	\]
	which would also make $i$ a cochain map. The reason is
	that we did not succeed to find any relation between
	the $\Leb^2$ scalar product of Whitney's elementary forms
	and the \textsc{dec} scalar product
	\eqnref{eqn:defScalarProductDEC}. This means that although there
	is a worked-out interpolation estimate for the space
	spanned by $\tilde \omega^\simplexs$, $\simplexs \in \complex^k$,
	by \cite{Dodziuk76}, it gives no possibility to compare
	solutions of variational problems that were computed
	using the \textsc{dec} scalar product.
\end{remark_nn}
\bibremend
\begin{proposition}
	Suppose that $r\complex g$ is a piecewise flat,
	$(\theta,h)$-small and absolutely well\hyp centred
	realised simplicial complex,
	that means all circumcentres $\lambda_\simplexs$ have
	barycentric coordinates $\lambda_\simplexs^i > \alpha$,
	and that the circumradii are bounded by $\beta h$.
	Then if $\bar g$ is a second piecewise flat metric
	with $\absval{(g - \bar g)\sprod v w} \leq c h^2
	\absval v \absval w$, it holds for $c':=\frac{c \beta}{\alpha \theta}$:
	\begin{align*}
		\ibetrag[\Leb^2]{\omega^\simplexs_g - \omega^\simplexs_{\bar g}}
			& \simleq c' h^2 \ibetrag[\Leb^2]{\omega^\simplexs_g} \\
		\ibetrag[\Leb^2]{\bard_g \omega^\simplexs_g - \bard_{\bar g} \omega^\simplexs_{\bar g}}
			& \simleq c' h^2 \ibetrag[\Leb^2]{\bard_g \omega^\simplexs_g} \\
		\absval{\dprod{\alpha_\simplexs \omega^\simplexs_g}{\beta_{\simplexs'}\omega^{\simplexs'}_g}_g
			- \dprod{\alpha_\simplexs \omega^\simplexs_{\bar g}}{\beta_{\simplexs'}\omega^{\simplexs'}_{\bar g}}_{\bar g}}
			& \simleq c' h^2 \dprod{\alpha_\simplexs \omega^\simplexs_g}{\beta_{\simplexs'}\omega^{\simplexs'}_g}_g
	\end{align*}
\end{proposition}
\begin{proof}
	The difference between $\bard_g$ and $\bard_{\bar g}$ is easiest, because
	it only involves simplex volumes like $\absval[g] \simplexs$ and
	$\absval[\bar g] \simplexs$. These are close to each other
	by \ref{prop:comparisonEuclidSimplexVolumeForms}. The approximation
	of the scalar product involves comparison between the
	neighbourhood volumes $\absval[g]{U_g(\simplexs)}$ and
	$\absval[\bar g]{U_{\bar g}(\simplexs)}$. These can be estimated
	if we know how the circumcentres are distorted. By
	\eqnref{eqn:inverseOfCayleyMengerMatrix}, these are controlled
	by the distortion of the Cayley--Menger matrix inverse $M_+\inv$,
	and inverses of symmetric matrices are treated by
	\ref{prop:comparisonEuclidSimplexMetricsOnForms}
	(which we apply to $M_+$ instead of $g$):
	\[
		\absval{q^i - \bar q^i} \simleq ch^2 r \absval{\grad_g \lambda^i}
	\]
	(where $r$ is the circumradius with respect to $g$)
	because $4r^2$ and $\absval{v^i}^2 = \absval{\grad_g \lambda^i}^2$ are
	the corresponding diagonal entries of $M_+\inv$.
	By assumption, this is smaller than $c' h^2 \absval{q^i}$,
\end{proof}

\subsection{Connection to the \BV{} Derivative}

\begin{goal}
	Recall that piecewise constant 
	functions possess distributional derivatives,
	which are $(n-1)$-dimen\-sional measures concentrated on the
	jump sets. Their analogue for differential forms are the
	currents from geometric measure theory. (In order to avoid
	``currential derivative'' or similar terms, we will speak of
	\BV{} derivatives.) If our definition of
	discrete exterior derivatives is meaningful, it should be
	connected to this sort of derivative. In fact, the \BV{}
	derivative of a piecewise constant $k$-form $\alpha$ also fulfills
	Stokes' theorem if the jump set is transversal to the integration
	domain. But as their support is $(n-1)$-dimensional, we will see that
	its scaling behaviour does match the one of full-dimensional
	$(k+1)$-forms such as $\bard\alpha$.
\end{goal}
\begin{definition}
	\label{def:BVforms}
	\begin{subeqns}
	The \begriff{comass} of a $k$-covector $\alpha$
	is the absolute value of its largest component, equivalently:
	the maximum over all applications of $\alpha$ to simple
	unit $k$-vectors:
	\[
		\norm[*] \alpha = \max \alpha(e_{i_0}\wedge\cdots\wedge e_{i_k}).
	\]
	For completeness, we also define
	that the \begriff{mass} of a $k$-vector is the norm dual to
	the comass: $\norm[*] v = \max_{\norm[*]\alpha = 1} \alpha(v)$.
	A differential form $\alpha \in \Leb^1_{\loc}\Omega^k(M)$
	has locally \begriff{bounded variation} (is locally of \BV) if
	\begin{equation}
		\label{eqn:defBV}
		\sup_{\mathclap{\substack{\beta \in \Cont_0^1\Omega^{k+1}(U)\\ \norm[*] \beta \leq 1}}} \, \dprod \alpha {\delta \beta}
		\quad \text{is finite}
		\qquad \forall U \subset\subset M,
	\end{equation}
	where of course $\Cont_0^1\Omega^k(U)$ denotes the space of
	continuously differentiable $k$-forms on $M$ with compact
	support inside $U$.
	The space of $k$-forms with locally
	bounded variation is called $\BV_{\loc}\Omega^k$.
	The globalisation to the space $\BV\Omega^k$ is as usual.
	\end{subeqns}
\end{definition}
\begin{fact}[cf. \citealt{Evans92}, thm. 5.1]
	\label{prop:BVstructureThm}
	\begin{subeqns}
	For each $\alpha \in \BV_{\loc}\Omega^k$, there is
	a Borel-regular measure $\mu$ on $M$ and a $\mu$-integrable
	$(k+1)$-form $d^\BV \alpha$ such that
	\begin{equation}
		\label{eqn:BVstructureTheorem}
		\dprod \alpha {\delta \beta} = \Int_M \sprod{d^\BV \alpha} \beta \d\mu
		\qquad \forall \beta \in \Cont^1_0\Omega^k.
	\end{equation}
	We will mostly write $\dprod{d^\BV\alpha}\beta$ as abbreviation
	of the right-hand side.
	\end{subeqns}
\end{fact}
\begin{remark_nn}
	This formulation of the \begriff{\BV{} structure theorem}
	is the one normally used for functions of bounded variation.
	For differential forms, one calls the supremum in
	\eqnref{eqn:defBV} the \begriff{mass} of the current
	(linear form on $\Omega^k$) $\beta \mapsto
	\dprod \alpha {\delta\beta}$, and then observes that every
	current of finite mass is ``representable by integration''
	in the meaning of the theorem (\citealt[sec. 4.1.7]{Federer69}, or
	\citealt[sec. 4.3\textsc b]{Morgan00}). To obtain uniqueness
	of $d^\BV\alpha$, one usually requires it to have
	unit-mass everywhere, and the pointwise scaling then comes
	from $\mu$. As we are only interested in $d^\BV\alpha$
	for $\alpha \in \Pol\inv\Omega^k$, it will be more
	adequate to use a non-unit-length $(k+1)$-form
	and the volume form
	of $\Rand U(\simplexs)$, $\simplexs \in \complex^k$,
	for $\mu$.
		
	For the proof of \ref{prop:BVstructureThm}, we refer to
	\citeauthor{Evans92} (\textit{loc. cit.}), because it only consists of the observation
	that $\beta \mapsto \dprod \alpha {\delta\beta}$
	has a norm-preserving continuation to $\Cont^0_0\Omega^k$,
	and the application of Riesz' representation theorem.
\end{remark_nn}
\begin{proposition}
	\label{prop:BVderivativeOfOmegas}
	For the basis elements $\omega^\simplext$ of
	$\Pol\inv\Omega^k$, the \BV{} derivative is given by
	$\mu = \dvol_{\Rand U(\simplext)}$ and
	$d^\BV\omega^\simplext = \nu \wedge \omega^\simplext$,
	where $\nu$ is the outer normal on $U(\simplext)$.
\end{proposition}
\begin{proof}
	\begin{subeqns}
	If $\beta \in \Cont^1_0\Omega^{k+1}$, the
	product $\sprod{\omega^\simplext}{\delta \beta}$
	is supported only in $U(\simplext)$, where we
	can apply the classical Green's formula because the integrand
	is smooth. So
	\begin{equation}
		\label{eqn:BVderivativeOfOmegat}
		\dprod{d^\BV\omega^\simplext} \beta
			\overset{(\eqnref{eqn:BVstructureTheorem})}= \dprod{\omega^\simplext}{\delta\beta}
			= \Int_{\mathclap{\Rand U(\simplext)}} \omega^\simplext \wedge * \beta
				+ \dprod{d\omega^\simplext}\beta
			= \Int_{\mathclap{\Rand U(\simplext)}}
				\sprod{\nu \wedge \omega^\simplext} \beta \, \dvol_{\Rand U(\simplext)},
	\end{equation}
	the last equality by usual multilinear algebra
	and $d\omega^\simplext = 0$ almost everywhere,
	\end{subeqns}
\end{proof}
\begin{proposition}
	\begin{subeqns}
	There is a variant of Stokes' theorem for the
	\BV{} derivative of $\Pol\inv$ forms: If we define
	\begin{equation}
		\label{eqn:defkDimensionalBVintegral}
		\Int_{r\simplexs} d^\BV\omega^\simplext
			:= \Int_{\mathclap{r\simplexs \cap \Rand U(\simplext)}} \omega^\simplext,
	\end{equation}
	then
	\begin{equation}
		\label{eqn:StokesForBVderivative}
		\Int_{r\simplexs} d^\BV \alpha
			= \Int_{\Rand r\simplexs} \alpha
		\qquad \text{for all $\alpha \in \Pol\inv\Omega^k$, $\simplexs \in \complex^k$}.
	\end{equation}
	\end{subeqns}
\end{proposition}
\begin{proof}
	The homotopy formula is
	easy for constant forms: If $d\alpha = 0$, then
	$0 = \int_U d\alpha = \int_{\Rand U} \alpha$,
	hence $\int_A \alpha = \pm \int_B \alpha$ if the
	integration domains $A$ and $B$ bound a common
	$(k+1)$-dimensional domain $U$, the sign
	depending on the orientation of $B$.
	This is the case for
	\[
		\Rand (r\simplexs \cap U(\simplext))
			\halfquad = \halfquad \Rand r\simplexs \cap U(\simplext)
			\halfquad \cup \halfquad r\simplexs \cap \Rand U(\simplext).
	\]
	So the formula is clear for $\alpha = \omega^\simplext$
	by definition of the ``domain integral'' over $r\simplexs$,
	and by linearity, it hence holds for all $\alpha \in \Pol\inv\Omega^k$,
\end{proof}
\begin{remark_nn}
	\begin{subenum}
	\item	The notation \eqnref{eqn:defkDimensionalBVintegral}
	might seem to obscure the actual integration process
	over a subdomain, but we would like to emphasise the analogue
	to $\bard\alpha$, which fulfills the same
	Stokes formula.
	\item	For a smoothly bounded $(k+1)$-dimensional integration
	domain $U$ and differential forms $\alpha \in \Omega^k$, $\beta \in \Omega^{k-1}$,
	it is always true that
	\[
		\Int_U \alpha = \Int_U \sprod \alpha{\dvol_U} \dvol_U,
		\qquad
		\Int_{\Rand U} \beta = \pm \Int_{\Rand U} \sprod{\nu \wedge \beta}{\dvol_U} \dvol_{\Rand U},
	\]
	where the sign is the same as in $\dvol_U = \pm \nu \wedge \dvol_{\Rand U}$.
	Therefore, the notation \eqnref{eqn:defkDimensionalBVintegral}
	can also be interpreted as
	\[
		\Int_{r\simplexs} d^\BV\alpha =
			\Int_{\mathclap{r\simplexs \cap \Rand U(\simplext)}} \sprod{d^\BV\alpha}{\dvol_{r\simplexs}}
			\dvol_{r\simplexs \cap \Rand U(\simplext)}
	\]
	with $d^\BV\omega^\simplext = \nu \wedge \omega^\simplext$
	as in \ref{prop:BVderivativeOfOmegas}. The notational problem
	is mainly that the \BV{} derivative is supported
	on a codimension-$1$-set, which makes the integral
	in Green's formula $(n-1)$-dimensional instead of
	$n$-dimensional, and the left-hand side integral
	in Stokes' formula $k$-dimensional instead of
	$(k+1)$-dimensional. Unfortunately, we do not know
	a common notation covering both.
	\item	The formula stays correct (with an appropriate
	notational adaption) for \textsc{sbv} forms
	(introduced by \citealt{Giorgi88}, as overview
	we refer to \citealt{Ambrosio00}) which are \BV{} forms whose
	derivative measure $\mu$ consists of parts
	$\mu_{ac}$ and $\mu_s$
	which are absolutely continuous with respect to
	the $n$-dimensional and to the $(n-1)$\hyp dimensional
	Hausdorff measure in $r\complex$, if $\mu_s$
	is supported on a set that is transversal to the
	integration domain $U$. For the proof, one can use
	the approximation of $\alpha \in \fnspacefont{SBV}\Omega^k$
	by convolution with smooth Gaussian kernels. If the jump set
	of $\alpha$, i.\,e. the support of $\mu_s$, is transveral
	to $U$, then the convergence is uniform almost
	everywhere on $\Rand U$, and so the integrals $\int_{\Rand U} \alpha_i$
	of the mollified forms $\alpha_i$ tend to $\int_{\Rand U} \alpha$
	and give a well-defined interpretation of $\int_U d^\BV\alpha$.
	\item	This means that for $\beta \in \Leb^\infty\Omega^{k+1}$
	which is smooth inside each $U(\simplexs)$, $\simplexs \in \complex^{k+1}$,
	we have
	\[
		\dprod \alpha{\delta\beta}
			= \sum_\simplexs \halfquad \Int_{\mathclap{\Rand U(\simplexs)}} \alpha \wedge * \beta
			+ \dprod{d^\BV \alpha}\beta
		\qquad \forall \alpha \in \Pol\inv\Omega^k.
	\]
	\end{subenum}
\end{remark_nn}
\begin{proposition}
	\label{prop:discreteAndBVderivative}
	$\dprod{d^\BV\alpha}\beta = \frac n {k+1} \dprod{\bard\alpha}\beta$
	for all $\alpha \in \Pol\inv\Omega^k$ and all $\beta \in \Pol^0\Omega^{k+1}$.
\end{proposition}
\begin{proof}
	Due to \eqnref{eqn:BVderivativeOfOmegat},
	it suffices to consider, for each $\simplexs \in \complex^{k+1}$
	and each $\simplexs \in \complex^k$,
	\[
		\dprod{d^\BV\omega^\simplext}{\beta}_{U(\simplexs)}
			= \Int_{\mathclap{U(\simplexs) \cap \Rand U(\simplext)}} \omega^\simplext \wedge * \beta
	\]
	which can be spelled out by using the $(n-1)$-flags $\simplexa$ in
	$U(\simplexs) \cap \Rand U(\simplext)$:
	\[
		\Int_{\mathclap{U(\simplexs) \cap \Rand U(\simplext)}} \omega^\simplext \wedge * \beta
			= \sum_\simplexa \Int_{\stds^{n-1}} \omega^\simplext \wedge * \beta
			= \frac 1 {(n-1)!} \sum_\simplexa  (\omega^\simplext \wedge * \beta)(b_\simplexa),
	\]
	where $b_\simplexa$ is the pull-back of an orthonormal basis
	of $r'\simplexa$. The flags occuring in this
	sum are of the form $(\flag 0,\dots,\hat\simplext,\simplexs,\dots,\flag n)$,
	cf. \ref{rem:decompositionOfRandUt}. Using the vectors
	$v_{\flag i,\flag{i+1}}$ from \ref{prop:centreDiffsArePerpendicular},
	we have inside each $r'\simplexa$
	\[
		\begin{split}
		b_\simplexa & = v_{\flag 0,\flag 1}\wedge \cdots \wedge \hspace{4ex} v_{\flag{k-1},\simplexs} \hspace{4ex}
			\wedge v_{\simplexs,\flag{k+2}} \wedge \cdots \wedge v_{\flag{n-1},\flag n} \\
			& = v_{\flag 0,\flag 1}\wedge \cdots \wedge (v_{\flag{k-1},\simplext} + v_{\simplext,\simplexs})
			\wedge v_{\simplexs,\flag{k+2}} \wedge \cdots \wedge v_{\flag{n-1},\flag n},
		\end{split}
	\]
	the factors in the latter product are all mutually perpendicular.
	Now observe
	\[
		(\omega^\simplext)^\sharp =
			\frac{v_{\flag 0,\flag 1} \wedge \cdots \wedge v_{\flag{k-1},\simplext}}
			{\absval{v_{\flag 0,\flag 1} \wedge \cdots \wedge v_{\flag{k-1},\simplext}}},
		\qquad
		(*\omega^\simplexs)^\sharp =
			\frac{v_{\simplexs,\flag{k+2}} \wedge \cdots \wedge v_{\flag{n-1},\flag n}}
			{\absval{v_{\simplexs,\flag{k+2}} \wedge \cdots \wedge v_{\flag{n-1},\flag n}}}.
	\]
	By orthogonality of all vectors in $b_\simplexa$,
	the application $(\omega^\simplext \wedge * \beta)(b_\simplexa)$, usually
	comprising all permutations of the factors, splits as
	\[
		\begin{split}
		(\omega^\simplext \wedge * \beta)(b_\simplexa)
			& = \omega^\simplext(v_{\flag 0,\flag 1}\wedge \cdots \wedge v_{\flag{k-1},\simplext})
			  \, (*\beta)(v_{\simplexs,\flag{k+2}} \wedge \cdots \wedge v_{\flag{n-1},\flag n}) \\
			& = \absval{v_{\flag 0,\flag 1}\wedge \cdots \wedge v_{\flag{k-1},\simplext}}
			    \quad \,\,\sprod{\beta}{\omega^\simplexs}
			    \absval{v_{\simplexs,\flag{k+2}} \wedge \cdots \wedge v_{\flag{n-1},\flag n}}.
		\end{split}
	\]
	Summation over all flags $(\flag 0,\dots,\hat\simplext,\simplexs,\dots,\flag n)$
	then gives, by \ref{prop:volumeOfPrimalAndDualCell},
	\[
		\Int_{\mathclap{U(\simplexs) \cap \Rand U(\simplext)}} \omega^\simplext \wedge * \beta
			= \frac{k! (n-k-1)!}{(n-1)!} \, \absval \simplext \, \absval{*\simplexs}
				\sprod{\omega^\simplexs} \beta
			= \frac n {k+1} \, \frac{\absval \simplext}{\absval \simplexs}
				\absval{U(\simplexs)} \sprod{\omega^\simplexs} \beta,
	\]
\end{proof}

\subsection{Approximation Estimate for $\Pol\inv\Omega^k$}

\begin{lemma}
	\label{prop:PolinvMatrixFullRank}
	Let us denote the set of multiindices $I = (i_1,\dots,i_k)$
	with $i_j < i_{j+1}$ and $1 \leq i_j \leq n$ for all $j$, by
	$\binom nk$ (which in fact is its cardinality).
	Suppose $\complex$ is a simplicial complex with only
	one $n$-simplex $\simplexe$ with a non-degenerate
	flat metric. Then the $\binom nk \times \binom{n+1}{k+1}$ matrix
	\[
		M_{(k)}(\simplexe) := \Big( \Int_{r\simplexe} \omega^\simplext(v_I) \Big)
				_{I \in \binom nk, \simplext \in \complex^k}
	\]
	has full rank $\binom nk$, where $v_{\{i_1,\dots,i_k\}} := \bigwedge v_{i_j}$
	for an arbitrary basis $v_j$ of $Tr\simplexe$.
\end{lemma}
\begin{proof}
	The choice of $v_j$ does not matter, because a change
	of this basis only results in a multiplication with
	a non-singular $\binom nk \times \binom nk$-matrix from
	the left. Furthermore, it suffices to show that
	the matrix $\tilde M_{(k)} := (\int \dvol_{r\simplext}(v_I))_{I,\simplext}$
	has full rank, because it only differs from $M_{(k)}$ by
	factors depending on the volumes $\absval{U(\simplext)}$,
	which must be non-vanishing for at least $\binom nk$
	of the $\simplext$ (which happens if the circumcentre
	lies on a facet of $\simplext$). Now if we choose
	$v_i = ri-r0$, we can transform the situation
	onto the
	unit simplex $D$ with vertices $0,e_1,\dots,e_n$,
	where then $v_i = e_i$, and the volume forms of
	simplices $\simplext$ containing
	the vertex $0$ have a particularly easy expression:
	\[
		\dvol_{r\simplext} = dx^{i_1} \wedge \dots \wedge dx^{i_k}
		\qquad
		\text{for } \simplext = \{0,i_1,\dots,i_k\}.
	\]
	So $\dvol_{r(\{0\}\cup I)}(v_{I'}) = 1$ if $I = I'$ and
	$0$ else, hence these rows of $\tilde M_{(k)}$ are
	already linearly independent,
\end{proof}
\begin{proposition}
	\label{prop:Hminus1estimateForPolinvOnlyOneStage}
	Suppose $r\complex$ is a simplicial complex with a piecewise flat and
	$(\theta,h)$-small metric, $\alpha \in \Sob^{1,1}\Omega^k$. Then there are $\alpha^0,
	\alpha^1 \in \Pol\inv\Omega^k$ and, if $\lambda_\simplexs^i > 0$
	for all components of the $\lambda_\simplexs$ defining subdivision
	neighbourhoods $U(\simplexs)$, there is $\alpha^2 \in \Pol\inv\Omega^k$,
	such that the $\Leb^2$ norms of $\alpha^0$, $\bard\alpha^1$ and
	$\bardelta\alpha^2$ are estimated by the corresponding
	norms of $\alpha$ up to a constant, and (with the
	Poincaré constant $\tilde C_\boxdot$ from
	\ref{rem:PoincareConstantWithoutScalingFactor})
	\begin{align*}
		\dprod{\alpha - \alpha^0} \beta & \simleq
			\tilde C_\boxdot h \ibetrag \alpha \ibetrag{\nabla \beta}
			&& \forall \beta \in \Sob^1\Omega^k, \\
		\dprod{d\alpha - \bard\alpha^1} \beta & \simleq
			\tilde C_\boxdot h \ibetrag{d\alpha} \ibetrag{\nabla \beta}
			&& \forall \beta \in \Sob^1\Omega^{k+1}, \\
		\dprod{\delta \alpha - \delta \alpha^2} \beta & \simleq
			\tilde C_\boxdot h \ibetrag{\delta \alpha} \ibetrag{\nabla \beta}
			&& \forall \beta \in \Sob^1\Omega^{k-1}.
	\end{align*}
\end{proposition}
\begin{proof}
	\begin{subeqns}
	Let us introduce a space $\Pol\inv_{elw}\Omega^k$ of
	``elementwise'' $\Pol\inv$ forms, spanned by
	\[
		\omega^{\simplexs,\simplexe} := \begin{cases}
							\dvol_{r\simplexs} & \text{in } U(\simplexs)\cap r\simplexe \\
							0 & \text{elsewhere}
							\end{cases}
		\qquad
		\text{for } \simplexs \subset \simplexe, \simplexe \in \complex^k.
	\]
	Then we can find $\tilde \alpha^0$, $\tilde \alpha^1$ and
	$\tilde \alpha^2$ such that, for each $\simplexe \in \complex^n$,
	\begin{equation}
		\label{eqn:Hminus1estimateForPolinvOnlyOneStageTilde}
		\begin{aligned}
			\int_{r\simplexe} \sprod{\tilde \alpha^0}{\bar \beta}
				& = \int_{r\simplexe} \sprod \alpha{\bar\beta}
				&& \text{for all constant } \bar\beta \in \Omega^k,\\
			\int_{r\simplexe} \sprod{\bard \tilde \alpha^1}{\bar \beta}
				& = \int_{r\simplexe} \sprod{d\alpha}{\bar\beta}
				&& \text{for all constant } \bar \beta \in \Omega^{k+1}
		\end{aligned}
	\end{equation}
	and similarly for $\bardelta\alpha^2$.
	In fact, writing $\tilde \alpha^0 = \tilde \alpha^0_{\simplext,\simplexe}
	\omega^{\simplext,\simplexe}$ and inserting $\beta = v_I^\flat$
	from above with $I \in \binom nk$, the equations
	for determining coefficients $\tilde \alpha^0_{\simplext,\simplexe}$
	have $M_{(k)}$ as system matrix, which has full rank
	by \ref{prop:PolinvMatrixFullRank}. The integral
	$\int \sprod{\bard\tilde\alpha^1}{\bar \beta}$ reduces to
	a boundary integral by \ref{prop:discreteAndBVderivative}
	which does not include $\bard$ anymore, this boundary integral
	is invariant under affine transformations, and the
	problem is solvable on the unit simplex $D$. If all
	$\lambda_\simplexs^i$ are positive, the $\bardelta_\simplexs^\simplext$
	are just a row- and column-rescaling of $\bard_\simplext^\simplexs$
	by non-zero factors.
	
	As a second step, let $\alpha^0 = \alpha^0_\simplext \omega^\simplext
	\in \Pol\inv\Omega^k$ be defined by the condition that,
	for each $\simplext \in \complex^k$ and all $\simplexs \in \complex^{k+1}$,
	\[
		\begin{aligned}
		\int_{U(\simplext)} \sprod{\alpha^0}{\bar \beta}
			& = \int_{U(\simplext)} \sprod{\tilde\alpha^0}{\bar \beta}
			&& \text{for all constant } \bar \beta' \in \Omega^k, \\
		\int_{U(\simplexs)} \sprod{\bard\alpha^1}{\bar \beta}
			& = \int_{U(\simplexs)} \sprod{\bard\tilde\alpha^1}{\bar\beta}
		&& \text{for all constant } \bar \beta' \in \Omega^{k+1}
		\end{aligned}
	\]
	(and similarly for $\bardelta\alpha^2$),
	which means averaging $\tilde \alpha^0$ over all parts
	$U(\simplext) \cap r\simplexe$ with $\simplext \subset \simplexe$.
	These forms have the desired properties:
	Norm-preservation is clear by construction.
	For the approximation, observe that $\beta \in \Omega^k$ can be replaced
	by some $\bar \beta$ that is constant in each element
	$r\simplexe$, and the error is estimated by the
	Poincaré inequality \ref{prop:poincareIneqVanishingBdryValues}:
	$\ibetrag[\Leb^2]{\beta - \bar \beta} \leq \tilde C_\boxdot h
	\ibetrag[\Leb^2]{\nabla \beta}$. And similarly it can
	be replaced by some $\bar \beta'$ that is constant in each
	neighbourhood $U(\simplext)$. This gives
	\[
		\begin{split}
		\dprod{\alpha - \alpha^0} \beta
			& = \dprod{\alpha - \tilde \alpha^0}\beta
			  + \dprod{\tilde \alpha^0 - \alpha^0} \beta \\
			& \simleq \dprod{\alpha - \tilde \alpha^0}{\bar\beta}
			  + \dprod{\tilde \alpha^0 - \alpha^0}{\bar \beta'}
			  + \tilde C_\boxdot h \ibetrag \alpha \ibetrag{\nabla \beta},
		\end{split}
	\]
	and the scalar products vanish by choice of $\tilde \alpha^0$
	and $\alpha^0$. Exactly
	the same computation is feasible for
	$\dprod{d \alpha - \bard\alpha^1} \beta$ and
	$\dprod{\delta \alpha - \bardelta\alpha^2} \beta$,
	\end{subeqns}
\end{proof}
\begin{proposition}
	\label{prop:Hminus1estimateForPolinv}
	\begin{subeqns}
	For a complex consisting of only one $n$-simplex $\simplexe$,
	let $\bard_{(k)}(\simplexe)$
	be the $\binom{n+1}{k+2} \times \binom{n+1}{k+1}$
	matrix representation $(\bard_\simplexs^\simplext)
	_{\simplexs \in \complex^{k+1},\simplext \in \complex^k}$
	of $\bard: \Pol\inv\Omega^k
	\to \Pol\inv\Omega^{k+1}$.
	
	Suppose $\complex$ is a simplicial complex
	with a piecewise flat and $(\theta,h)$-small metric.
	Assume that
	the $\binom{n+1}{k+1} \times \binom{n+1}{k+1}$-matrix
	\begin{equation}
		\label{eqn:assumptionDECinterpolation}
		\begin{pmatrix} M_{(k)}(\simplexe) \\[1ex] M_{(k+1)}(\simplexe) \bard_{(k)}(\simplexe) \end{pmatrix}
		\quad
		\text{has full rank for each $\simplexe \in \complex^n$.}
	\end{equation}
	Let $\tilde C_\boxdot$ be the
	Poincaré constant from \ref{rem:PoincareConstantWithoutScalingFactor}.
	Then
	for each $\alpha \in \Omega^k$, there is
	$\bar \alpha \in \Pol\inv\Omega^k$ with
	$\ibetrag[\Leb^p]{\bar \alpha} \simleq \ibetrag[\Leb^p] \alpha$,
	$\ibetrag[\Leb^p]{\bard \bar \alpha} \simleq \ibetrag[\Leb^p]{d \alpha}$
	and
	\begin{equation}
		\label{eqn:Hminus1estimateForPolinv}
		\begin{aligned}
			\dprod{\alpha - \bar \alpha}\beta &
				\simleq \tilde C_\boxdot h \ibetrag \alpha \ibetrag{\nabla \beta}
				&& \forall \beta \in \Sob^1\Omega^k \\
			\dprod{d\alpha - \bard \bar \alpha}\beta &
				\simleq \tilde C_\boxdot h \ibetrag{d\alpha} \ibetrag{\nabla \beta}
				&& \forall \beta \in \Sob^1\Omega^{k+1}.
		\end{aligned}
	\end{equation}
	\end{subeqns}
\end{proposition}
\begin{proof}
	The assumption \eqnref{eqn:assumptionDECinterpolation}
	guarantees that we can find one single $\tilde \alpha
	\in \Pol_{elw}\inv\Omega^k$ fulfilling both
	equations in
	\eqnref{eqn:Hminus1estimateForPolinvOnlyOneStageTilde}
	at the same time,
\end{proof}
\begin{remark}
	\begin{subenum}
	\item
	We did not succeed to verify \eqnref{eqn:assumptionDECinterpolation}
	in the general case, but there are at least no structural
	obstructions for it to hold: $M_{(k)}$ has full rank $\binom nk$,
	and $\bard_{(k)}$, which is the transposed of $\Rand_{k+1}$
	from \ref{prop:rankOfBoundaryMatrix} with rows scaled by
	$\absval \simplext$ and columns scaled by $\absval \simplexs$,
	has rank $\binom n{k+1}$, which add up to $\binom{n+1}{k+1}$.
	\item
	Using an $\Leb^p$ Poincaré inequality (\citealt[thm. 4.5.2]{Evans92})
	instead of \ref{prop:poincareIneqVanishingBdryValues}
	leads to $\dprod{\alpha - \alpha^0} \beta \leq
	\tilde C_{\boxdot,p} h \ibetrag[\Leb^p] \alpha \ibetrag[\Leb^p]{\nabla \beta}$
	and similar for $d\alpha - \bard\alpha^1$ and
	$\delta \alpha - \bardelta\alpha^2$.
	\item
	The estimates \eqnref{eqn:Hminus1estimateForPolinv} are
	formulated as $\Sob\inv$ norm estimates. By inserting mollified
	characteristic forms or forms with small
	support and $\absval{\int \beta} = 1$ (``Dirac forms''),
	one can localise the convergence.
	\item
	We do not call \ref{prop:Hminus1estimateForPolinvOnlyOneStage}
	and \ref{prop:Hminus1estimateForPolinv}
	``interpolation estimates'', as $\alpha^0$, $\alpha^1$,
	$\alpha^2$ or $\bar \alpha$ may have
	nothing to do with $\alpha$ pointwise, but only in
	integral mean. In contrast to interpolation,
	the integral of $\alpha$ and $\alpha^0$ over
	smaller spaces like boundaries of the $U(\simplext)$
	will in general not converge,
	as \ref{prop:discreteAndBVderivative} shows.
	For an example, see also \ref{ex:DirPbInPolInv}.
	The section title ``interpolation of
	\textsc{dec}'' does not refer to interpolation of
	smooth functions, but to the process of extending
	the simplicial definitions in \ref{sec:introductionToDec}
	to $\Leb^\infty$ forms in \ref{prop:interpolationOfDec}.
	\end{subenum}
\end{remark}

\cleardoublepage \addtocontents{toc}{\protect\pagebreak\protect\vspace*{2.77cm}}
\addtocontents{toc}{\protect\thispagestyle{plain}}
%
%
               \chapter{Applications}
%
%
\label{sec:applications}


\section{Real-Valued Variational Problems}
\label{sec:fullDimensionalApproximation}


\begin{situation}
	\label{sit:xIsDiffeomorphism}
	Using the results of the preceding sections, we do not speak
	of a manifold and its triangulation, but directly suppose
	that $M = r\complex$ is a realised $n$-dimensional regular
	simplicial complex (compact, as usual),
	endowed with a piecewise flat and $(\theta,h)$-small
	Riemannian metric $g^e$, as well as with a smooth metric $g$
	fulfilling \ref{prop:comparisongandge} and
	\ref{prop:estimateOfChristoffelOperator}
	with $C_{0,1}' h < 1$.
	Except for \ref{sit:curvedDomain}sqq.,
	we assume that if $M$ has a boundary, it follows
	the boundary of the Karcher simplices. Therefore, the
	homeomorphism property of $x$ remains unchanged.
\end{situation}
\begin{remark}
	\begin{subenum}
	\item
	Obviously, the spaces $\Cont^{k,\alpha}$ of strongly differentiable
	and Lipschitz functions for $g$ and $g^e$ (defined in the classical
	meaning for $g$ and by \ref{def:differentiableStructureOnrK}
	for $g^e$) are different,
	but as the Sobolev norms for differentiation orders $k = 0,1,2$ are
	equivalent, the spaces $\SobW^{k,r}(Mg)$ and $\SobW^{k,r}(Mg^e)$ coincide.
	\item The convergence of curve length and geodesic distance,
	treated in \citet[sec. 4.1]{Hildebrandt06} for the case of
	embedded surfaces, is already covered by \ref{lem:lengthOfCurvesWithTwoCoincPoints}
	in our setting.
	\end{subenum}
\end{remark}

\subsection{The Dirichlet Problem for Functions}

\begin{goal}
	In this section, we will deal with approximations of the Dirichlet
	problem, that is solving a weak version of $\laplace u = f$,
	where $\laplace$ is the Laplace--Beltrami operator of $M$.
	The Laplacian of $k$-differential forms
	will be dealt in the subsequent section.
	
	As first step, we will give a short review of the usual proof
	for convergence of Galerkin approximations to the Dirichet
	problem as can be found for instance in \cite{Braess07}.
	In the second step, we will add the usual
	error terms resulting from the ``variational crime'' to use
	$g^e$ instead of $g$. This is standard in the \textsc{fe} theory for
	geometric \textsc{pde}'s initiated by \cite{Dziuk88}, but often
	not separated from the error of the first step.
\end{goal}
\begin{definition}
	\label{def:Pol}
	Situation as in \ref{sit:xIsDiffeomorphism}. Denote by
	$\Sob^1_0$ the space of weakly differentiable
	functions with vanishing trace on $\Rand M$ and
	by $\Pol^1$ the space of globally continuous,
	piecewise linear functions (here, ``linear'' of course means
	usual linearity in the parameter domain $r\simplexe$, $\simplexe \in \complex^n$),
	by $\Pol^1_0$ the same but with vanishing boundary values.
	The norm of an operator on function spaces will
	be denoted by $\iopnorm \argdot$.
\end{definition}
\begin{definition}
	For $v,w \in \Pol^1$, recall the definition
	$\Lap(u,v) := {\dprod{du}{dv}}_{\Leb^2(Mg)}$
	from \eqnref{eqn:defLapDir} and that
	the (homogeneous) weak $g$-Dirichlet problem is the task to find
	$u \in \Sob^1_0(M)$ such that $\Lap(u,v) = {\dprod fv}_{Mg}$ for all
	$v \in \Sob^1_0$,
	we shortly write $Lu = f$ with an operator $L:
	\Sob^1_0 \to (\Sob^1_0)^*$.
	The \begriff{$g$-Galerkin solution} to the
	Dirichlet problem with respect to the trial space $\Pol^1_0$
	is the solution $u_h$ to $\Lap(u_h,v) = {\dprod fv}_{Mg}$
	for all $v \in \Pol^1_0$.
	Naturally, there is also the notion of a $g^e$-Galerkin solution.
\end{definition}
\begin{remark}
	By \ref{prop:DirichletProblemForkForms}, we know that the
	Dirichlet problem has no solution for general $f \in \Leb^2$,
	but only for $f \perp \Harm$, and the solution is
	unique up to harmonic components, in other words: there
	is a unique solution in $\Harm^\perp$. But the space of harmonic
	functions is one-dimensional, consisting only of
	the constant functions---and these are ruled out by the boundary
	value requirements.
\end{remark}
\begin{fact}[\citealt{Schwarz95}, also cf.
	\ref{rem:GaffneysInequality}]
	\label{prop:dirichletPbSolvabilityAndH2Regularity}
	The de Rham complex ($\Sob^{1,0}\Omega,d)$ of a smooth
	compact Riemannian manifold is a Fredholm complex,
	so the Dirichlet problem
	is uniquely solvable, and $\ibetrag[\Leb^2]{du} \leq
	C_\boxdot \ibetrag[\Leb^2] f$ with the Poincaré constant
	$C_\boxdot$ from \ref{prop:coercivityOnHarmPerpTang}.
	This means that $L\inv$ is a bounded linear operator.
	
	If $\Rand M$ is piecewise smooth or convex (that means,
	convex where it is not smooth), then $M$
	is \textbf{$\Sob^2$-regular}, i.\,e. there is a constant $C_\boxdot$
	depending on $M$, but not on $f$, with $\ibetrag[\Sob^2] u
	\leq C_\boxdot \ibetrag[\Leb^2] f$, that means that
	$\iopnorm[\Leb^2,\Sob^2]{L\inv} \leq C_\boxdot$ in this case.
\end{fact}
\begin{lemma}[\textbf{Céa}]
	\label{prop:CeasLemma}
	Situation as in \ref{sit:xIsDiffeomorphism}.
	Let $u$ be the Dirichlet potential
	and $u_h$ be the $g$-Galerkin solution to $f \in \Leb^2$.
	Then $u_h$ is the
	orthogonal projection of $u$ onto $\Pol^1_0$ with respect to
	$\Lap(\argdot,\argdot)$.
\end{lemma}
\begin{proof}
	As $\Pol^1_0 \subset \Sob^1_0$, also $u$ fulfills
	$\Lap(u,v) = {\dprod fv}_{\Leb^2}$
	for all $v \in \Pol^1_0$ by which $u_h$ was defined. So we have the
	so-called ``Galerkin orthogonality''
	$\Lap(u-u_h,v) = 0$ for all such $v \in \Pol^1_0$,
	which is the characterising property of the projection error,
\end{proof}
\begin{corollary}
	Situation as in \ref{sit:xIsDiffeomorphism}.
	Let $\Pi$ be the orthogonal projection $\Sob^1_0 \to \Pol^1_0$ with
	respect to $\Lap(\argdot,\argdot)$ and $\Pi^\perp := \id - \Pi$
	be the projection error. Then for any $k$ for which
	both sides are defined,
	\[
		\ibetrag[\Sob^k]{u - u_h} \leq \ibetrag[\Leb^2] f \, \iopnorm[\Leb^2,\Sob^k]{\Pi^\perp L\inv}.
	\]
\end{corollary}
\begin{proposition}
	\label{prop:dirichletPbH1Estimate}
	Situation as in \ref{sit:xIsDiffeomorphism}
	with dimension $n \leq 3$. Then $\iopnorm[\Sob^2,\Sob^1]{\Pi^\perp} \simleq \theta^{-1}h$,
	and if additionally $M$ is $\Sob^2$-regular, then
	\[
		\ibetrag[\Sob^1]{u - u_h} \simleq
		C_\boxdot h\theta^{-1}\ibetrag[\Leb^2] f.
	\]
\end{proposition}
\begin{proof}
	It suffices to show that there is \textit{one} $u_h \in \Pol^1$
	with $\ibetrag[\Sob^1]{u-u_h} \simleq \theta^{-1} h \ibetrag[\Sob^2] u$,
	then the projection of $u$ will produce a smaller error than this $u_h$.
	As usual, we take $u_h$ to be the Lagrange interpolation of $u$
	(which is well-defined, as $\Sob^2 \subset \Cont^0$ in dimension
	$\leq 3$, cf. \citealt{Adams75}, Theorem 5.4.\textsc c).
	And this interpolation estimate is exactly
	\ref{prop:InterpolationEstimateLpForRealvaluedFunctions},
\end{proof}
\begin{sloppypar}
\begin{proposition}[\textbf{Aubin--Nitsche}]
	\label{prop:dirichletPbL2Estimate}
	Situation as in \ref{sit:xIsDiffeomorphism}.
	Then $\iopnorm[\Leb^2,\Leb^2]{\Pi^\perp L\inv}
	\leq \iopnorm[\Leb^2,\Sob^1]{\Pi^\perp L\inv}^2$. Under the same
	conditions as in \ref{prop:dirichletPbH1Estimate},
	\[
		\ibetrag[\Leb^2]{u - u_h} \simleq
		C_\boxdot^2 h^2\theta^{-2}\ibetrag[\Leb^2] f.
	\]
\end{proposition}
\end{sloppypar}
\begin{proof}
	First, note that for a right-hand side $g$, the solution $L\inv g$
	is characterised by ${\dprod gv}_{Mg} = \Lap(L\inv g,v)$
	for all $v \in \Sob^1_0$. Now for a right-hand side $f \in \Leb^2$,
	consider
	\[
		\begin{split}
		\ibetrag[\Leb^2(Mg)]{\Pi^\perp L\inv f}
			& = \sup_{g \in \Leb^2} \frac{{\dprod{\Pi^\perp L\inv f}g}_{\Leb^2}}{\ibetrag[\Leb^2]g} \\
			& = \sup \frac{\Lap(\Pi^\perp L\inv f,L\inv g)}{\ibetrag[\Leb^2]g}
			\overset{(*)} = \sup \frac{\Lap(\Pi^\perp L\inv f, \Pi^\perp L\inv g)}{\ibetrag[\Leb^2] g} \\[1ex]
			& \hspace{28.5ex}\leq \ibetrag[\Sob^1]{\Pi^\perp L\inv f} \iopnorm[\Leb^2,\Sob^1]{\Pi^\perp L\inv},
		\end{split}
	\]
	where we have used in $(*)$ that $\Pi$ and hence $\Pi^\perp$ is a
	$\Lap$-orthogonal projection,
\end{proof}
\begin{remark}
	It would of course be possible to consider other interpolation
	procedures than just nodal Lagrange interpolation, for example
	averaged Taylor polynomials as in \citet[section 4.1]{Brenner02},
	which would circumvent the dimension restrictions. However,
	the emphasis of this thesis lies more on the different possible
	applications of the Karcher simplex construction
	than on optimal results for the Dirichlet problem.
\end{remark}
\begin{lemma}
	\label{prop:estimateDirAndDire}
	Situation as in \ref{sit:xIsDiffeomorphism}.
	Let $F(v) := {\dprod vf}_{M,g}$, and let $\Lap^e$ and $F^e$
	be defined similar to $\Lap$ and $F$, but with $g^e$ instead
	of $g$ everywhere.
	Then $|(\Lap-\Lap^e)$ $(v,w)| \simleq C_0' h^2 \ibetrag[\Leb^2]{dv} \ibetrag[\Leb^2]{dw}$
	and $\absval{(F - F^e)v} \simleq C_0' h^2 \ibetrag[\Leb^2]v$.
\end{lemma}
\begin{proof}
	Exactly as in \ref{prop:normEquivalences},
\end{proof}
\bibrembegin
\begin{remark_nn}
	In the understanding of \cite{Hildebrandt06}, the ``weak
	Laplacian'' $L_g$ is a mapping $\Sob^1 \to
	(\Sob^1)^*$, $L_g u: v \mapsto \Lap(u,v)$.
	In this setting, \ref{prop:estimateDirAndDire}
	can be seen as a convergence result for
	the weak Laplacians: $\iopnorm[\Sob^1,(\Sob^1)^*]{L_g - L_{g^e}}
	\simleq C_0' h^2$.
\end{remark_nn}
\bibremend
\begin{proposition}
	\label{prop:estimateDirichletProblem}
	Situation as in \ref{sit:xIsDiffeomorphism} with
	$\Sob^2$-regular $M$. Let $u_h, u_h^e \in \Pol^1_0$
	be the Galerkin solutions to $L_gu = F$ and $L_{g^e} u^e = F^e$. Then
	\[
		\ibetrag[\Leb^2]{u_h - u^e_h} + C_\boxdot \ibetrag[\Leb^2]{du_h-du_h^e}
		\simleq C_0' C_\boxdot^2 h^2 \ibetrag[\Leb^2] f.
	\]
\end{proposition}
\begin{proof}
	During this proof, $\ibetrag \argdot$ always
	means $\ibetrag[\Leb^2(Mg)] \argdot$.
	Let us first consider the derivative term on the left-hand side:
	For some $v$ with $\ibetrag v = 1$, we have
	\[
		\begin{split}
		\ibetrag{du_h - du_h^e} & = \Lap(u_h - u_h^e, v) \\
			& \leq \absval{\Lap(u_h, v) - \Lap^e(u_h^e, v)}
			  + \absval{\Lap^e(u_h^e, v) - \Lap(u_h^e, v)} \\
			& \leq \absval{(F - F^e)v}
				 + \absval{(\Lap^e - \Lap)(u_h^e, v)} \\
			& \simleq C_0' h^2 \ibetrag f \ibetrag v
				 + C_0' h^2 \ibetrag{du_h^e} \ibetrag{dv}
		\end{split}
	\]
	Then use
	$\ibetrag{du_h^e} \leq C_\boxdot \ibetrag f$ from
	\ref{prop:dirichletPbSolvabilityAndH2Regularity}.
	For the estimate of $\ibetrag{u_h - u_h^e}$, use
	the Poincaré inequality again,
\end{proof}
\begin{remark_nn}
	As in the euclidean setting, the proofs carry over
	to an arbitrary continuous,
	strongly $\Sob^1_0$-elliptic bilinear form on $\Sob^1$
	instead of $\Lap$.
\end{remark_nn}

\subsection{Variational Problems in $\Omega^k$}

\begin{assumption}
	\label{asm:PolOmegaIsApproximating}
	Situation as in \ref{sit:xIsDiffeomorphism}.
	For $k = 0,\dots,n$, let there be finite-dimensional subspaces
	$\Pol\Omega^k$ of $\Sob^{1,0}\Omega^k$ (or $\Sob^{0,1}\Omega^k$,
	if needed) with $\Leb^2$ and $\Sob^{1,1}$ approximation order $h$
	analogous to \ref{prop:InterpolationEstimateForRealvaluedFunctions}:
	\[
		\min_{v_h \in \Pol\Omega^k}
			  \ibetrag[\Leb^2]{v - v_h}
			+ \ibetrag[\Leb^2]{dv - dv_h}
			+ \underbrace{\ibetrag[\Leb^2]{\delta v - \delta v_h}}
				_{\text{only for \eqnref{eqn:feApproxOfHodgeDecomp}}}
			\leq \alpha h \ibetrag[\Sob^2] v
	\]
	and similar for $\tang^* v = \tang^* v_h = 0$ or
	$\nor v = \nor v_h = 0$.
	Furthermore, assume that the Dirichlet problem is $\Sob^2$-regular
	and the Hodge decomposition $u = da + \delta b + c$
	is $\Sob^1$-regular, which means
	$\ibetrag[\Sob^1]{da} \simleq \ibetrag[\Sob^1] u$ etc.
	We abbreviate $\dprod\argdot\argdot_{\Leb^2(Mg^e)}$ as
	$\dprod\argdot\argdot_e$.
\end{assumption}
\begin{proposition}
	\label{prop:feApproxOfHodgeDecomp}
	\begin{subeqns}
	Assume \ref{asm:PolOmegaIsApproximating}.
	Let $u = da + \delta b + c$ be the Hodge decomposition
	of $u \in \Sob^{1,1}\Omega^k$, which can be computed as
	$a = \argmin F[u]$ over $a \in \Sob^{1,1}\Omega^{k-1}_{\tang}$ and
	$b = \argmin G[u]$ over $b \in \Sob^{1,1}\Omega^{k+1}_{\nor}$ as
	in \ref{prop:HodgeDecomposition}. If
	$a_h = \argmin F[u]$ over $a_h \in \Pol\Omega^{k-1}_{\tang}$ and
	$b_h = \argmin G[u]$ over $b_h \in \Pol\Omega^{k+1}_{\nor}$, then
	\begin{equation}
		\label{eqn:feApproxOfHodgeDecomp}
		\ibetrag[\Leb^2]{da - da_h} + \ibetrag[\Leb^2]{\delta b - \delta b_h}
		\leq \alpha h \ibetrag[\Sob^1] u.
	\end{equation}
	If $u = da_e + \delta b_e + c_e$ is the Hodge decomposition
	with respect to $g^e$, and if $a_{h,e}$ and $b_{h,e}$
	are defined similiarly, then
	\begin{align}
		\label{eqn:distortionSmoothHodgeDecomp}
		  \ibetrag[\Leb^2]{da - da_e}
		+ \ibetrag[\Leb^2]{\delta b - \delta b_e}
		+ \ibetrag[\Leb^2]{c - c_e} & \simleq C_0' h^2 \ibetrag[\Leb^2] u, \\
		  \ibetrag[\Leb^2]{da_h - da_{h,e}}
		+ \ibetrag[\Leb^2]{\delta b_h - \delta b_{h,e}}
			& \simleq C_0' h^2 \ibetrag[\Leb^2] u.
	\end{align}
	\end{subeqns}
\end{proposition}
\begin{proof}
	\textit{ad primum:} By the Euler--Lagrange equation
	$\dprod{da}{dv} = \dprod u{dv}$
	for all $v \in \Sob^{1,1}\Omega^{k+1}_{\tang}$ and $\dprod{da_h}{dv}
	= \dprod u{dv}$ for all $v \in \Pol\Omega^{k+1}_{\tang}$, we know that
	$da_h$ is the $\Leb^2$-best approximation of $da$ in $d(\Pol\Omega^k_{\tang})$,
	which is smaller than $\alpha h \ibetrag{\nabla da}$ by assumption.
	
	\textit{ad sec.:} If $\dprod{da_{h,e}}{dv}_e - \dprod u{dv}_e = 0$,
	then $\dprod{da_{h,e}}{dv} - \dprod u{dv} \simleq C_0' h^2
	(\ibetrag{da_{h,e}} + \ibetrag u) \ibetrag v$ and hence
	$\dprod{da_h - da_{h,e}}{dv} \simleq C_0' h^2
	(\ibetrag{da_{h,e}} + \ibetrag u) \ibetrag v$ for all
	$v \in \Pol\Omega^k_{\tang}$. The same calculation is valid
	for $da_h - da_{h,e}$ instead of $da - da_e$. The
	$c - c_e$ estimate comes out as the remainder,
\end{proof}
\bibrembegin
\begin{remark}
	\begin{subenum}
	\item	\eqnref{eqn:distortionSmoothHodgeDecomp}
	is our analogue of thm. 3.4.6 in \citet{Wardetzky06}.
	\bibremend
	\item	In general, there will be no exact finite-dimensional Hodge
	decomposition in $\Pol\Omega^k$, as we have not required any
	connection between $d(\Pol\Omega^k)$ and $\Pol\Omega^{k+1}$.
	There is a Hodge decomposition in the space
	of Whitney forms with convergence proven by
	\citet[thm. 4.9]{Dodziuk76}. Variational problems in a
	specific space $\Pol\inv\Omega^k$ of piecewise constant forms
	will be treated in \ref{obs:DECisExteriorCalculus}sqq.
	\item	\label{rem:distortionFEECHodgeDecomp}
	The \textsc{feec} setting of \citename{Arnold} \etal{}
	only has a weak Hodge decomposition $u = da_h + \tilde b_h + c_h$
	of $u \in \Pol\Omega^k$
	as in \ref{prop:HodgeDecompositionInH10}, but as its parts are
	also orthogonal projections, there is an estimate
	\[
		\ibetrag[\Leb^2]{da_h - da_{h,e}}
		+ \ibetrag[\Leb^2]{\tilde b_h - \tilde b_{h,e}}
		+ \ibetrag[\Leb^2]{c_h - c_{h,e}} \simleq C_0' h^2 \ibetrag[\Leb^2] u
	\]
	corresponding to \eqnref{eqn:distortionSmoothHodgeDecomp}.
	\end{subenum}
\end{remark}
\bibremend
\parag{Mixed form of Dirichlet problem.}
	\cite{Arnold06,Arnold10} have shown how to construct finite-dimensional
	subcomplexes $(\Pol\Omega,d)$ of $(\Sob^{1,0}\Omega,d)$ and solve the
	mixed Dirichlet problem therein.
	\cite{Holst12} have extended this to the situation where the domain
	of the Sobolev space and the finite-dimensional approximation are
	endowed with different, but close Riemannian metrics $g$ and $g^e$,
	which leads to the situation that the inclusion map
	$\Pol\Omega^k(Mg^e) \to \Sob^{1,0}\Omega^k(Mg)$ is not
	norm-preserving anymore, but only an almost-isometric map.
	Their setting directly applies to Finite Element computations on
	the Karcher--Delaunay triangulation:
\begin{proposition_nn}[\citealt{Holst12}, thm. 3.10]
	\label{prop:feDiffFormsDirichletProblem}
	Assume \ref{asm:PolOmegaIsApproximating}, and use the
	notation from \ref{obs:mixedFormDirichletPb}sq.
	For $f \in \Leb^2\Omega^k$, let $(\sigma,u,p) \in \Pol\homfont S$ and
	$(\sigma_e, u_e, p_e) \in \Pol \homfont S_e$
	be the solution of the mixed formulation
	\ref{prop:wellposednessOfWeakDirichletProblemOnkForms}
	of the Dirichlet problem in
	$M g$ and $M g^e$
	respectively, where $\Pol \homfont S = \Pol\Omega^k_{\tang}
	\times \Pol\Omega^{k+1}_{\tang} \times \Pol\Harm^k_{\tang}$ is a
	stable choice of trial spaces
	from \citet[eqn 7.14]{Arnold06}, and
	$\Pol\homfont S_e$ differs from $\Pol\homfont S$ only 
	by the last factor $(\Pol\Harm^k_e)_{\tang}$, the harmonic trial
	functions with respect to $g^e$. Then
	\[
		\ibetrag[\Sob^1]{\sigma - \sigma_e}
		+ \ibetrag[\Sob^1]{u - u_e}
		+ \ibetrag[\Leb^2]{p - p_e} \simleq \frac{C_0'} \gamma h^2 \ibetrag[\Leb^2] f,
	\]
	where $\gamma$ is the inf-sup constant as
	in \ref{prop:wellposednessOfWeakDirichletProblemOnkForms}
	(but over $\Pol\Omega^k$).
\end{proposition_nn}
\begin{proof}
	The solution $s=(\sigma,u,p)$ with respect to the ``correct''
	scalar product $g$ fulfills $b(s,t) = F(t)$ for every
	test triple $t = (\tau,v,q) \in \Pol\homfont S$.
	On the other hand, the distorted solution $s_e$ fulfills
	$b_e(s_e,t_e) = F_e(t_e)$ for all $t_e \in
	\Pol\homfont S_e$ with the obvious definition of $b_e$ and $F_e$.
	As the trial spaces only differ in the
	last term $q$, we have
	\[
		b_e(s_e,t_e) = b_e(s_e,t) + \dprod{u_e}{q_e-q}_e
			= b_e(s_e,t) + \dprod{u_e}q _e
	\]
	(because $u_e \perp \Pol\Harm^k_e$). Now observe
	$\dprod{u_e}q = \dprod{\pi u_e}q$, where $\pi$ is the
	orthogonal projection onto $\Pol\Harm^k$, and by
	\ref{rem:distortionFEECHodgeDecomp} the projection of a
	$\Pol\Harm^k_e$ element onto $\Pol\Harm^k$ is small. Hence
	\[
		\dprod{u_e} q _e
		= \dprod{u_e}q + (\dprod{u_e}q _e - \dprod{u_e} q)
		\simleq C_{0,1}' h^2 \ibetrag u \ibetrag q.
	\]
	Weakening the right-hand side, we obtain
	$\absval{b_e(s_e,t) - F_e(t)} \simleq C_0' h^2 \ibetrag s \ibetrag t$.
	By the scalar product comparison \ref{prop:estimateDirAndDire},
	also $\absval{b(s_e,t) - F(t)} \simleq C_0' h^2 \ibetrag s \ibetrag t$,
	and taking this together with $b(s,t) = F(t)$, we have
	\[
		b(s - s_e,t) \simleq C_0' h^2 \ibetrag s \ibetrag t.
	\]
	Now, by the inf-sup-condition \eqnref{eqn:infsupCondForb},
	$\gamma \ibetrag{s - s_e} \leq \sup_t b(s - s_e,t) / \ibetrag t$,
\end{proof}

\subsection{Dirichlet Problems with Curved Boundary}
	The case that the analytical and the computational domain actually
	coincide is not the only interesting problem. When for example
	a Dirichlet problem on the unit disk in hyperbolic space is considered,
	a Karcher triangulation with respect to the whole hyperbolic space will not
	exactly cover the unit disk. But the treatment of such a boundary
	approximation is standard in Finite Element
	theory, and the main task is to carefully inspect
	which arguments have to be modified because they rely on the
	Euclidean structure of the domain.
	We give a presentation according to \cite{Doerfler98}
	and do not treat the difference between $g$ and $g^e$, as this
	comparison can be done separately by using
	\ref{prop:estimateDirichletProblem} after
	\ref{prop:estimateCurvedDomains}.
	
	The usual setup for boundary approximation is that a domain
	$\Omega$ is replaced by a simplicial domain $\Omega_h$ whose
	boundary vertices lie on $\Rand \Omega$. By $(n-1)$-dimensional
	interpolation estimates, one then gets that $\Rand \Omega$
	and $\Rand \Omega_h$ are only $\simleq h^2\kappa$ far apart,
	where $\kappa$ bounds the curvature of $\Rand \Omega$ and $h$
	the mesh size of $\Omega_h$. We translate this, for
	$\Omega \subset M$, into the following
\begin{situation}
	\label{sit:curvedDomain}
	Let $M = r\complex$ be a piecewise flat and $(\theta,h)$-small
	realised simplicial complex. Let $\Omega \subset M$ be a
	\discretionary{full-}{dimensional}{full-dimensional} domain
	and $\Omega_h = r\bar\complex$ a realised full-dimensional
	subcomplex, connected by a ``normal graph map'' $\Phi:
	\Rand\Omega \to \Rand\Omega_h$, $p \mapsto \exp_p \dist \nu$,
	where $\dist: \Rand\Omega \to \R$ is Lipschitz-continuous and
	$\nu$ is the outer normal on $\Rand \Omega$, with
	the following properties: First, the retraction
	inverse $(p,t) \mapsto \exp_t t\dist\nu$ is injective
	(to ensure that no topology change may happen). Seond,
	it is ``short'' in the send that $\absval\dist \leq \alpha h^2$,
	$\operatorname{Lip} \d \leq \alpha h \leq 1$, and $\absval{\nabla d\dist} \leq \alpha$
	(where $\dist$ is smooth) for some $\alpha \in \R$.
	Let all principal curvatures
	of $\Rand \Omega$ in $M$ be bounded by $\kappa$.
	This implies that for small $h$ the norms of $d\Phi$ and $\nabla d\Phi$ are bounded,
	see section \ref{sec:graphSubmanifolds}.
\end{situation}
\begin{lemma}
	\label{prop:curvedDomsEstimateAnsatzFcts}
	\sloppypar
	Situation as in \ref{sit:curvedDomain}.
	If $v \in \Sob^1(\Omega_h)$, then $\ibetrag[\Leb^2(\Omega_h\setminus\Omega)] v
	\simleq \alpha h^2 \ibetrag[\Leb^2(\Omega_h \setminus \Omega)]{dv}$ and
	$\ibetrag[\Leb^2(\Rand\Omega\cap\Omega_h)]v \simleq
	\sqrt \alpha h \ibetrag[\Leb^2(\Omega_h \setminus \Omega)]{dv}$ for small $h$.
\end{lemma}
\begin{proof}
	\begin{subeqns}
	\textit{ad primum:}
	It suffices to show the claim for smooth $v$.
	Consider $\lambda \in \Omega_h
	\setminus \Omega$. As $\dist(\lambda,\Rand\Omega_h)
	\simleq \alpha h^2$, there is a curve $\gamma[\lambda]: \mu \leadsto
	\lambda$ for some $\mu \in \Rand\Omega_h$ with length
	$\simleq \alpha h^2$. If $h$ is small, this
	curve can be supposed to be a straight line lying entirely in one
	simplex of $\bar\complex$. As $v(\mu) = 0$,
	\begin{equation}
		\label{eqn:curveBdryvIsIntdv}
		v(\lambda) = \Int_{\gamma[\lambda]} dv\,\dot \gamma.
	\end{equation}
	Suppose $\gamma$ is arclength-parametrised.
	Now we can again apply the arguments from the proof of
	\ref{prop:InterpolationEstimateLpForRealvaluedFunctions}
	(keeping in mind that $\gamma[\lambda]$ has length $h$ there,
	but $\alpha h^2$ here):
	\[
		\Int_{\mathclap{\Omega_h \cap \Omega}} \absval v^2
			\quarquad\overset{(\eqnref{eqn:curveBdryvIsIntdv})}\leq\quarquad
				\Int_{\mathclap{\Omega_h \cap \Omega}}
				\halfquad \Big(\Int_{\mathclap{{\gamma[\lambda]}}} \absval{dv}\Big)^2
			\quarquad \overset{(\eqnref{eqn:interpolEstimatRealValued1})}\simleq \quarquad
				\alpha h^2 \Int_{\mathclap{\Omega_h \cap \Omega}}
				\hspace{2ex}\Int_{\mathclap{{\gamma[\lambda]}}} \absval{dv}^2
		\halfquad \overset{(\eqnref{eqn:interpolEstimatRealValued2})}\simleq \halfquad
			\alpha^2 h^4 \ibetrag[\Leb^2(\Omega_h\setminus\Omega)]{dv}^2
	\]
	\textit{ad sec.:} Because $\Rand\Omega_h$ is a graph over
	$\Rand\Omega$, the inverse is also true: $\Rand \Omega$
	is a graph (usually not normal) over $\Rand\Omega_h$, so	
	we can introduce coordinates in which
	a simplex $\simplexf$ of
	$\Rand\Omega_h$ lies in the $x_m$-plane and $\Rand\Omega$
	is parametrised by $(x_1,\dots,x_{m-1}) \mapsto
	(x_1,\dots,x_{m+1}, \rho)$. Then
	\[
		\begin{split}
		\ibetrag[\Leb^2(\Rand\Omega \cap r\simplexf)] v^2
			&= \Int_\simplext \absval v^2 \sqrt{1 + \absval{d\rho}^2}
			\overset{(\eqnref{eqn:estimateForProjectionDifferential})}\simleq \Int_{r\simplexf} \absval v^2 \\
			&\overset{(\eqnref{eqn:curveBdryvIsIntdv})}\simleq
				\Int_{r\simplexf} \Big(\Int_{\gamma[\lambda]} \absval{dv} \Big)^2
			\halfquad \overset{(\eqnref{eqn:interpolEstimatRealValued1})}\simleq
				\alpha h^2 \ibetrag[\Leb^2(\Omega_h\setminus\Omega)]{dv}^2, \\[-5ex]
		\end{split}
	\]
	\end{subeqns}
\end{proof}
\begin{lemma}
	\label{prop:curvedDomsExtensionu}
	Situation as in \ref{sit:curvedDomain}.
	For $v: M \to \R$, which is $\Sob^2$ continuous in $\Omega$ and $M \setminus \Omega$,
	let $[v]$ be the jump of $v$ across $\Rand\Omega$. If $h$ is small,
	there is a continuous extension $\bar u$ of $u \in \Sob^2(\Omega)$ onto $\Omega\cup\Omega_h$ such
	that $\bar u|_\Omega = u$, $\ibetrag[\Sob^2(\Omega_h \setminus \Omega)]{\bar u}
	\simleq \ibetrag[\Sob^2(\Omega)] u$ and
	$\ibetrag[\Leb^2(\Rand\Omega\cap\Omega_h)]{[d\bar u\,\nu]} \simleq
	\ibetrag[\Sob^1(\Omega)] u$.
\end{lemma}
\begin{proof}
	By assumption,
	all points in $\Omega_h \setminus \Omega$ are covered by the homotopy
	\[
		\Phi_t: \quad p \mapsto \exp_p t \nu,
	\]
	where at each $p \in \Rand\Omega \cap \Omega_h$, the parameter $t$ is chosen within
	$]0;\dist(p)]$ (in particular, points with
	negative $\dist(p)$ are excluded, as they would
	parametrise $\Omega \setminus \Omega_h$ instead of
	$\Omega_h \setminus \Omega$).
	For an image point of $\Phi_t$,
	set $\bar u(\exp_p t\nu) := u(\exp_p -t\nu)$,
	the reflection along $\Rand \Omega$. This $\bar u$ is
	continuous, and $[d\bar u\,\nu] = \pm 2 du\,\nu$.
	The $\Sob^2$ norm-preservation follows from the assumptions on $\Phi$
	(but note that $\bar u$ is not $\Sob^2$ in $\Omega_h \cup \Omega$ due to
	the jump on $\Rand\Omega$, even though $\Phi_t$ is smooth),
\end{proof}
\begin{proposition}
	\label{prop:estimateCurvedDomains}
	Situation as in \ref{sit:curvedDomain}. Let $u \in \Sob^2_0(\Omega)$ be
	the solution of $Lu = f$ with respect to $\Omega$,
	and let $u_h \in \Pol^1_0(\Omega_h)$
	be the Galerkin solution over $\Omega_h$ for an extension of the
	right-hand side $f$ by zero onto $\Omega_h \setminus \Omega$.
	Then $\ibetrag[\Leb^2(\Omega)]{du - du_h}
	\simleq \sqrt \alpha h \ibetrag[\Sob^2(\Omega)] u$ for small $h$,
	where $u_h$ has been extended
	by zero in $\Omega\setminus\Omega_h$.
\end{proposition}
\begin{proof}
	\begin{subeqns}
	Let $\bar u$ be the extension of $u$ from
	\ref{prop:curvedDomsExtensionu}. Assume we can show
	\begin{equation}
		\label{eqn:curvedDomsConvergence}
		\qquad \ibetrag[\Sob^1(\Omega_h)]{\bar u - u} \simleq
			\ibetrag[\Sob^1(\Omega_h)]{\bar u - v}
			+ \alpha h^2 \ibetrag[\Sob^2(\Omega_h \setminus\Omega)]{\bar u}
			+ \sqrt \alpha h \ibetrag[\Leb^2(\Rand\Omega\cap\Omega_h)]{[d\bar u(\nu)]}
	\end{equation}
	for every $v \in \Pol^1(\Omega_h)$. Then the claim is proven by
	\ref{prop:InterpolationEstimateLpForRealvaluedFunctions} and
	\ref{prop:curvedDomsExtensionu}. Supposed $v \in \Pol^1$, observe
	that in $\ibetrag{dv - du_h} = \sup \, \dprod{dv - du_h}{dw} / \ibetrag{dw}$,
	it suffices to take $w \in \Pol^1$. So
	we have
	\[
		\begin{split}
		\ibetrag[\Leb^2(\Omega_h)]{d\bar u - du_h}
			& \leq \ibetrag{d \bar u - dv} + \ibetrag{dv - du_h}
			= \ibetrag{d \bar u - dv} + \sup_{w \in \Pol^1} \frac{\dprod{dv - du_h}{dw}}{\ibetrag{dw}} \\
			& = \ibetrag{d \bar u - dv} + \sup_{w \in \Pol^1} \frac{\dprod{dv - d\bar u}{dw} + \dprod{d\bar u - du_h}{dw}}{\ibetrag{dw}} \\
			& \leq 2 \ibetrag{d \bar u - dv} + \sup_{w \in \Pol^1} \frac{\dprod{d\bar u - du_h}{dw}}{\ibetrag{dw}}.
		\end{split}
	\]
	And now, if $\bar f$ is the extension of $f$
	by $\bar f = 0$ in $\Omega_h \setminus \Omega$,
	\[
		\begin{split}
		{\dprod{d\bar u - du_h}{dw}}_{\Leb^2(\Omega_h)}
			& = \Int_{\mathclap{\Omega_h\cap\Omega}} \sprod{d\bar u}{dw} - fw
			+ \Int_{\mathclap{\Omega_h \setminus\Omega}} \sprod{d\bar u}{dw} - \bar f w \\
			& = \Int_{\Omega_h\cup\Omega} \underbrace{(-\laplace u - f)}_{=0} w
			+ \Int_{\mathclap{\Rand\Omega_h\cap\Omega}} w \, du\,\nu
			+ \Int_{\mathclap{\Omega_h \setminus\Omega}} -w \laplace u
			+ \Int_{\mathclap{\Rand(\Omega_h \setminus\Omega)}} w\, d\bar u(-\nu),
		\end{split}
	\]
	as $-\nu$ is the outer normal of $\Omega_h \setminus\Omega$. So
	this gives
	\[
		{\dprod{d\bar u - du_h}{dw}}_{\Leb^2(\Omega_h)}
		\leq \ibetrag[\Leb^2(\Omega_h \setminus \Omega)]{\laplace \bar u}
			\ibetrag[\Leb^2(\Omega_h \setminus \Omega)] w
			+ \ibetrag[\Leb^2(\Rand\Omega\cap\Omega_h)]{[du(\nu)]}
			\ibetrag[\Leb^2(\Rand\Omega \cap\Omega_h)] w,
	\]
	which shows, together with \ref{prop:curvedDomsEstimateAnsatzFcts}
	for the $w$ norms,
	the claimed estimate \eqnref{eqn:curvedDomsConvergence},
	\end{subeqns}
\end{proof}

\subsection{Heat Flow}

\begin{goal}
	As a short outlook on Galerkin methods for parabolic problems,
	we consider the approximation of heat flow under
	perturbations of metric. We decided to exclude the general
	convergence theory (see e.\,g. \citealt[chap. 1]{Thomee06})
	and concentrate on the difference between Galerkin
	approximations with respect to $g$ and $g^e$.
\end{goal}
\begin{proposition}
	Situation as in \ref{sit:xIsDiffeomorphism}.
	For a time interval $\interv 0 a$,
	let $u_h, u_{h,e}$ be the time\hyp continuous Galerkin approximation to the heat flow with
	initial value $u_0 \in \Pol^1_0$ and right-hand side
	$f \in \Leb^\infty(\interv 0 a, \Leb^2(Mg))$ for
	metrics $g$ and $g^e$ respectively, that means
	\begin{align*}
		\dprod{\dot u_h}v + \dprod{du_h}{dv} & = \dprod fv
			& \forall v \in \Pol^1,&&
				u_h|_{t=0} = u_0, \\
		{\dprod{\dot u_{h,e}}v}_e + {\dprod{du_{h,e}}{dv}}_e & = {\dprod fv}_e
			& \forall v \in \Pol^1,&&
				u_{h,e}|_{t=0} = u_0,
	\end{align*}
	where ${\dprod\argdot\argdot}_e$ is the abbreviation
	for ${\dprod\argdot\argdot}_{Mg^e}$.
	Then their difference can be estimated by
	\[
		\ibetrag[\Leb^\infty(\Leb^2)]{u_h - u_{h,e}}
			\simleq C_0' C_\boxdot h^2 \ibetrag[\Sob^1]{u_0}
				+ C_0' C_\boxdot h^2 \Big(\int \ibetrag[\Leb^2]{f(t)}^2\Big)^\half.
	\]
\end{proposition}
\begin{proof}
	The proof follows the line of the usual convergence proof
	for parabolic problems as in \citet[thm. 1.2]{Thomee06}: Consider
	$\eps := u_h - u_{h,e}$. By the defining equations for $u_h$
	and $u_{h,e}$, we have
	\[
		\dprod{\dot \eps}v + \dprod{d \eps}{dv} = \dprod fv - \dprod{\dot u_{h,e}}v - \dprod{du_{h,e}}{dv}.
	\]
	By \ref{prop:normEquivalences} and \ref{prop:estimateDirAndDire},
	we have
	\[
		\absval{\dprod{\dot u_{h,e}}v - \dprod{du_{h,e}}{dv} - \dprod fv}
		\simleq C_{0,1}' h^2 \big( \ibetrag{\dot u_{h,e}} \ibetrag v
		+ \ibetrag{du_{h,e}} \ibetrag{dv} + \ibetrag f \ibetrag v \big),
	\]
	where all norms are $\Leb^2$ norms. So we have for $v = \eps$,
	together with the Poincaré constant $C_\boxdot$ from
	\ref{prop:poincareIneqVanishingBdryValues},
	\[
		\smallfrac 1 2\, {\textstyle \ddt} \ibetrag \eps^2 + \ibetrag{d\eps}^2
			\simleq C_0' C_\boxdot h^2 \big( \ibetrag{\dot u_{h,e}}
			+ \ibetrag{du_{h,e}} + \ibetrag f \big) \ibetrag{d\eps}.
	\]
	Then Young's inequality gives $2c ab \leq c^2 a^2 + b^2$,
	hence we obtain a separated summand $\ibetrag{d\eps}^2$ on the right-hand side,
	which can be cancelled (the suppressed constant belongs to $c$):
	\[
		\smallfrac 1 2\, {\textstyle \ddt} \ibetrag \eps^2
			\simleq (C_0'C_\boxdot h^2)^2 \big(\ibetrag{\dot u_{h,e}}^2
			+ \ibetrag{du_{h,e}}^2 + \ibetrag f^2\big)
	\]
	Integration over $\interv 0 a$ gives, as $\eps|_{t=0} = 0$,
	\[
		\ibetrag \eps^2 \leq (C_0' C_\boxdot h^2)^2
			\int \ibetrag{\dot u_{h,e}}^2
			+ \ibetrag{du_{h,e}}^2 + \ibetrag f^2.
	\]
	From the usual regularity theory for parabolic
	problems (\citealt[eqn. 1.20, case $m = 0$]{Thomee06}), we know
	that $\int (\ibetrag{\dot u}^2 + \ibetrag[\Sob^1]u^2)
	\simleq \ibetrag[\Sob^1]{u_0} + \int \ibetrag f^2$, which
	shows the desired estimate for $\eps$,
\end{proof}
\begin{proposition}
	Situation as above.
	Let $u_h^n, u^n_{h,e}$ be the
	Galerkin approximation to the heat flow with
	implicit Euler time discretisation with respect
	to $g$ and $g^e$ respectively, that means
	\begin{align*}
		\dprod{\bar\partial u_h^n}v + \dprod{du_h^n}{dv} & = \dprod fv
			\quad \forall v \in \Pol^1_0,&
				u_h^0 = u_0, \\
		{\dprod{\bar \partial u_{h,e}^n}v}_e + {\dprod{du_{h,e}^n}{dv}}_e & = {\dprod fv}_e
			\quad \forall v \in \Pol^1_0,&
				u_{h,e}^0 = u_0
	\end{align*}
	for the backward difference quotient
	$\bar\partial v^n := \frac 1 \tau (v^n - v^{n-1})$.
	Then their difference at time $t = n\tau$ can be estimated by
	$\ibetrag[\Leb^2]{du_h - d u_{h,e}^n} \simleq Kh^2 t$,
	where $K$ depends on the geometry, $\ibetrag[\Leb^\infty(\Leb^2)]f$ and
	$\ibetrag[\Sob^1]{u_0}$.
\end{proposition}
\begin{proof}
	As before, let $\eps^n := u_h^n - u_{h,e}^n$. Then
	\[
		\dprod{\bar\partial \eps^n}v + \dprod{d\eps}{dv}
			= \dprod fv - \dprod{\bar\partial u_{h,e}^n}v - \dprod{du_{h,e}}{dv},
	\]
	and the right-hand side is bounded by
	\begin{align*}
		\absval{\dprod{\bar\partial u_{h,e}^n}v - {\dprod{\bar\partial u_{h,e}^n}v}_e}
		& + \absval{\dprod{du_{h,e}}{dv} - {\dprod{du_{h,e}}{dv}}_e}
		+ \absval{\dprod fv - {\dprod fv}_e} \\
		& \simleq C_0' h^2 \big(\ibetrag{\bar\partial u_{h,e}^n} \ibetrag v
			+ \ibetrag{du_{h,e}} \ibetrag{dv} + \ibetrag f \ibetrag v \big) \\
		& \simleq C_0' C_\boxdot h^2 \big(\ibetrag{\bar\partial u_{h,e}^n}
			+ \ibetrag{du_{h,e}} + \ibetrag f \big) \ibetrag{dv}.
	\end{align*}
	Denote the whole term in parentheses as $\Lambda$. As before, it is bounded in terms of
	the given data. Then again the choice $v = \eps^n$
	gives
	\[
			\ibetrag{\eps^n}^2 - \dprod{\eps^{n-1}}{\eps^n} + \ibetrag{d\eps^n}^2
				\simleq C_0' C_\boxdot h^2 \tau \Lambda \ibetrag{d\eps^n}^2
	\]
	and so
	\[
		\ibetrag{\eps^n}^2 + \ibetrag{d\eps^n}^2 \simleq C_0' C_\boxdot h^2 \tau \Lambda \ibetrag{d\eps^n}^2
				+ C_\boxdot^2 \ibetrag{d\eps^{n-1}}\ibetrag{d\eps^n}.
	\]
	And of course $\ibetrag{d\eps^n}^2$ is smaller than the last left-hand side,
	which gives $\ibetrag{d\eps^n} \simleq C_{0,1}' C_\boxdot h^2 \tau \Lambda
	+ C_\boxdot^2 \ibetrag{d\eps^{n-1}}$. Then the claim follows
	by induction over $n$,
\end{proof}

\subsection{Discrete Exterior Calculus}

\begin{observation}
	\label{obs:DECisExteriorCalculus}
	As we have noticed in \ref{rem:GaffneysInequality},
	all variational problems from section \ref{sec:functionalAnalysis}
	are uniquely solvable in $(\Pol\inv\Omega^k,\bard)$ like in
	$(\Omega^k,d)$ by the construction of $\Pol\inv\Omega^k$ as a
	(co-)chain complex. As $(\Pol\inv\Omega^k,\bard)$ just a gentle
	way of writing the simplicial cochain complex $(C^k, \Rand^*)$,
	its (co-)homology is isomorphic to the de Rham complex' one
	(a short direct proof, called ``the theorem of de Rham'', is
	given in \citealt[sec. \textsc{iv}.29]{Whitney57}, although
	\citealt{deRham31} proved isomorphy to singular, not simplicial
	cohomology). Therefore, we can hope for approximating
	smooth solutions of variational problems by ones in $\Pol\inv\Omega^k$.
\end{observation}
\begin{situation}
	\label{sit:varPbsInDEC}
	Let $r\complex$ be a realised oriented regular $n$-dimensional
	simplicial complex without boundary with a piecewise flat, $(\theta,h)$-small metric $g$. Let
	$\lambda_\simplexs$, $\simplexs \in \complex^*$, be the simplices'
	circumcentres, and suppose $\lambda_\simplexs^i > 0$
	for all their components (i.\,e. $r\complex g$ is well-centred).
	Assume \eqnref{eqn:assumptionDECinterpolation} and that the
	Hodge decomposition $u = da + \delta b + c$ is $\Sob^1$-regular, meaning
	$\ibetrag[\Sob^1]{da} \simleq \ibetrag[\Sob^1] u$ etc. We use
	the Poincaré inequality in the form
	\ref{rem:PoincareConstantWithoutScalingFactor}.
\end{situation}
\begin{proposition}
	\label{prop:DirichletPbInDEC}
	Situation as in \ref{sit:varPbsInDEC}. For a function
	$f \in \Sob^1$, let $u \in \Sob^2$ be the solution of the Poisson
	problem $\dprod{du}{dv} = \dprod fv$ for all $v \in \Sob^1$,
	and let $u_h \in \Pol\inv\Omega^0$ be the solution of
	$\dprod{du_h}{dv_h} = \dprod f {v_h}$ for all
	$v_h \in \Pol\inv\Omega^0$. Then
	\[
		\dprod{du - \bard u_h}{\bard v_h}
		\simleq \tilde C_\boxdot h
			( \ibetrag{\nabla f} \ibetrag{v_h}
			+ \ibetrag{\nabla du} \ibetrag{\bard v_h})
		\qquad \forall v_h \in \Pol\inv\Omega^k.
	\]
\end{proposition}
\begin{proof}
	Let $v$ and $v_h$ be connected by
	\eqnref{eqn:Hminus1estimateForPolinv}. Then
	\[
		\dprod{du - \bard u_h}{\bard v_h}
			= \dprod{du}{dv} - \dprod{\bard u_h}{\bard v_h}
			  + \dprod{du}{\bard v_h - dv}
			= \dprod f{v - v_h} + \dprod{du}{\bard v_h - dv},
	\]
	and both terms can be estimated as claimed,
\end{proof}
\begin{proposition}
	\label{prop:approxHodgeDecompInDEC}
	Situation as in \ref{sit:varPbsInDEC}. Let $u = da + \delta b + c$
	be the Hodge decomposition of $u \in \Sob^{1,1}\Omega^k$,
	and let $\bar u = \bard a_h + \bardelta b_h + c_h$ be the
	Hodge decomposition of its $\Leb^2$-orthogonal
	projection onto $\Pol\inv\Omega^k$. Then
	\[
		\begin{aligned}
		\dprod{da - \bard a_h}{\bard v_h}
			& \simleq \tilde C_\boxdot h \ibetrag[\Sob^1]{u} \ibetrag[\Leb^2]{\bard v_h} \\
		\dprod{\delta b - \bardelta b_h}{\bardelta v_h}
			& \simleq \tilde C_\boxdot h \ibetrag[\Sob^1]{u} \ibetrag[\Leb^2]{\bardelta v_h}
		\end{aligned}
		\qquad \forall v_h \in \Pol\inv\Omega^k.
	\]
\end{proposition}
\begin{proof}
	We know that $da$ is characterised by
	$\dprod{da}{dv} = \dprod u{dv}$ for all $v \in \Sob^{1,0}\Omega^k$.
	Naturally, $a_h$ is characterised by
	$\dprod{\bard a_h}{\bard v_h} = \dprod{\bar u}{\bard v_h}$ for all
	$v_h \in \Pol\inv\Omega^k$, but the right-hand side is
	$\dprod u{\bard v_h}$ if $\bard u$ is the orthogonal projection
	onto $\Pol\inv$. So we can proceed exactly like before,
	but using \ref{prop:Hminus1estimateForPolinvOnlyOneStage}
	to connect only $dv$ and $\bard v_h$ instead of $v$ and $v_h$:
	\[
		\dprod{da - \bard a_h}{\bard v_h}
			= \dprod{da}{dv} - \dprod{\bard a_h}{\bard v_h}
			+ \dprod{da}{\bard v_h - dv}
			= \dprod u {dv - \bard v_h} + \dprod{da}{\bard v_h - dv},
	\]
	the $\ibetrag{\nabla da}$ produced by the latter
	term can be estimated by $\ibetrag[\Sob^1] u$
	by assumption. The same procedure is feasible for
	$\delta b$ and $\bardelta b_h$ (where another
	test form $v$ can be employed such that $\delta v$
	is close to $\bardelta v_h$),
\end{proof}
\begin{proposition}
	\label{prop:mixedFormDirichletPbInDEC}
	\makeatletter
	\def\@currentlabel{\arabic{equation}}\label{prop:mixedFormDirichletPbInDEC:woChapter}
	\makeatother
	Define $\homfont S^1 :=
	\Sob^1\Omega^{k-1} \times \Sob^1\Omega^k \times
	\Sob^1\Harm^k$ and $\Pol\inv\homfont S :=
	\Pol\inv\Omega^{k-1} \times \Pol\inv\Omega^k
	\times \Pol\inv\Harm^k$.
	Suppose $s = (\sigma, u, p) \in \homfont S^1$
	is a solution of the Poisson
	problem in mixed form as in \ref{obs:mixedFormDirichletPb},
	and $s_h = (\sigma_h, u_h, p_h) \in \Pol\inv\homfont S$
	is the solution of the
	corresponding finite-dimensional problem. Then for all
	$t_h = (\tau_h, v_h, q_h) \in \Pol\inv\homfont S$,
	\[
		b(s - s_h, t_h) \simleq \tilde C_\boxdot h
			 (\ibetrag[\Leb^2]{\nabla f}
			+ \ibetrag[\Leb^2]{\nabla s}) \ibetrag[\Sob^{1,0}]{t_h},
	\]
	where the left-hand side is of course not to be taken literally
	as in \eqnref{eqn:defBilfb}, but with $\Pol\inv$ exterior
	derivatives for $s_h$ and $t_h$, i.\,e. consisting
	of terms like $\dprod{du - \bard u_h}{\bard v_h}$ etc.
\end{proposition}
\begin{proof}
	In the spirit of \ref{prop:DirichletPbInDEC}, we start with
	\[
		b(s - s_h, t_h) = b(s,t) - b(s_h, t_h) + b(s, t - t_h).
	\]
	As before, the first two terms are $\dprod f{v - v_h}$,
	which is well-controlled by the right-hand side of the claim.
	In $b(s, t - t_h)$, there are many easy terms, which we
	do not explicitely discuss once more. Only
	$\dprod u {q - q_h} = \dprod u{q_h}$ is iteresting.
	Estimating it actually means bounding the difference between
	$\Sob^1\Harm^k$ and $\Pol\inv\Harm^k$. Let $u = da + \delta b$
	be the Hodge decomposition of $u$ (which does not contain
	a harmonic term), and choose $\bard a_h$, $\bardelta b_h$
	close to them in the sense of \ref{prop:Hminus1estimateForPolinvOnlyOneStage}.
	Then, as $q_h \in \Pol\inv\Harm^k$,
	\[
		\dprod u {q_h} = \dprod{da + \delta b}{q_h}
			= \dprod{da - \bard a_h}{q_h} + \dprod{\delta b - \bardelta b_h}{q_h} 
			 \simleq \tilde C_\boxdot h (\ibetrag{\nabla da} + \ibetrag{\nabla db}) \ibetrag{q_h},
	\]
\end{proof}
\begin{remark}
	\begin{subenum}
	\item	In \ref{prop:approxHodgeDecompInDEC}, we cannot say
	anything about $c - c_h$ (yet), as we would need to control
	the terms in $\dprod{c - c_h}{v_h} = \dprod{\bard a_h - da}{v_h}
	+ \dprod{\bardelta b_h - \delta b}{v_h}$, but we only have
	control over the scalar product with $\bard v_h$ or
	$\bardelta v_h$ respectively.
	\item	As remarked in \ref{rem:bdryValuesForDEC},
	the correct treatment of variational boundary value problems
	in $\Pol\inv\Omega^k$ would require a modification of
	$\bardelta$ at boundary simplices.
	\item	Employing $\Pol\inv$ forms also as test functions
	is unsatisfactory, as they are no classical objects
	to test with, so the results are not easily comparable to
	usual estimates. However, we are not sure which forms
	would be the right ones to test with: Perhaps
	forms that are ``almost constant'' in some small,
	but non-shrinking region would be good to
	obtain an average value for such a region. $\Sob^{1,1}$
	forms are not the right candidates: If for example
	in the Poisson problem $\dprod{du - \bard u_h}{dv}$
	converged for all ``well-behaving'' $v \in \Sob^1$,
	then also $\ibetrag{du}^2 - \ibetrag{\bard u_h}^2$
	would converge, which is not the case:
	\end{subenum}
\end{remark}
\begin{example}
	\label{ex:DirPbInPolInv}
	Consider an equilateral triangle mesh in the $xy$ plane
	with unit edge length, rotated such
	that one of the edges is parallel to the $x$-axis.
	Now consider the constant vector field $v = (1,0)$,
	and let us compute its $\Pol\inv\Omega^1$ approximation
	according to \ref{prop:Hminus1estimateForPolinv}:
	In a triangle $ijk$, where $ij$ points in $x$ direction,
	we have $\int_{ijk}\sprod{v_h}{(1,0)} = \frac 1 3 \betrag{ijk}
	(\alpha_{ij} + \frac 1 2 \alpha_{ik} + \frac 1 2 \alpha_{kj})$.
	So $\alpha_{ij} = 2$ and
	$\alpha_{ik} = \alpha_{kj} = 1$ gives the
	correct integral mean (the scalar product with
	$(0,1)$ is obvious). There are
	other combinations matching $\fint_{ijk} v$
	(for example the ``obvious'' choice $\alpha_{ij} = 3$ and
	$\alpha_{ik} = \alpha_{kj} = 0$), but
	this one also captures its exterior derivative:
	As $v$ is constant, $dv = 0$, and $\bard v_h = 0$ for
	$v_h = 2 \omega^{ij} + \omega^{ik} + \omega^{kj}$.
	
	This vector field $v$ is also the gradient of the
	function $f: (x,y) \mapsto x$. If $f_h \in \Pol\inv\Omega^0$ has
	the same values as $f$ on all vertices, then $\bard f_h = v_h$.
	The homogeneous Dirichlet problem in $\Pol\inv\Omega^0$
	is equivalent to the Dirichlet problem in $\Pol^1$
	by \eqnref{eqn:dualEdgeFormulaDEC}, for which reason
	$f_h$ will be the $\Pol\inv$ harmonic function with
	prescribed boundary values $f_h|_{\Rand\complex}$, no matter
	where the boundary is drawn. Thus, $f$ and $f_h$ with
	derivatives $v$ and $v_h$ are the solutions compared in
	\ref{prop:DirichletPbInDEC}. As $\absval v = 1$
	everywhere, we have $\ibetrag[\Leb^2(ijk)]{df}^2 = \betrag{ijk}$
	in each triangle, but $\ibetrag[\Leb^2(ijk)]{\bard f_h}^2
	= \frac 1 3 \betrag{ijk} (4 + 1 + 1) = 2 \betrag{ijk}$.
	The discrepancy $\ibetrag[\Leb^2(ijk)]{\bard f_h}^2
	= 2 \ibetrag[\Leb^2(ijk)]{df}^2$ does not shrink
	with smaller and smaller edge lengths, so the Dirichlet
	energy of the $\Pol\inv$ approximations will always
	be twice as large as of the analytical solution.
\end{example}


\newsectionpage
\section{Approximation of Submanifolds}
\label{sec:graphSubmanifolds}


\subsection{Extrinsic and Intrinsic Karcher Triangulation}

\begin{definition}
	Let the piecewise smooth $n$-dimensional submanifold $S \subset M$ be
	given by a bijective triangulation $y: r\complex \to S$
	with vertices $p_i = y(ri)$. If $y|_{r\simplexe}$ for each $\simplexe
	\in \complex^n$ is a barycentric mapping with respect to the
	induced metric $g|_S$, then $y$ is called an
	\begriff{intrinsic} Karcher triangulation. If moreover
	each $y|_{r\simplexe}$ is also a barycentric mapping with
	respect to the metric of $M$, then $y$ is called an
	\begriff{extrinsic Karcher triangulation}.
\end{definition}
\begin{goal}
	The possibility for the existence of an intrinsic
	Karcher triangulation has been dealt with in section
	\ref{sec:KarcherDelaunayTriangulation}. The question
	of this section will now be how well an extrinsic
	Karcher triangulation, induced by the same complex
	$\complex$ and the same vertex set $\{p_i\}$, approximates $S$.
	Note that such an extrinsic Karcher triangulation
	is always an interpolation of the given triangulation
	of $S$ in the sense of \ref{prop:estimateDistancexAndyH1}.
\end{goal}
\begin{proposition}
	Let $y$ be a piecewise $(\theta,h)$-small
	intrinsic Karcher triangulation $y$ of
	$S$ with $\norm{W_\nu}
	\leq \kappa$ for all Weingarten maps $W_\nu$.
	Suppose that all vertices $p_i = y(ri)$, $i \in \simplexe$,
	lie in a common
	convex ball with respect to $g$ for each $\simplexe
	\in \complex^n$.
	Then for small edge lengths
	$\ell_{ij} := \dist_S(p_i,p_j)$ and $\bar \ell_{ij} :=
	\dist(p_i,p_j)$
	with respect to $g|_S$ and $g$, it holds
	$\absval{\ell_{ij} - \bar \ell_{ij}}
	\simleq \kappa h \theta\inv \bar\ell_{ij}$.
\end{proposition}
\begin{proof}
	There exists an extrinsic Karcher triangulation $x$ of
	some set $S' \subset M$ with the same combinatorics and vertices
	as $y$ (that means: interpolating $y$)
	because the vertices of each simplex
	are contained in a convex ball.
	We do not know if $S'$ is a manifold,
	because the fullness of the extrinsic simplices is not
	clear \textit{a priori}, but will be a result of the
	length estimate.
	
	Let us show the claim for the edge $\gamma: e_i \leadsto e_j$ in $r\simplexe$
	with tangent $v := \dot \gamma$.
	The estimate \ref{prop:dxMinusdyAtVertices} can be extended
	to the whole edge:
	\[
		\absval{(dx - P dy)v} \simleq h\theta\inv \frac{\absval{(\nabla dx - P \nabla dy)(v,v)}}{\absval{dy\,v}}
	\]
	As edges are mapped to geodesics,
	$\nabla dx(v,v) = 0$ and $\tang \nabla dy(v,v) = 0$.
	And as $\sprod{\nabla_{dy\,v} dy\,v}\nu = -\sprod {\nabla_{dy\,v}\nu}{dy\,v}$
	for any normal $\nu$ to $S$, we have $\absval{\nabla dy(v,v)}
	\leq \kappa \absval{dy\,v}^2$. So
	\[
		\absval{\ell_{ij} - \bar \ell_{ij}}
			\leq \int \bigabsval{\absval{dx\,v} - \absval{dy\,v}}
			\leq \int \absval{(dx - Pdy) v}
			\leq \kappa h\theta\inv \int \absval{dy\,v},
	\]
\end{proof}
\begin{corollary}
	\label{prop:estimateExtrIntrBaryMapping}
	\begin{subeqns}
	Situation as before, additionally
	$\norm{W_\nu} + h \norm{\nabla W_\nu} \leq \kappa$
	for all Weingarten maps.
	Let $\bar \ell_{ij}$ also induce
	a $(\theta,h)$-full metric $g^e$ on $r\complex$. Then
	for small $h$,
	\begin{gather}
		\absval{(y^*g - g^e)\sprod vw}
			\simleq (C_0 h^2 + \kappa h\theta\inv) \absval v\, \absval w,\\
		\label{eqn:connectionEstimateIntrinsicKarcherTriang}
		\absval{\nabla^{y^*g}_v w - \nabla^{g^e}_v w}
			\simleq \tilde C_{0,1}' h \absval v\, \absval w,
	\end{gather}
	where $\tilde C_{0,1} = (C_{0,1} + \kappa^2)\theta^{-1}$.
	The second estimate also holds for any other piecewise
	flat metric $g^e$ on $r\complex$.
	\end{subeqns}
\end{corollary}
\begin{proof}
	The metric estimate comes from the edge length comparison
	above, and the connection estimate from
	\ref{prop:estimateOfChristoffelOperator} does not depend on the
	chosen metric, as long as it is flat. Due to the Gauß equation
	(e.\,g. \citealt[thm. 4.7.2]{Jost11}), the intrinsic curvature
	tensor of $S$ is bounded by $C_0 + \norm{W_\nu}^2$ and its derivative
	by $C_1 + \norm{W_\nu} \norm{\nabla W_\nu}$,
\end{proof}
\begin{remark}
	The observation that \eqnref{eqn:connectionEstimateIntrinsicKarcherTriang}
	also holds for any other piecewise flat metric on $\complex$ means
	that if $g$ is approximated up to second order
	by a better-suited approximation of edge lengths $\ell_{ij}$ than just
	$\bar \ell_{ij}$, then the approximation of the connection
	remains unchanged.
	
	Nevertheless, taken as it is,
	\ref{prop:estimateExtrIntrBaryMapping} says that a simple interpolation
	of a given triangulation, just as in Euclidean space, is not
	the best candidate for geometry approximation. Henceforth,
	the rest of this section is devoted to the normal graph
	mapping, which reveals better approximation properties.
\end{remark}

\subsection{General Properties of Normal Graphs}

\begin{definition}
	\begin{subeqns}
	Let $S \subset M$ be an $n$-dimensional compact boundaryless
	smooth submanifold.
	A second submanifold $S' \subset M$ is said to be a
	\begriff{normal graph} over $S$ if there is
	a normal vector field $Z$ on $S$ such that
	\begin{equation}
		\label{eqn:graphParametrisation}
		\Phi: a \mapsto \exp_a Z|_a
	\end{equation}
	is a bijective mapping $S \to S'$. Where we need it,
	we will also consider the ``smooth transition'' $S \leadsto S'$
	via the homotopy
	\begin{equation}
		\Phi_t: a \mapsto \exp_a t Z|_a.
	\end{equation}
	Parallel transport along $t \mapsto \Phi_t(p)$ from $\Phi_a(p)$
	to $\Phi_b(p)$ will
	be denoted by $P^{b,a}$.
	\end{subeqns}
\end{definition}
\bibrembegin
\begin{remark}
	\begin{subenum}
	\item
	The term ``normal graph'' or ``normal height map''
	is mostly used in the context of triangular approximation
	of surfaces in $\R^3$, e.\,g. in \cite{Hildebrandt06}.
	In the context of manifold-valued pde's, it
	is more common to consider the \begriff{geodesic
	homotopy} $\Phi_t$, see \ref{rem:Leb2Width},
	which also \cite{Grohs13} use. In particular, their
	control of distortion \emph{along $\Phi_t$}
	is equivalent to our control of the distortion
	\emph{by $\Phi$}.
	\bibremend
	\item
	Here, as usual, we do not want to treat global
	properties of $M$, so we always tacitly assume
	$\absval Z < \inj M$.
	\item
	Any other $n$-dimensional submanifold $S' \subset M$
	that is near enough to have a bijective orthogonal
	projection $S' \to S$ can be represented as a normal
	graph over $S$. Here ``normal projection'' means mapping
	some $p \in S'$ onto the point $q \in S$ minimising
	$\dist(p,q)$. The largest $\eps$ such that the orthogonal
	projection $\B_\eps(S) \to S$ is well-defined is called
	the \begriff{reach} of $S$ (introduced by \citealt[def. 4.1]{Federer59},
	for a recent overview see \citealt{Thaele08}).
	Another formulation for the same thing is that
	\[
		\tilde \Phi: TS^\perp \to M,
		(p,Z) \to \exp_p Z
	\]
	is a diffeomorphism from
	$O_\eps := \{\nu \in TS^\perp \with \absval \nu < \eps\}$
	onto its image.
	\item	It is well-known (see e.\,g. \citealt[eqn 1.11]{Hildebrandt12})
	that for $M = \R^m$, the map
	$\Phi$ is locally a diffeomorphism if $\absval Z \norm{W_\nu} < 1$
	for all Weingarten maps $W_\nu$. (Note that
	this bound can only capture the local geometry of $S$ but
	cannot see if some part of $S$ that is intrinsically far
	from a point $p \in S$ comes close to $p$ in the
	surrounding space $M$.)
	Different to the usual argument involing the curvature
	radius of $S$ and osculating spheres, one can
	use Jacobi fields as in \ref{parag:graphCoordinates}
	(we use the notation from there) to see this:
	
	We already know $d\Phi_t\,\dot p = J(t)$ for a Jacobi field
	with $J(0) = \dot p$ and $\dot J(0) = \nabla_{\dot p} \nu$, so
	$\ddt \Phi_t^*g\sprod{\dot p}{\dot p} = g\sprod J {\dot J}$.
	Now let $\nu$ be a unit normal field. Then
	$\dot J = W_\nu J$. If $\dot p$ is the eigenvector in
	direction of the largest eigenvalue $\kappa$, we have
	\[
		\ddt \Phi_t^* g\sprod{\dot p}{\dot p} = 2 \kappa(t) \absval{J(t)},
	\]
	where $\kappa(t)$ is the eigenvalue of the
	Weingarten map $W_\nu$ in direction $\dot p$ at $\Phi_t(p)$. 
	In our case $R = 0$, the Riccati equation
	\eqnref{eqn:RiccatiEqnForWeingartenMap} gives
	$\dot W_\nu = - W_\nu^2$, so the eigenvalues
	$\kappa_i$ also evolve by $\dot \kappa_i = -\kappa_i^2$.
	This differential equation has solution
	$\kappa_i(t) = (t - \frac 1 {\kappa_i(0)})\inv
	= \frac{\kappa_i(0)}{\kappa_i(0)t - 1}$, hence
	\[
		\ddt \Phi_t^* g\sprod{\dot p}{\dot p}
			= \frac{2\kappa}{\kappa t - 1} \Phi_t^*g \sprod{\dot p}{\dot p}.
	\]
	This is solved by $\Phi_t^*g\sprod{\dot p}{\dot p}
	= (1 - \kappa t)^2 g\sprod{\dot p}{\dot p}$, so
	$g$ is positive definite for $\kappa t < 1$.
	\item	\label{rem:tangPartOfNablaZ}
	If we represent $Z = f^i \nu_i$ with parallel
	unit normal vector fields $\nu_i$ and scalar functions $f^i$,
	then $\nabla Z = df^i \otimes \nu_i + f^i \nabla \nu_i$,
	which splits into tangential and normal parts
	\[
		\tang \nabla Z = f^i \nabla \nu_i,
		\qquad
		\nor \nabla Z = df^i \otimes \nu_i.
	\]
	Hence, although $\norm{\nabla Z}$ usually shrinks slower
	than $\absval Z$ for $\absval f \to 0$, the tangential part
	is $\norm{\tang \nabla Z} \simleq \kappa \absval f$, where
	$\kappa$ is an upper bound for the Weingarten maps $W_\nu$ of $S$.
	Similary, $\norm{\tang \nabla^2 Z} \simleq \kappa \absval{df}
	+ \kappa' \absval f$ if $\kappa'$ bounds all $\norm{\nabla W_\nu}$.
	\end{subenum}
\end{remark}
\bibremend
\begin{situation}
	\label{sit:normalGraph}
	Let $S'$ be given as a normal graph over $S$ by a vector
	field $Z$ with $\dist := \absval Z \leq \eps^2$ and
	$\norm{d \dist} \leq \eps$ everywhere, and
	let the Weingarten maps of $S$ be bounded by
	$\norm{W_\nu} + \eps \norm{\nabla W_\nu} \leq \kappa$.
	This means $\norm{\nabla Z} \simleq \kappa \eps$.
	For simplicity, assume $\kappa \leq \kappa^2$.
\end{situation}
\begin{proposition}
	\label{prop:boundOnReach}
	Situation as in \ref{sit:normalGraph}.
	The map $\Phi: a \mapsto \exp_a Z$ is
	locally a diffeomorphism if
	$\absval \dist (\kappa + \sqrt{C_0}) < 1$ everywhere.
\end{proposition}
\begin{proof}
	Let us suppose $Z$ has unit length at some point and see
	for which $t$ the map $\Phi_t$ is locally a diffeomorphism.
	Note that
	\eqnref{eqn:RiccatiEqnForWeingartenMap} gives
	$\absval{\dot \kappa_i} \leq C_0 + \absval{\kappa_i^2}$
	for any eigenvalue of a Weingarten map.
	The equation $\dot u = -C_0 - u^2$ leads to a subsolution
	\[
		u(t) = \sqrt{C_0} \tan \Big(\sqrt{C_0}t - \arctan \smallfrac{\kappa_i(0)}{\sqrt{C_0}} \Big).
	\]
	Regarding $\ddt \Phi_t^* g\sprod{\dot p}{\dot p} \leq \kappa(t)
	\Phi_t^* g\sprod{\dot p}{\dot p}$, which has positive solutions 
	(for positive initial data) as long
	as $\kappa$ is bounded, it suffices to know where the first
	pole of $\kappa_i(t)$ can occur. The first pole of $u$ (which must also
	bound the position of the first pole of $\kappa_i$) is where
	$\sqrt{C_0}t - \arctan \smallfrac{\kappa_i(0)}{\sqrt{C_0}}
	= \pm \smallfrac \pi 2$. Now a simple function inspection shows
	\[
		\frac 1 {1 + s} < \smallfrac \pi 2 + \arctan s,
		\qquad
		- \frac 1 {1 + s} > -\smallfrac \pi 2 + \arctan s,
	\]
	so there will be no pole as long as $\absval{\sqrt{C_0} t}
	< (1 + \smallfrac \kappa {\sqrt{C_0}})\inv$,
\end{proof}
\begin{observation}
	\label{prop:linearisationOfdPhi}
	\begin{subeqns}
	By \ref{prop:derivativeOfExpAlongCurve},
	the differential of $\Phi_t$
	is $d\Phi_t V = J(t)$ for the Jacobi field with
	$J(0) = V$, $\dot J(0) = \nabla_V Z$.
	By \eqnref{eqn:JacobiFieldsUsualEstimate}, this means
	for $Q_t: V \mapsto V + t\nabla_V Z$
	\begin{equation}
		\absval{d\Phi_t V - P^{t,0}Q_tV} \simleq C_0 \dist^2 t^2 (1 + \kappa \eps t) \absval V
	\end{equation}
	and
	\begin{equation}
		\absval{d\Phi_s V - P^{t,0}V} \simleq \kappa \eps \absval V
			+ C_0 \dist^2 t^2 (1 + \kappa \eps t)\absval V.
	\end{equation}
	This estimate is scale-invariant with respect
	to scaling of $Z$: If $Z' = \alpha Z$,
	then $\Phi_{t/\alpha}' = \Phi_t$. So
	$tZ = t' Z'$ is scale-invariant,
	and the estimate only contains $tZ$, never
	$Z$ alone.
	
	But due to \ref{rem:tangPartOfNablaZ},
	this distortion happens mostly in normal direction,
	the tangential change is of higher order:
	\begin{equation}
		\label{eqn:linearisationOfTangdPhi}
		\norm{\tang(d\Phi - P)} \simleq \kappa \dist + C_0 \dist^2
	\end{equation}
	\end{subeqns}
\end{observation}
\begin{proposition}
	\label{prop:estimateForProjectionDifferential}
	\begin{subeqns}
	Situation as in \ref{sit:normalGraph}.
	Consider some point $p \in \operatorname{reach} S$
	with projection $\psi(p)$ onto $S$. If
	$p = \exp_{\psi(p)} \dist \nu$ for some unit
	normal vector $\nu$ with $\norm{W_\nu} \leq \kappa$, and if
	$\kappa \eps < \frac 1 2$, the orthogonal
	projection $\psi$ satisfies
	\begin{equation}
		\norm{d\psi - Q_\nu\inv \tang P^{p,\psi(p)}} \simleq C_0 \dist^2,
		\qquad
		\norm{d\psi - \tang P^{p,\psi(p)}} \simleq \kappa \dist + C_0 \dist^2,
	\end{equation}
	where $Q_\nu$ is the linear map $T_p S \to T_p M$, $V \mapsto V + \dist \nabla_V \nu$.
	If $\dist \nu$ is replaced by some other normal
	vector field $Z$ with $\exp_{\psi(p)} Z = p$,
	and $Q$ is replaced by $V \mapsto V + \nabla_V Z$,
	it holds
	\begin{equation}
		\label{eqn:estimateForProjectionDifferential}
		\norm{d\psi - (\tang Q)\inv \tang P^{p,\psi(p)}} \simleq C_0 \dist^2,
		\qquad
		\norm{d\psi -\tang P^{p,\psi(p)}} \simleq \kappa \dist + C_0 \dist^2.
	\end{equation}
	\end{subeqns}
\end{proposition}
\begin{proof}
	Let us first show that $\tang Q$ does not depend
	on how $Z$ is chosen at points $\neq \psi(p)$.
	If $Z = \dist \nu$ in a neighbourhood of $\psi(p)$, where
	$\nu$ is a parallel unit normal field and $\dist$ is
	constant, then $Z$ is parallel, and so
	$\nabla Z = \tang \nabla Z = \dist W_\nu$. For any other $Z$,
	$\tang \nabla Z$ stays the same, and only some part
	$\nor \nabla Z \neq 0$ is added. That means
	$\tang Q_\nu = Q_\nu$ on $T_{\psi(p)} S$. So we will only prove
	\eqnref{eqn:estimateForProjectionDifferential}.
	
	\textit{ad primum:}	
	Observe that the operator $\tang Q: T_p S \to T_p S$
	fulfills $\norm{\tang Q - \id} \simleq \kappa \eps < \frac 1 2$
	by assumption,
	hence is invertible with $\norm{(\tang Q)\inv - \id}
	\simleq \frac{\kappa \eps}{1 - \kappa \eps}$ by
	\ref{prop:estimateIdMinusInverseOperator}, which gives
	$\norm{(\tang Q)\inv} \simleq 1 + \frac{\kappa \eps}{1 - \kappa \eps}
	= \frac 1 {1-\kappa \eps} < 2$. Hence
	the claim is proven if we can show
	$\norm{\tang(Q d\psi - P)} \simleq C_0 \dist^2$.
	
	Consider some vector $V \in T_pM$ and split
	$V = V_p + V_\nu$ as in \eqnref{eqn:JacobiFieldsForOrthProjection}.
	Then $\tang Q d\psi(V) = \dot p + \tang \nabla_{\dot p} Z
	= J_p(0) + \dot J_p(0)$ on the one hand, and $\tang P V = PV_p = P J_p(1)$
	by \eqnref{eqn:JacobiFieldsOrthSplitting} on the other.	
	So \eqnref{eqn:JacobiFieldsUsualEstimate} gives
	\[
		\absval{J_p(1) - P(J_p(0) + \dot J_p(0))}
			\simleq C_0 \dist^2 \absval{\dot p}
			\simleq C_0 \dist^2 \absval V.
	\]
	\textit{ad sec.:} We have
	$\norm{d\psi - \tang P}
	\leq \norm{d\psi - (\tang Q)\inv \tang P} + \norm{(\tang Q)\inv \tang - \tang}$.
	The first norm has been estimated above, and the second is
	$\simleq \kappa \dist$ because $\norm{\tang Q \tang - \tang} \simleq
	\kappa \dist$ due to \ref{rem:tangPartOfNablaZ} and the boundedness of
	$(\tang Q)\inv$,
\end{proof}
\bibrembegin
\begin{remark_nn}
	For $C_0 = 0$, this (exact) representation of the projection
	differential
	is the one in \citet[thm. 3.2.1]{Wardetzky06} and
	\citet[lemma 4]{Morvan04}.
\end{remark_nn}
\bibremend
\begin{proposition}
	\label{prop:estimateProjectionDifference}
	Situation as in \ref{sit:normalGraph}.
	Let $\tang'$ be the orthogonal projection
	$TM|_{S'} \to TS'$. Then for small $\eps$,
	we have $\norm{P\tang' - \tang P}
	\simleq \kappa \eps + C_0 \dist^2$. This means that the angles
	$\measuredangle(T_{\Phi(p)}S', PT_pS)$ and
	$\measuredangle((T_{\Phi(p)}S')^\perp,$ $PT_pS^\perp)$
	between the corresponding tangent and normal spaces
	must be bounded by this factor, too.
			\label{prop:estimateNuMinusNut}
	Therefore, normals $\nu_i$ to $S$
	can be extended to
	normal fields $\nu_{i,t}$ along $\Phi_t$
	with $\absval{\nu_{i,t} - P^{t,0} \nu_i}
	\simleq \kappa \eps + C_0 \dist^2$.
\end{proposition}
\begin{proof}
	For the time of this proof, let us write
	the terminal value $J(1)$ of a Jacobi field along $t\mapsto \Phi_t(p)$
	with initial values $J(0) = \dot q$ and $\dot J(0) = \dot \nu$
	as $T(\dot q,\dot \nu)$. Linearity of the Jacobi equation translates
	into linearity of $T$. In this notation, the splitting from
	\eqnref{eqn:JacobiFieldsForOrthProjection} says that
	a vector $V \in TM|_{S'}$ can be represented as
	$V = T(\dot p,\tang \dot \nu) + T(0,\nor \dot \nu)$. We argue that
	its projection $\tang' V$ onto $TS'$ is almost $T(\dot p,\nabla_{\dot p}Z)$.
	
	In fact, all tangent vectors on $S'$ have the form
	$T(\dot q,\nabla_{\dot q}Z)$ for some $\dot q \in TS$.
	Now consider
	\[
		\absval{V - T(\dot q,\nabla_{\dot q}Z)}^2
			= \absval{T(\dot p - \dot q, \dot \nu - \nabla_{\dot q}Z)}^2.
	\]
	This is minimal among all $\dot r$
	if $\tang' V = T(\dot q,\nabla_{\dot q}Z)$, this means
	its norm has vanishing derivative in direction $(\dot r,\nabla_{\dot r}Z)$.
	Because $T$ is linear, this gives
	\[
		0 = \sprod{T(\dot p - \dot q,\dot \nu - \nabla_{\dot q}Z)}{T(\dot r,\nabla_{\dot r}Z)}
		\qquad \forall \dot r \in T_p S.
	\]
	Now recall that $T(U,W) = P(U + \dist W) + O(C_0 \dist^2)$,
	hence this is
	\[
		= \sprod{\dot p - \dot q}{\dot r + \dist \nabla_{\dot r} Z}
		+ \dist \sprod{\dot \nu - \nabla_{\dot q}Z}{\dot r}
		+ \dist^2 \sprod{\dot \nu - \nabla_{\dot q}Z}{\nabla_{\dot r} Z} + O(C_0 \dist^2).
	\]
	If $\dot p = \dot q$, the first term vanishes,
	and (using that $\kappa \dist$ is small) the remaining ones are estimated
	from above by $\kappa \dist \absval V\,\absval{\dot r}$.
	Because the minimisation is well-conditioned at this
	position, the optimal $\dot q$ is $\dot p + O((\kappa \dist + C_0 \dist^2) \absval V)$.
	
	Now recall from
	\eqnref{eqn:JacobiFieldsOrthSplitting}
	that $P\tang P V = T(\dot p, \tang \nabla_{\dot p} Z)$, which gives that
	the claim $|P \tang' V - \tang P V|
	= \absval{(\tang' - P \tang P)V} = \absval{T(0, \nor \nabla_{\dot p}Z)}
	+ O((\kappa \dist + C_0 \dist^2) \absval V) \simleq (\kappa \eps + C_0 \dist^2)\absval V$
	is just
	the usual Jacobi field estimate \eqnref{eqn:JacobiFieldsUsualEstimate},
\end{proof}
\begin{corollary}
	Omitting the last paragraph of the proof, one gets
	$\norm{\tang P \tang' - \tang P} \simleq \kappa \dist + C_0 \dist^2$.
\end{corollary}
\bibrembegin
\begin{remark_nn}
	This is analogous to the classical statement
	$\norm{P(P_h - \1)P} \simleq \dist$ up to constants
	depending on the geometry from
	\citename{Dziuk} \textit{et al.}, where $P$
	is the projection onto $TS$ and $P_h$
	the projection onto $TS'$.
\end{remark_nn}
\bibremend

\subsection{Geometric Distortion by the Graph Mapping}

\begin{lemma}
	\label{prop:comparisonQXQYwithXY}
	Situation as in \ref{sit:normalGraph}.
	Then for
	$Q: U \mapsto U + \nabla_U Z$,
	\[
		\betrag{\sprod{QU}{QV} - \sprod UV}
		\simleq (\kappa^2\dist + C_0 \dist^2) \betrag U \betrag V
		\qquad \forall U,V \in T_p S.
	\]
\end{lemma}
\begin{proof}
	\sloppypar
	Just because $\sprod{U + \nabla_U Z}{V + \nabla_V Z}
	- \sprod UV = \sprod{\nabla_U Z}V + \sprod{\nabla_V Z}U
	+ \sprod{\nabla_U Z}{\nabla_V Z}$ and \eqnref{eqn:linearisationOfTangdPhi},
\end{proof}
\begin{conclusion}
	\label{prop:comparisonPhigWithg}
	Situation as in \ref{sit:normalGraph}.
	Pulled back to $S$, the $S'$ metric
	$\Phi^* g|_p\sprod UV
	= g|_{\Phi(p)} \sprod{d\Phi\,U}{d\Phi\,V}$
	fulfills
	\[
		\bigabsval{(\Phi^* g - g)\sprod UV}
		\simleq (\kappa^2 \dist + C_0 \dist^2) \betrag U \betrag V.
	\]
\end{conclusion}
\begin{proof}
	Is a direct application
	of \eqnref{eqn:linearisationOfTangdPhi} and
	\ref{prop:comparisonQXQYwithXY}. We especially
	remark that the difference between
	$g|_p$ and $g|_{\Phi(p)}$ does not need
	to be handled explicitely, as $P^{p,\Phi(p)}$
	is an isometry with respect to these two metrics,
\end{proof}
\bibrembegin
\begin{remark}
	\label{rem:metricDistortionTensorA}
	As $\Phi^*g$ and $g$ are equivalent metrics,
	$A:=d\Phi^t d\Phi$ (where $d\Phi^t$ denotes the $g$-adjoint
	of $d\Phi$) is a self-adjoint automorphism
	of $T_p S$ such that
	$\Phi^*g\sprod UV = g\sprod{AU}V$, called
	the \begriff{metric distortion tensor}
	by \citet[p. 53]{Wardetzky06}.
	In the numerical literature, it
	is common not to compare the Riemannian
	metrics, but to estimate directly
	$\norm{\smallfrac{G^e}G A - \id}$,
	which already includes the
	volume element change (cf. the proof
	of \ref{prop:normEquivalences}),
	see \cite{Dziuk88, Demlow09, Heine05}.
	
	For a comparison with the tensor $J$
	from \ref{rem:metricDistortionTensorJ},
	consider $\Psi := \Phi\inv$.
	If $M$ is the Euclidean space
	$\R^n$ and $S'$ is a piecewise flat submanifold,
	its metric $g^e := g|_{S'}$ is piecewise flat.
	The metric $g|_S$ pulls back to a metric $\Psi^*g$
	on $S'$, and there is $J$ such that
	$\Psi^*g\sprod UV = g^e\sprod{JU}V$.
	So the transformations $A$ and $J$ perform
	inverse tasks.
\end{remark}
\bibremend
\begin{proposition}
	\label{prop:connectionPulledBackByPhi}
	Situation as in \ref{sit:normalGraph}.
	For a given vector
	$U$ and a vector field $V$ on $S$, define the ``connection
	distortion''
	$W := \nabla_{d\Phi_t U} d\Phi_t V -
	d\Phi_t (\tang \nabla_U V)$. This vector field
	obeys the differential equation
	\[
		\ddot W = R(Z,W)Z + \dot F
		\qquad \text{for }
		F := R(Z,d\Phi_t U)d\Phi_t V + \nabla^2_{d\Phi_t U,d\Phi_t V} Z
	\]
	with initial values $W(0) = 0$ and $\dot W(0) = F(0)$.
\end{proposition}
\begin{proof}
	Let us abbreviate
	$U^t := d\Phi_t U$, $V^t := d\Phi_t V$, and
	denote the parallel translation of $Z$ along
	$t \mapsto \exp_p tZ$ also as $Z$. Let
	$K := \nabla_{U^t} V^t$ and
	$J := d\Phi_t (\tang \nabla_U V)$. Then we want
	to determine $W = K - J.$
	
	By \ref{parag:graphCoordinates}, $J$ is a Jacobi
	field, i.\,e. $\ddot J = R(Z,J)Z$. An
	inhomogeneous Jacobi equation describes $K$:	
	\[
		\ddot K = R(Z,K)Z + D_t R(Z,U^t)V^t + D_t \nabla^2_{U^t,V^t} Z.
	\]
	In fact,
	consider a variation $\gamma(r,s)$ of geodesics (in $M$),
	i.\,e. we assume that $s \mapsto \gamma(r, s)$ is
	a geodesic for each fixed $s$, with
	$\partial_s \gamma(0,0) = V$ and
	$\partial_r \gamma(0,0) = U$.
	Transport this along $t$ as
	$c(r,s,t) := \exp_{\gamma(r,s)} t Z|_{\gamma(r,s)}$.
	Then we want to determine $K = D_r \partial_s$, so we consider
	\[
		\begin{split}
			\ddot K = D_t D_t D_r \partial_s
				& = D_t D_r D_t \partial_s + D_t R(\partial_t,\partial_r)\partial_s \\
				& = D_t D_r D_s \partial_t + D_t R(\partial_t,\partial_r)\partial_s,
		\end{split}
	\]
	the first term of which is
	\[
		\begin{split}
			\nabla_Z \nabla_{U^t} \nabla_{V^t} Z
				& = \nabla_Z \nabla^2_{U^t,V^t} Z + \nabla_Z \nabla_{\nabla_{U^t}V^t} Z\\
				& = \nabla_Z \nabla^2_{U^t,V^t} Z + \nabla_{\nabla_{U^t}V^t} \nabla_Z Z + R(Z, \nabla_{U^t}V^t),Z\\
				& = D_t \nabla^2_{U^t,V^t}Z \,\, + \,\, \qquad 0 \qquad\quad\! + R(Z,K)Z.
		\end{split}
	\]
	The initial value is computed in exactly the same way,
\end{proof}
\pagebreak[3]
\begin{proposition}
	\label{prop:weakEstimateForPullBackConnectionDifference}
	\begin{subeqns}
	Situation as before. If $C_0 \absval Z^2 + \norm{\nabla Z}
	\leq \frac 1 2$, then
	\[
		\absval{\nabla_{d\Phi\, U} d\Phi\, V - d\Phi\, \tang \nabla_U V}
		\simleq
			\absval U \absval V (\norm{\nabla^2 Z} + C_0 \absval Z)
			+ C_0  \absval{\nabla_U V} \absval Z^2.
	\]
	If we only consider the tangential part $\tang W$,
	then even
	\begin{equation}
		\label{eqn:estimateForPullBackConnectionDifference}
		\absval{\tang \nabla_{d\Phi\, U} d\Phi\, V - d\Phi\, \tang \nabla_U V} 
		\simleq
			\absval U \absval V \big( \norm{\tang \nabla^2 Z}
			+ \norm{\nabla Z} + C_0 \absval Z^2 \big)
			+ C_0 \absval{\nabla_U V} \absval Z^2.
	\end{equation}
	\end{subeqns}
\end{proposition}
\begin{proof}
	\textit{Preparatory step one:} Let us first establish the boundedness of $W = K - J$
	and show $\absval W \simleq \absval{\nabla_U V} + at$, where
	$a := \absval{U^t} \absval{V^t} (\norm{\nabla^2 Z} + C_0 \absval Z)$:
	The $t$-derivative of $K$ is, as $\partial_t = Z$,
	\[
		\begin{split}
			\nabla_Z \nabla_{U^t} V^t & = \nabla_{U^t} \nabla_Z V^t + R(Z,U^t)V^t \\
				& = \nabla_{U^t} \nabla_{V^t} Z + R(Z,U^t)V^t 
				\quad = \quad
				\nabla^2_{U^t,V^t}Z + \nabla_{\nabla_{U^t}V^t} Z + R(Z,U^t)V^t,
		\end{split}
	\]
	so $\ddt \absval K \leq \absval{\dot K} \leq
	a + \norm{\nabla Z} \, \absval K$. As $\absval Z$ is short by assumption, we
	have $\absval{U^t} \simleq \absval U$ and $\absval{V^t} \simleq \absval V$.
	The differential inequality of the
	form $\dot u \leq a + b u$ gives $u \leq (u_0 + \frac ab)
	\e^{b t} - \frac ab$.
	For $bt \leq \frac 1 2$, this function is dominated
	by $u_0 + 2 b t(u_0 + \frac ab) \leq 2(u_0 + at)$.
	
	For the bound on $J$, we have $\absval J \leq
	\absval{J(0)} + t \absval{\dot J(0)}$ as usual,
	and $J(0) = \nabla_U V$, $\dot J(0) = \nabla_Z J(0)$
	shows that these terms are already contained in the $K$ estimate.

	\textit{Preparatory step two:} Now let us show
	\[
		\Bigabsval{W(t) - \Int_0^t P^{t,\tau}F(\tau) \d \tau} 
		\simleq C_0 t^2 \absval Z^2 (\absval{\nabla_U V} + at).
	\]
	The proof idea is from \citet[thm 5.5.2]{Jost11}.
	Let $A := \int_0^t PF$. This is a vector field
	fulfilling $\ddot A = \dot F$ with the same initial values
	as $W$, namely
	$A(0) = 0$ and $\dot A(0) = F(0)$. Furthermore,
	let $w: \interv 0t \to \R$ be the solution of
	$\ddot w = C_0 \absval Z^2 \absval W$ with initial
	values $w(0) = \dot w(0) = 0$. Then,
	for some parallel vector field $E$ along $t$,
	define $v := (\sprod{W - A}E - w)/t$ and obtain
	\[
		\ddt (\dot v t^2) =
		\ddt \big((\sprod{\dot W - \dot A} E - \dot w)t - \sprod{W - A}E + w\big)
		= (\sprod{\ddot W - \ddot A}E - \ddot w)t \leq 0.
	\]
	This means that $\dot v t^2 \leq 0$, hence
	$\dot v \leq 0$. Now $\sprod{W - A}E - v$ has a double root at
	$t = 0$, so $v(0) = 0$ and thus $v \leq 0$ everywhere. And
	because $E$ was arbitrary, this already means
	\[
		\absval{W - A} \leq w.
	\]
	Therefore we are done if we can bound $u$ by the
	right-hand side of the proposition.
	But as we know that $\absval W \simleq
	\absval{\nabla_U V} + at$, we can simply
	integrate $\ddot w = C_0 \absval Z^2 \absval W$ twice and obtain
	the desired estimate (the argument is
	the same as in the proof of
	\ref{prop:estimateSecondOrderODE}).

	\textit{ad primum:}
	The final estimate for $\absval W$
	comes from $\absval{\int PF} \leq \int \absval F$,
	together with $\absval{F(t)} \simleq \absval{F(0)} \leq a$
	for $t \leq 1$, because
	the same holds for $U^t$ and $V^t$, and the norm
	of $\nabla^2 Z|_{\Phi_t(p)}$ is the same as the norm
	of $\nabla^2 Z|_p$ because $Z$ is parallel along
	$t \mapsto \Phi_t(p)$.

	\textit{ad sec.:}
	For the estimate of $\absval{\tang W}$ let $\tang^t$
	be the orthogonal projection onto the tangent space of $\Phi_t(S)$
	and consider
	\[
		\tang^t W(t) - \Int_0^t \tang^t P^{t,\tau} F(\tau) \d \tau
			= \tang^t W(t) - \Int_0^t P^{t,0} \tang^t P^{0,\tau} F(\tau)
			+ (\tang^t P^{t,\tau} - P^{t,0} \tang P^{0,\tau}) F(\tau) \d \tau.
	\]
	Then $\absval{\tang P F} \simleq \absval U \absval V \norm{\tang \nabla^2 Z}$,
	and the projection difference is estimated by $t \norm{\nabla Z}
	+ C_0 t^2 \absval Z^2$ in \ref{prop:estimateProjectionDifference},
\end{proof}
\begin{theorem}
	\label{prop:comparisonygAndge}
	\begin{subeqns}
	Let $y: r\complex \to S$ be the triangulation of a smooth submanifold
	$S \subset M$ with Weingarten maps $W_\nu$ bounded by
	$\norm{W_\nu} + h \norm{\nabla W_\nu} \leq \kappa$
	and $x: r\complex \to S'$ an extrinsic Karcher triangulation
	with the same vertices $p_i$ and $y = \psi(x)$. Suppose
	$g^e$ is a $(\theta,h)$-small piecewise flat metric on
	$r\complex$ induced by edge lengths $\dist_{S'}(p_i,p_j)
	= \dist(p_i,p_j)$. Then for small $h$, it holds
	$\dist(x,y) \simleq h^2 \theta\inv \inorm[\Leb^\infty]{\nabla dx - P \nabla dy}$
	and
	\begin{align*}
		\absval{(y^*g - g^e)\sprod vw} & \simleq (\kappa^2 \dist + C_0 \dist^2) \absval v \absval w,\\
		\absval{\nabla^{y^*g}_v w - \nabla^{g^e}_v w} & \simleq \kappa h \theta^{-2}
			\inorm[\Leb^\infty]{\nabla dx - P \nabla dy} \absval v \absval w + ho.,
	\end{align*}
	where ``ho.'' stands for higher-order terms
	whose coefficients depend on $C_0$, $\kappa$,
	$\theta$, $\absval v$, $\absval w$ and $\absval{\nabla^{g^e}_v w}$.
	The norm on the left-hand side may be induced
	by either $x^*g$, $y^*g$, or $g^e$, because all three are
	equivalent.
	\end{subeqns}
\end{theorem}
\begin{proof}
	By the estimate \ref{prop:estimateDistancexAndy} for $\dist(x,y)$,
	$S'$ is a normal graph over $S$ for small $h$ with
	$\absval Z = \dist(x,y) \simleq h^2 \theta\inv \inorm{\nabla dx - P\nabla dy}$.
	Morally, it is clear that $\norm{\nabla Z}$ must be controlled by
	$\inorm{\nabla dx - P\nabla dy}$, too. In fact, we can precisely
	compute this for $V = dy\,v$:
	\[
		\nabla_V Z = \nabla_{dy\,v}(-X_x|_y) = \dot J(1)
	\]
	for a Jacobi field along $x \leadsto y$ with $J(0) = dx\,v$ and
	$J(1) = dy\,v$ by combining \ref{prop:derivativeOfX} and
	\ref{prop:derivativeXwrtoBasePoints} (we postpone the computation
	to the next section because it is more relevant there), so by
	\eqnref{eqn:JacobiFieldsUsualEstimate}
	$\absval{\nabla_V Z - (dx - Pdy)v} \simleq C_0 \dist^2(x,y)$,
	hence \ref{prop:estimateDistancexAndyH1} gives
	\[
		\norm{\nabla Z} \simleq h \theta\inv \inorm{\nabla dx - P \nabla dy} + C_0 \dist^2.
	\]
	Now because $x = \Phi \circ y$ and hence $dx = d\Phi dy$,
	\ref{prop:comparisonPhigWithg} gives
	\[
		\absval{(x^*g - y^*g)\sprod vw} \simleq (\kappa^2 \dist + C_0 \dist^2)
			\absval[y^*g] v \absval[y^*g] w.
	\]
	The comparison of $x^*g$ and $g^e$ is done in
	\ref{prop:comparisongandge}.---Analogously, we only compare
	$\nabla^{y^*g}$ to $\nabla^{x^*g}$ and refer to
	\ref{prop:estimateOfChristoffelOperator} for the rest:
	For a vector $v$ and a vector field $w$,
	\ref{prop:weakEstimateForPullBackConnectionDifference} gives
	(together with \ref{rem:tangPartOfNablaZ})
	\[
		\absval{\nabla^{S'}_{dx\,v}dx\,w - d\Phi \nabla^S_{dy\,v}dy\,w}
			\simleq \kappa h^2 \theta\inv \inorm{\nabla dx - P \nabla dy} + ho.
	\]
	By definition of the pull-back connection,
	$\nabla^{S'}_{dx\,v}dx\,w = dx \nabla^{x^*g}_v w$ and thus
	\[
		\absval{d\Phi dy(\nabla^{x^*g}_v w - \nabla^{y^*g}_v w)}
			\simleq \kappa h^2 \theta\inv \inorm{\nabla dx - P \nabla dy} + ho.
	\]
	Together with $\norm{d\Phi\inv} \simleq 1$ and $\norm{dy\inv} \simleq
	1/h\theta$, this shows the claim,
\end{proof}

\subsection{The Weak Shape Operator}

\begin{lemma}
	\label{prop:divergenceTangentialAndNormal}
	Let $S \subset M$ be a smooth submanifold
	with boundary $\Rand S$ in $M$.
	Let $U$ be a smooth vector field on $S$,
	not neccessarily tangential to $S$. Then $U$
	may be extended to a vector field on $M$
	in such a way that $\Div^M \tang U = \Div^S \tang U$
	and $\Div^M \nor U = - \sprod U H$,
	where $H = \nabla^M_{e_i}e_i$ for
	any orthonormal basis $e_i$ of $T_p S$ is the
	\begriff{mean curvature vector} of $S$.
\end{lemma}
\begin{proof}
	It suffices to find a local extension of $U$
	to some small neighbourhood of $S$.
	Let $e_1,\dots,e_n,\nu_{n+1},\dots,\nu_m$
	be an orthonormal basis of $T_pM$.
	Then $\Div^M U = \sprod{\nabla^M_{e_i} U}{e_i}
	+ \sprod{\nabla^M_{\nu_j} U}{\nu_j}$.
	If $U$ is extended parallel in normal
	direction, the latter term vanishes.
	
	Regarding the tangential part,
	observe $\sprod{\nabla^M_{e_i} \tang U}{e_i}
	= \sprod{\tang \nabla^M_{e_i} \tang U}{e_i}$,
	and because $\tang \nabla^M = \nabla^S$, this
	is $\Div^S \tang U$.
	
	Now consider $\nor U = \alpha^j\nu_j$. Again,
	if $U$ is constant in normal direction,
	$\Div^M \nor U = \sprod{(\partial_i \alpha^j) \nu_j}{e_i}
	+ \sprod{\alpha^j \nabla^M_{e_i}\nu_j}{e_i}$,
	the former term vanishes, the latter one is
	$\alpha^j \sprod{\nabla^M_{e_i}\nu_j}{e_i}
	= - \alpha^j \sprod{\nabla^M_{e_i}e_i}{\nu_j}$,
\end{proof}
\begin{lemma}
	\label{prop:relationHandWeakMeanCurvature}
	Let $S \subset M$ be a smooth submanifold
	with boundary $\Rand S$ in $M$.
	Then for smooth vector fields $V$ and $W$ on $M$,
	\[
		\Int_S \sprod W \nu \Div^M V + \sprod W \nu \sprod VH
		+ \sprod{\nabla^M_V \nu} W + \sprod \nu {\nabla^M_V W}
		= \Int_{\Rand S} \sprod W \nu \sprod V \tau,
	\]
	where $\tau$ is the outer normal of $\Rand S$ in $S$.
\end{lemma}
\begin{proof}
	Let $f := \sprod W \nu$. By the
	divergence theorem (\citealt[thm. 14.23]{Lee03}), we have
	\[
		\Int_S \Div^S(fV) = \Int_{\Rand S} \sprod{fV}\tau.
	\]
	Now by product rule and \ref{prop:divergenceTangentialAndNormal},
	$\Div^S(fV) = f\Div^S V + Vf = f\Div^M V - f \sprod VH
	+ Vf$,
\end{proof}
\begin{corollary}
	Let $S \subset M$ be a smooth submanifold without
	boundary. The
	operators $s_\nu, \sigma_\nu: \VF(TM|_S) \times \VF(TM|_S)
	\to \R$, defined by
	\[
		s_\nu(V,W) := \Int_S \sprod{\nabla^M_V \nu}W
				+ \sprod W \nu \sprod VH,
		\quad
		\sigma_\nu(V,W) := -
		\Int_S \sprod W \nu \Div^M V
				+ \sprod \nu {\nabla^M_V W}
		\vspace{-.5ex}
	\]
	coincide for smooth $S$ and each normal field $\nu$.
	On tangential vector fields, $s_\nu(V,W) = \sigma_\nu(V,W)
	= -\int \secondFF_\nu(V,W)$.
	If $S$ were only piecewise
	smooth, $\sigma_\nu$ would still be well-defined.
	It is called the \begriff{weak shape operator} of $S$.
\end{corollary}
\begin{proposition}
	Situation as in \ref{sit:normalGraph}.
	Then there are normal fields $\nu$ for $S$ 
	and $\nu'$ for $S'$ such that $\sigma_\nu$ approximates the weak
	shape operator $\sigma'_{\nu'}$ of $S'$, which is
	\[
		\sigma'_{\nu'}(V,W) = -\Int_{S'} \sprod W{\nu'} \Div^M V
				+ \sprod{\nu'}{\nabla^M_V W},
		\vspace{-.5ex}
	\]
	up to first order:
	$\absval{(\sigma_\nu(V,W) - \sigma'_{\nu'}(PV,PW)}
	\simleq \eps \ibetrag[\Sob^1(TM|_S)] V \ibetrag[\Sob^1(TM|_S)] W$.
\end{proposition}
\begin{proof}
	The derivatives of $V$ and $W$ are not
	distorted at all,
	$P\nabla_V W = \nabla_{PV} PW$, see
	\ref{parag:fermiCoordinates}, hence
	$\Div^M PV = \Div^M V$.
	According to \ref{prop:estimateNuMinusNut},
	the normals $\nu$ and $\nu'$ can be chosen
	such that they do not differ by more than $\eps$,
\end{proof}
\bibrembegin
\begin{remark_nn}
	This is our analogue of the weak shape
	operator convergence result from
	\citet[thm. 8]{Hildebrandt11} (similar,
	but more extensive, \citealt[thm. 2.4]{Hildebrandt12}).
	However, \citename{Hildebrandt} defines $\sigma'_{\nu'}$
	with $\Div^{S'}$ instead of $\Div^M$, which
	did not meet our own calculations.
\end{remark_nn}
\bibremend

%
\newsectionpage
\section{Variation of Karcher Simplex Volume}
\label{sec:VariationKarcherTriangs}
%

\begin{notation}
	Different from other sections, we will mostly write $\vol S$
	instead of $\absval S$ for the volume of some set $S \subset M$,
	but where we find it meaningful, we will mix both notations.
\end{notation}
%
\begin{goal}
	Consider a full-dimensional Karcher simplex $x(r\simplexe)$
	with vertices $p_0,\dots,p_m$. If $p_i$ is moved with
	velocity $\dot p_i = dx\,u_i$, leading to a family
	$x_t$ of barycentric mappings, we would like to compute
	the derivative $\ddt|_{t = 0} \vol_g(x_t(r\simplexs))$
	of a subsimplex volume. In a second step, we want to show
	that this derivative is close to $\ddt|_{t = 0}
	\vol_{g^e}(r\simplexs_t)$, where the vertices $e_i$ of
	$r\simplexs$ are moved with velocity $u_i$.
\end{goal}
\begin{fact}[see e.\,g. \citealt{Jost11}, eqn. 4.8.3]
	\label{prop:smoothVariationOfArea}
	Let $\Phi_t: S \to S^t$
	be a normal variation of the smooth submanifold
	$S \subset M$ with a velocity $Z = \partial_t \Phi_t$
	that has compact support. If $H$ is the mean curvature
	vector of $S$, then
	\[
		\left.{\textstyle \ddt}\right|_{t = 0} \vol{S^t}
			= - \Int_S \Div^S Z
			\halfquad \overset{(\ref{prop:relationHandWeakMeanCurvature})} =
			\Int_S \sprod Z H,
	\]
	where the last equality only holds if $Z$ vanishes at the
	boundary of $S$, whereas the divergence formula is
	also correct for velocities with support up to the boundary.
\end{fact}
\begin{situation}
	\label{sit:variationOfVertex}
	Let $\complex$ be a regular $m$-dimension simplicial
	complex, $x: r\complex \to M$ be a Karcher
	triangulation with respect to vertices
	$p_i \in M$ such that
	$d_{ri}x_\simplexe$ is bijective for every vertex of
	every element and $x$ induces metrics
	$x^*g$ and $g^e$ as usual, where $g^e$
	is $(\theta,h)$-small, fulfilling
	\ref{prop:comparisongandge} and
	\ref{prop:estimateOfChristoffelOperator}.	
	Let $\bar\complex$ be an $n$-dimensional
	subcomplex of $\complex$ such that $S := x(r\bar\complex)$
	is an $n$-dimensional submanifold of $M$.
\end{situation}

\subsection{Variation of Karcher Triangulations}

\begin{lemma}
	\label{prop:derivativeXwrtoBasePoints}
	Let $p \in M$ be some point and $X$ be its half
	squared distance gradient from \ref{rem:definitionOfX}.
	Recall from \ref{prop:derivativeOfX} that
	if the evaluation point $a$ is moved on a
	curve $a(r)$, then $\nabla_{\dot a(r)} X
	= \sigma \dot J_r(\sigma)$, where $J_r$ is a Jacobi
	field along the geodesic $p \leadsto a(r)$
	with $J_r(0) = 0$ and $J_r(\sigma) = \dot a(r)$.
	
	Now if $p$ moves on a curve $p(t)$, this
	induces a new variation vector field $\dot X_{p(t)} := D_t X_{p(t)}$.
	It satisfies $\dot X_{p(t)}|_a = \sigma \dot J_t(\sigma)$,
	where $J_t$ is a Jacobi field along $p(t) \leadsto a$
	with boundary values $J_t(0) = \dot p(t)$ and $J_t(\sigma) = 0$.
	Combining both variations, we obtain
	$\nabla_{\dot a(r)} \dot X_{p(t)}
	= \sigma D_r \dot J_r(\sigma)
	= \sigma D_t \dot J_r(\sigma)$
\end{lemma}
\begin{remark_nn}
	\begin{subenum}
	\item
	Note that $\dot J$ and $\dot X$ denote derivatives
	with respect to different directions. In
	this section, the parameter $r$ takes the rôle
	of $t$ in section \ref{sec:DistanceFunction}.
	\bibrembegin
	\item
	In the notation of \cite{Grohs13},
	$\dot X = \nabla_1 \log(a,p)$.
	\bibremend
	\end{subenum}
\end{remark_nn}
\begin{proof}
	\textit{ad primum:}
	Without restriction, we can assume that $p(t)$
	and $a(r)$ describe geodesics, as no second derivatives of them
	enter.
	Analogous to \ref{prop:derivativeOfX}, consider
	a variation of geodesics
	\[
		c(r,s,t) := \exp_{p(t)} \big(\smallfrac s \sigma (\exp_{p(t)})\inv a(r) \big).
	\]
	As in \ref{prop:derivativeOfX}, one can see that
	\[
		X_{p(t)}|_{c(r,s,t)} = \smallfrac s \sigma \partial_s c(r,s,t),
		\qquad
		J_r = \partial_r,
		\qquad
		J_t = \partial_t
	\]
	because $\partial_r$ and $\partial_t$ are Jacobi fields with
	the desired values at $s = 0$ and $s = \sigma$.
	
	\textit{ad sec.:} We have $\smallfrac 1 \sigma \nabla_{\dot a(r)} \dot X_{p(t)}
	= D_r D_t \partial_s c(r,s,\sigma)
	= D_t D_r \partial_s c(r,s,\sigma) + R(\partial_r,\partial_t)\partial_s$,
	and because $c(r,s,\sigma) = a(r)$ is independent of $t$,
	we have $\partial_t = 0$ there. And the former term is
	$D_t D_s \partial_r c(r,s,\sigma) = D_t \dot J_r(\sigma)$,
\end{proof}
\begin{lemma}
	\label{prop:estimatesOnDrJs}
	Situation as before. Abbreviate $\ell := \dist(a(r),p(t))$,
	$V := \dot a(r)$, $W := \dot p(t)$ and $T:= \partial_s$.
	Then if $C_0 \ell^2 < \frac{\pi^2}4$,
	\[
		\begin{split}
		\absval{D_t J_r} & \leq 90 C_{0,1}(r) \smallfrac{s(\sigma - s)}\sigma\,
			\absval V \, \absval W \, \absval T(r) \\
			& \leq 90 C_{0,1}(r) \, s \, \absval V \, \absval W \, \absval T(r),
		\end{split}
		\qquad
		\absval{D_t \dot J_r} \leq 50 C_{0,1}(r) \, \absval V \, \absval W \, \absval T(r).
	\]
\end{lemma}
\begin{proof}
	The claim is similar to \ref{thm:estimatesOnDsJ},
	and so is the proof: We again have to
	find a differential equation for $U := D_t J_r$ to
	apply \ref{prop:estimateSecondOrderODE}.
	By the usual laws of covariant differentiation
	and $\partial_r = J_r$, $\partial_t = J_t$, we have
	\[
		\begin{split}
		D_t \ddot J_r = D_t D_s D_s \partial_r
					& = D_s D_t D_s \partial_r + R(J_t, T)\dot J_r \\
					& = D_s D_s D_t \partial_r + D_s R(J_t,T)J_r + R(J_t, T)\dot J_r \\
					& = D^2_{ss} U + (D_s R)(J_t,T)J_r + R(\dot J_t,T)J_r + 2 R(J_t,T)\dot J_r.
		\end{split}
	\]
	On the other hand, using the Jacobi equation,
	\[
		- D_t \ddot J_r = D_t R(J_r,T)T = (D_t R)(J_r,T)T + R(U,T)T + R(J_r,\dot J_t)T + R(J_r,T)\dot J_t.
	\]
	So with $A := -R(\argdot,T)T$, we have $\ddot U = AU + B$,
	where $\norm A \leq C_0 \absval T^2$. Let us assume that we consider
	some geodesic with $\absval T = 1$ (as usual, the correct power
	of $\absval T$ follows from a scaling argument). Then,
	because \ref{thm:estimatesOnDtJ} holds for $J_r$ as well as for
	$J_t$,
	\[
		\begin{split}
		\absval B   & = |(D_s R)(J_t,T)J_r + (D_t R)(J_r,T)T \\
					& \qquad + R(\dot J_t,T)J_r + 2 R(J_t,T)\dot J_r + R(J_r,\dot J_t)T + R(J_r,T)\dot J_t | \\
					& \leq 2 C_1 \absval V\,\absval W + 15 C_0 \smallfrac 1 \sigma \absval V\,\absval W.
		\end{split}
	\]
	The claim on $D_t \dot J_r$ follows from
	$D_t \dot J_r = D_t D_s \partial_r = D_s D_t \partial_r + R(\partial_t,\partial_r)\partial_s$,
\end{proof}
\begin{lemma}
	\label{prop:derivativexwrtoVertexVariation}
	Notation as before, and $A_\lambda$ as in 
	\ref{prop:derivativesOfF}.
	For varying vertices $p_i(t) \in M$,
	let $x_t$ be the corresponding Karcher triangulations,
	and $V = dx\,v$. Then
	\[
		A_\lambda \dot x_t|_\lambda = - \lambda^i
		\dot X_i|_{x_t(\lambda)},
		\qquad
		A_\lambda \nabla_V \dot x = - v^i \dot X_i -
		\lambda^i \nabla^2_{V,\dot x} X_i - \lambda^i \nabla_V \dot X_i - A_v \dot x.
	\]
\end{lemma}
\begin{proof}
	\begin{subeqns}
	\textit{ad primum:}
	Along $t \mapsto x_t(\lambda)$, consider
	the vector field $U(t) := \lambda^i X_i(t)|_{x_t(\lambda)}$.
	As has been stated in \ref{rem:definitionOfX},
	this vector field vanishes for all $r$,
	so $\dot U = 0$. For those who believe,
	the shorthand proof is
	\begin{equation}
		\label{eqn:derivativesOfXwrtoVertexVariation}
		\begin{split}
		0 =
		D_t|_{t = 0} (\lambda^i X_i(t)|_{x_t})
			& = \lambda^i D_t|_{t = 0} (X_i(t)|_{x_0}) + \lambda^i D_t|_{t = 0} (X_i(0)|_{x_t}) \\
			& = \lambda^i(\dot X_i + \nabla_{\dot x_0} X_i)|_{x_0},
		\end{split}
	\end{equation}
	a rather uncommon use of the chain rule.
	For all others, this argument is justified
	by a calculation in coordinates: Let $\lambda^i X_i = U^j \partial_j$
	and $\dot x(t) = \dot x^j(t) \partial_j$ be the needed coordinate representation.
	As $U^k(t) = 0$ for all $k$ and all $t$, we have $\dot U^k(t) = 0$,
	and this is
	by usual Euclidean chain rule
	$\partial_t U^k(t,x^1(t),\dots, x^n(t)) + \partial_\ell U^k(t,x^1(t),\dots,x^n(t))
	\dot x^\ell(t)$. And if all $U^k$ vanish, we can add
	$\Chris \ell jk U^j \dot x^\ell$ without harm, which gives
	\[
		0 = (\partial_t U^k
		  +  \partial_\ell U^k \dot x^\ell + \Chris \ell jk U^j \dot x^\ell)\partial_k
		= D_t U + \nabla_{\dot x} U.
	\]
	\textit{ad sec.:}
	Differentiating \eqnref{eqn:derivativesOfXwrtoVertexVariation}
	once more leads to
	\[
		0 = \lambda^i \nabla_V \nabla_{\dot x} X_i + \nabla_V (\lambda^i \dot X_i)
		  = \lambda^i \nabla^2_{V,\dot x} X_i + \lambda^i \nabla_{\nabla_V \dot x} X_i
		  + v^i \dot X_i + \lambda^i \nabla_V \dot X_i + v^i \nabla_{\dot x} X_i,
	\]
	\end{subeqns}
\end{proof}
\begin{proposition}
	\label{prop:estimateDotxMinusdxu}
	Situation as in \ref{sit:variationOfVertex},
	and let the variation of $p_i$ be given by $\dot p_i(t)
	= dx \, w_i$ for some vector $w_i \in Tr\simplexe$. Then
	for $u := \lambda^i w_i$, we have
	\[
		\absval{\dot x - dx\,u} \simleq C_{0,1}' h \absval{\dot x},
		\qquad
		\absval{\nabla_{dx\,v} \dot x - dx\nabla^{g^e}_v u}
		\simleq C_{0,1}' h \absval u \, \absval{\dot x}.
	\]
\end{proposition}
\begin{proof}
	\textit{ad primum:}
	At the vertex $e_i$ of the standard simplex, $d_{e_i}x\, w$ and the variation
	$\dot p_i$ agree. At another point $\lambda \in \stds$,
	\[
		d_\lambda x\, w_i = Pd_{e_i}x\, w_i + O(C_{0,1}' h^2 \absval{w_i})
			= P\dot p_i + O(C_{0,1}' h^2 \absval{w_i})
	\]
	by \ref{prop:comparisondxvAtDifferentPoints}, which means
	$dx\,u = \lambda^i P\dot p_i + O(C_{0,1}' h \absval u)$
	by definition of $u$, and
	\begin{align*}
		\dot x
			& = - \lambda^i \dot X_i + O(C_0' h^2 \absval{\dot x})
					&& \text{by \ref{prop:derivativexwrtoVertexVariation}} \\
			& = \,\, \lambda^i P \dot p_i + O(C_0' h^2 \absval{\dot x})
					&& \text{by \ref{prop:derivativeXwrtoBasePoints}
						and \ref{thm:estimatesOnDtJ}}.
	\end{align*}
	\textit{ad sec.:}
	The derivative of $u$ is, by Euclidean
	calculus, just $\nabla^{g^e}_v u = v^i w_i$, so the latter
	term is $dx\nabla^{g^e}_v u = v^i dx\,w_i$.
	The covariant derivative $\nabla^2 X_i$ in
	\ref{prop:derivativexwrtoVertexVariation}
	is estimated by \ref{corol:estimateForNablaX}, and the
	$\nabla_V \dot X$ term by \ref{prop:estimatesOnDrJs},
	so we have
	\[
		\begin{split}
		\absval{A_\lambda \nabla_{dx\,v} \dot x - v^i dx\,w_i}
		& \leq v^i \absval{\dot X_i + dx\,w_i} + \lambda^i \absval{\nabla^2_{V,\dot x} X_i}
			+ \lambda^i \absval{\nabla_V \dot X_i} + \absval{A_v \dot x} \\
		& \simleq C_0 h^2 v^i \absval{\dot p_i} + C_{0,1}' h \absval v \, \absval{\dot x},
		\end{split}
	\]
\end{proof}

\subsection{Discrete Vector Fields}

\begin{definition}
	\label{def:vertexTangentSpaces}
	For a piecewise barycentric mapping $x: r\complex \to M$, let
	\[
		\bar \Pol\VF_x := TM|_{x(\complex^0)} =
		\bigsqcup_{\mathclap{p_i, i \in \complex^0}} T_{p_i} M
	\]
	be the disjoint union of all vertex tangent spaces.
	For $\bar U = (U_i) \in \bar\Pol\VF_x$, define a piecewise interpolation:
	It induces a variation of $x$ by defining $x_t[\bar U]$ to
	be the piecewise barycentric mapping with respect
	to vertices $\exp_{p_i} tU_i$ (where we keep $t$ so small
	that $x(\lambda)$ and $x_t(\lambda)$ stay in a convex neighbourhood
	of each other). We call
	$\bar U \mapsto U := \dot x_t[\bar U]|_{t = 0}$ the
	\begriff{$\Pol^1$-interpolation} of $\bar U$ and
	\[
		\Pol^1\VF_x := \{U \mit \bar U \in \bar\Pol\VF_x \}
	\]
	the set of piecewise smooth, globally
	continuous test vector fields. 
\end{definition}
\begin{observation}
	As a finite sum of vector spaces with scalar products $g$ and $g^e$,
	$\bar \Pol\VF_x$ carries the natural inner products
	\[
		\ell^2g \dprod{\bar V}{\bar W} = \sum_i g\sprod{V_i}{W_i},
		\qquad
		\ell^2g^e\dprod{\bar V}{\bar W} = \sum_i g^e\sprod{v_{\simplexe,i}}{w_{\simplexe,i}},
	\]
	whereas $\Pol^1\VF_x$ has the scalar products $\Leb^2g$ and $\Leb^2g^e$
	that are induced from $\Leb^2\VF(TM)$:
	\[
		\Leb^2g\dprod VW = \Int_{x(r\complex)} g\sprod VW,
		\qquad
		\Leb^2g^e\dprod VW
			= \sum_{\simplexe \in \complex^n} \int_{r\simplexe}
				g^e\sprod{\bar v}{\bar w},
	\]
	where $V = dx\,\bar v$ and $W = dx\,\bar w$.
	As both are isomorphic finite-dimensional vector spaces, all these
	norms are equivalent. The equivalence constants between
	$\ell^2g$ and $\ell^2g^e$ and between $\Leb^2g$ and $\Leb^2g^e$
	are the ones from \eqnref{eqn:comparisongge} and
	\eqnref{eqn:estimateL1normSecondOrder}, whereas the equivalence
	constant between $\Leb^2g^e$ and $\ell^2g^e$ depends on the
	maximal and minimal simplex volume.
\end{observation}
\begin{definition}
	Situation as in \ref{sit:variationOfVertex}, and
	$\bar U \in \bar\Pol\VF_x$.
	Inside every simplex $\simplexe \in \complex^m$,
	the vector $U_i$, $i \in \simplexe$, can be
	represented as $U_i = d_{ri}x_\simplexe\,w_i^\simplexe$
	(without any summation). Define a piecewise linear,
	globally discontinuous
	vector field $u|_{r\simplexe} := \lambda^i w_i^\simplexe$
	and a piecewise smooth vector field $\bar u$ by
	requiring $dx \, \bar u = U$ everywhere.
\end{definition}
\begin{conclusion}
	\label{prop:inducedTriangVariationVersusLinearInterpolation}
	By definition, $dx \nabla^{x^*g}_v \bar u = \nabla_{dx\,v} \dot x_t[\bar U]$.
	From \ref{prop:estimateDotxMinusdxu}, we hence know that
	\[
		\absval{u - \bar u} \simleq C_{0,1}' h \absval u,
		\qquad
		\absval{\nabla^{x^*g}_v \bar u - \nabla^{g^e}_v u}
		\simleq C_{0,1}' h \absval v \absval u,
	\]
	where all norms are $\absval[x^*g] \argdot$ norms.
	The same estimates hold for the jump $[dx\,u]_\simplexf
	= dx_{\simplexe} \lambda^i w_i^\simplexe
	- dx_{\simplexe'} \lambda^i w_i^{\simplexe'}$ of
	$u$ across a facet $\simplexf = \simplexe \cap \simplexe'$.
\end{conclusion}

\subsection{Area Differentials}

\begin{observation}
	\label{obs:volumeDifferentialPiecewise}
	\begin{subeqns}
	Situation as in \ref{sit:variationOfVertex}.
	If the vertices $p_i$ of a Karcher
	triangulation vary smoothly with velocity
	$(U_i) \in \bar \Pol\VF_x$, the area change of
	$S = x(r\bar \complex)$
	is also smooth, hence has a differential
	\begin{equation}
		\label{eqn:defdvolN}
		d\vol_{\bar \complex}: \bar\Pol\VF_x \to \R.
	\end{equation}
	The volume is additive in $\simplexs \in \bar \complex^n$,
	and the variations of different vertices are
	linearly independent, so it
	suffices to compute the differential $d\vol^i_{\simplexs,g}:
	T_{ri}r\complex \to \R$ of
	$\absval[g]{x(r\simplexs)}$ with respect to the variation
	of $x(ri)$, $i \in \simplexs$. Correspondingly, let
	$d\vol^i_{\simplexs,g^e}$ be the analogous differential of
	$\absval[g^e]\simplexs$.
	\end{subeqns}
\end{observation}
\begin{remark_nn}
	We do not think that a notational distinction between this
	area differential and the volume form
	from \ref{def:PolOmegak} is neccessary. For readers who
	disagree, we remark that in \eqnref{eqn:defdvolN}, the
	$d$ denotes a differential and is hence written in
	italics, whereas it is upright as part of the volume form
	$\dvol$.
\end{remark_nn}
\begin{proposition}
	\label{prop:estimateAreaDifferentialOneSimplex}
	Situation as in \ref{sit:variationOfVertex}.
	Then $\absval{d\vol^i_{\simplexs,g} - d\vol^i_{\simplexs,g^e}}
	\simleq C_{0,1}' h \betrag[g^e]{\simplexs}$.
\end{proposition}
\begin{proof}
	By \ref{prop:smoothVariationOfArea}, $d\vol^i_{\simplexs,g}(w)
	= - \int \Div^S Z$, where $Z = \dot x$ is induced by the
	vertex variation $\dot p_i = dx\,w$. 
	If $\tilde v_j$
	form a $g$-orthogonal basis of $Tr\simplexs$, this is
	$\int \sprod{\nabla_{dx\,\tilde v_j} \dot x}{dx\,\tilde v_j}$.
	By the comparison of volume elements for $g$ and $g^e$
	in \ref{prop:comparisonEuclidSimplexVolumeForms},
	\[
		d\vol^j_{\simplexs,g}(w) = - \Int_{\mathclap{x(r\simplexs),g}} g\sprod{\nabla_{dx\,\tilde v_j} \dot x}{dx\,\tilde v_j}
			= -(1 + O(C_0' h^2)) \Int_{\mathclap{x(r\simplexs),g^e}} g\sprod{\nabla_{dx\,\tilde v_j} \dot x}{dx\,\tilde v_j},
	\]
	and, noting that there is a $g^e$-orthonormal basis $v_j$ of $Tr\simplexs$
	with $\absval{v_j - \tilde v_j} \simleq C_0' h^2$
	by \ref{prop:estimateEntriesOfOrthoBasis}, the integrand is
	\begin{align*}
		g\sprod{\nabla^g_{dx\,\tilde v_j} \dot x}{dx\,\tilde v_j}
			& = x^*g \sprod{\nabla^{g^e}_{\tilde v_j} u}{\tilde v_j} + O(C_{0,1}' h \absval u)
				&& \text{by \ref{prop:estimateDotxMinusdxu}} \\
			& = x^*g \sprod{\nabla^{g^e}_{v_j} u}{v_j} + O(C_{0,1}' h \absval u) + O(C_0' h^2 \norm{\nabla u}) \\
			& = g^e\sprod{\nabla^{g^e}_{v_j} u}{v_j} + O(C_{0,1}' h \absval u) + O(C_0' h^2 \norm{\nabla u}),
	\end{align*}
	and the last right-hand-side term is $\Div^{(r\simplexs,g^e)} u$,
\end{proof}
\begin{remark}
	\begin{subenum}
	\item
	It is common knowledge that $d\vol^i_\stds = d\lambda^i \absval \stds$,
	proven by inserting $\Div u = d\lambda^i(w)$ for $u = \lambda^i w$
	into \ref{prop:smoothVariationOfArea}. Classically, one says
	for triangles (\citealt[eqn. 4.3]{Polthier02}) that the gradient of the
	area functional with respect to vertex variations
	is the $\frac \pi 2$ rotation of the opposite
	edge vector, which is exactly $(d\lambda^i)^\sharp$.
	
	One has to take care to transfer this to the subsimplex situation,
	because $\Div^{(r\simplexs,g^e)} u$ $= 0$ if $w$ is perpendicular to $r\simplexs$,
	so one needs a form $d\lambda^i_\simplexs$ that acts
	like $d\lambda^i$ on $Tr\simplexs$ and vanishes on $Tr\simplexs^\perp$. For
	example, if $r\simplexs = \conv(e_0,\dots,e_n)$ and $v_0 =
	\grad \lambda^0 \in T\stds$,
	\[
		d\lambda^i_\simplexs(w) = d\lambda^i
				\Big(w - \frac{\sprod{w_0}w}{\absval{v_0}^2} v_0\Big)
			= w^i - \smallfrac{w^0}{\absval{d\lambda^0}^2}
				\sprod{d\lambda^i}{d\lambda^0}.
	\]
	It is easier to transfer the gradient $v^i_\simplexs$ of $\lambda^i$ in $r\simplexs$
	to $Tr\simplexe$, as its vector components stay the same:
	$\Div^{(\simplexs,g^e)} u = g^e\sprod{v^i_\simplexs} w$. By
	\eqnref{eqn:defOuterNormalvi}, this $v^i_\simplexs$
	is characterised as the vector in $Tr\simplexs$
	that is perpendicular to the facet $\simplexs \setminus\{i\}$ opposite $i$ with
	lengths $h_i\inv$, the reciprocal of $r\simplexs$'s height above
	$\simplexs \setminus\{i\}$ from \eqnref{eqn:defOuterNormalvi}.
	\item
	For computational purposes, \ref{prop:estimateAreaDifferentialOneSimplex} is
	insatisfactory, as only $d\vol_{\simplexs,g^e}(u)$ is
	numerically accessible, not $d\vol_{\simplexs,g^e}(\bar u)$.
	But this is only an easy combination with
	\ref{prop:inducedTriangVariationVersusLinearInterpolation}, which will
	be spelled out for the area gradients in the following paragraphs.
	\end{subenum}
\end{remark}

\subsection{Area Gradients}

\begin{observation}
	The area differentials $d\vol_{\simplexs,g}$
	and $d\vol_{\simplexs,g^e}$ can be expressed as
	gradient with respect to different norms on
	$\bar \Pol\VF_x$.
	The gradients with respect to $\ell^2g$
	and $\ell^2g^e$ correspond to the
	``mean curvature vector'' of
	\cite{Polthier02}, whereas the gradients
	with respect to
	or $\Leb^2g$ and $\Leb^2g^e$ give the
	construction from \cite{Dziuk91}.
\end{observation}
\begin{corollary}
	If $H_{\ell^2g} \in T_{p_i} M$ is the gradient
	of $\betrag[g]{x(r\simplexs)}$ with respect to a variation of $p_i$, and
	$H_{\ell^2g^e} = \betrag[g^e]{\simplexs} \grad \lambda^i$ is the corresponding
	gradient of $\betrag[g^e]{\simplexs}$, then
	$\absval{dx \, H_{\ell^2g^e} - H_{\ell^2g}}$
	$\simleq C_{0,1}' h \absval{H_{\ell^2g^e}}$.
\end{corollary}
\begin{definition}
	The \begriff{discrete mean curvature vector} $H_{\Leb^2g}
	\in \Pol^1\VF_x$ is the solution of $\Leb^2g\dprod{H_{\Leb^2g}} V
	= \dvol_{\bar \complex,g}(V)$ for all $V \in \Pol^1\VF$. The approximate
	mean curvature vector $H_{\Leb^2g^e} \in \Pol^1\VF_x$ is
	the solution of $\Leb^2g^e\dprod{H_{\Leb^2g^e}} V = d\vol_{\bar \complex,g^e}(V)$
	for all $V \in \Pol^1\VF$.
\end{definition}
\begin{observation}
	\begin{subeqns}
	The differentials of the right-hand sides can be represented
	as $\Leb^2$ products of linear maps: If vectors $Z_i = dx\,w_i
	\in T_{p_i}M$ induce a variation $\dot x = dx\,\bar z$ of $x$,
	and if we define $z = \lambda^i w_i$,
	as well as $d\lambda: e_i \mapsto \grad \lambda^i$,
	then
	\begin{align}
		\Leb^2g\dprod{H_{\Leb^2g}} Z & = \dprod{dx}{\nabla^{x^*g} \bar z} _{\Leb^2g(Tr\complex \otimes x^*TM)},
			\label{eqn:L2gRepresentationOfAreaDifferential} \\
		\Leb^2g^e\dprod{H_{\Leb^2g^e}} Z & = \dprod{d\lambda}{\nabla^{g^e} z} _{\Leb^2g^e(Tr\complex \otimes x^*TM)}.
			\label{eqn:L2geRepresentationOfAreaDifferential}
	\end{align}
	To see \eqnref{eqn:L2gRepresentationOfAreaDifferential},
	we take an orthonormal basis $E_i = dx\,y_i$ in
	$\Div^S Z = \sprod{\nabla_{E_i} Z}{E_i}$.
	Then we have an integrand of the form
	$\sprod{\alpha(y_i)}{\beta(y_i)}$ for two linear
	maps $\alpha,\beta$. A computation in coordinates easily
	shows that this is $\sprod \alpha\beta$.
	
	For \eqnref{eqn:L2geRepresentationOfAreaDifferential},
	the computation is even simpler: $z = \lambda^iv_i$
	(without summation) for the gradient $v_i$ of $\lambda^i$
	has derivative $\nabla z = v_i\otimes d\lambda^i$, which maps $e_k$
	to $v_k$, so $\Div z = v^i_i = \sprod{\nabla_{e_i}z}{\grad \lambda^i}$.
	\end{subeqns}
\end{observation}
\begin{proposition}
	Situation as before. Then
	$\ibetrag[\Leb^2]{H_{\Leb^2g} - H_{\Leb^2g^e}}
	\simleq C_{0,1}' h (1 + h \ibetrag[\Leb^2]{H_{\Leb^2g^e}})$.
\end{proposition}
\begin{proof}
	Similar to \ref{prop:estimateDirichletProblem}: The functionals
	on the right-hand side of
	\eqnref{eqn:L2geRepresentationOfAreaDifferential}
	only differ by a factor of
	$1 + O(C_{0,1}' h \absval S)$, and
	the bilinear forms on the left fulfill
	$|\Leb^2g\sprod UV -  \Leb^2g^e \sprod UV|
	\simleq C_0' h^2 \ibetrag U\,\ibetrag V$, so
	\pagebreak[3]
	\[
		\begin{split}
		\ibetrag[\Leb^2g]{H_{\Leb^2g} - H_{\Leb^2g^e}}^2
			& = \Leb^2g\sprod{H_{\Leb^2g} - H_{\Leb^2g^e}}{H_{\Leb^2g} - H_{\Leb^2g^e}} \\
			& \leq \absval{\Leb^2g\sprod{H_{\Leb^2g}}{H_{\Leb^2g} - H_{\Leb^2g^e}}
					- \Leb^2g^e\sprod{H_{\Leb^2g^e}}{H_{\Leb^2g} - H_{\Leb^2g^e}}} \\
			& \qquad + \absval{(\Leb^2g - \Leb^2g^e)\sprod{H_{\Leb^2g^e}}{H_{\Leb^2g} - H_{\Leb^2g^e}}} \\
			& \simleq \absval{(d\vol_{\bar \complex,g} - d\vol_{\bar \complex,g^e})(H_{\Leb^2g} - H_{\Leb^2g^e})} \\
			& \qquad + C_0' h^2 \ibetrag{H_{\Leb^2g^e}} \, \ibetrag{H_{\Leb^2g} - H_{\Leb^2g^e}} \\
			& \simleq C_{0,1}' h \ibetrag{H_{\Leb^2g} - H_{\Leb^2g^e}}
				(1 + h \ibetrag{H_{\Leb^2g^e}}),\mathrlap{\hspace{17.9ex}\qedsymbol\qedhere} \\[-4ex]
		\end{split}
	\]
\end{proof}


\newsectionpage
\section{The Manifold-Valued Dirichlet Problem}
\label{sec:mfDirichletProblem}


\begin{goal}
	Let $N\gamma$ and $Mg$ be two smooth compact Riemannian manifolds
	of dimension $n$ and $m$.
	We have seen in \ref{prop:KarcherTriangIsTriang} that a
	sufficiently dense generic point set in $N$ gives us
	an almost-isometry $N \to r\complex$, and in
	\ref{prop:estimateDistancexAndyH1} that smooth functions
	$r\complex \to M$ can be interpolated by piecewise barycentric
	mappings $r\simplexs \to M$, $\simplexs \in \complex^n$,
	that are globally continuous. Now it is a natural attempt
	to consider Galerkin approximations e.\,g. to the
	Dirichlet problem in $\Sob^1(N,M)$: Suppose $N$ has a Lipschitz
	boundary that can be resolved by the triangulation.
	Given a smooth harmonic function
	$y: N \to M$, let $x: r\complex \to M$ be the
	function that minimises the Dirichlet energy among all
	functions that are piecewise barycentric mappings and agree
	with $y$ at the boundary vertices.
	How well is $y$ then approximated by $x$?
	
	A popular example
	for such functions $y$ is the minimal submanifold
	problem: If $S \subset M$ is a minimal submanifold,
	then the identity mapping $S \to M$ is harmonic
	(\citealt[eqn. 4.8.12]{Jost11}). One already sees that
	to keep the notation consistent among the chapters,
	we denote smooth harmonic functions etc. by $y$
	and the interpolation in the sense of \ref{sit:interpolationInManifolds}
	by $x$.
\end{goal}
\bibrembegin
\begin{remark_nn}
	The results presented in this section are generally
	the same as in \cite{Grohs13}, but although their
	interpolation procedure is the same as ours
	(but including higher-order interpolation, see
	\ref{rem:higherOrderInterpolSander}),
	their functional analytic approach is
	slightly different: They do not use the distance
	measure $\rho_1$ and its corresponding Poincaré
	inequality, but consider functionals that are
	(in a weak sense) convex along geodesic homotopies.
\end{remark_nn}
\bibremend

\subsection{The General Galerkin Approach}

\begin{definition}
	The \begriff{Dirichlet energy} of $y \in \Cont^1(N,M)$
	is
	\vspace{-0.5ex}
	\[
		\Dir(y) := \frac 1 2 \Int_N \absval{dy}^2.
		\vspace{-0.5ex}
	\]
	Recall that the norm on $TN \otimes y^*TM$ induced
	by $\gamma$ and $g$ has a representation
	in local coordinates $u^\alpha$ for $N$ and $v^i$ for $M$ as
	$\absval{dy}^2 = \gamma^{\alpha\beta} g_{ij} y^i_{,\alpha} y^j_{,\beta}$.
\end{definition}
\begin{proposition}[\begriff{Dirichlet principle}, \citealt{Jost11}, eqn. 8.1.13]
	\label{prop:DirichletPrincipleOnMfs}
	$y \in \Cont^1(N,M)$ is a critical point for $\Dir$
	with given boundary values
	iff $\dprod{dy}{\nabla V} = 0$ for all compactly supported vector
	fields $V \in \VF(y^*TM)$.
\end{proposition}
\begin{proof}
	Every compactly supported vector field $V$ along $y$ induces
	a variation $y_t := \exp_y(tV)$ of $y$ that does not
	change the boundary values. By the usual calculus
	of variations, $\ddt|_{t = 0} \Dir(y_t) = \dprod{dy}{D_t dy}$.
	So we only have to compute the $t$-derivative of
	$dy_t = y_{,\alpha}^i(t) \partial_i \otimes du^\alpha$
	with the notation from \ref{prop:derivsOfCoordVFs}. As
	$D_t = \nabla_V$ in the usual sloppy notation and
	$D_t(du^\alpha) = 0$, we have at the origin of normal
	coordinates in $N$
	\vspace{-0.5ex}
	\[
		\begin{split}
		D_t dy_t = D_t(y^i_{,\alpha} \partial_i) \otimes du^\alpha
			& = V^i_{,\alpha} \partial_i \otimes du^\alpha + y^i_{,\alpha} \nabla_V \partial_i \otimes du^\alpha \\
			& = V^i_{,\alpha} \partial_i \otimes du^\alpha + y^i_{,\alpha} V^j \Chris ijk \partial_k \otimes du^\alpha\\[-1.5ex]
		\end{split}
	\]
	On the other hand, regarding \eqnref{eqn:connectionInPulledBackBundle},
	\[
		\nabla V = \nabla_{\partial_\alpha} (V^i \partial_i) \otimes du^\alpha
			= V^i_{,\alpha} \partial_i \otimes du^\alpha + V^i y^j_{,\alpha} \Chris ijk \partial_k \otimes du^\alpha,
		\vspace{-3ex}
	\]
\end{proof}
\begin{definition}
	\begin{subenum}
	\item	For $N = r\complex$,
	let $\Pol^1(N,M)$ be the space of piecewise barycentric,
	globally continuous mappings (which obviously depends on
	the simplicial structure and not
	only on its manifold structure, but we do not explicitely denote
	this). Obviously the domain of $\Dir$
	can be extended to include also piecewise smooth mappings, so
	every $v \in \Pol^1(N,M)$ has finite Dirichlet energy.
	\item	For $a,b \in M$, denote $a \sim b$ if there is
	a unique shortest geodesic $a \leadsto b$ in $M$.
	Say that $x,y: N \to M$ are \begriff{close} of
	$x(p) \sim y(p)$ for almost $p \in N$.
	\item	On $\Cont^1(N,M)$, define the $\Leb^r$ metric $\rho_{0,r}$ and
	the $\Sob^{1,r}$ ``distance measure'' $\rho_{1,r}$ by
	\begin{align*}
		\rho_{0,r}(x,y) & := \Big(\Int_N \dist^r(x(p),y(p)) \d p\Big)^{\nicefrac 1 r}, \\
		\rho_{1,r}(x,y) & := \Big(\Int_N \left\{\begin{aligned} & \norm{d_p x - P d_p y}^r && \text{if $x(p) \sim y(p)$} \\
		                        & \infty        && \text{else}
		                        \end{aligned} \right\} \d p\Big)^{\nicefrac 1 r},
	\end{align*}
	with the usual modification for $r = \infty$.
	We abbreviate $\rho_0 := \rho_{0,2}$ and
	$\rho_1 := \rho_{1,2}$.
	Let $\Sob^1(N,M)$ be the completion
	of $\Cont^1(N,M)$ with respect to $\rho_0 + \rho_1$.
	\end{subenum}
\end{definition}
\begin{lemma}
	\label{prop:holonomyEstimate}
	Let $\gamma$ be a closed curve in a convex region of $M$, and let $P$
	be the parallel transport along $\gamma$. If $C_0 L^2(\gamma) < \pi^2$,
	then
	$\norm{P_\gamma - \id} \leq \frac 1 2 C_0 \Len^2(\gamma)$.
\end{lemma}
\begin{proof}
	The parallel transport is continuous with respect to
	$\Leb^\infty$ convergence in the space of loops
	$\interv 0 1 \to M$, so it suffices to show the
	claim for smooth $\gamma$.
	To fix notation, let us say $\gamma: \interv 0 1 \to M$,
	$\gamma(0) = \gamma(1) = p$.
	As this curve lies entirely in a convex region, it
	can be represented as $\gamma(t) = \exp_p V(t)$
	with a vector field $V: \interv 0 1 \to T_pM$. Define a
	homotopy $c(s,t) := \exp_p sV(t)$ between $\gamma$
	and the lazy loop. Denote the $s$-parameter lines
	by $c_t$ and the $t$-parameter lines by $\gamma_s$. By
	\ref{prop:estimateOfParallelTransportDeriv}, we have
	\[
		P_\gamma - \id = \Int_0^1 \Int_0^1 P_s^{1,t}
			R(\dot c_t, \dot \gamma_s) P_s^{t,0} \d t \d s
	\]
	The coordinate vectors $\dot c_t = \partial_s c$ and $\dot \gamma_s
	= \partial_t c$
	can be explicitely computed: $\partial_s c(s,t)
	= P_t^{s,0} V(t)$
	because $c_t$ is a geodesic with initial
	velocity $V(t)$, and $\partial_t c(s,t) = J_t(s)$
	for a Jacobi field $J_t$ along $c_t$ with
	values $J_t(0) = 0$, $J_t(1) = \dot \gamma(t)$
	and $\dot J_t(0) = \dot V(t)$ by
	\ref{prop:derivativeOfExpAlongCurve} (any two of these
	conditions determine $J_t$ uniquely). So we have $\absval{\dot c_t}
	= \absval V$ and
	$\absval{\dot \gamma_s} \leq \absval{\dot \gamma}$
	because the Jacobi field grows monotonously in $s$,
	see the proof of \ref{thm:estimatesOnDtJ}.
	Because the parallel transports along $\gamma_s$ are
	isometries, we obtain
	\[
		\norm{P_\gamma - \id} \leq \int \int C_0 \absval{\dot c_t}\, \absval{\dot \gamma_s}
			\leq C_0 L(\gamma) \max \absval V.
	\]
	And $\absval{V(t)}$ is the distance from $p$
	to $\gamma(t)$, which cannot be larger than
	$\frac 1 2 \Len(\gamma)$,
\end{proof}
\begin{proposition}[\textbf{``triangle inequality''}]
	\label{prop:triangInequalRho}
	\sloppypar
	If $x,y,z \in \Cont^1(N,M)$ with $\rho_{0,\infty}(x,y)
	+ \rho_{0,\infty}(y,z) \leq \ell$,
	then
	\[
		\rho_1(x,z) \leq \rho_1(x,y) + \rho_1(y,z) + \smallfrac 1 2 C_0 \ell^2 \Dir^\half(z)
	\]
\end{proposition}
\begin{proof}
	Pointwise, we have
	\[
		  \norm{dx - P^{x,z}dz} \leq
		  \norm{dx - P^{x,y}dy}
		+ \norm{dy - P^{y,z}dz}
		+ \norm{P^{x,z} - P^{x,y}P^{y,z}}\, \norm{dz}.
	\]
	The
	difference between the parallel transports is
	the holonomy of the loop $x \leadsto y \leadsto z \leadsto x$,
	which is smaller than $\frac 1 2 C_0 \ell^2$ by
	\ref{prop:holonomyEstimate}. Now the claim is
	a simple application of
	Minkowski's inequality
	in $\Leb^2(M,\R)$,
\end{proof}
\begin{remark}
	\begin{subenum}
	\item	That $\rho_{0,r}$ is indeed a metric is proven with the same argument
	as the usual Minkowsky inequality, see e.\,g.
	\citet[lemma 1.18]{Alt06}. In contrast,
	\ref{prop:triangInequalRho} gives only a distorted
	triangle inequality for $\rho_1$.
	Nevertheless, $\Sob^1(N,M)$ can be defined as
	the completion of $\Cont^1(N,M)$ with respect
	to $\rho_0 + \rho_1$, because this term does not
	disturb the usual completion construction
	for metric spaces, see e.\,g. \citet[no. 0.20]{Alt06},
	and we never need to use the triangle inequality
	explicitely.
	\item	Because of
	\[
		\absval{\Dir(x) - \Dir(y)} \simleq \rho_1^2(x,y)
	\]
	(the hidden constant comes from the comparison
	of $\betrag{dx - Pdy}$ with $\norm{dx - Pdy}$),
	the definition above ensures that every $u \in \Sob^1(N,M)$
	indeed has finite Dirichlet energy. Nevertheless,
	not all functions with finite Dirichlet energy are contained
	in our definition of $\Sob^1(N,M)$, but only those
	that are limits of smooth function sequences.
	The usual counterexample is
	the function $u \mapsto \frac u{\absval u}$ from the
	unit ball to the unit sphere minimises $\Dir$ in dimension
	$m \geq 3$ and larger (\citealt[sec. 6]{Hildebrandt77}; general
	regularity theory is given in \citealt{Schoen82}).
	So the usual definition of $\SobW^{1,2}(N,M)$ as
	\[
		\{y \in \SobW^{1,2}(N,\R^k) \mit y(p) \in M \text{ a. e.}\},
	\]
	where $M$ is embedded in $\R^k$, neccessarily has the drawback
	that $\Cont^1(N,M)$ is not dense in $\SobW^{1,2}(N,M)$ in
	dimension $3$ and larger (\citealt[sec. 4]{Schoen83}).
	In allusion to \cite{Meyers64}, \citet[p. 266]{Jost88}
	states this as $\Sob^1(N,M) \neq \SobW^{1,2}(N,M)$.
	By our use of $\Sob^1(N,M)$,
	we restrict ourselves to functions that \emph{can}
	be smoothly approximated. This space is well-suited for
	approximation questions, but the wrong one to show existence
	of solutions. For an overview over difficulties and pitfalls
	of the harmonic mapping problem, we refer to the survey of
	\cite{Jost88}.
	\item	Consequently, two functions $x,y \in \Cont^1(N,M)$ are close iff
	$\rho_1(x,y)$ is finite.
	\item	\label{rem:Leb2Width}
	If $x,y \in \Cont^1(N,M)$ are close, the
	geodesics $x(p) \leadsto y(p)$
	give rise to a \begriff{geodesic homotopy} $h:x \leadsto y$,
	i.\,e. a smooth mapping $N \times \interv 01 \to M$,
	$(p,s) \mapsto h_s(p)$,
	such that $h_0 = x$ and $h_1 = y$
	and $s \mapsto h_s(p)$ is a geodesic for any $p$.
	It minimises the energy
	\vspace{-1ex}
	\[
		E(h) := \Int_N \Int_0^1 \absval{{\textstyle\dds} h_s(p)}^2 \d s \d p
	\]
	\enlargethispage{0.5ex}
	over all homotopies in the same class
	(\citealt[lemma 8.5.1]{Jost11}). In fact, $E(h) = \rho_0^2(x,y)$
	if $h$ is the geodesic homotopy $x \leadsto y$ and
	$x,y$ are close, because $\absval{\dds h_s(p)}$ is independent of $s$
	in this case. This is also called the
	\begriff{$\Leb^2$-width}
	of the geodesic homotopy (\citealt{Kokarev13}, the older
	literature mostly uses the $\Leb^\infty$-width $\rho_{0,\infty}$
	from \citealt{Siegel84}).
	\end{subenum}
\end{remark}
\begin{proposition}[\begriff{Poincaré inequality}]
	\label{prop:poincareIneqConstructive}
	Suppose
	$\Rand N$ is smooth, all
	Weingarten maps of $\Rand N$
	with respect to $N$ are
	bounded by $\norm{W_\nu} \leq \kappa$
	everywhere, and no point in $N$ has
	distance larger than $r$ to $\Rand N$. Then
	\[
		\ibetrag[\Leb^2(N)] f^2 \leq 2 r C_N \ibetrag[\Leb^2(\Rand N)] f^2
			+ 4 r^2 \ibetrag[\Leb^2(N)]{df}^2
		\qquad \text{with } C_N := \e^{r\max(\kappa,\sqrt{C_0})}.
	\]
\end{proposition}
\begin{proof}
	Without regarding the constants, it would be
	very easy to reduce this case to the Poincaré inequality
	for vanishing boundary values \ref{prop:poincareIneqVanishingBdryValues}.
	As a very personal attitude, we would like to
	circumvent the contradiction argument there.
	Let us first consider a positive $\Cont^1$ function $g: N \to \R$.
	
	As \citet[prop. 3.5]{Mantegazza03} have shown, the distance
	field $\dist := \dist(\argdot, \Rand N)$ is $\Cont^1$
	except on an $(n-1)$-dimensional set $S$
	(in fact, they deal with the distance field
	of an arbitrary submanifold $K$ for boundaryless $N$,
	but the case of $K = \Rand N$ is also possible).
	By the coarea formula (\citealt[thm. 3.4.2]{Evans92}),
	$\int g$ can be computed by integration over the
	$t$-level sets $N^t := \{p \in N\setminus S \mit$ $\dist = t\}$,
	where points in $S$ can be omitted because it is a null set:
	\[
		\Int_N g = \Int_0^r \Big(\Int_{N^t} g\Big) \d t
	\]
	(note that $\absval{\grad \dist} = 1$, so there is no additional
	weighting factor).
	There are smooth homotopies $h^t$ retracting each level
	set $N^t$ to the boundary,
	defined on a subset $N^t_0$ of $\Rand N$,
	with $h^t_0 = \id$ and $h^t_t(N^t_0) = N^t$,
	following the gradient field of $\dist$. The intermediate
	mappings $h^t_s$ cover sets $N^t_s \subset N^s$,
	and the $N^t$ integral can be computed by the
	fundamental theorem of calculus for $a(s) = \int_{N^t_s} g$
	as $a(t) = a(0) + \int_0^s \dot a(s) \d s$. The
	derivative of the integrals is composed of the
	integrand's change along $s$-lines and the changing
	of the volume element:
	\[
		\dds \Int_{N^t_s} g
			= \Int_{N^t_s} dg(\dot h^t_s) + \Int_{N^t_s} g \spur W_s,
	\]
	where $\dot h^t_s$ denotes the $s$-derivative of the homotopy and
	$W_s$ is the Weingarten operator of the distance set $N^s$
	from \eqnref{eqn:RiccatiEqnForWeingartenMap}.
	Here we have used that $\tau(s) := \spur W_s$ is
	the derivative of the volume element
	(\citealt[eqn. 1.5.4]{Karcher89} or, in a more general
	setup, \citealt[eqn. 9.4.17]{Delfour11}).
	Now by \eqnref{eqn:RiccatiEqnForWeingartenMap},
	the function $\tau(s)$ obeys $\dot \tau \leq C_0 - \tau^2$
	with initial value $\tau(0) \leq \kappa$ by assumption.
	This differential inequality delivers us
	$\tau \leq K := \max(\kappa, \sqrt{C_0})$. (Note that
	not the absolute value of $\tau$ can be bounded,
	only $\tau$ itself---in fact $\tau \to -\infty$
	where $dh^t_s$ becomes singular.) So we have
	\[
		\dds \Int_{N^t_s} g \leq \Int_{N^t_s} \absval{dg}
			+ K \Int_{N^t_s} g
	\]
	or $\dot a \leq b + K a$ with $b$
	being the integral over $\absval{dg}$. This
	differential inequality has the supersolution
	$a(0) \e^{Kt} + \int_0^t b(s) \d s$, which hence is a bound
	for $a(t)$. That means
	\[
		\Int_N g \leq \Int_0^r a(t) \d t
		\leq r\e^{Kr} \Int_{\Rand N} g + r \Int_N \absval{dg}.
	\]
	Now for $f \in \Sob^1$, let $g = f^2$. The
	latter term becomes $\absval{d(f^2)} = 2 f \absval{df}$, and
	its integral is estimated by $2 \ibetrag f \ibetrag{df}$
	by Hölder. Then apply Young's inequality
	$uv \leq \delta u^2 + \frac 1 {4\delta} v^2$ with
	$\delta = \frac 1 {4r}$ to obtain
	\[
		r \ibetrag[\Leb^2]{d(f^2)} \leq \smallfrac 1 2 \ibetrag[\Leb^2] f^2 + 2 r^2 \ibetrag[\Leb^2]{df}^2,
	\]
\end{proof}
\begin{corollary}
	\label{prop:poincareIneqForDistance}
	Situation as before.
	Suppose $x,y \in \Sob^1(N,M)$ are close maps with
	$\dist(x,y)(p)$ $\leq \eps$ for all boundary points $p \in \Rand N$. Defining
	$C_N' := C_N \sqrt r$, it holds
	$\rho_0(x,y) \simleq C_N' \eps + r \rho_1(x,y)$.
\end{corollary}
\begin{proof}
	Consider the function $f := \dist(x,y): N \to \R$.
	It has differential $df(V) = g\sprod{Y_y}{(dx - Pdy)V}$
	by \ref{prop:derivativesOfDist}, and hence
	$\absval{df} \leq \norm{dx - Pdy}$,
\end{proof}
\begin{remark_nn}
	\begin{subenum}
	\item
	The Poincaré inequality in the form above also
	holds for differential forms, with the covariant
	derivative on the right-hand side. In fact, consider
	$u \in \Sob^1\Omega^k$ and $f := \absval u$.
	Then, because $\nabla$ is metric,
	$\absval{df} = \sprod{\nabla u} u / \absval u
	\leq \absval{\nabla u}$.
	\bibrembegin
	\item
	By the same method of proof, the Poincaré
	inequality of \citet[thm. 0.4]{Kappeler03}
	can be significantly shortened. They prove that
	if $N,M$ are closed and compact and $M$ has
	negative sectional curvature, then any two
	homotopic mappings $x,y \in \Cont^1(N,M)$ satisfy
	$\rho_0(x,y) \simleq 1 + \Dir(x)^\half + \Dir(y)^\half$.
	\bibremend
	\end{subenum}
\end{remark_nn}
\begin{situation}
	\label{sit:galerkinMethodInMfs}
	For simplicity, we assume $N = r\complex$
	(otherwise, concatenate the results below with
	\ref{prop:estimateDirAndDire}). Suppose the metric $\gamma$ of $N$
	is piecewise $(\frac 1 2,h)$-small, so that we can
	omit the fullness parameter. If $y: N \to M$
	is a smooth function, we assume that its piecewise
	barycentric interpolation $x$ is close to $y$,
	which is the case for small enough $C_0 h^2$.
\end{situation}
\begin{proposition}[\begriff{Galerkin orthogonality}]
	\label{prop:galerkinEstimateForDirPbOnMfs}
	Situation as in \ref{sit:galerkinMethodInMfs}.
	Let $y \in \Sob^1(N,M)$ be a critical point of $\Dir$
	with respect to compactly supported variations,
	and let $x \in \Pol^1(N,M)$ be a critical point of $\Dir$
	with respect to variations $W \in \Pol^1\VF_x$
	as in \ref{def:vertexTangentSpaces} that
	vanish at boundary vertices, such that $x(p_i) = y(p_i)$ on
	all boundary vertices. Then if $x$ and $y$ are close,
	\[
		\dprod{dx - Pdy}{\nabla W} = 0
		\qquad
		\forall W \in \Pol^1\VF_x, W|_{\Rand N} = 0.
	\]
\end{proposition}
\begin{proof}
	Because $x$ and $y$ are close, the parallel transport
	induces a bundle isomorphism $x^*TM \to y^*TM$.
	Because piecewise smooth vector fields are in $\Sob^1$
	and the variation on the whole boundary vanishes
	if it vanishes on the vertices (recall that the
	barycentric mapping on a subsimplex only depends
	on the vertices of this subsimplex), we obtain that
	$PW$ is an admissible variation field along $y$
	for all $W \in \Pol^1\VF_x$. Therefore $\dprod{Pdy}{\nabla W} = 0$
	by \ref{prop:DirichletPrincipleOnMfs}, and similarly
	$\dprod{dx}{\nabla W} = 0$,
\end{proof}
\begin{corollary}
	\label{prop:galerkinSolutionInMfsIsBestApprox}
	Situation as before. Then $\ibetrag[\Leb^2]{dx - Pdy}
	\leq \inf_{W \in \Pol^1\VF_x} \ibetrag[\Leb^2]{dx - Pdy - \nabla W}$,
	because $\ibetrag{dx - Pdy}^2 = \dprod{dx-Pdy - \nabla W}{dx-Pdy}
	\leq \ibetrag{dx-Pdy - \nabla W} \ibetrag{dx-Pdy}$.
\end{corollary}

\subsection{Approximation Properties of Karcher Triangulation Variations}

\begin{lemma}
	Situation as in \ref{sit:galerkinMethodInMfs}.
	Let $V$ be an $\Sob^1$
	vector field along $x$. Then for any $i \in \complex^0$,
	there is a variation $p_i(t)$ of $x(ri)$ such that
	the vector field $\dot X_i$ from
	\ref{prop:derivativexwrtoVertexVariation} satisfies
	$\ibetrag[\Leb^2(s_i)]{V - \dot X_i}
	\simleq h(1 + C_{0,1}h) \ibetrag[\Leb^2(s_i)]{\nabla V} + C_{0,1}h^2 \ibetrag[\Leb^2(s_i)] V$,
	where $s_i$ denotes the star of $ri$, i.\,e. the union
	of all simplices $r\simplexs$ with $i \in \simplexs$.
\end{lemma}
\begin{proof}
	Abbreviate $p := x(ri)$ and write $X$ instead of $X_i$ for the
	time of the proof.
	%
	In the first step,
	let us consider a smooth vector fields $V$.
	Choose the variation of $x(ri)$ such that $\dot p(0) = V|_p$.
	Then by \ref{prop:derivativexwrtoVertexVariation}, $\dot X = V$
	at $p$ and hence
	\[
		(V - \dot X)|_q = \Int_\gamma P\nabla_{\dot \gamma}(V - \dot X)
			= \Int_\gamma P\nabla_{\dot \gamma} V
			- \Int_\gamma P\nabla_{\dot \gamma} \dot X,
	\]
	where $\gamma: p \leadsto q$. The second integral should disappear
	in the result. By \ref{prop:estimatesOnDrJs}, we have
	$\absval{\nabla_{\dot \gamma} \dot X} \simleq C_{0,1} h \absval{\dot \gamma} \,
	\absval{\dot p}$. So we end up with $\absval{\dot p}$, which is a
	point evaluation of $V$ and hence undesired. Express
	$\dot p = V|_p = PV|_{\gamma(t)} - \int P\nabla_{\dot \gamma} V$,
	where the integral only runs from $0$ to $t$. Then
	\[
		\absval[g|_q]{V - \dot X} \leq \int \norm{\nabla V} + C_{0,1} h
			\int \big(\absval V + \int \norm{\nabla V} \big)
			\simleq (1 + C_{0,1}h)\Int_\gamma \norm{\nabla V} + C_{0,1} h\Int_\gamma \absval V,
	\]
	Squaring both sides and applying Hölder's inequality
	as in \eqnref{eqn:interpolEstimatRealValued2} gives
	\[
		\Int_{s_i} \absval{V - \dot X}^2 \simleq h (1 + C_{0,1} h)\Int_{s_i} \norm{\nabla V}^2
			+ C_{0,1} h^2 \Int_{s_i} \absval V^2.
	\]
	So for a smooth vector field $V$, we have constructed an interpolation.
	The best approximation in $\Leb^2$ must of course also fulfill
	this inequality. And by continuity of the $\Leb^2$-orthogonal projection,
	this holds for every vector field of class $\Sob^1$,
\end{proof}
\begin{proposition}
	\label{prop:interpolByTriangVariations}
	Situation as in \ref{sit:galerkinMethodInMfs}.
	Let $Q$ be an $\Sob^1$ section of $T^*N \otimes x^*TM$. Then
	there is some $W \in \Pol^1\VF_x$ with $\ibetrag[\Leb^2]{Q - \nabla W}
	\simleq h \ibetrag[\Leb^2]{\nabla Q} + C_{0,1}h^2 \ibetrag Q$.
	The hidden constant depends on $n$, $m$ and $\absval N$.
\end{proposition}
\begin{proof}
	\begin{subeqns}
	It suffices to show the claim for the $\Leb^2(N,M)$ operator
	norm in the left-hand side instead of $\Leb^2$ norm:
	\begin{equation}
		\label{eqn:interpolByTriangVariations}
		\iopnorm[\Leb^2]{Q - \nabla W}
		\overset !\simleq h \ibetrag[\Leb^2]{\nabla Q} + C_{0,1}h^2 \ibetrag Q.
	\end{equation}
	In fact, let $v$ be the unit vector field on $N$
	realising $\norm Q$ everywhere. Then $v \in \Leb^2(x^*TM)$
	and hence $\ibetrag Q^2 = \int \betrag Q^2
	\simleq \int \norm Q^2 = \int \betrag{Qv}^2
	\simleq \iopnorm Q^2 \ibetrag v^2 \simleq
	\iopnorm Q^2 \absval N^2$. So let us prove
	\eqnref{eqn:interpolByTriangVariations}.

	In any simplex, $Q$ can be applied to vectors
	$ri - rj$ and their linear combinations. Choose
	a norm-preserving $\Sob^1$ extension of $Q$ such that
	it can also be applied to vectors $ri$. Then define,
	on each star $s_i$, a vector field $V_i := Q ri$.
	Then $Qv = v^iV_i$ for any $v \in TN|_{s_i}$. Now let
	$\dot X_i$ be the $\Leb^2$ best approximation to $-V_i$
	on $s_i$. Then
	\[
		\begin{split}
		\Big(\int \absval{Qv + v^i \dot X_i}^2\Big)^\half
			& = \Big(\int (v^i)^2 \absval{V_i + \dot X_i}^2\Big)^\half \\
			& \leq \ibetrag{v^i} \ibetrag{V_i + \dot X_i} \\
			& \simleq \ibetrag{v^i} \big( h (1 + C_{0,1} h) \ibetrag{\nabla V_i}
				+ C_{0,1} h^2 \ibetrag{V_i}\big) \\
			& \simleq \ibetrag v \big(h(1 + C_{0,1} h) \inorm{\nabla Q}
				+ C_{0,1} h^2 \inorm Q \big).
		\end{split}
	\]
	Now recall that $\absval{A_\lambda \nabla_{dx\,v}\dot x
	+ v^i \dot X_i} \simleq
	C_{0,1}h \absval{dx\,v}\,\absval{\dot x}$ from
	\ref{prop:derivativexwrtoVertexVariation}
	in combination with \ref{corol:estimateForNablaX} and
	\ref{prop:estimatesOnDrJs}, and $\norm{A_\lambda - \id}
	\simleq C_0 h^2$ from \ref{prop:EstimateFordx}. This gives
	$\ibetrag{\nabla_{dx\,v}\dot x + v^i\dot X_i}
	\simleq C_0 h^2 \ibetrag{\nabla_{dx\,v}\dot x} + C_0 h^2 \ibetrag{\dot X_i} \ibetrag{v^i}$.
	So we have for $W := \dot x \in \Pol^1\VF_x$
	\[
		\begin{split}
		\ibetrag{\nabla_{dx\,v}W - Qv}
			& \leq \ibetrag{\nabla_{dx\,v}W + v^i \dot X_i} + \ibetrag{v^i \dot X_i + Qv} \\
			& \simleq C_{0,1} h^2 \ibetrag v \inorm{\nabla W}
				+ h (1 + C_{0,1} h) \inorm{\nabla Q} + C_{0,1} h^2 \inorm Q.
		\end{split}
	\]
	Because $\nabla W$ is almost an $\Leb^2$ best approximation of $Q$,
	we have $\inorm{\nabla W} \simleq \inorm Q$, which completes the proof,
	\end{subeqns}
\end{proof}
\begin{theorem}
	\label{prop:errorEstimateForDirPbOnMfs}
	Situation as in \ref{prop:galerkinEstimateForDirPbOnMfs}. Then
	\[
		\rho_0(x,y) + \rho_1(x,y)
		\simleq \rho_{0,\Rand N}(x,y) + h \ibetrag{\nabla dy} + C_{0,1}' h,
	\]
	where the hidden constant depends on $m$ and the geometry of $N$.
\end{theorem}
\begin{proof}
	Applying \ref{prop:interpolByTriangVariations} to
	$Q = dx - Pdy$, there is $W \in \Pol^1\VF_x$ with
	\[
		\ibetrag{dx - P dy - \nabla W}
		\simleq h \ibetrag{\nabla dx - \nabla P dy}
		+ C_{0,1}h^2 \ibetrag{dx - Pdy}.
	\]
	Because $C_{0,1} h^2$ is assumed to be small, say $\leq \frac 1 2$,
	we can neglect the latter term. Due to \ref{prop:EstimateForNabladx},
	$\betrag{dx} \simleq C_{0,1}' h \absval N$, and $\nabla P dy - P \nabla dy$
	can be shown to be bounded with an argument like in
	\ref{prop:estimateDistancexAndy} (spelled out in detail, this amounts
	to a rought $\Leb^\infty$ estimate for $\dist(x,y)$, which is provable by
	a suitable modification of the standard first-order $\Leb^\infty$
	estimate as in \citealt[p. 89]{Braess07}). Then the
	claim is proven by \ref{prop:galerkinSolutionInMfsIsBestApprox}
	and \ref{prop:poincareIneqForDistance},
\end{proof}

\cleardoublepage 
%
%
               \chapter{Outlook}
%
%

There are several research directions that would naturally
continue the course of this dissertation, but which could not be
further investigated due to time constraints:

\begin{subenum123}
\item	Whereas the weak formulation and approximation of
extrinsic curvature is obviously bound to the embedding
of a submanifold, the weak form of Ricci curvature
as in \cite{Fritz12} could be formulated intrinsically.
\item	The measure-valued
approximation of Lipschitz--Killing curvatures
of submanifolds in Euclidean space
from \cite{CohenSteiner06} could possibly be carried
over to situations where the surrounding space
itself has curvature.
\item	The level set approach (\citealt{Osher88}, \citealt{Osher03a})
that was used to approximate \textsc{pde}'s on surfaces
(\citealt{Dziuk08}), surface flows (\citealt{Deckelnick01})
or, as combination of both, \textsc{pde}'s on
evolving surfaces (\citealt[sec. 8]{Dziuk13}) can
directly be carried over to submanifolds.
\item	Assumption \eqnref{eqn:assumptionDECinterpolation}
has to be verified, perhaps under additional
conditions. The testing with $\Pol\inv$ forms in
\ref{prop:DirichletPbInDEC}--\ref{prop:mixedFormDirichletPbInDEC:woChapter}
should be sharpened or at least re-interpreted with
the use of more classical test functions.
\item	Our definition of the barycentric mapping $x$
is implicit and needs to know gradients of the squared
distance function $\dist^2$. The exact computation of
geodesic distances is very expensive, and the task
to find fast and accurate approximations is
a current research problem, cf. \cite{Crane13},
\cite{Campen13} and references therein.
The use of any of these $\dist^2$ approximations
to compute the barycentric mapping would lead to
a computationally feasible approximation of $x$.
\item	After this, or restricted to 3-manifolds where geodesic distances can
be exactly computed (or sufficiently well approximated),
the minimal surface algorithms from
\cite{Brakke92}, \cite{Pinkall93}, \cite{Renka95}, and
\cite{Dziuk99} can be applied, for example in hyperbolic
three-space $\mathbbm H^3$, the product $\mathbbm H^2 \times \R$
of hyperbolic 2-space and the real line, or products with
twisted metrics.
\item	Variational methods in shape space,
as have been dealt by \cite{Rumpf11}, can be extended
e.\,g. to the computation of minimal submanifolds
(whose dimension can be freely chosen) or
multi-dimensional regression.
\end{subenum123}

\cleardoublepage
\bibliography{_Dissertation}

\pagestyle{empty}
\cleardoublepage
\subsection*{Zusammenfassung}
\label{offizielleZsfsg}

\begin{subenum123}
\item
Sei $(M,g)$ eine unberandete, kompakte Riemannsche
Mannigfaltigkeit und $\stds$ das $n$-dimensionale Standardsimplex.
Für $n+1$ gegebene Punkte $p_i \in M$ betrachten wir mit
\cite{Karcher77} die Funktion
\[
	E: M \times \stds \to \R, \quad
	(a,\lambda) \mapsto \lambda^0 \dist^2(a,p_0) + \dots +
	                    \lambda^n \dist^2(a,p_n),
\]
worin $\dist$ der geodätische Abstand in $M$ sei. Liegen alle
$p_i$ in einem hinreichend kleinen geodätischen Ball, so ist
$x: \lambda \mapsto \argmin_a E(a,\lambda)$
eine wohldefinierte Funktion $\stds \to M$ (\ref{prop:xIsGlobalMinimiser}).
Wir nennen $s := x(\stds)$ das Karcher-Simplex mit
Ecken $p_i$. Auf $\stds$ sei eine
flache Riemannsche Metrik $g^e$ durch Vorgabe von
Seitenlängen $\dist(p_i,p_j)$ definiert. Wenn alle
Seitenlängen kleiner als $h$ sind und
$\vol(\stds,g^e) \geq \alpha h^n$ für ein $\alpha > 0$ ist,
so zeigen wir in \ref{prop:comparisongandge} und
\ref{prop:estimateOfChristoffelOperator}
\begin{equation}
	\label{eqn:metricEstimateAppendix}
	\tag{\textsc a.1a}
	\absval{(x^*g - g^e)\sprod vw} \leq c h^2 \absval v \absval w,
	\qquad
	\absval{(\nabla^{x^*g} - \nabla^{g^e})_v w}
	\leq c h \absval v \absval w
\end{equation}
mit einer nur vom Krümmungstensor $R$ von $(M,g)$
und $\theta$ abhängigen Konstanten $c$.
Daraus folgen mit wenig Aufwand Interpolationsabschätzungen
für Funktionen $u: s \to \R$ (\ref{prop:InterpolationEstimateForRealvaluedFunctions})
und $y: s \to N$ für eine zweite Riemannsche Mannigfaltigkeit $N$
(\ref{prop:interpolNtoMEstimateSingleSimplex}).
Auch erlaubt diese Simplexdefinition, auf Grundlage der
Voronoi-Zerlegung von \citename{Leibon} und \citename{Letscher}
(2000) eine Karcher-Delaunay-Triangulierung zu definieren
(\ref{def:globalMappingx}).

Daher können wir im folgenden ganz $(M,g)$ als trianguliert
annehmen. Auf jedem Simplex ist $g$ durch eine Metrik
$g^e$ mit \eqnref{eqn:metricEstimateAppendix}
approximiert, und schwach differenzierbare
Funktion $u \in \Sob^1(M,g)$ lassen sich durch stückweise
polynomielle $u_h \in \Pol^1(M)$ approximieren.
In der üblichen Weise (\citealt{Dziuk88}, \citename{Holst}
und \citename{Stern} 2012) lassen sich daher Variationsprobleme
wie das Poissonproblem (\ref{prop:estimateDirichletProblem},
\ref{prop:feDiffFormsDirichletProblem},
\ref{prop:errorEstimateForDirPbOnMfs}) oder die Hodge-Zerlegung
(\ref{prop:feApproxOfHodgeDecomp})
in $\Sob^1(M,g)$ mit denjenigen in $\Sob^1(M,g^e)$
und ihren Ga\-ler\-kin-Approximationen
in $\Pol^1(M)$ vergleichen.

Anknüpfend an die gängige Finite-Elemente-Theorie für
Probleme auf Untermannigfaltigkeiten des $\R^m$ lassen
sich auch Untermannigfaltigkeiten $S \subset M$ durch
Karcher\hyp Simplexe approximieren. Der dabei auftretende
Geometriefehler ist gleich dem für Untermannigfaltigkeiten
des $\R^m$ zuzüglich eines Terms $ch^2$
(\ref{prop:comparisonygAndge}).

\item
Sei $M$ die geometrische Realisierung eines simplizialen
Komplexes $\complex$. Die simpliziale Kohomologie
$(C^k(\complex), \Rand^*)$ ist von \citename{Desbrun}
und \citename{Hirani} (\citeyear{Hirani03},
\citeyear{Desbrun05}) als dis\-kre\-tes äußeres Kalkül
(\textsc{dec}) interpretiert worden. Wir definieren
Räume $\Pol\inv\Omega^k \subset \Leb^\infty\Omega^k$
und äußere Differentiale und geben eine isometrische
Kokettenabbildung $C^k \to \Pol\inv\Omega^k$
an (\ref{prop:interpolationOfDec}).
Damit ist die Berechnung von Variationsproblemen
im diskreten äußeren Kalkül auf Variationsprobleme in einem
Raum von nicht-konformen Ansatz-Differentialformen zurückgeführt.
Wir untersuchen die Approximationseigenschaften von $\Pol\inv\Omega^k$
in $\Sob^1\Omega^k$ (\ref{prop:Hminus1estimateForPolinvOnlyOneStage},
\ref{prop:Hminus1estimateForPolinv}) und vergleichen die Lösungen
von Variationsproblemen in ihnen (\ref{prop:DirichletPbInDEC}--%
\ref{prop:mixedFormDirichletPbInDEC:woChapter}).
\end{subenum123}

\end{document}